\documentclass[11pt]{article}
\usepackage{amsmath, amsfonts, amssymb, amscd, amsthm, bbm}
\usepackage{graphicx}
\usepackage{caption, subcaption}

\usepackage[all]{xy}

\newtheorem{thm}{Theorem}
\newtheorem{prop}[thm]{Proposition}
\newtheorem{lem}[thm]{Lemma}

\newtheorem{cor}[thm]{Corollary}
\newtheorem{rem}[thm]{Remark}
\newtheorem{df}[thm]{Definition}
\newtheorem{ex}[thm]{Example}

\renewcommand{\epsilon}{\varepsilon}
\renewcommand{\phi}{\varphi}

\renewcommand{\deg}{\operatorname{deg}}
\renewcommand{\P}{\operatorname{P}}

\newcommand{\BB}{\mathbb}
\newcommand{\g}{\mathfrak}

\newcommand{\im}{\operatorname{Im}}
\newcommand{\re}{\operatorname{Re}}
\newcommand{\tr}{\operatorname{Tr}}

\newcommand{\HC}{\BB H_{\BB C}}

\newcommand{\degt}{\widetilde{\operatorname{deg}}}
\newcommand{\degtt}{(\operatorname{deg}+2)}

\newcommand{\M}{\operatorname{Mx}}

\newcommand{\sgn}{\operatorname{sign}}
\newcommand{\DR}{\operatorname{Diag}}

\usepackage[OT2,T1]{fontenc}
\newcommand\textcyr[1]{{\fontencoding{OT2}\fontfamily{wncyr}\selectfont #1}}
\newcommand{\Zh}{\textit{\textcyr{Zh}}}
\newcommand{\Sh}{\textit{\textcyr{Sh}}}

\begin{document}

\title{\bf Quaternionic Analysis, Representation Theory and Physics II}
\author{Igor Frenkel and Matvei Libine}
\maketitle

\begin{abstract}
We develop further quaternionic analysis introducing left and right doubly
regular functions. We derive Cauchy-Fueter type formulas for these doubly
regular functions that can be regarded as another counterpart of Cauchy's
integral formula for the second order pole, in addition to the one studied
in the first paper with the same title.
We also realize the doubly regular functions as a subspace of the
quaternionic-valued functions satisfying a Euclidean version of Maxwell's
equations for the electromagnetic field.

Then we return to the study of the original quaternionic analogue of Cauchy's
second order pole formula and its relation to the polarization of vacuum.
We find the decomposition of the space of quaternionic-valued functions
into irreducible components that include the spaces of doubly left and right
regular functions.
Using this decomposition, we show that a regularization of the vacuum
polarization diagram is achieved by subtracting the component corresponding
to the one-dimensional subrepresentation of the conformal group.
After the regularization, the vacuum polarization diagram is identified with
a certain second order differential operator which yields a quaternionic
version of Maxwell equations.

Next, we introduce two types of quaternionic algebras consisting of spaces
of scalar-valued and quaternionic-valued functions.
We emphasize that these algebra structures are invariant under the action
of the conformal Lie algebra.
This is done using techniques that appear in the study of the vacuum
polarization diagram.
These algebras are not associative, but we can define an infinite family of
$n$-multiplications, and we conjecture that they have the structures of
weak cyclic $A_{\infty}$-algebras.
We also conjecture the relation between the multiplication operations
of the scalar and non-scalar quaternionic algebras with the $n$-photon
Feynman diagrams in the scalar and ordinary conformal QED.

We conclude the article with a discussion of relations between quaternionic
analysis, representation theory of the conformal group, massless quantum
electrodynamics and perspectives of further development of these subjects.
\end{abstract}

\section{Introduction}

The starting point in the development of quaternionic analysis by Fueter and
others was an exact analogue of Cauchy's integral formula for complex
holomorphic functions
\begin{equation}  \label{Cauchy}
f(w) = \frac 1{2\pi i} \oint \frac {f(z)}{z-w}\,dz
\end{equation}
involving the first order pole.
This counterpart of (\ref{Cauchy}) is usually referred to as the Cauchy-Fueter
formulas for the quaternionic analogues of holomorphic functions known as
left and right regular functions. Thus there are two versions:
\begin{align}
f(W) &= \frac1{2\pi^2} \int_S k(Z-W) \cdot Dz \cdot f(Z),  \label{lr-intro}  \\
g(W) &= \frac1{2\pi^2} \int_S g(Z) \cdot Dz \cdot k(Z-W),  \label{rr-intro}
\end{align}
where
\begin{equation}  \label{CF-kernel}
k(Z-W) = \frac{(Z-W)^{-1}}{\det(Z-W)},
\end{equation}
$Dz$ is a certain quaternionic valued $3$-from, the contour of integration
$S$ is homotopic to a $3$-sphere $S^3$ around $W$ in $\BB H$,
$f: \BB H \to \BB S$ is left regular, and $g: \BB H \to \BB S'$ is right
regular. Here, $\BB S$ and $\BB S'$ are two dimensional left and right modules
over $\BB H$. (Usually people consider $\BB H$-valued functions.)
The quaternionic conformal group $SL(2,\BB H)$ and its Lie algebra
$\mathfrak{sl}(2,\BB H)$ have natural actions on the spaces of left and right
regular functions, which are analogous to the actions of the (global) conformal
group $SL(2,\BB C)$ and its Lie algebra $\mathfrak{sl}(2,\BB C)$ on the space
of holomorphic functions.
In spite of this indisputable parallel between complex and quaternionic
analysis, further attempts to extend the analogy between the two theories
have been met with substantial difficulties. In particular, neither left nor
right regular functions form a ring and, therefore, cannot be regarded as a
full counterpart of the ring of holomorphic functions.
Additionally, generalizations of the Cauchy-Fueter formulas
(\ref{lr-intro})-(\ref{rr-intro}) to higher order poles are not at all
straightforward.

In our first paper with the same title \cite{FL1} we proposed to approach
quaternionic analysis from the point of view of representation theory of
the conformal group $SL(2,\BB H)$ and its Lie algebra $\mathfrak{sl}(2,\BB H)$.
In particular, we explored the parallel between quaternionic and complex
analysis from this representation theoretic point of view.
This approach allowed us to discover a quaternionic counterpart
of Cauchy's integral formula for the second order pole
\begin{equation}  \label{2pole-intro}
\frac{df}{dw}(w) = \frac 1{2\pi i} \oint \frac {f(z)\,dz}{(z-w)^2}
\end{equation}
by interpreting the square of the Cauchy-Fueter kernel (\ref{CF-kernel}) as
a kernel of an intertwining operator for $\mathfrak{sl}(2,\BB H)$.
As explained in Introduction of \cite{FL1}, the derivative operator
$\frac{d}{dz}$ can be interpreted as an intertwining operator between
certain representations of $SL(2,\BB C)$.
We show in \cite{FL1} that the quaternionic counterpart of (\ref{2pole-intro})
dictated by representation theory of the quaternionic conformal group has the
form
\begin{equation}  \label{Mx-intro}
(\M F)(W) = \frac{12i}{\pi^3} \int_{Z \in U(2)}
k(Z-W) \cdot F(Z) \cdot k(Z-W) \,dV,
\end{equation}
where the operator $\frac{d}{dz}$ in Cauchy's formula (\ref{2pole-intro})
is replaced by a certain second order differential operator that we call
``Maxwell operator''
$$
\M F = \overrightarrow{\nabla} F \overleftarrow{\nabla} - \square F^+.
$$
The operator $\M$ is an intertwining operator between certain actions of
$\mathfrak{sl}(2,\BB H)$ on the space of quaternionic valued polynomials
(or its analytic completion).

In complex analysis, Cauchy's formulas for the first and second order poles
(\ref{Cauchy}), (\ref{2pole-intro}) admit immediate generalizations to
holomorphic functions on the punctured complex plane
$\BB C^{\times} = \BB C \setminus \{0\}$ by choosing the contour of integration
to be the difference of loops around zero and infinity.
In quaternionic analysis, there is a similar generalization of the
Cauchy-Fueter formulas (\ref{lr-intro})-(\ref{rr-intro}) to regular functions
on $\BB H^{\times} = \BB H \setminus \{0\}$
by choosing the contour of integration to be the difference of two $3$-cycles
around zero and infinity (as well as more general domains).
However, a generalization of the quaternionic analogue of the second order pole
formula (\ref{2pole-intro}) to functions on $\BB H^{\times}$ presents substantial
difficulties and is directly related to the divergence of the Feynman diagram
for vacuum polarization, as was indicated in \cite{FL1}.
In the present paper we resolve this problem similarly to our derivation of
the quaternionic second order pole formula for the scalar valued functions in
\cite{FL3}.
Recall that $\BB D^+$ and $\BB D^-$ are certain open domains in
$\BB H \otimes \BB C$, both having $U(2)$ as Shilov boundary.
The idea is to separate the singularities in the second order pole
by considering maps $J^{\varepsilon_1 \varepsilon_2}$, where
$\varepsilon_1, \varepsilon_2 = \pm$, from a space of quaternionic valued
functions ${\cal W}'$ to (a completion of) the tensor product of left and
right regular functions
$$
(J^{\varepsilon_1 \varepsilon_2} F)(Z_1,Z_2) = \frac {12i}{\pi^3} \int_{W \in U(2)}
k(W-Z_1) \cdot F(W) \cdot k(W-Z_2) \,dV,
$$
$$
\text{with} \quad Z_1 \in \BB D^{\varepsilon_1}, \quad Z_2 \in \BB D^{\varepsilon_2},
$$
and then taking the limits as $Z_1$, $Z_2$ approach the common boundary $U(2)$
of $\BB D^{\pm}$.
While the maps $J^{++}$ and $J^{--}$ are well defined on the diagonal $Z_1=Z_2$
(as in \cite{FL1}), the maps $J^{+-}$ and $J^{-+}$ have singularities
which cancel each other in the sum (as in \cite{FL3}).
Setting $Z_1=Z_2$ in the total map
$$
J=J^{++}+J^{--}-J^{+-}-J^{-+}
$$
yields an extension of our second order pole formula to functions on
$\BB H^{\times}$
\begin{equation}  \label{J=Mx}
(JF)(Z,Z) = (\M F)(Z).
\end{equation}

The appearance of singularity in $J^{+-}$, $J^{-+}$ is related to the presence
of the one-dimensional representation in the subquotient of ${\cal W}'$.
For this reason, the generalization of our second order pole formula from
\cite{FL1} to the ``quaternionic Laurent polynomials'' -- i.e. polynomial
functions defined on $\BB H^{\times}$ -- requires a detailed study of the
$\mathfrak{sl}(2,\BB H)$-module structures of the spaces ${\cal W}'$ and
${\cal W}$ of quaternionic valued functions (now defined on $\BB H^{\times}$)
and the homomorphism $\M: {\cal W}' \to {\cal W}$.
It turns out that each of these modules contains $13$ composition factors,
and they can be studied using a larger complex originally considered
as equation (57) in \cite{FL1}:
\begin{equation}  \label{q-complex-intro}
\xymatrix@R=1pc{
& & & \hat{\cal F} \ar[dr] & & & \\
\BB C \ar[r] & \Sh' \ar[r] & {\cal W}' \ar[ur] \ar[rr]^{\M} \ar[dr] & &
{\cal W} \ar[r] & \Sh  \ar[r] & \BB C, \\
& & & \hat{\cal G} \ar[ur] & & &}
\end{equation}
where $\Sh'$ is the space of scalar valued functions on $\BB H^{\times}$
and $\Sh$ is its dual space.
We show that both $\Sh$ and $\Sh'$ have six irreducible components,
five of which reappear in ${\cal W}$ and ${\cal W}'$, and the trivial
one-dimensional subrepresentation $\BB C$ of $\Sh'$ is annihilated by
$\nabla^+$.

The spaces $\hat{\cal F}$ and $\hat{\cal G}$ are
$$
\hat{\cal F} = \{ f: \BB H^{\times} \to \BB S \otimes \BB S \}
\quad \text{and} \quad
\hat{\cal G} = \{ g: \BB H^{\times} \to \BB S' \otimes \BB S' \}.
$$
They contain subspaces ${\cal F}$ and ${\cal G}$ of ``doubly
left and right regular functions'' respectively that share
many similarities with the usual left and right regular functions.
They are preserved by the actions of the conformal group and can be defined as
kernels of certain linear differential operators. Besides, they also satisfy
a quaternionic analogue of Cauchy's integral formula for the second order
pole (\ref{2pole-intro}) as follows:
\begin{align}
\nabla (W \cdot f)(W) &= \frac1{\pi^2} \int_S k_1(Z-W) \cdot Dz \cdot Z
\cdot f(Z),  \label{dlr-intro}  \\
(g(W) \cdot W) \overleftarrow{\nabla} &= \frac1{\pi^2} \int_S g(Z) \cdot Z
\cdot Dz \cdot k_1(Z-W),  \label{drr-intro}
\end{align}
where
$$
k_1(Z-W) = \frac12 \nabla k(Z-W),
$$
$f: \BB H \to \BB S \otimes \BB S$ is a doubly left regular function
and $g: \BB H \to \BB S' \otimes \BB S'$ is a doubly right regular function.
Like the Cauchy-Fueter formulas (\ref{lr-intro})-(\ref{rr-intro}),
formulas (\ref{dlr-intro})-(\ref{drr-intro}) extend to functions
on $\BB H^{\times}$ and more general domains.

We emphasize that in quaternionic analysis there are two analogues
of Cauchy's integral formula for the second order pole (\ref{2pole-intro}).
On the one hand, the kernel  $(z-w)^{-2}$ in (\ref{2pole-intro}) can be viewed
as a square of $(z-w)^{-1}$ with the quaternionic counterpart having form
(\ref{Mx-intro}). On the other hand, the kernel  $(z-w)^{-2}$ can also be
regarded as a derivative of $(z-w)^{-1}$, in which case the quaternionic
counterpart is given by (\ref{dlr-intro})-(\ref{drr-intro}).
Although the two quaternionic analogues (\ref{Mx-intro}) and
(\ref{dlr-intro})-(\ref{drr-intro}) of Cauchy's integral formula for the
second order pole appear to be quite different, they turn out to be
complementary to each other.
In fact, one can identify the doubly left and right regular functions with
self-dual and anti-self-dual solutions of the Maxwell equations
(in Euclidean signature).
Moreover, (\ref{dlr-intro})-(\ref{drr-intro}) can be combined into a single
formula valid for any quaternionic functions $A(Z)$ annihilated by the
Maxwell operator:
\begin{equation}  \label{Max-intro}
\M A(Z)=0, \qquad A: \BB H^{\times} \to \BB H.
\end{equation}
Comparing (\ref{Mx-intro}) and (\ref{Max-intro}) clearly demonstrates the
complementary nature of the two quaternionic analogues (\ref{Mx-intro}) and
(\ref{dlr-intro})-(\ref{drr-intro}) of the second order pole formula
and that the Maxwell equations play a key role in both versions!

Our study of the quaternionic complex and the decomposition of the
representations involved in (\ref{q-complex-intro}) uses extensively the
basis of $K$-types of $(\mathfrak g, K)$-modules with
$\mathfrak g = \mathfrak{gl}(2,\BB H) \otimes \BB C
\simeq \mathfrak{gl}(4,\BB C)$ and $K=U(2) \times U(2)$.
Therefore, the algebra generated by the matrix coefficients of
$GL(2,\BB C)$ in finite dimensional representations
($t^l_{n\,\underline{m}}(Z)$'s and $N(Z)^k$'s) plays a key role in our approach.
This algebra can be viewed as a counterpart of the algebra of
Laurent polynomials in complex analysis, which is the algebra of
the matrix coefficients of $GL(1,\BB C)$.
The algebra of matrix coefficients of $GL(2,\BB C)$ technically is more
complicated than the familiar algebra of Laurent polynomials, but conceptually
various results and formulas in many aspects are similar.
The bases of $K$-types that we are using have certain advantages over the
equivalent picture in the Minkowski space, where simpler continuous bases are
natural (see e.g. \cite{FL3}, Section 8, for the relation between the two types
of bases).
In particular, the $K$-bases allow us to conveniently isolate the
one-dimensional
irreducible component in representations ${\cal W}$ and ${\cal W}'$,
which plays crucial role in regularization of the vacuum polarization.
This component is less transparent in the Minkowski picture and is hidden
in the traditional physics approach to this regularization.
The $K$-type approach is extensively used in analysis of various
representations of real semisimple groups, see e.g. \cite{Le} and
especially Section 8 dedicated to the $K$-types of representations of $SU(2,2)$.

We saw in our first paper \cite{FL1} that in order to introduce unitary
structures on the spaces of harmonic as well as (left and right) regular
functions, one must replace the quaternionic conformal group $SL(2,\BB H)$
with $SU(2,2)$, which is another real form of $SL(4,\BB C)$.
The group $SU(2,2)$ in turn can be identified with the conformal group of
the Minkowski space $\BB M$.
Similarly, the unitarity of the spaces of doubly (left and right) regular
functions and, equivalently, the space of solutions of the Maxwell equations
(\ref{Max-intro}) modulo the image of $\Sh'$ require the same Minkowski
space $\BB M$.

Moreover, the quaternionic complex (\ref{q-complex-intro}) in Minkowski
space realization can be identified with the complex of differential forms
on the Minkowski space with the zero light cone removed.
The program of study of vector bundles on the covering space of the
compactified Minkowski space and representations of the conformal group
in the spaces of sections was suggested by I.~Segal as a mathematical approach
to studying of four-dimensional field theory.
In particular, the irreducible components of the complex of differential
forms were identified by his student S.~Paneitz \cite{P}.
Thus quaternionic analysis and representation theory of the conformal group
associated to Minkowski space are deeply intertwined and mutually beneficial.
Another example of this link is provided by the realization of irreducible
representations of the most degenerate series (depending on $n \in \BB Z$)
as solutions of certain differential operators on $\BB M$ \cite{JV1}.
For $n = \pm 1$, these spaces are exactly the spaces of left/right regular
functions.
And for $n = \pm 2$, they can be identified with the doubly left/right
regular functions.
One can also define $n$-regular functions for any $n \in \BB N$, then the
Cauchy-Fueter type integral formulas for these functions can be interpreted
as a quaternionic analogue of Cauchy's integral formula for the $n$-th order
pole, where the Cauchy kernel is treated as the $(n-1)$-st derivative of
$(z-w)^{-1}$.

The quaternionic second order pole formulas described above show that the
analogy with the complex case is not straightforward.
So, it is not surprising that the quaternionic counterpart of the algebra of
complex holomorphic functions is far from obvious.
In this paper we suggest a certain candidate for a quaternionic algebra,
again, based on representation theory of the conformal group.
We already noted that tensor products of representations of the
most degenerate series of $SU(2,2)$ (and its Lie algebra)
depending on $n \in \BB Z$ do not contain representations of the same class.
For $n = \pm 1$, these representations are exactly the spaces of left/right
regular functions. And for $n = \pm 2$, these representations can be identified
with the doubly left/right regular functions.
Therefore, one cannot expect a group-invariant algebra structure on these
function spaces.
Thus we have to consider the class of representations that comes next after
the most degenerate series of $SU(2,2)$ -- the middle series.
Such representations appear in the complex of quaternionic spaces
(\ref{q-complex-intro}), and one can consider other similar representations,
for example, the space of $\BB C$-valued functions $\Zh$ studied in \cite{FL3}.
Clearly, the space $\Sh'$ provides a trivial example of a quaternionic algebra
of $\BB C$-valued functions with pointwise multiplication.
On the other hand, the best candidate for a quaternionic algebra of
quaternionic-valued functions appears to be a closely related space ${\cal W}'$.
In order to understand the quaternionic algebra structure, we start with the
algebra of scalar-valued quaternionic functions $\Zh$.
It is similar to, but in certain ways simpler than ${\cal W}'$.
In both cases we first embed our spaces into larger algebras
$$
I: \Zh \to \operatorname{completion \: of} {\cal H} \otimes {\cal H}, \qquad
J: {\cal W}' \to \operatorname{completion \: of} {\cal V} \otimes {\cal V}',
$$
where $I$ is described in \cite{FL3} and $J$ appears in our study of
vacuum polarization in this paper.
Note that ${\cal H}$ is the space of harmonic functions on $\BB H^{\times}$
and ${\cal V}$, ${\cal V}'$ are the spaces of left and right regular
functions on $\BB H^{\times}$. Both ${\cal H}$ and ${\cal V}$, ${\cal V}'$
are representations of the most degenerate series corresponding to $n=0$
and $n=\pm 1$ respectively.
The multiplication in the larger algebra is defined using the invariant
pairings on ${\cal H}$ and between ${\cal V}$ and ${\cal V}'$.
Thus, to finish our construction of quaternionic multiplication operation,
we need to find appropriate inverses of the maps $I$ and $J$.
This can be done in two ways by considering certain subtle limits and,
therefore, one can define a one-parameter family of invariant multiplications
by taking linear combinations of those limits.
The case of ${\cal W}'$ differs from $\Zh$ in that $\ker J$ is non-trivial
and, therefore, we actually define multiplication on ${\cal W}'/\ker J$.
An additional subtlety of the ${\cal W}'$ case is that in the construction
of the inverse of $J$ we lose the one-dimensional irreducible component of
the representation ${\cal W}'/\ker J$; this is the component containing
the identity element. Thus we can only define a one-parameter family of
multiplication-like operations
$$
({\cal W}'/\ker J) \otimes ({\cal W}'/\ker J) \to
({\cal W}'/\ker J)/\BB C \simeq {\cal W}'/\ker\M
$$
that are invariant under the action of the conformal group.
These maps can be lifted to genuine multiplication operations
$$
({\cal W}'/\ker J) \otimes ({\cal W}'/\ker J) \to {\cal W}'/\ker J
$$
similarly to the procedure of adjoining of unit to algebras without units.
It is an interesting problem to find a way to define the multiplication on
${\cal W}'/\ker J$ directly without the procedure of adjoining the unit.
Both quaternionic algebras $\Zh$ and ${\cal W}'/\ker J$ defined in this paper
are not associative. However, our construction allows immediate generalizations
to invariant $n$-multiplications (multiplications of $n$ factors) with $n=1$
corresponding to the identity operator and $n=2$ to the multiplication
described above. We conjecture that the resulting $n$-multiplications satisfy
the quadratic relations of weak cyclic $A_{\infty}$-algebras.
Thus the quaternionic analogue of the algebra of complex holomorphic functions
might have a much richer structure than its classical counterpart!
Another interesting question is how to characterize the $n$-multiplications as
intertwining operators for the conformal group.

We see repeatedly in \cite{FL1} and in this paper that the Minkowski space
reformulation of various structures of quaternionic analysis leads to
profound relations with different structures of four-dimensional conformal
field theory, particularly with the conformal QED.
One can ask about the physical meaning of the two quaternionic algebras
$\Zh$ and ${\cal W}'/\ker J$, including their $n$-multiplications and
possible relations between them.
Our second conjecture is that they are related to the $n$-photon diagrams
in the scalar and non-scalar conformal QED (Figure \ref{n-photon}).
We discuss this conjecture in more detail in Subsection \ref{last_subsection}
at the end of the paper.
For now, we only mention the relation between our two conjectures.
Namely, it was discovered relatively recently (more than fifty years
after creation of modern QED) that $n$-photon diagrams (as well as
$n$-gluon diagrams) satisfy certain quadratic relations, known in
physics literature as the BCFW relations
(see \cite{BCFW}, \cite{BBBV} and references therein).

These relations provide strong evidence in support of and a link between our
two conjectures.
In particular, if the second conjecture is correct, then the BCFW relations
yield the associativity-type relations for the quaternionic algebras
giving them the structures of weak cyclic $A_{\infty}$-algebras,
thus making the first conjecture valid as well.
This way studying quaternionic algebras and relating structures will produce
further connections between quaternionic analysis and four-dimensional quantum
field theory that we predicted in the first paper with the same
title \cite{FL1}. We expect that these connections will be beneficial for
both disciplines: quaternionic analysis will be enriched by many beautiful
structures and quantum field theory will find its purely mathematical
formulation.


\begin{figure}
\begin{center}
\begin{subfigure}[b]{0.25\textwidth}
\centering
\includegraphics[scale=0.4]{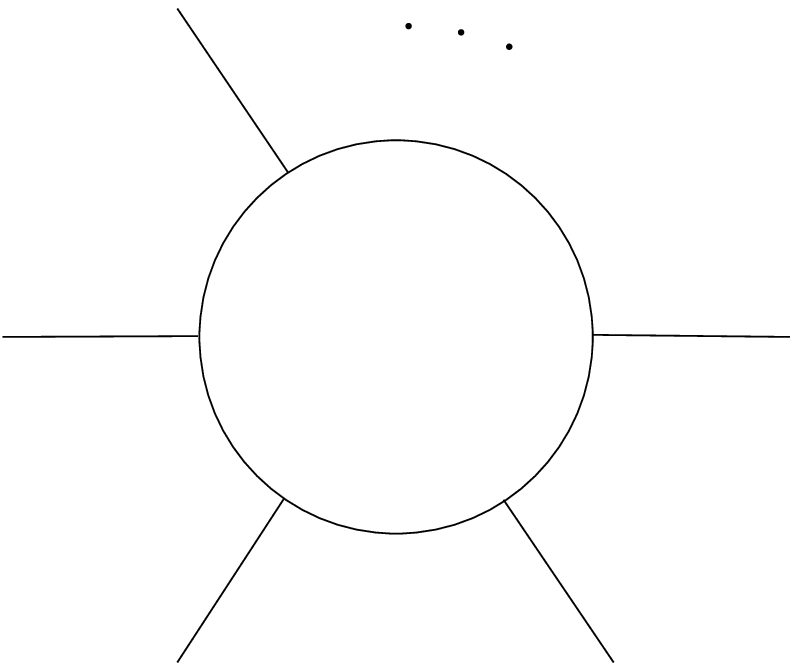}
\caption{Scalar case}
\end{subfigure}
\qquad
\begin{subfigure}[b]{0.4\textwidth}
\centering
\includegraphics[scale=0.4]{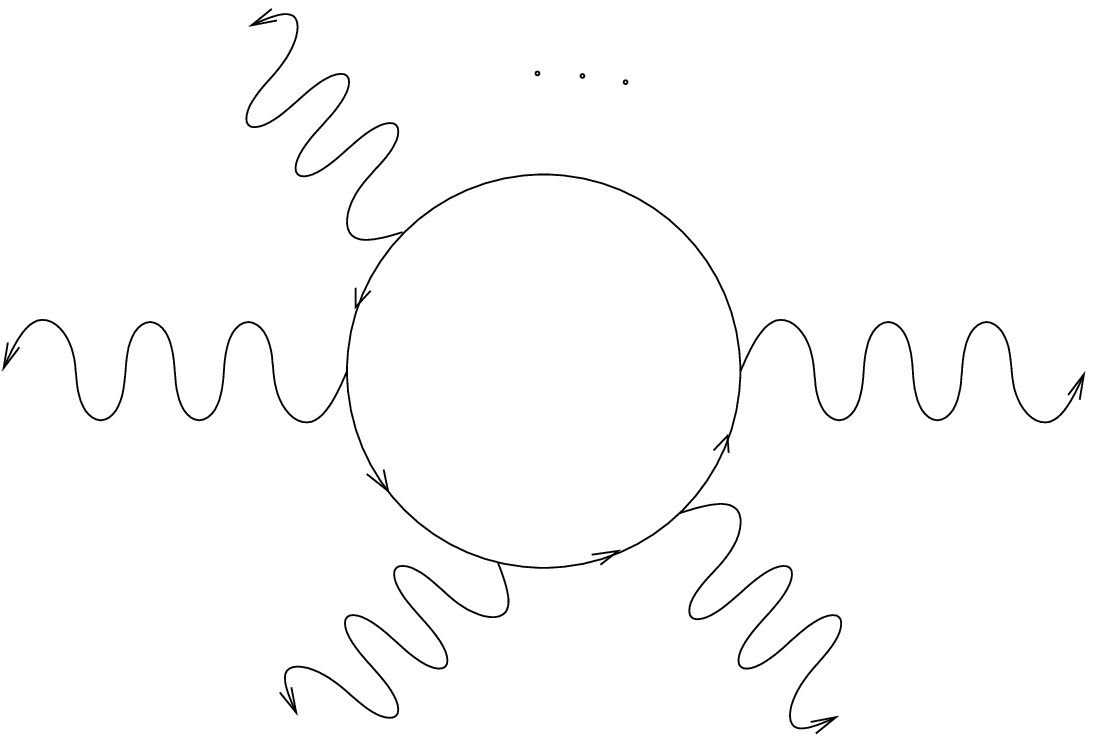}
\caption{Spinor case}
\end{subfigure}
\end{center}
\caption{$n$-photon Feynman diagrams.}
\label{n-photon}
\end{figure}

For technical reasons the paper is organized differently from the order of
this discussion. In Section \ref{DR-section} we define and study left and
right doubly regular functions. This is done in complete parallel with the
theory of regular functions reviewed in \cite{FL1}:
we prove analogues of the Cauchy-Fueter formulas,
construct action of the conformal group,
decompose the Cauchy-Fueter kernel for these functions
and study the invariant bilinear form.
In the last subsection we generalize the notion of doubly regular functions
to $n$-regular functions.
The Cauchy-Fueter type formulas for these functions are proved in a
separate paper \cite{FL4}.
In Section \ref{QCC-section} we describe the quaternionic chain complex
(\ref{q-complex-intro}), which plays a central role in this paper.
Then we decompose the representations $\Sh$ and $\Sh'$ into irreducible
components.
In Section \ref{DRinQCC-section} we proceed to the study of representations
${\cal W}$ and ${\cal W}'$. First, we analyze the kernel of $\M$ in
${\cal W}'$ which contains the image of $\Sh'$ and the irreducible components
isomorphic to doubly left and right regular functions.
We carefully identify the $K$-types and explicit forms of various intertwining
functors related to the quaternionic chain complex (\ref{q-complex-intro}).
Some of the intertwining functors are expressed by quaternionic analogues of
Cauchy's integral formula (\ref{2pole-intro}).
In Section \ref{decomp-section} we complete the decomposition of
representations ${\cal W}$ and ${\cal W}'$.
Besides the five irreducible components coming from $\Sh$ and $\Sh'$ and four
irreducible components of doubly regular functions, we identify four additional
irreducible components, including the trivial one-dimensional representation.
These four components are crucial to understanding of polarization of vacuum
and definition of quaternionic algebra that are subject of
Section \ref{VP-section}.
In Section \ref{VP-section} we generalize our result from \cite{FL1} and
extend the quaternionic analogue of Cauchy's integral formula for the
second order pole from ${\cal W}'^+$ to ${\cal W}'$.
Our main technical tool is a certain operator
$$
J: {\cal W}' \to \operatorname{completion \: of} {\cal V} \otimes {\cal V}'.
$$
The quotient space ${\cal W}'/\ker J$ has four irreducible components,
including the trivial one-dimensional subrepresentation.
In the subsequent section, ${\cal W}'/\ker J$ will be equipped with
an algebra structure.
In Section \ref{Alg-section} we first construct a scalar quaternionic algebra
using previous results from \cite{FL3}.
Then we proceed in a similar fashion to our main goal of constructing
the quaternionic algebra structure on the space ${\cal W}'/\ker J$.
The latter version is in many aspects similar to its scalar counterpart, but
has a richer structure.
In Subsection \ref{last_subsection} we discuss relations between quaternionic
algebras and Feynman diagrams of massless QED as well as future problems and
perspectives of our direction of quaternionic analysis.
In Section \ref{errata-section} we provide some comments about our earlier
papers \cite{FL1, FL3} that are relevant to the present article.

Since this paper is a continuation of \cite{FL1,FL3},
we follow the same notations and instead of introducing those
notations again we direct the reader to Section 2 of \cite{FL3}.

\section{Doubly Regular Functions}  \label{DR-section}

\subsection{Definitions}

We continue to use notations established in \cite{FL1}.
In particular, $e_0$, $e_1$, $e_2$, $e_3$ denote the units of the classical
quaternions $\BB H$ corresponding to the more familiar $1$, $i$, $j$, $k$
(we reserve the symbol $i$ for $\sqrt{-1} \in \BB C$).
Thus $\BB H$ is an algebra over $\BB R$ generated by $e_0$, $e_1$, $e_2$, $e_3$,
and the multiplicative structure is determined by the rules
$$
e_0 e_i = e_i e_0 = e_i, \qquad
(e_i)^2 = e_1e_2e_3 = - e_0, \qquad
e_ie_j=-e_ie_j, \qquad 1 \le i< j \le 3,
$$
and the fact that $\BB H$ is a division ring.
Next we consider the algebra of complexified quaternions
(also known as biquaternions) $\HC = \BB C \otimes_{\BB R} \BB H$ and
write elements of $\HC$ as
$$
Z = z^0e_0 + z^1e_1 + z^2e_2 + z^3e_3, \qquad z^0,z^1,z^2,z^3 \in \BB C,
$$
so that $Z \in \BB H$ if and only if $z^0,z^1,z^2,z^3 \in \BB R$:
$$
\BB H = \{ X = x^0e_0 + x^1e_1 + x^2e_2 + x^3e_3; \: x^0,x^1,x^2,x^3 \in \BB R \}.
$$
Recall that we denote by $\BB S$ (respectively $\BB S'$)
the irreducible 2-dimensional left (respectively right) $\HC$-module,
as described in Subsection 2.3 of \cite{FL1}.
The spaces $\BB S$ and $\BB S'$ can be realized as respectively
columns and rows of complex numbers.
Then
\begin{equation}  \label{SotimesS}
\BB S \otimes \BB S' \simeq \HC.
\end{equation}
Note that $\BB S \otimes \BB S$ and $\BB S' \otimes \BB S'$
are respectively left and right modules over $\HC \otimes \HC$.

We introduce four first order differential operators
\begin{align*}
\nabla^+ \otimes 1 &= (e_0 \otimes 1) \frac{\partial}{\partial x^0}
+ (e_1 \otimes 1) \frac{\partial}{\partial x^1}
+ (e_2 \otimes 1) \frac{\partial}{\partial x^2}
+ (e_3 \otimes 1) \frac{\partial}{\partial x^3}, \\
1 \otimes \nabla^+ &= (1 \otimes e_0) \frac{\partial}{\partial x^0}
+ (1 \otimes e_1) \frac{\partial}{\partial x^1}
+ (1 \otimes e_2) \frac{\partial}{\partial x^2}
+ (1 \otimes e_3) \frac{\partial}{\partial x^3}, \\
\nabla \otimes 1 &= (e_0 \otimes 1) \frac{\partial}{\partial x^0}
- (e_1 \otimes 1) \frac{\partial}{\partial x^1}
- (e_2 \otimes 1) \frac{\partial}{\partial x^2}
- (e_3 \otimes 1) \frac{\partial}{\partial x^3}, \\
1 \otimes \nabla &= (1 \otimes e_0) \frac{\partial}{\partial x^0}
- (1 \otimes e_1) \frac{\partial}{\partial x^1}
- (1 \otimes e_2) \frac{\partial}{\partial x^2}
- (1 \otimes e_3) \frac{\partial}{\partial x^3},
\end{align*}
which can be applied to functions with values in
$\BB S \otimes \BB S$ or $\BB S' \otimes \BB S'$ as follows.
If $U$ is an open subset of $\BB H$ or $\HC$ and
$f: U \to \BB S \otimes \BB S$ is a differentiable function,
then these operators can be applied to $f$ on the left. For example,
$$
(\nabla^+ \otimes 1) f = (e_0 \otimes 1) \frac{\partial f}{\partial x^0}
+ (e_1 \otimes 1) \frac{\partial f}{\partial x^1}
+ (e_2 \otimes 1) \frac{\partial f}{\partial x^2}
+ (e_3 \otimes 1) \frac{\partial f}{\partial x^3}.
$$
Similarly, these operators can be applied on the right to
differentiable functions $g: U \to \BB S' \otimes \BB S'$;
we often indicate this with an arrow above the operator. For example,
$$
g (\overleftarrow{\nabla^+ \otimes 1})
= \frac{\partial g}{\partial x^0} (e_0 \otimes 1)
+ \frac{\partial g}{\partial x^1} (e_1 \otimes 1)
+ \frac{\partial g}{\partial x^2} (e_2 \otimes 1)
+ \frac{\partial g}{\partial x^3} (e_3 \otimes 1).
$$

The tensor product $\BB S \otimes \BB S$ decomposes into a direct sum
of its symmetric part $\BB S \odot \BB S$ and antisymmetric part
$\BB S \wedge \BB S$:
$$
\BB S \otimes \BB S =
(\BB S \odot \BB S) \oplus (\BB S \wedge \BB S).
$$
Similarly, $\BB S' \otimes \BB S'$ decomposes into a direct sum
of its symmetric and antisymmetric parts:
$$
\BB S' \otimes \BB S' =
(\BB S' \odot \BB S') \oplus (\BB S' \wedge \BB S').
$$

\begin{df}  \label{dr-definition}
Let $U$ be an open subset of $\BB H$.
A ${\cal C}^1$-function $f: U \to \BB S \odot \BB S$ is
{\em doubly left regular} if it satisfies
$$
(\nabla^+ \otimes 1) f =0 \quad \text{and} \quad (1 \otimes \nabla^+) f =0
$$
for all points in $U$.
Similarly, a ${\cal C}^1$-function $g: U \to \BB S' \odot \BB S'$ is
{\em doubly right regular} if
$$
g (\overleftarrow{\nabla^+ \otimes 1}) =0 \quad \text{and} \quad
g (\overleftarrow{1 \otimes \nabla^+}) =0
$$
for all points in $U$.
\end{df}

Since
$$
(\nabla \otimes 1)(\nabla^+ \otimes 1)
= (\nabla^+ \otimes 1)(\nabla \otimes 1)
= (1 \otimes \nabla)(1 \otimes \nabla^+)
= (1 \otimes \nabla^+)(1 \otimes \nabla)
= \square,
$$
$$
\square = \frac{\partial^2}{(\partial x^0)^2}+
\frac{\partial^2}{(\partial x^1)^2} + \frac{\partial^2}{(\partial x^2)^2}+
\frac{\partial^2}{(\partial x^3)^2},
$$
doubly left and right regular functions are harmonic.

One way to construct doubly left regular functions is to start with a harmonic
function $\phi: \BB H \to \BB S \odot \BB S$, then
$(\nabla \otimes \nabla) \phi$ is doubly left regular.
Similarly, if $\phi: \BB H \to \BB S' \odot \BB S'$ is harmonic, then
$\phi (\overleftarrow{\nabla \otimes \nabla})$ is doubly right regular.

We also can talk about doubly regular functions defined on open subsets of
$\HC$. In this case we require such functions to be holomorphic.

\begin{df}
Let $U$ be an open subset of $\HC$.
A holomorphic function $f: U \to \BB S \odot \BB S$ is
{\em doubly left regular} if it satisfies $(\nabla^+ \otimes 1) f =0$
and $(1 \otimes \nabla^+) f =0$ for all points in $U$.

Similarly, a holomorphic function $g: U \to \BB S' \odot \BB S'$ is
{\em doubly right regular} if
$g (\overleftarrow{\nabla^+ \otimes 1}) =0$ and
$g (\overleftarrow{1 \otimes \nabla^+}) =0$ for all points in $U$.
\end{df}

Let ${\cal DR}$ and ${\cal DR}'$ denote respectively the spaces of (holomorphic)
doubly left and right regular functions on $\HC$, possibly with singularities.

\begin{thm}  \label{dr-action-thm}
\begin{enumerate}
\item
The space ${\cal DR}$ of doubly left regular functions
$\HC \to \BB S \odot \BB S$ (possibly with singularities)
is invariant under the following action of $GL(2,\HC)$:
\begin{multline}  \label{pi_dl}
\pi_{dl}(h): \: f(Z) \: \mapsto \: \bigl( \pi_{dl}(h)f \bigr)(Z) =
\frac{(cZ+d)^{-1} \otimes (cZ+d)^{-1}}{N(cZ+d)}
\cdot f\bigl( (aZ+b)(cZ+d)^{-1} \bigr),  \\
h^{-1} = \bigl( \begin{smallmatrix} a & b \\ c & d \end{smallmatrix} \bigr)
\in GL(2,\HC).
\end{multline}
\item
The space ${\cal DR}'$ of doubly right regular functions
$\HC \to \BB S' \odot \BB S'$ (possibly with singularities)
is invariant under the following action of $GL(2,\HC)$:
\begin{multline}  \label{pi_dr}
\pi_{dr}(h): \: g(Z) \: \mapsto \: \bigl( \pi_{dr}(h)g \bigr)(Z) =
g \bigl( (a'-Zc')^{-1}(-b'+Zd') \bigr)
\cdot \frac{(a'-Zc')^{-1} \otimes (a'-Zc')^{-1}}{N(a'-Zc')}, \\
h = \bigl( \begin{smallmatrix} a' & b' \\ c' & d' \end{smallmatrix} \bigr)
\in GL(2,\HC).
\end{multline}
\end{enumerate}
\end{thm}

\begin{proof}
It is easy to see that the formulas describing the actions
$\pi_{dl}$ and $\pi_{dr}$ also produce well-defined actions on the spaces
of all functions on $\HC$ (possibly with singularities) with values in
$\BB S \otimes \BB S$ and $\BB S' \otimes \BB S'$ respectively.
These actions preserve the subspaces of functions with values in
$\BB S \odot \BB S$ and $\BB S' \odot \BB S'$.
Differentiating $\pi_{dl}$ and $\pi_{dr}$, we obtain actions of the Lie algebra
$\mathfrak{gl}(2,\HC)$, which we still denote by $\pi_{dl}$ and $\pi_{dr}$
respectively.
Using notations
$$
\partial = \begin{pmatrix} \partial_{11} & \partial_{21} \\
\partial_{12} & \partial_{22} \end{pmatrix} = \frac 12 \nabla, \qquad
\partial^+ = \begin{pmatrix} \partial_{22} & -\partial_{21} \\
-\partial_{12} & \partial_{11} \end{pmatrix} = \frac 12 \nabla^+,
\qquad \partial_{ij} = \frac{\partial}{\partial z_{ij}},
$$
we can describe these actions of the Lie algebra.

\begin{lem}  \label{Lie-alg-action}
The Lie algebra action $\pi_{dl}$ of $\mathfrak{gl}(2,\HC)$ on ${\cal DR}$
is given by
\begin{align*}
\pi_{dl} \bigl( \begin{smallmatrix} A & 0 \\ 0 & 0 \end{smallmatrix} \bigr) &:
f(Z) \mapsto - \tr (AZ \partial) f,  \\
\pi_{dl} \bigl( \begin{smallmatrix} 0 & B \\ 0 & 0 \end{smallmatrix} \bigr) &:
f(Z) \mapsto - \tr (B \partial) f,  \\
\pi_{dl} \bigl( \begin{smallmatrix} 0 & 0 \\ C & 0 \end{smallmatrix} \bigr) &:
f(Z) \mapsto \tr (ZCZ \partial +CZ) f + (CZ \otimes 1 + 1 \otimes CZ) f,  \\
\pi_{dl} \bigl( \begin{smallmatrix} 0 & 0 \\ 0 & D \end{smallmatrix} \bigr) &:
f(Z) \mapsto \tr (ZD \partial +D) f + (D \otimes 1 + 1 \otimes D) f.
\end{align*}

Similarly, the Lie algebra action $\pi_{dr}$ of $\mathfrak{gl}(2,\HC)$ on
${\cal DR}'$ is given by
\begin{align*}
\pi_{dr} \bigl( \begin{smallmatrix} A & 0 \\ 0 & 0 \end{smallmatrix} \bigr) &:
g(Z) \mapsto - \tr (AZ \partial +A) g - g (A \otimes 1 + 1 \otimes A),  \\
\pi_{dr} \bigl( \begin{smallmatrix} 0 & B \\ 0 & 0 \end{smallmatrix} \bigr) &:
g(Z) \mapsto - \tr (B \partial) g,  \\
\pi_{dr} \bigl( \begin{smallmatrix} 0 & 0 \\ C & 0 \end{smallmatrix} \bigr) &:
g(Z) \mapsto \tr (ZCZ \partial +ZC) g + g(ZC \otimes 1 + 1 \otimes ZC),  \\
\pi_{dr} \bigl( \begin{smallmatrix} 0 & 0 \\ 0 & D \end{smallmatrix} \bigr) &:
g(Z) \mapsto \tr (ZD \partial) g.
\end{align*}
\end{lem}

\begin{proof}
These formulas are obtained by differentiating (\ref{pi_dl}) and (\ref{pi_dr}).
\end{proof}

We return to the proof of Theorem \ref{dr-action-thm}.
Since the Lie group $GL(2,\HC) \simeq GL(4,\BB C)$ is connected, it is
sufficient to show that, if $f \in {\cal DR}$, $g \in {\cal DR}'$ and
$\bigl( \begin{smallmatrix} A & B \\ C & D \end{smallmatrix} \bigr) \in
\mathfrak{gl}(2,\HC)$, then
$\pi_{dl} \bigl( \begin{smallmatrix} A & B \\ C & D \end{smallmatrix} \bigr)f
\in {\cal DR}$ and
$\pi_{dr} \bigl( \begin{smallmatrix} A & B \\ C & D \end{smallmatrix} \bigr)g
\in {\cal DR}'$.
Consider, for example, the case of 
$\pi_{dl} \bigl( \begin{smallmatrix} 0 & 0 \\ C & 0 \end{smallmatrix} \bigr)f$,
the other cases are similar. We have:
\begin{multline*}
(\nabla^+ \otimes 1) \pi_{dl}
\bigl( \begin{smallmatrix} 0 & 0 \\ C & 0 \end{smallmatrix} \bigr)f
= (\nabla^+ \otimes 1) \bigl( \tr(ZCZ \partial +CZ)f + (CZ \otimes 1)f \bigr)
+ (1 \otimes CZ) (\nabla^+ \otimes 1)f \\ + (1 \otimes C)
(e_0 \otimes e_0 + e_1 \otimes e_1 + e_2 \otimes e_2 + e_3 \otimes e_3)f,
\end{multline*}
the first summand is zero essentially because the space of left regular
functions is invariant under the action $\pi_l$ (equation (22) in \cite{FL1}),
the second summand is zero because $f$ satisfies $(\nabla^+ \otimes 1) f=0$,
and the third summand is zero by Lemma \ref{antisymmetric}.
\end{proof}

\begin{lem}  \label{antisymmetric}
Let $t \in \BB S \odot \BB S$ and $t' \in \BB S' \odot \BB S'$,
then
$$
(e_0 \otimes e_0 + e_1 \otimes e_1 + e_2 \otimes e_2 + e_3 \otimes e_3)t=0 \quad
\text{in $\BB S \otimes \BB S$}
$$
and
$$
t'(e_0 \otimes e_0 + e_1 \otimes e_1 + e_2 \otimes e_2 + e_3 \otimes e_3)=0 \quad
\text{in $\BB S' \otimes \BB S'$}.
$$
\end{lem}


\begin{proof}
Under the standard realization of $\BB H$ as a subalgebra of $\BB C^{2 \times 2}$,
we have:
$$
e_0 = \bigl(\begin{smallmatrix} 1 & 0 \\ 0 & 1 \end{smallmatrix}\bigr), \qquad
e_1 = \bigl(\begin{smallmatrix} 0 & -i \\ -i & 0 \end{smallmatrix}\bigr), \qquad
e_2 = \bigl(\begin{smallmatrix} 0 & -1 \\ 1 & 0 \end{smallmatrix}\bigr), \qquad
e_3 = \bigl(\begin{smallmatrix} -i & 0 \\ 0 & i \end{smallmatrix}\bigr).
$$
Then by direct computation using Kronecker product
(see also Subsection \ref{matrix-subsection}) we obtain
$$
e_0 \otimes e_0 + e_1 \otimes e_1 + e_2 \otimes e_2 + e_3 \otimes e_3
= \left( \begin{smallmatrix} 0 & 0 & 0 & 0 \\ 0 & 2 & -2 & 0 \\
0 & -2 & 2 & 0 \\ 0 & 0 & 0 & 0 \end{smallmatrix} \right).
$$
Since the elements of $\BB S \odot \BB S$ and $\BB S' \odot \BB S'$,
when realized as $4$-tuples, have equal second and third entries,
they are annihilated by the above matrix, and the result follows.
\end{proof}

\subsection{Cauchy-Fueter Formulas for Doubly Regular Functions}

In this section we derive Cauchy-Fueter type formulas for doubly regular
functions from the classical Cauchy-Fueter formulas for left and right
regular functions.

\begin{lem}  \label{zf-regular}
Let $f(Z)$ be a doubly left regular function, then the
$\BB S \otimes \BB S$-valued functions $(1 \otimes Z) f(Z)$ and
$(Z \otimes 1) f(Z)$ are ``left regular'' in the sense that they satisfy
$$
(\nabla^+ \otimes 1) \bigl[ (1 \otimes Z) f(Z) \bigr] = 0, \qquad
(1 \otimes \nabla^+) \bigl[ (Z \otimes 1) f(Z) \bigr] = 0.
$$

Similarly, if $g(Z)$ is a doubly right regular function, then the
$\BB S' \otimes \BB S'$-valued functions $g(Z)(1 \otimes Z)$
and $g(Z)(Z \otimes 1)$ are ``right regular'' in the sense that they satisfy
$$
\bigl[ g(Z) (1 \otimes Z) \bigr] (\overleftarrow{\nabla^+ \otimes 1}) = 0,
\qquad
\bigl[ g(Z) (Z \otimes 1) \bigr] (\overleftarrow{1 \otimes \nabla^+}) = 0.
$$
\end{lem}

\begin{proof}
We have:
$$
(\nabla^+ \otimes 1) \bigl[ (1 \otimes Z) f(Z) \bigr]
= (e_0 \otimes e_0 + e_1 \otimes e_1 + e_2 \otimes e_2 + e_3 \otimes e_3)f(Z)
+ (1 \otimes Z) \bigr[ (\nabla^+ \otimes 1) f(Z) \bigr],
$$
the first summand is zero by Lemma \ref{antisymmetric} and
the second summand is zero because $f$ satisfies $(\nabla^+ \otimes 1) f=0$.
Proofs of the other assertions are similar.
\end{proof}

Let $\degtt$ denote the degree operator plus two times the identity.
For example, if $f$ is a function on $\BB H$,
$$
\degtt f = x^0\frac{\partial f}{\partial x^0} +
x^1\frac{\partial f}{\partial x^1} + x^2\frac{\partial f}{\partial x^2}
+ x^3\frac{\partial f}{\partial x^3} +2f.
$$
Similarly, we can define operators $\deg$ and $\degtt$ acting on functions
on $\HC$. For convenience we recall Lemma 8 from \cite{FL2}
(it applies to both cases).

\begin{lem}
\begin{equation}  \label{deg-nabla}
2\degtt = Z^+\nabla^+ + \nabla Z = \nabla^+Z^+ + Z\nabla.
\end{equation}
\end{lem}

Define
\begin{equation}  \label{k_1-kernel}
k_1(Z-W) = \frac14 (\nabla \otimes \nabla) \biggl( \frac1{N(Z-W)} \biggr)
\end{equation}
(the derivatives can be taken with respect to either $Z$ or $W$ variable
-- the result is the same); this is a function of $Z$ and $W$ taking values in
$\HC \otimes \HC$, it is spelled out in equation (\ref{k_1-explicit}).
We also consider holomorphic $3$-forms $Dz \otimes Z$ and $Z \otimes Dz$
on $\HC$ with values in $\HC \otimes \HC$.
Then we obtain the following analogue of the Cauchy-Fueter formulas for
doubly regular functions.

\begin{thm}  \label{Fueter-doubly-reg}
Let $U \subset \BB H$ be an open bounded subset with piecewise ${\cal C}^1$
boundary $\partial U$. Suppose that $f(Z)$ is doubly left regular on a
neighborhood of the closure $\overline{U}$, then
\begin{multline*}
\frac1{2\pi^2} \int_{\partial U} k_1(Z-W) \cdot (Dz \otimes Z) \cdot f(Z) \\
= \frac1{2\pi^2} \int_{\partial U} k_1(Z-W) \cdot (Z \otimes Dz) \cdot f(Z) =
\begin{cases}
\degtt f(W) & \text{if $W \in U$;} \\
0 & \text{if $W \notin \overline{U}$.}
\end{cases}
\end{multline*}
If $g(Z)$ is doubly right regular on a neighborhood of the closure
$\overline{U}$, then
\begin{multline*}
\frac1{2\pi^2} \int_{\partial U} g(Z) \cdot (Dz \otimes Z)\cdot k_1(Z-W) \\
= \frac1{2\pi^2} \int_{\partial U} g(Z) \cdot (Z \otimes Dz)\cdot k_1(Z-W) =
\begin{cases}
\degtt g(W) & \text{if $W \in U$;} \\
0 & \text{if $W \notin \overline{U}$.}
\end{cases}
\end{multline*}
\end{thm}

\begin{proof}
By Lemma \ref{zf-regular}, the $\BB S \otimes \BB S$-valued function
$(1 \otimes Z) f(Z)$ satisfies
$(\nabla^+ \otimes 1) \bigl[ (1 \otimes Z) f(Z) \bigr] = 0$.
From the classical Cauchy-Fueter formula for left regular functions,
we obtain:
\begin{equation}  \label{old-Fueter}
\frac1{2\pi^2} \int_{\partial U} k_{1/2}(Z-W) \cdot (Dz \otimes 1) \cdot
(1 \otimes Z) f(Z) =
\begin{cases}
(1 \otimes W) f(W) & \text{if $W \in U$;} \\
0 & \text{if $W \notin \overline{U}$,}
\end{cases}
\end{equation}
where
$$
k_{1/2}(Z-W) = \frac{(Z-W)^{-1}}{N(Z-W)} \otimes 1
= - \frac12 (\nabla_Z \otimes 1) \biggl( \frac1{N(Z-W)} \biggr)
= \frac12 (\nabla_W \otimes 1) \biggl( \frac1{N(Z-W)} \biggr).
$$
Applying $1 \otimes \nabla$ to both sides of (\ref{old-Fueter})
(the derivative is taken with respect to $W$),
\begin{multline*}
\frac1{\pi^2} \int_{\partial U} k_1(Z-W) \cdot (Dz \otimes Z) \cdot f(Z) =
\begin{cases}
(1 \otimes \nabla)(1 \otimes W) f(W) & \text{if $W \in U$;} \\
0 & \text{if $W \notin \overline{U}$}
\end{cases} \\
= \begin{cases}
2\degtt f(W) & \text{if $W \in U$;} \\
0 & \text{if $W \notin \overline{U}$,}
\end{cases}
\end{multline*}
where the last equality follows from (\ref{deg-nabla}),
since $(1 \otimes \nabla^+) f =0$.
The other cases are similar.
\end{proof}

We have an analogue of Liouville's theorem for doubly regular functions:

\begin{cor}  \label{Liouville}
Let $f: \BB H \to \BB S \odot \BB S$ be a function that is doubly
left regular and bounded on $\BB H$, then $f$ is constant.
Similarly, if $g: \BB H \to \BB S' \odot \BB S'$ is a function that is
doubly right regular and bounded on $\BB H$, then $g$ is constant.
\end{cor}

\begin{proof}
The proof is essentially the same as for the (classical) left and right
regular functions on $\BB H$, so we only give a sketch of the first part.
From Theorem \ref{Fueter-doubly-reg} we have:
$$
\frac{\partial}{\partial x_{11}} \degtt f(X) =
\frac1{2\pi^2} \int_{S^3_R} \frac{\partial k_1(Z-X)}{\partial x_{11}} \cdot
(Dz \otimes Z) \cdot f(Z)
$$
where $S^3_R \subset \BB H$ is the three-dimensional sphere of radius $R$
centered at the origin
$$
S^3_R = \{ X \in \BB H ;\: N(X)=R^2 \}
$$
with $R^2>N(X)$.
If $f$ is bounded, one easily shows that the integral on the right hand side
tends to zero as $R \to \infty$. Thus
$\frac{\partial}{\partial x_{11}} \degtt f =0$. Similarly,
$$
\frac{\partial}{\partial x_{12}} \degtt f =
\frac{\partial}{\partial x_{21}} \degtt f =
\frac{\partial}{\partial x_{22}} \degtt f = 0.
$$
It follows that $\degtt f$ and hence $f$ are constant.
\end{proof}

\subsection{Expansion of the Cauchy-Fueter Kernel for Doubly Regular Functions}  \label{matrix-subsection}

We often identify $\HC$ with $2 \times 2$ matrices with complex entries.
Similarly, it will be convenient to identify $\HC \otimes \HC$
with $4 \times 4$ matrices with complex entries using the Kronecker product.
Let $\BB C^{n \times n}$ denote the algebra of $n \times n$ complex matrices.
If $A = \bigl( \begin{smallmatrix} a_{11} & a_{12} \\
a_{21} & a_{22} \end{smallmatrix} \bigr),
B = \bigl( \begin{smallmatrix} b_{11} & b_{12} \\
b_{21} & b_{22} \end{smallmatrix} \bigr) \in \BB C^{2 \times 2}$,
then their Kronecker product is
$$
A \otimes B = \begin{pmatrix} a_{11}B & a_{12}B \\ a_{21}B & a_{22}B \end{pmatrix}
= \begin{pmatrix} a_{11}b_{11} & a_{11}b_{12} & a_{12}b_{11} & a_{12}b_{12} \\
a_{11}b_{21} & a_{11}b_{22} & a_{12}b_{21} & a_{12}b_{22} \\
a_{21}b_{11} & a_{21}b_{12} & a_{22}b_{11} & a_{22}b_{12} \\
a_{21}b_{21} & a_{21}b_{22} & a_{22}b_{21} & a_{22}b_{22} \end{pmatrix}
\in \BB C^{4 \times 4}.
$$
Similarly, if we identify $\BB S$ and $\BB S'$ with columns and rows of two
complex numbers respectively, then
\begin{equation}  \label{StensorS}
\begin{pmatrix} z_1 \\ z_2 \end{pmatrix} \otimes
\begin{pmatrix} w_1 \\ w_2 \end{pmatrix}
= \left( \begin{smallmatrix} z_1w_1 \\ z_1w_2 \\ z_2w_1 \\ z_2w_2
\end{smallmatrix} \right), \qquad
(z_1, z_2) \otimes (w_1, w_2) = (z_1w_1, z_1w_2, z_2w_1, z_2w_2).
\end{equation}
A $4$-tuple belongs to $\BB S \odot \BB S$ or $\BB S' \odot \BB S'$
if and only if its second entry equals the third entry.
It is easy to see that Kronecker product satisfies
$$
(A \otimes B) (C \otimes D) = (AC) \otimes (BD).
$$

Recall that the Cauchy-Fueter kernel for doubly regular functions is defined by
(\ref{k_1-kernel}).
From its realization as a $4 \times 4$ matrix, we find:
\begin{multline}  \label{k_1-explicit}
k_1(Z-W) = (\partial_Z \otimes \partial_Z) \biggl( \frac1{N(Z-W)} \biggr)
= \begin{pmatrix} \partial_{11} \left(\begin{smallmatrix} \partial_{11} &
\partial_{21} \\ \partial_{12} & \partial_{22} \end{smallmatrix}\right) &
\partial_{21} \left(\begin{smallmatrix} \partial_{11} & \partial_{21} \\
\partial_{12} & \partial_{22} \end{smallmatrix}\right) \\
\partial_{12} \left(\begin{smallmatrix} \partial_{11} & \partial_{21} \\
\partial_{12} & \partial_{22} \end{smallmatrix}\right) &
\partial_{22} \left(\begin{smallmatrix} \partial_{11} & \partial_{21} \\
\partial_{12} & \partial_{22} \end{smallmatrix}\right) \end{pmatrix}
\biggl( \frac1{N(Z-W)} \biggr)  \\
= \frac2{N(Z-W)} (Z-W)^{-1} \otimes (Z-W)^{-1}  \\
- \frac{1/2}{N(Z-W)^2}
(e_0 \otimes e_0 + e_1 \otimes e_1 + e_2 \otimes e_2 + e_3 \otimes e_3).
\end{multline}

Next we recall the matrix coefficients $t^l_{n\,\underline{m}}(Z)$'s of $SU(2)$
described by equation (27) of \cite{FL1} (cf. \cite{V}):
\begin{equation}  \label{t}
t^l_{n\,\underline{m}}(Z) = \frac 1{2\pi i}
\oint (sz_{11}+z_{21})^{l-m} (sz_{12}+z_{22})^{l+m} s^{-l+n} \,\frac{ds}s,
\qquad
\begin{smallmatrix} l = 0, \frac12, 1, \frac32, \dots, \\ m,n \in \BB Z +l, \\
 -l \le m,n \le l, \end{smallmatrix}
\end{equation}
$Z=\bigl(\begin{smallmatrix} z_{11} & z_{12} \\
z_{21} & z_{22} \end{smallmatrix}\bigr) \in \HC$,
the integral is taken over a loop in $\BB C$ going once around the origin
in the counterclockwise direction.
We regard these functions as polynomials on $\HC$.
Using Lemma 22 from \cite{FL1} repeatedly, we compute:
\begin{multline*}
(\partial \otimes \partial) t^l_{n\,\underline{m}}(Z) = \begin{pmatrix}
\partial_{11} \left(\begin{smallmatrix} \partial_{11} & \partial_{21} \\
\partial_{12} & \partial_{22} \end{smallmatrix}\right) &
\partial_{21} \left(\begin{smallmatrix} \partial_{11} & \partial_{21} \\
\partial_{12} & \partial_{22} \end{smallmatrix}\right) \\
\partial_{12} \left(\begin{smallmatrix} \partial_{11} & \partial_{21} \\
\partial_{12} & \partial_{22} \end{smallmatrix}\right) &
\partial_{22} \left(\begin{smallmatrix} \partial_{11} & \partial_{21} \\
\partial_{12} & \partial_{22} \end{smallmatrix}\right) \end{pmatrix}
t^l_{n\,\underline{m}}(Z) \\
= \left(\begin{smallmatrix}
(l-m)(l-m-1) t^{l-1}_{n+1\,\underline{m+1}} & (l-m)(l-m-1) t^{l-1}_{n\,\underline{m+1}} &
(l-m)(l-m-1) t^{l-1}_{n\,\underline{m+1}} & (l-m)(l-m-1) t^{l-1}_{n-1\,\underline{m+1}} \\
(l+m)(l-m) t^{l-1}_{n+1\,\underline{m}} & (l+m)(l-m) t^{l-1}_{n\,\underline{m}} &
(l+m)(l-m) t^{l-1}_{n\,\underline{m}} & (l+m)(l-m) t^{l-1}_{n-1\,\underline{m}} \\
(l-m)(l+m) t^{l-1}_{n+1\,\underline{m}} & (l-m)(l+m) t^{l-1}_{n\,\underline{m}} &
(l-m)(l+m) t^{l-1}_{n\,\underline{m}} & (l-m)(l+m) t^{l-1}_{n-1\,\underline{m}} \\
(l+m)(l+m-1) t^{l-1}_{n+1\,\underline{m-1}} & (l+m)(l+m-1) t^{l-1}_{n\,\underline{m-1}} &
(l+m)(l+m-1) t^{l-1}_{n\,\underline{m-1}} & (l+m)(l+m-1) t^{l-1}_{n-1\,\underline{m-1}}
\end{smallmatrix}\right) (Z).
\end{multline*}
Since $\square t^l_{n\,\underline{m}}(Z) =0$, by observation made after
Definition \ref {dr-definition}, the columns and rows of this
$4 \times 4$ matrix are respectively doubly left and right regular.

\begin{lem}  \label{doubly-reg-fns-basis}
Let
$$
l = 0, \frac12, 1, \frac32, \dots, \qquad
m,n \in \BB Z +l, \qquad
-l-1 \le m \le l+1, \qquad
-l \le n \le l.
$$
The functions $\HC \to \BB S \odot \BB S$
$$
F_{l,m,n}(Z) = \begin{pmatrix} (l-m)(l-m+1) t^l_{n\,\underline{m+1}}(Z) \\
(l-m+1)(l+m+1) t^l_{n\,\underline{m}}(Z) \\
(l-m+1)(l+m+1) t^l_{n\,\underline{m}}(Z) \\
(l+m)(l+m+1) t^l_{n\,\underline{m-1}}(Z) \end{pmatrix}
$$
and
$$
F'_{l,m,n}(Z) = \frac1{N(Z)}
\begin{pmatrix} (l-n+1)(l-n+2) t^{l+1}_{n-1\,\underline{m}}(Z^{-1}) \\
(l-n+1)(l+n+1) t^{l+1}_{n\,\underline{m}}(Z^{-1}) \\
(l-n+1)(l+n+1) t^{l+1}_{n\,\underline{m}}(Z^{-1}) \\
(l+n+1)(l+n+2) t^{l+1}_{n+1\,\underline{m}}(Z^{-1}) \end{pmatrix}
$$
are doubly left regular.
Similarly, the functions $\HC \to \BB S' \odot \BB S'$
$$
G_{l,m,n}(Z) = \bigl( t^l_{m+1\,\underline{n}}(Z), t^l_{m\,\underline{n}}(Z),
t^l_{m\,\underline{n}}(Z), t^l_{m-1\,\underline{n}}(Z) \bigr)
$$
and
$$
G'_{l,m,n}(Z) = N(Z)^{-1} \cdot \bigl( t^{l+1}_{m\,\underline{n-1}}(Z^{-1}),
t^{l+1}_{m\,\underline{n}}(Z^{-1}), t^{l+1}_{m\,\underline{n}}(Z^{-1}),
t^{l+1}_{m\,\underline{n+1}}(Z^{-1}) \bigr)
$$
are doubly right regular.
\end{lem}


\begin{proof}
The result can be derived either by direct computations using
Lemmas 22 and 23 in \cite{FL1} or from Proposition 24 in \cite{FL1}.
\end{proof}

\begin{prop}  \label{DR-irred-prop}
Let
\begin{equation} \label{FF'}
{\cal F}^+ = \BB C\text{-span of } \bigl\{ F_{l,m,n}(Z) \bigr\}, \qquad
{\cal F}^- = \BB C\text{-span of } \bigl\{ F'_{l,m,n}(Z) \bigr\},
\end{equation}
\begin{equation}  \label{GG'}
{\cal G}^+ = \BB C\text{-span of } \bigl\{ G_{l,m,n}(Z) \bigr\}, \qquad
{\cal G}^- = \BB C\text{-span of } \bigl\{ G'_{l,m,n}(Z) \bigr\},
\end{equation}
$$
l = 0, \frac12, 1, \frac32, \dots, \qquad
m,n \in \BB Z +l, \qquad
-l-1 \le m \le l+1, \qquad
-l \le n \le l.
$$
Then $(\pi_{dl},{\cal F}^+)$, $(\pi_{dl},{\cal F}^-)$, $(\pi_{dr},{\cal G}^+)$
and $(\pi_{dr},{\cal G}^-)$ are irreducible representations of
$\mathfrak{sl}(2,\HC)$ (as well as $\mathfrak{gl}(2,\HC)$).
\end{prop}

The result can be proved directly by finding the $K$-types of
$(\pi_{dl},{\cal F}^{\pm})$ and $(\pi_{dr},{\cal G}^{\pm})$, computing
the actions of
$\bigl(\begin{smallmatrix} 0 & B \\ 0 & 0 \end{smallmatrix}\bigr)$ and
$\bigl(\begin{smallmatrix} 0 & 0 \\ C & 0 \end{smallmatrix}\bigr)$,
then showing that any non-zero vector generates the whole space.
Alternatively, it follows from Corollary \ref{FG-irred}.

Next we derive two matrix coefficient expansions of the Cauchy-Fueter
kernel for doubly regular functions (\ref{k_1-kernel}) in terms of these
functions $F_{l,m,n}$, $F'_{l,m,n}$, $G_{l,m,n}$, $G'_{l,m,n}$.
This is a doubly regular function analogue of Proposition 26 from \cite{FL1}
for the usual regular functions (see also Proposition \ref{Prop26}).

\begin{prop}
We have the following expansions
\begin{multline}  \label{k_1-expansion-1}
k_1(Z-W) = \sum_{l,m,n} F_{l,m,n}(Z) \cdot G'_{l,m,n}(W)
= \frac1{N(W)} \sum_{l,m,n}
\begin{pmatrix} (l-m)(l-m+1) t^l_{n\,\underline{m+1}}(Z) \\
(l-m+1)(l+m+1) t^l_{n\,\underline{m}}(Z) \\ (l-m+1)(l+m+1) t^l_{n\,\underline{m}}(Z) \\
(l+m)(l+m+1) t^l_{n\,\underline{m-1}}(Z) \end{pmatrix}  \\
\times
\bigl( t^{l+1}_{m\,\underline{n-1}}(W^{-1}), t^{l+1}_{m\,\underline{n}}(W^{-1}),
t^{l+1}_{m\,\underline{n}}(W^{-1}), t^{l+1}_{m\,\underline{n+1}}(W^{-1}) \bigr),
\end{multline}
which converges uniformly on compact subsets in the region
$\{ (Z,W) \in \HC \times \HC^{\times}; \: ZW^{-1} \in \BB D^+ \}$, and
\begin{multline}  \label{k_1-expansion-2}
k_1(Z-W) = \sum_{l,m,n} F'_{l.m,n}(Z) \cdot G_{l,m,n}(W)
= \frac1{N(Z)} \sum_{l,m,n}
\begin{pmatrix} (l-n+1)(l-n+2) t^l_{n-1\,\underline{m}}(Z^{-1}) \\
(l-n+1)(l+n+1) t^l_{n\,\underline{m}}(Z^{-1}) \\
(l-n+1)(l+n+1) t^l_{n\,\underline{m}}(Z^{-1}) \\
(l+n+1)(l+n+2) t^l_{n+1\,\underline{m}}(Z^{-1}) \end{pmatrix}  \\
\times
\bigl( t^l_{m+1\,\underline{n}}(W), t^l_{m\,\underline{n}}(W),
t^l_{m\,\underline{n}}(W), t^l_{m-1\,\underline{n}}(W) \bigr),
\end{multline}
which converges uniformly on compact subsets in the region
$\{ (Z,W) \in \HC^{\times} \times \HC; \: WZ^{-1} \in \BB D^+ \}$.
The sums are taken first over all $m=-l-1,-l,\dots,l+1$, 
$n=-l,-l+1,\dots,l$, then over $l=0,\frac12,1,\frac32,\dots$.
\end{prop}


\begin{proof}
See the discussion after equation (21) in \cite{FL3} for the definition of
the open domain $\BB D \subset \HC$.
Using our previous calculations (\ref{k_1-explicit}) and
Proposition 26 in \cite{FL1} (see also Proposition \ref{Prop26}), we obtain:
\begin{multline*}
k_1(Z-W) = (\partial_Z \otimes \partial_Z) \biggl( \frac1{N(Z-W)} \biggr)
= -(\partial_Z \otimes 1) \biggl( 1 \otimes \frac{(Z-W)^{-1}}{N(Z-W)} \biggr)  \\
= \frac1{N(W)} (\partial_Z \otimes 1) \left( 1 \otimes
\sum_{l,m,n'} \begin{pmatrix} (l-m+1) t^{l+\frac12}_{n' \, \underline{m+ \frac 12}}(Z)  \\
(l+m+1) t^{l+\frac12}_{n' \, \underline{m- \frac 12}}(Z) \end{pmatrix}
\cdot \bigl( t^{l+1}_{m \, \underline{n'- \frac 12}}(W^{-1}),
t^{l+1}_{m \, \underline{n'+ \frac 12}}(W^{-1}) \bigr) \right)  \\
= \frac1{N(W)} \sum_{l,m,n'} \begin{pmatrix} A_{11}(l,m,n') & A_{12}(l,m,n') \\
A_{21}(l,m,n') & A_{22}(l,m,n') \end{pmatrix},
\end{multline*}
where $l=-\frac12,0,\frac12,1,\frac32,\dots$, $m=-l-1,-l, \dots, l+1$,
$n'=-l-\frac12,-l+\frac12, \dots, l+\frac12$, and
\begin{multline*}
A_{11}(l,m,n') =
\partial_{11} \begin{pmatrix} (l-m+1) t^{l+\frac12}_{n' \, \underline{m+ \frac 12}}(Z)  \\
(l+m+1) t^{l+\frac12}_{n' \, \underline{m- \frac 12}}(Z) \end{pmatrix}
\cdot \bigl( t^{l+1}_{m \, \underline{n'- \frac 12}}(W^{-1}),
t^{l+1}_{m \, \underline{n'+ \frac 12}}(W^{-1}) \bigr)  \\
= \begin{pmatrix} (l-m)(l-m+1) t^{l}_{n'+\frac12 \, \underline{m+1}}(Z)  \\
(l-m+1)(l+m+1) t^{l}_{n'+\frac12 \, \underline{m}}(Z) \end{pmatrix}
\cdot \bigl( t^{l+1}_{m \, \underline{n'- \frac 12}}(W^{-1}),
t^{l+1}_{m \, \underline{n'+ \frac 12}}(W^{-1}) \bigr),
\end{multline*}
\begin{multline*}
A_{12}(l,m,n') =
\partial_{21} \begin{pmatrix} (l-m+1) t^{l+\frac12}_{n' \, \underline{m+ \frac 12}}(Z)  \\
(l+m+1) t^{l+\frac12}_{n' \, \underline{m- \frac 12}}(Z) \end{pmatrix}
\cdot \bigl( t^{l+1}_{m \, \underline{n'- \frac 12}}(W^{-1}),
t^{l+1}_{m \, \underline{n'+ \frac 12}}(W^{-1}) \bigr)  \\
= \begin{pmatrix} (l-m)(l-m+1) t^{l}_{n'-\frac12 \, \underline{m+1}}(Z)  \\
(l-m+1)(l+m+1) t^{l}_{n'-\frac12 \, \underline{m}}(Z) \end{pmatrix}
\cdot \bigl( t^{l+1}_{m \, \underline{n'- \frac 12}}(W^{-1}),
t^{l+1}_{m \, \underline{n'+ \frac 12}}(W^{-1}) \bigr),
\end{multline*}
\begin{multline*}
A_{21}(l,m,n') =
\partial_{12} \begin{pmatrix} (l-m+1) t^{l+\frac12}_{n' \, \underline{m+ \frac 12}}(Z)  \\
(l+m+1) t^{l+\frac12}_{n' \, \underline{m- \frac 12}}(Z) \end{pmatrix}
\cdot \bigl( t^{l+1}_{m \, \underline{n'- \frac 12}}(W^{-1}),
t^{l+1}_{m \, \underline{n'+ \frac 12}}(W^{-1}) \bigr)  \\
= \begin{pmatrix} (l-m+1)(l+m+1) t^{l}_{n'+\frac12 \, \underline{m}}(Z)  \\
(l+m)(l+m+1) t^{l}_{n'+\frac12 \, \underline{m-1}}(Z) \end{pmatrix}
\cdot \bigl( t^{l+1}_{m \, \underline{n'- \frac 12}}(W^{-1}),
t^{l+1}_{m \, \underline{n'+ \frac 12}}(W^{-1}) \bigr),
\end{multline*}
\begin{multline*}
A_{22}(l,m,n') =
\partial_{22} \begin{pmatrix} (l-m+1) t^{l+\frac12}_{n' \, \underline{m+ \frac 12}}(Z)  \\
(l+m+1) t^{l+\frac12}_{n' \, \underline{m- \frac 12}}(Z) \end{pmatrix}
\cdot \bigl( t^{l+1}_{m \, \underline{n'- \frac 12}}(W^{-1}),
t^{l+1}_{m \, \underline{n'+ \frac 12}}(W^{-1}) \bigr)  \\
= \begin{pmatrix} (l-m+1)(l+m+1) t^{l}_{n'-\frac12 \, \underline{m}}(Z)  \\
(l+m)(l+m+1) t^{l}_{n'-\frac12 \, \underline{m-1}}(Z) \end{pmatrix}
\cdot \bigl( t^{l+1}_{m \, \underline{n'- \frac 12}}(W^{-1}),
t^{l+1}_{m \, \underline{n'+ \frac 12}}(W^{-1}) \bigr).
\end{multline*}
Since the terms with $l=-\frac12$ are zero, we can restrict $l$ to
$0,\frac12,1,\frac32,\dots$. We have:
$$
\sum_{l,m,n'} A_{11}(l,m,n') = \sum_{l,m,n}
\begin{pmatrix} (l-m)(l-m+1) t^{l}_{n \, \underline{m+1}}(Z)  \\
(l-m+1)(l+m+1) t^{l}_{n \, \underline{m}}(Z) \end{pmatrix}
\cdot \bigl( t^{l+1}_{m \, \underline{n-1}}(W^{-1}),
t^{l+1}_{m \, \underline{n}}(W^{-1}) \bigr),
$$
$$
\sum_{l,m,n'} A_{21}(l,m,n') = \sum_{l,m,n}
\begin{pmatrix} (l-m+1)(l+m+1) t^{l}_{n \, \underline{m}}(Z)  \\
(l+m)(l+m+1) t^{l}_{n \, \underline{m-1}}(Z) \end{pmatrix}
\cdot \bigl( t^{l+1}_{m \, \underline{n-1}}(W^{-1}),
t^{l+1}_{m \, \underline{n}}(W^{-1}) \bigr),
$$
where $n=n'+\frac12$, and
$$
\sum_{l,m,n'} A_{12}(l,m,n') = \sum_{l,m,n}
\begin{pmatrix} (l-m)(l-m+1) t^{l}_{n \, \underline{m+1}}(Z)  \\
(l-m+1)(l+m+1) t^{l}_{n \, \underline{m}}(Z) \end{pmatrix}
\cdot \bigl( t^{l+1}_{m \, \underline{n}}(W^{-1}),
t^{l+1}_{m \, \underline{n+1}}(W^{-1}) \bigr),
$$
$$
\sum_{l,m,n'} A_{22}(l,m,n') = \sum_{l,m,n}
\begin{pmatrix} (l-m+1)(l+m+1) t^{l}_{n \, \underline{m}}(Z)  \\
(l+m)(l+m+1) t^{l}_{n \, \underline{m-1}}(Z) \end{pmatrix}
\cdot \bigl( t^{l+1}_{m \, \underline{n}}(W^{-1}),
t^{l+1}_{m \, \underline{n+1}}(W^{-1}) \bigr),
$$
where $n=n'-\frac12$; in all cases $n=-l,-l+1,\dots,l$.
This proves the first expansion. The other expansion is proved similarly.
\end{proof}

\subsection{Doubly Regular Functions on $\BB H^{\times}$}  \label{deg+2-inverse-subsection}

In this subsection we show that if a (left or right) doubly regular function is
defined on all of $\BB H^{\times}$, then the operator $\degtt$ can be inverted.
This will be needed, for example, when we define the invariant bilinear pairing
for such functions.

We start with a doubly left regular function
$f: \BB H^{\times} \to \BB S \odot \BB S$
and derive some properties of such functions.
Of course, doubly right regular functions
$g: \BB H^{\times} \to \BB S' \odot \BB S'$ have similar properties.
Let $0<r<R$, then, by the Cauchy-Fueter formulas for doubly regular functions
(Theorem \ref{Fueter-doubly-reg}),
$$
\degtt f(W) =
\frac1{2\pi^2} \int_{S^3_R} k_1(Z-W) \cdot (Dz \otimes Z) \cdot f(Z)
- \frac1{2\pi^2} \int_{S^3_r} k_1(Z-W) \cdot (Dz \otimes Z) \cdot f(Z),
$$
for all $W \in \BB H$ such that $r^2<N(W)<R^2$, where
$S^3_R \subset \BB H$ is the sphere of radius $R$ centered at the origin
$$
S^3_R = \{ X \in \BB H ;\: N(X)=R^2 \}.
$$
Define functions $\tilde f_+: \BB H \to \BB S \odot \BB S$
and $\tilde f_-: \BB H^{\times} \to \BB S \odot \BB S$ by
$$
\tilde f_+(W) =
\frac1{2\pi^2} \int_{S^3_R} k_1(Z-W) \cdot (Dz \otimes Z) \cdot f(Z),
\qquad R^2>N(W),
$$
$$
\tilde f_-(W) =
- \frac1{2\pi^2} \int_{S^3_r} k_1(Z-W) \cdot (Dz \otimes Z) \cdot f(Z),
\qquad r^2<N(W).
$$
Note that $\tilde f_+$ and $\tilde f_-$ are doubly left regular and
that $\tilde f_-(W)$ decays at infinity at a rate $\sim N(W)^{-2}$.

For a function $\phi$ defined on $\BB H$ or, slightly more generally,
on a star-shaped open subset of $\BB H$ centered at the origin, let
$$
\bigl( \degtt^{-1} \phi \bigr)(Z) = \int_0^1 t \cdot \phi(tZ) \,dt.
$$
Similarly, for a function $\phi$ defined on $\BB H^{\times}$ and decaying
sufficiently fast at infinity, we can define $\degtt^{-1} \phi$ as
$$
\bigl( \degtt^{-1} \phi \bigr)(Z) = - \int_1^{\infty} t \cdot \phi(tZ) \,dt.
$$
Then
$$
\degtt \bigl( \degtt^{-1} \phi \bigr)
= \degtt^{-1} \bigl( \degtt \phi \bigr) = \phi
$$
for functions $\phi$ that are either defined on star-shaped open subsets of
$\BB H$ centered at the origin or on $\BB H^{\times}$ and decaying sufficiently
fast at infinity. (In the same fashion one can also define $\degtt^{-1}\phi$
for functions defined on star-shaped open subsets of $\HC$ centered at
the origin or on $\HC^{\times}$ and decaying sufficiently fast at infinity.)

We introduce functions
$$
f_+ = \degtt^{-1} \tilde f_+ \qquad \text{and} \qquad
f_- = \degtt^{-1} \tilde f_-.
$$

\begin{prop}
Let $f: \BB H^{\times} \to \BB S \odot \BB S$ be a doubly left regular
function. Then $f(X)=f_+(X)+f_-(X)$, for all $X \in \BB H^{\times}$.
\end{prop}

\begin{proof}
Let $f_{-2} = f-f_+-f_-$, we want to show that $f_{-2} \equiv 0$.
Note that $f_{-2}: \BB H^{\times} \to \BB S \odot \BB S$ is a doubly left
regular function such that $\degtt f_{-2} \equiv 0$, hence $f_{-2}$ is
homogeneous of degree $-2$. Let
$$
f_{-3} = \frac{\partial}{\partial x_{11}} f_{-2}, \qquad
f_{-3}: \BB H^{\times} \to \BB S \odot \BB S,
$$
then $f_{-3}$ is a doubly left regular function that is homogeneous
of degree $-3$.

By the Cauchy-Fueter formulas for doubly regular functions
(Theorem \ref{Fueter-doubly-reg}),
\begin{multline*}
f_{-3}(X) =
- \frac1{2\pi^2} \int_{S^3_R} k_1(Z-X) \cdot (Dz \otimes Z) \cdot f_{-3}(Z)  \\
+ \frac1{2\pi^2} \int_{S^3_r} k_1(Z-X) \cdot (Dz \otimes Z) \cdot f_{-3}(Z),
\end{multline*}
where $R,r>0$ are such that $r^2<N(X)<R^2$.
By Liouville's theorem (Corollary \ref{Liouville}), the first integral defines
a doubly left regular function on $\BB H$ that is either constant or unbounded.
On the other hand, the second integral defines a doubly left regular function
on $\BB H^{\times}$ that decays at infinity at a rate $\sim N(W)^{-2}$.
We conclude that $f_{-3} \equiv 0$, hence $f_{-2} \equiv 0$ as well.
\end{proof}

\begin{df}  \label{deg-inv-def}
Let $f: \BB H^{\times} \to \BB S \odot \BB S$ be a doubly left regular
function. We define
$$
\degtt^{-1}f = \degtt^{-1}f_+ + \degtt^{-1}f_-.
$$
Similarly, we can define $\degtt^{-1}g$ for doubly right regular functions
$g: \BB H^{\times} \to \BB S' \odot \BB S'$.
\end{df}

From the previous discussion we immediately obtain:

\begin{prop}
Let $f: \BB H^{\times} \to \BB S \odot \BB S$ be a doubly left regular function
and $g: \BB H^{\times} \to \BB S' \odot \BB S'$ a doubly right regular function.
Then
\begin{align*}
\degtt \bigl( \degtt^{-1} f \bigr) = \degtt^{-1} \bigl(\degtt f \bigr) &= f, \\
\degtt \bigl( \degtt^{-1} g \bigr) = \degtt^{-1} \bigl( \degtt g \bigr) &= g.
\end{align*}
\end{prop}

From the expansions of the Cauchy-Fueter kernel (\ref{k_1-expansion-1}) and
(\ref{k_1-expansion-2}) we immediately obtain an analogue of Laurent series
expansion for doubly regular functions.

\begin{cor}  \label{Laurent-expansion}
Let $f: \BB H^{\times} \to \BB S \odot \BB S$ be a doubly left regular
function, write $f=f_++f_-$ as in the above proposition.
Then the functions $f_+$ and $f_-$ can be expanded as series
$$
f_+(X) = \sum_l \Bigl( \sum_{m,n} a_{l,m,n} F_{l,m,n}(X) \Bigr), \qquad
f_-(X) = \sum_l \Bigl( \sum_{m,n} b_{l,m,n} F'_{l,m,n}(X) \Bigr).
$$

If $g: \BB H^{\times} \to \BB S' \odot \BB S'$ is a doubly right regular
function, then it can be expressed as $g=g_+ + g_-$ in a similar way,
and the functions $g_+$ and $g_-$ can be expanded as series
$$
g_+(X) = \sum_l \Bigl( \sum_{m,n} c_{l,m,n} G_{l,m,n}(X) \Bigr), \qquad
g_-(X) = \sum_l \Bigl( \sum_{m,n} d_{l,m,n} G'_{l,m,n}(X) \Bigr).
$$
\end{cor}

Formulas expressing the coefficients $a_{l,m,n}$, $b_{l,m,n}$, $c_{l,m,n}$ and
$d_{l,m,n}$ will be given in Corollary \ref{Laurent-coeff}.

\subsection{Invariant Bilinear Pairing for Doubly Regular Functions}

We define a pairing between doubly left and right regular functions as follows.
If $f(Z)$ and $g(Z)$ are doubly left and right regular functions on
$\HC^{\times}$ respectively, then by the results of the previous subsection
$\degtt^{-1}f$ and $\degtt^{-1}g$ are well defined, and we set
\begin{equation}   \label{pairing1}
\langle f, g \rangle_{\cal DR} = \frac1{2\pi^2}
\int_{Z \in S^3_R} g(Z) \cdot (Z \otimes Dz)\cdot \bigl( \degtt^{-1} f \bigr)(Z),
\end{equation}
where $S^3_R \subset \BB H \subset \HC$ is the sphere of radius $R$
centered at the origin
$$
S^3_R = \{ X \in \BB H ;\: N(X)=R^2 \}.
$$
Recall that by Lemma 6 in \cite{FL1} the $3$-form $Dz$ restricted to $S^3_R$
becomes $Z\,dS/R$, where $dS$ is the usual Euclidean volume element on $S^3_R$.
Thus we can rewrite (\ref{pairing1}) as
\begin{multline*}
\langle f, g \rangle_{\cal DR} = \frac1{2\pi^2} \int_{Z \in S^3_R} g(Z) \cdot
(Z \otimes Z) \cdot \bigl( \degtt^{-1} f \bigr)(Z) \,\frac{dS}R  \\
= \frac1{2\pi^2} \int_{Z \in S^3_R} g(Z) \cdot (Dz \otimes Z)
\cdot \bigl( \degtt^{-1} f \bigr)(Z).
\end{multline*}
Since $(1 \otimes \nabla^+) \degtt^{-1} f =0$ and, by Lemma \ref{zf-regular},
$\bigl[ g(Z) (Z \otimes 1) \bigr] (\overleftarrow{1 \otimes \nabla^+}) = 0$,
the integrand of (\ref{pairing1}) is a closed $3$-form and the contour of
integration can be continuously deformed. In particular, this pairing does
not depend on the choice of $R>0$.

\begin{prop}
If $f(Z)$ and $g(Z)$ are doubly left and right regular functions
on $\HC^{\times}$ respectively, then
\begin{multline}  \label{deg-switch}
\langle f, g \rangle_{\cal DR} = \frac1{2\pi^2} \int_{Z \in S^3_R} g(Z) \cdot
(Z \otimes Dz) \cdot \bigl( \degtt^{-1} f \bigr)(Z) \\
= - \frac1{2\pi^2} \int_{Z \in S^3_R} \bigl( \degtt^{-1} g \bigr)(Z) \cdot
(Z \otimes Dz) \cdot f(Z).
\end{multline}
\end{prop}

\begin{proof}
Since the expression
$$
\int_{Z \in S^3_R} g(Z) \cdot (Z \otimes Dz) \cdot f(Z)
$$
is independent of the choice of $R>0$, we have:
\begin{multline*}
0 = \frac{d}{dt} \biggr|_{t=1} \biggl( \int_{Z \in  S^3_{tR}} g(Z)
\cdot (Z \otimes Dz) \cdot f(Z) \biggr)  \\
= \int_{Z \in  S^3_R} \Bigl( \bigl( \degtt g \bigr)(Z) \cdot (Z \otimes Dz)
\cdot f(Z)
+ g(Z) \cdot (Z \otimes Dz) \cdot \bigl( \degtt f \bigr)(Z) \Bigr).
\end{multline*}
From this (\ref{deg-switch}) follows.
\end{proof}

\begin{cor}
If $f(Z)$ and $g(Z)$ are doubly left and right regular functions on $\HC$
respectively and $W \in \BB D_R^+$ (open domains $\BB D_R^{\pm}$ were defined by
equation (22) in \cite{FL3}), the Cauchy-Fueter formulas for doubly
regular functions (Theorem \ref{Fueter-doubly-reg}) can be rewritten as
$$
f(W) = \bigl\langle k_1(Z-W), f(Z) \bigr\rangle_{\cal DR}
\qquad \text{and} \qquad
g(W) = - \bigl\langle g(Z), k_1(Z-W) \bigr\rangle_{\cal DR}.
$$
\end{cor}

We can rewrite the bilinear pairing (\ref{pairing1})
in a more symmetrical way.
Let $0<r<R$ and $0<r_1<R<r_2$. Using the Cauchy-Fueter formulas for doubly
regular functions (Theorem \ref{Fueter-doubly-reg}), substituting
\begin{multline*}
  g(Z) = \frac1{2\pi^2} \int_{W \in S^3_{r_2}}
  \bigl( \degtt^{-1} g \bigr)(W) \cdot (W \otimes Dw) \cdot k_1(W-Z) \\
  - \frac1{2\pi^2} \int_{W \in S^3_{r_1}}
  \bigl( \degtt^{-1} g \bigr)(W) \cdot (W \otimes Dw) \cdot k_1(W-Z),
  \qquad Z \in \BB D^+_{r_2} \cap \BB D^-_{r_1},
\end{multline*}
into (\ref{pairing1}) and shifting contours of integration, we obtain:
\begin{multline}  \label{pairing1-symm}
\langle f, g \rangle_{\cal DR}  \\
= \frac1{4\pi^4}
\iint_{\genfrac{}{}{0pt}{}{Z \in S^3_r}{W \in S^3_R}} \bigl( \degtt^{-1} g \bigr)(W)
\cdot (W \otimes Dw) \cdot k_1(W-Z) \cdot (Z \otimes Dz) \cdot
\bigl( \degtt^{-1} f \bigr)(Z)  \\
- \frac1{4\pi^4} \iint_{\genfrac{}{}{0pt}{}{Z \in S^3_R}{W \in S^3_r}}
\bigl( \degtt^{-1} g \bigr)(W) \cdot (W \otimes Dw) \cdot k_1(W-Z) \cdot
(Z \otimes Dz) \cdot \bigl( \degtt^{-1} f \bigr)(Z).
\end{multline}

\begin{prop}  \label{invariant}
The bilinear pairing (\ref{pairing1}) is $\mathfrak{gl}(2,\HC)$-invariant.
\end{prop}

\begin{proof}
It is sufficient to show that the pairing is invariant under
$SU(2) \times SU(2) \subset GL(2,\HC)$, dilation matrices
$\bigl(\begin{smallmatrix} \lambda & 0 \\ 0 & 1 \end{smallmatrix}\bigr)
\in GL(2,\HC)$, $\lambda \in \BB R$, $\lambda>0$, inversions
$\bigl(\begin{smallmatrix} 0 & 1 \\ 1 & 0 \end{smallmatrix}\bigr)
\in GL(2,\HC)$ and
$\bigl(\begin{smallmatrix} 0 & B \\ 0 & 0 \end{smallmatrix}\bigr)
\in \mathfrak{gl}(2,\HC)$, $B \in \HC$.

First, let
$h= \bigl(\begin{smallmatrix} a & 0 \\ 0 & d \end{smallmatrix}\bigr)
\in GL(2,\HC)$, $a,d \in \BB H$, $N(a)=N(d)=1$, $\tilde Z = a^{-1}Zd$.
Using Proposition 11 from \cite{FL1} we obtain:
\begin{multline*}
  2\pi^2 \cdot \langle \pi_{dl}(h)f, \pi_{dr}(h)g \rangle_{\cal DR}
  = \int_{Z \in S^3_R}
(\pi_{dr}(h)g)(Z) \cdot (Z \otimes Dz) \cdot
\bigl( \degtt^{-1} (\pi_{dl}(h)f) \bigr)(Z)  \\
= \int_{Z \in S^3_R} g(a^{-1}Zd) \cdot \frac{a^{-1} \otimes a^{-1}}{N(a)}
\cdot (Z \otimes Dz) \cdot \degtt^{-1} \Bigl( N(d) \cdot (d \otimes d)
\cdot f(a^{-1}Zd) \Bigr)  \\
= \int_{\tilde Z \in S^3_R} g(\tilde Z) \cdot (\tilde Z \otimes D\tilde z)
\cdot \bigl( \degtt^{-1} f \bigr)(\tilde Z)
= 2\pi^2 \cdot \langle f,g \rangle_{\cal DR}.
\end{multline*}

The calculations for
$h = \bigl(\begin{smallmatrix} \lambda & 0 \\ 0 & 1 \end{smallmatrix}\bigr)
\in GL(2,\HC)$, $\lambda \in \BB R$, $\lambda>0$, are similar.


Next, we recall that matrices
$\bigl(\begin{smallmatrix} 0 & B \\ 0 & 0 \end{smallmatrix}\bigr)
\in \mathfrak{gl}(2,\HC)$, $B \in \HC$,
act by differentiation (Lemma \ref{Lie-alg-action}).
For example, if
$B= \bigl(\begin{smallmatrix} 1 & 0 \\ 0 & 0 \end{smallmatrix}\bigr) \in \HC$,
using the symmetric expression (\ref{pairing1-symm}) we obtain:
\begin{multline*}
4\pi^4 \cdot \Bigl\langle \pi_{dl}\bigl(\begin{smallmatrix} 0 & B \\
0 & 0 \end{smallmatrix}\bigr)f, g \Bigr\rangle_{\cal DR}  \\
= \iint_{\genfrac{}{}{0pt}{}{Z \in S^3_r}{W \in S^3_R}} \bigl( \degtt^{-1} g \bigr)(W)
\cdot (W \otimes Dw) \cdot k_1(W-Z) \cdot (Z \otimes Dz) \cdot
\biggl(\degtt^{-1} \frac{\partial f}{\partial z_{11}} \biggr)(Z)  \\
- \iint_{\genfrac{}{}{0pt}{}{Z \in S^3_R}{W \in S^3_r}} \bigl( \degtt^{-1} g \bigr)(W)
\cdot (W \otimes Dw) \cdot k_1(W-Z) \cdot (Z \otimes Dz) \cdot
\biggl(\degtt^{-1} \frac{\partial f}{\partial z_{11}} \biggr)(Z)  \\
= \iint_{\genfrac{}{}{0pt}{}{Z \in S^3_r}{W \in S^3_R}} \bigl( \degtt^{-1} g\bigr)(W)
\cdot (W \otimes Dw) \cdot \biggl( \frac{\partial}{\partial w_{11}} k_1(W-Z)
\biggr) \cdot (Z \otimes Dz) \cdot \bigl( \degtt^{-1} f \bigr)(Z)  \\
- \iint_{\genfrac{}{}{0pt}{}{Z \in S^3_R}{W \in S^3_r}} \bigl( \degtt^{-1} g \bigr)(W)
\cdot (W \otimes Dw) \cdot \biggl( \frac{\partial}{\partial w_{11}} k_1(W-Z)
\biggr) \cdot (Z \otimes Dz) \cdot \bigl( \degtt^{-1} f \bigr)(Z)  \\
= - \iint_{\genfrac{}{}{0pt}{}{Z \in S^3_r}{W \in S^3_R}} \bigl( \degtt^{-1} g \bigr)(W)
\cdot (W \otimes Dw) \cdot \biggl( \frac{\partial}{\partial z_{11}} k_1(W-Z)
\biggr) \cdot (Z \otimes Dz) \cdot \bigl( \degtt^{-1} f \bigr)(Z)  \\
+ \iint_{\genfrac{}{}{0pt}{}{Z \in S^3_R}{W \in S^3_r}} \bigl( \degtt^{-1} g \bigr)(W)
\cdot (W \otimes Dw) \cdot \biggl( \frac{\partial}{\partial z_{11}} k_1(W-Z)
\biggr) \cdot (Z \otimes Dz) \cdot \bigl( \degtt^{-1} f \bigr)(Z)  \\
= - \iint_{\genfrac{}{}{0pt}{}{Z \in S^3_r}{W \in S^3_R}}
\biggl(\degtt^{-1} \frac{\partial g}{\partial w_{11}} \biggr)(W) \cdot
(W \otimes Dw) \cdot k_1(W-Z) \cdot (Z \otimes Dz) \cdot
\bigl( \degtt^{-1} f \bigr)(Z)  \\
+ \iint_{\genfrac{}{}{0pt}{}{Z \in S^3_R}{W \in S^3_r}}
\biggl(\degtt^{-1} \frac{\partial g}{\partial w_{11}} \biggr)(W) \cdot
(W \otimes Dw) \cdot k_1(W-Z) \cdot (Z \otimes Dz) \cdot
\bigl( \degtt^{-1} f \bigr)(Z)  \\
= - 4\pi^4 \cdot \Bigl\langle f, \pi_{dr}\bigl(\begin{smallmatrix} 0 & B \\
0 & 0 \end{smallmatrix}\bigr)g \Bigr\rangle_{\cal DR}.
\end{multline*}

Finally, if
$h = \bigl(\begin{smallmatrix} 0 & 1 \\ 1 & 0 \end{smallmatrix}\bigr)
\in GL(2,\HC)$, changing the variable to $\tilde Z = Z^{-1}$ -- which is
an orientation reversing map $S^3_R \to S^3_{1/R}$ -- and using
Proposition 11 from \cite{FL1}, we have:
\begin{multline*}
2\pi^2 \cdot \langle \pi_{dl}(h)f, \pi_{dr}(h)g \rangle_{\cal DR} \\
= \int_{Z \in S^3_R}
g(Z^{-1}) \cdot \frac{Z^{-1} \otimes Z^{-1}}{N(Z)} \cdot (Z \otimes Dz) \cdot
\degtt^{-1} \biggl( \frac{Z^{-1} \otimes Z^{-1}}{N(Z)} \cdot f(Z^{-1}) \biggr)  \\
= - \int_{Z \in S^3_R} g(Z^{-1}) \cdot
\frac{Z^{-1} \otimes (Z^{-1} \cdot Dz \cdot Z^{-1})}{N(Z)^2} \cdot
\bigl( \degtt^{-1} f \bigr)(Z^{-1}) \\
= - \int_{\tilde Z \in S^3_{1/R}} g(\tilde Z) \cdot
(\tilde Z \otimes D\tilde z) \cdot \bigl( \degtt^{-1} f \bigr)(\tilde Z)
= - 2\pi^2 \cdot \langle f, g \rangle_{\cal DR}.
\end{multline*}
(Note that the negative sign in
$\langle \pi_{dl}(h)f, \pi_{dr}(h)g \rangle_{\cal DR}
= - \langle f, g \rangle_{\cal DR}$
does not affect the invariance of the bilinear pairing under the Lie algebra
$\mathfrak{gl}(2,\HC)$.)
\end{proof}

Next we describe orthogonality relations for doubly regular functions.
Recall functions $F_{l,m,n}$, $F'_{l,m,n}$, $G_{l,m,n}$ and $G'_{l,m,n}$ introduced
in Lemma \ref{doubly-reg-fns-basis}, these are the functions that appear in
matrix coefficient expansions of the Cauchy-Fueter kernel
(\ref{k_1-expansion-1}) and (\ref{k_1-expansion-2}).

\begin{prop}  \label{orthogonality-dreg-prop}
We have the following orthogonality relations:
$$
\langle F_{l,m,n}, G'_{l',m',n'} \rangle_{\cal DR} =
- \langle F'_{l,m,n}, G_{l',m',n'} \rangle_{\cal DR} =
\delta_{ll'} \cdot \delta_{mm'} \cdot \delta_{nn'},
$$
$$
\langle F_{l,m,n}, G_{l',m',n'} \rangle_{\cal DR}
= \langle F'_{l,m,n}, G'_{l',m',n'} \rangle_{\cal DR} = 0.
$$
\end{prop}

\begin{proof}
Recall that, by Lemma 6 in \cite{FL1}, the $3$-form $Dz$ restricted to $S^3_R$
equals $Z\,dS/R$.
We continue to identify tensor products of matrices with their Kronecker
products.
Applying Lemma 23 from \cite{FL1} repeatedly, we compute
$$
(Z \otimes Z) \cdot  F_{l,m,n}(Z)
= \left( \begin{smallmatrix} (l-n+1)(l-n+2) t^{l+1}_{n-1\,\underline{m}}(Z) \\
(l+n+1)(l-n+1) t^{l+1}_{n\,\underline{m}}(Z) \\
(l-n+1)(l+n+1) t^{l+1}_{n\,\underline{m}}(Z) \\
(l+n+1)(l+n+2) t^{l+1}_{n+1\,\underline{m}}(Z) \end{smallmatrix} \right).
$$
When we multiply this by $G'_{l',m',n'}(Z)$ and integrate over $S^3_R$,
by the orthogonality relations (17) in \cite{FL3}
(see also equation (1) in $\S$6.2 of Chapter III in \cite{V}),
$$
\langle F_{l,m,n}, G'_{l',m',n'} \rangle_{\cal DR} =
\delta_{ll'} \cdot \delta_{mm'} \cdot \delta_{nn'}.
$$

Since 
$$
G_{l',m',n'}(Z) \cdot (Z \otimes Z) \cdot F_{l,m,n}(Z)
$$
is homogeneous of degree $2(l+l')>-2$ and the pairing (\ref{pairing1})
is independent of the choice of $R>0$, $\langle F_{l,m,n}, G_{l',m',n'} \rangle_{\cal DR}$
must be zero.
The cases involving $F'_{l,m,n}$ can be proved similarly.
\end{proof}

\begin{cor}  \label{Laurent-coeff}
The coefficients $a_{l,m,n}$, $b_{l,m,n}$, $c_{l,m,n}$ and $d_{l,m,n}$
of Laurent expansions of doubly regular functions given in
Corollary \ref{Laurent-expansion} are given by the following expressions:
\begin{center}
\begin{tabular}{lcl}
$a_{l.m,n} = \langle f, G'_{l,m,n} \rangle_{\cal DR}$, & \qquad &
$b_{l.m,n} = -\langle f, G_{l,m,n} \rangle_{\cal DR}$, \\
$c_{l.m,n} = -\langle F'_{l,m,n}, g \rangle_{\cal DR}$, & \qquad &
$d_{l.m,n} = \langle F_{l,m,n}, g \rangle_{\cal DR}$.
\end{tabular}
\end{center}
\end{cor}

\subsection{$n$-regular Functions}

One can generalize the notion of doubly regular functions to
triply regular functions, quadruply regular functions and so on.
Thus, left $n$-regular functions take values in
$$
\underbrace{\BB S \odot \dots \odot \BB S}_{\text{$n$ times}}
$$
and satisfy $n$ regularity conditions
$$
(1 \otimes \dots \otimes \underset{\text{$i$-th place}}{\nabla^+} \otimes
\dots \otimes 1)f=0, \qquad i=1,\dots,n.
$$
Similarly, one can define right $n$-regular functions with values in
$$
\underbrace{\BB S' \odot \dots \odot \BB S'}_{\text{$n$ times}}.
$$
The group $GL(2,\HC)$ acts on $n$-regular functions similarly to
(\ref{pi_dl})-(\ref{pi_dr}).
Then polynomial $n$-regular functions should yield realizations
of all the highest weight representations of the most degenerate series of
representations of the conformal Lie algebra $\mathfrak{sl}(2,\HC)$.
Those are often called the spin $\frac{n}2$ representations of positive and
negative helicities.
The spin $0$ case corresponds to the harmonic functions, while the
spin $\frac12$ case correspond to the usual left and right regular functions.
Such representations were considered by H.~P.~Jakobsen and M.~Vergne
in \cite{JV1}.

One can also derive an analogue of the Cauchy-Fueter formulas as well as
a bilinear pairing for $n$-regular functions, just as we did for the $n=2$
case in this section. However, the $n=2$ case appears to be special,
as the doubly regular functions can be realized as a certain subspace of
$\HC$-valued functions, which will be the subject of Section
\ref{DRinQCC-section}.

Functions with values in $\BB S \odot \dots \odot \BB S$ are a subspace of
all functions
$$
\HC^{\times} \to \underbrace{\BB S \otimes \dots \otimes \BB S}_{\text{$n$ times}}.
$$
As a special case of Schur-Weyl duality,
\begin{equation}  \label{Schur-Weyl}
\underbrace{\BB S \otimes \dots \otimes \BB S}_{\text{$n$ times}}
= \bigoplus_{k=0}^n M_k \otimes V_k,
\end{equation}
where $V_k$ is the irreducible representation of $SL(2,\BB C)$ of
dimension $k+1$ and $M_k$ is its multiplicity, which is also an irreducible
representation of the symmetric group on $n$ objects.
In particular,
$$
V_n = \underbrace{\BB S \odot \dots \odot \BB S}_{\text{$n$ times}}
\qquad \text{and} \qquad
V_0 = \underbrace{\BB S \wedge \dots \wedge \BB S}_{\text{$n$ times}}.
$$
When $n=2$, $M_1=0$ and
$$
\BB S \otimes \BB S = (\BB S \wedge \BB S) \oplus (\BB S \odot \BB S).
$$

\begin{prop}
As a representation of $\mathfrak{gl}(2,\HC)$, the space of maps
$$
\HC^{\times} \to \BB S \wedge \BB S
$$
is isomorphic to $\BB C[z_{11},z_{12},z_{21},z_{22}, N(Z)^{-1}]$ with
$\mathfrak{gl}(2,\HC)$ action obtained by differentiating the following
action of $GL(2,\HC)$:
$$
f(Z) \: \mapsto \: \frac1{N(cZ+d)^2}
\cdot f\bigl( (aZ+b)(cZ+d)^{-1} \bigr),  \qquad
h^{-1} = \bigl( \begin{smallmatrix} a & b \\ c & d \end{smallmatrix} \bigr)
\in GL(2,\HC).
$$

Similarly, as a representation of $\mathfrak{gl}(2,\HC)$, the space of maps
$$
\HC^{\times} \to \BB S' \wedge \BB S'
$$
is isomorphic to $\BB C[z_{11},z_{12},z_{21},z_{22}, N(Z)^{-1}]$ with
$\mathfrak{gl}(2,\HC)$ action obtained by differentiating another action of
$GL(2,\HC)$:
$$
f(Z) \: \mapsto \: \frac1{N(a'-Zc')^2}
\cdot f \bigl( (a'-Zc')^{-1}(-b'+Zd') \bigr),  \qquad
h = \bigl( \begin{smallmatrix} a' & b' \\ c' & d' \end{smallmatrix} \bigr)
\in GL(2,\HC).
$$
\end{prop}

\begin{proof}
Since $\BB S \wedge \BB S$ and $\BB S' \wedge \BB S'$ are one-dimensional,
the spaces of functions $\HC^{\times} \to \BB S \wedge \BB S$ and
$\HC^{\times} \to \BB S' \wedge \BB S'$ can be both identified with
$\BB C[z_{11},z_{12},z_{21},z_{22}, N(Z)^{-1}]$.
The actions of $\mathfrak{gl}(2,\HC)$ are obtained by taking the determinants of
$\pi_{dl}$ and $\pi_{dr}$ respectively, and the result follows from
(\ref{pi_dl})-(\ref{pi_dr}).
\end{proof}

Note that similar representations and their irreducible components were
considered in Subsections 3.1-3.2 in \cite{L}.

For general $n$, we have a decomposition of the space of functions
$\HC^{\times} \to \BB S \otimes \dots \otimes \BB S$ according to
(\ref{Schur-Weyl}).
In a separate paper \cite{FL4} we study $n$-regular functions in more detail.
For example, we prove that generalizations of the Cauchy-Fueter formulas to
such functions provide natural quaternionic analogues of Cauchy's
differentiation formula
$$
f^{(n-1)}(w) = \frac{(n-1)!}{2\pi i} \oint \frac {f(z)\,dz}{(z-w)^n}.
$$

\section{Quaternionic Chain Complex and Decomposition of $(\rho,\Sh)$,
$(\rho',\Sh')$ into Irreducible Components}  \label{QCC-section}

\subsection{Quaternionic Chain Complex}

We start with a sequence of maps (57) from \cite{FL1}:
\begin{equation}  \label{W-sequence}
\begin{CD}
(\rho', \Sh') @>{\partial^+}>> (\rho'_2, {\cal W}') @>{\M}>>
(\rho_2, {\cal W}) @>{\tr \circ \partial^+}>> (\rho, \Sh),
\end{CD}
\end{equation}
where
$$
\Sh = \Sh' = \bigl\{\text{$\BB C$-valued polynomial functions on
$\HC^{\times}$}\bigr\}
= \BB C[z_{11},z_{12},z_{21},z_{22}, N(Z)^{-1}],
$$
$$
{\cal W} = {\cal W}' = \bigl\{\text{$\HC$-valued polynomial functions on
$\HC^{\times}$}\bigr\}
= \HC \otimes \Sh,
$$
$$
\partial = \begin{pmatrix} \partial_{11} & \partial_{21} \\
\partial_{12} & \partial_{22} \end{pmatrix} = \frac 12 \nabla, \qquad
\partial^+ = \begin{pmatrix} \partial_{22} & -\partial_{21} \\
-\partial_{12} & \partial_{11} \end{pmatrix} = \frac 12 \nabla^+,
\qquad \M F = \nabla F \nabla - \square F^+.
$$
Since the compositions of any two consecutive maps are zero:
$$
\M \circ \partial^+ =0 \qquad \text{and} \qquad
(\tr \circ \partial^+) \circ \M =0,
$$
we call (\ref{W-sequence}) a {\em quaternionic chain complex}.
The Lie algebra $\g{gl}(2,\HC)$ acts on these spaces by differentiating
the following group actions:
\begin{align*}
\rho(h): \: f(Z) \quad &\mapsto \quad \bigl( \rho(h)f \bigr)(Z) =
\frac {f \bigl( (aZ+b)(cZ+d)^{-1} \bigr)}{N(cZ+d)^2 \cdot N(a'-Zc')^2}, \\
\rho'(h): \: f(Z) \quad &\mapsto \quad \bigl( \rho(h)f \bigr)(Z) =
f \bigl( (aZ+b)(cZ+d)^{-1} \bigr), \\
\rho_2(h): \: F(Z) \quad &\mapsto \quad \bigl( \rho_2(h)F \bigr)(Z) =
\frac {(cZ+d)^{-1}}{N(cZ+d)} \cdot F \bigl( (aZ+b)(cZ+d)^{-1} \bigr) \cdot
\frac {(a'-Zc')^{-1}}{N(a'-Zc')},  \\
\rho'_2(h): \: F(Z) \quad &\mapsto \quad \bigl( \rho'_2(h)F \bigr)(Z) =
\frac {(a'-Zc')}{N(a'-Zc')} \cdot F \bigl( (aZ+b)(cZ+d)^{-1} \bigr)
\cdot \frac {(cZ+d)}{N(cZ+d)},
\end{align*}
where $f \in \Sh$ or $\Sh'$, $F \in {\cal W}$ or ${\cal W}'$,
$h = \bigl(\begin{smallmatrix} a' & b' \\ c' & d' \end{smallmatrix}\bigr)
\in GL(2,\HC)$ and 
$h^{-1} = \bigl(\begin{smallmatrix} a & b \\ c & d \end{smallmatrix}\bigr)$.
Although $\Sh = \Sh'$ and ${\cal W} = {\cal W}'$ as vector spaces,
these notations indicate the action of $\g{gl}(2,\HC)$.
Also, in \cite{FL3} we treat $(\rho_1,\Zh)$, where $\Zh=\Sh$ as vector spaces,
but the action $\rho_1$ of $\g{gl}(2,\HC)$ is different from $\rho$, $\rho'$
considered here.
We have the following four analogues of Lemma 68 in \cite{FL1}:

\begin{lem}  \label{rho-action-lem}
The Lie algebra action $\rho$ of $\mathfrak{gl}(2,\HC)$ on $\Sh$ is given by
\begin{align*}
\rho \bigl( \begin{smallmatrix} A & 0 \\ 0 & 0 \end{smallmatrix} \bigr) &:
f(Z) \mapsto - \tr (AZ \partial + 2A) f,  \\
\rho \bigl( \begin{smallmatrix} 0 & B \\ 0 & 0 \end{smallmatrix} \bigr) &:
f(Z) \mapsto - \tr (B \partial) f,  \\
\rho \bigl( \begin{smallmatrix} 0 & 0 \\ C & 0 \end{smallmatrix} \bigr) &:
f(Z) \mapsto \tr (ZCZ \partial + 4CZ) f,  \\
\rho \bigl( \begin{smallmatrix} 0 & 0 \\ 0 & D \end{smallmatrix} \bigr) &:
f(Z) \mapsto \tr (ZD \partial + 2D) f.
\end{align*}
\end{lem}

\begin{lem}
The Lie algebra action $\rho'$ of $\mathfrak{gl}(2,\HC)$ on $\Sh'$ is given by
\begin{align*}
\rho' \bigl( \begin{smallmatrix} A & 0 \\ 0 & 0 \end{smallmatrix} \bigr) &:
f(Z) \mapsto - \tr (AZ \partial) f,  \\
\rho' \bigl( \begin{smallmatrix} 0 & B \\ 0 & 0 \end{smallmatrix} \bigr) &:
f(Z) \mapsto - \tr (B \partial) f,  \\
\rho' \bigl( \begin{smallmatrix} 0 & 0 \\ C & 0 \end{smallmatrix} \bigr) &:
f(Z) \mapsto \tr (ZCZ \partial) f,  \\
\rho' \bigl( \begin{smallmatrix} 0 & 0 \\ 0 & D \end{smallmatrix} \bigr) &:
f(Z) \mapsto \tr (ZD \partial) f.
\end{align*}
\end{lem}

\begin{lem}
The Lie algebra action $\rho_2$ of $\mathfrak{gl}(2,\HC)$ on ${\cal W}$
is given by
\begin{align*}
\rho_2 \bigl( \begin{smallmatrix} A & 0 \\ 0 & 0 \end{smallmatrix} \bigr) &:
F(Z) \mapsto - \tr (AZ \partial + A) F -FA,  \\
\rho_2 \bigl( \begin{smallmatrix} 0 & B \\ 0 & 0 \end{smallmatrix} \bigr) &:
F(Z) \mapsto - \tr (B \partial) F,  \\
\rho_2 \bigl( \begin{smallmatrix} 0 & 0 \\ C & 0 \end{smallmatrix} \bigr) &:
F(Z) \mapsto \tr (ZCZ \partial +2ZC) F + CZF + FZC,  \\
\rho_2 \bigl( \begin{smallmatrix} 0 & 0 \\ 0 & D \end{smallmatrix} \bigr) &:
F(Z) \mapsto \tr (ZD \partial +D) F +DF.
\end{align*}
\end{lem}

\begin{lem}  \label{rho'_2-action-lem}
The Lie algebra action $\rho'_2$ of $\mathfrak{gl}(2,\HC)$ on ${\cal W}'$
is given by
\begin{align*}
\rho'_2 \bigl( \begin{smallmatrix} A & 0 \\ 0 & 0 \end{smallmatrix} \bigr) &:
F(Z) \mapsto - \tr (AZ \partial + A) F +AF,  \\
\rho'_2 \bigl( \begin{smallmatrix} 0 & B \\ 0 & 0 \end{smallmatrix} \bigr) &:
F(Z) \mapsto - \tr (B \partial) F,  \\
\rho'_2 \bigl( \begin{smallmatrix} 0 & 0 \\ C & 0 \end{smallmatrix} \bigr) &:
F(Z) \mapsto \tr (ZCZ \partial +2ZC) F - ZCF - FCZ,  \\
\rho'_2 \bigl( \begin{smallmatrix} 0 & 0 \\ 0 & D \end{smallmatrix} \bigr) &:
F(Z) \mapsto \tr (ZD \partial +D) F -FD.
\end{align*}
\end{lem}

Next, we show that the maps in the quaternionic chain complex
(\ref{W-sequence}) are $\g{gl}(2,\HC)$-equivariant.

\begin{prop}  \label{equiv-prop1}
The map $\partial^+ : (\rho', \Sh') \to (\rho'_2, {\cal W}')$
in (\ref{W-sequence}) is $\g{gl}(2,\HC)$-equivariant.
\end{prop}

\begin{proof}
For $f(Z) \in \Sh'$, by direct computation we obtain:
\begin{multline*}
- \partial^+ \tr ( A Z \partial) f
= - \tr (A Z \partial) \partial^+ f
+ \begin{pmatrix} -a_{12}\partial_{12}-a_{22}\partial_{22} &
a_{12}\partial_{11}+a_{22}\partial_{21} \\ a_{11}\partial_{12}+a_{21}\partial_{22} &
-a_{11}\partial_{11}-a_{21}\partial_{21} \end{pmatrix} f  \\
= - \tr (A Z \partial) \partial^+ f - \tr(A) \partial^+ f + A \partial^+ f,
\end{multline*}
\begin{multline*}
\partial^+ \tr (ZCZ \partial) f = \tr (ZCZ \partial) \partial^+ f \\
+ \begin{pmatrix}
\begin{smallmatrix} c_{21}z_{11}\partial_{21}+c_{22}z_{21}\partial_{21}
+c_{21}z_{12}\partial_{22}+2c_{22}z_{22}\partial_{22} \\
+c_{12}z_{11}\partial_{12}+c_{12}z_{21}\partial_{22}+c_{22}z_{12}\partial_{12}
\end{smallmatrix} &
\begin{smallmatrix} -c_{11}z_{11}\partial_{21}-2c_{12}z_{21}\partial_{21}
-c_{11}z_{12}\partial_{22}-c_{12}z_{22}\partial_{22} \\
-c_{12}z_{11}\partial_{11}-c_{22}z_{12}\partial_{11}-c_{22}z_{22}\partial_{21}
\end{smallmatrix} \\ \\
\begin{smallmatrix} -c_{21}z_{11}\partial_{11}-c_{22}z_{21}\partial_{11}
-2c_{21}z_{12}\partial_{12}-c_{22}z_{22}\partial_{12} \\
-c_{11}z_{11}\partial_{12}-c_{11}z_{21}\partial_{22}-c_{21}z_{22}\partial_{22}
\end{smallmatrix} &
\begin{smallmatrix} 2c_{11}z_{11}\partial_{11}+c_{12}z_{21}\partial_{11}
+c_{11}z_{12}\partial_{12}+c_{12}z_{22}\partial_{12} \\
+c_{11}z_{21}\partial_{21}+c_{21}z_{12}\partial_{11}+c_{21}z_{22}\partial_{21}
\end{smallmatrix} \end{pmatrix} f  \\
= \tr (ZCZ \partial) \partial^+ f + 2\tr(ZC) \partial^+ f - ZC\partial^+ f
-(\partial^+ f)CZ.
\end{multline*}
The calculations showing that $\partial^+$ intertwines the actions of
$\bigl(\begin{smallmatrix} 0 & B \\ 0 & 0 \end{smallmatrix}\bigr)$ and
$\bigl(\begin{smallmatrix} 0 & 0 \\ 0 & D \end{smallmatrix}\bigr)$ are similar.
\end{proof}

\begin{prop}
The map $\M : (\rho'_2, {\cal W}') \to (\rho_2, {\cal W})$
in (\ref{W-sequence}) is $\g{gl}(2,\HC)$-equivariant.
\end{prop}

\begin{proof}
Note that $\M F = 4(\partial F \partial - \partial \partial^+ F^+)$,
where $F(Z) \in {\cal W}'$. Using
\begin{equation}  \label{bracket1}
\partial [ \tr(AZ\partial)F] = \tr(AZ\partial) (\partial F) + \partial AF,
\qquad
\partial^+ [ \tr(AZ\partial)F] = \tr(AZ\partial) (\partial^+ F)
+ A^+ \partial^+ F,
\end{equation}
we obtain
\begin{multline*}
\frac14 \M \bigl( - \tr (AZ \partial + A) F +AF \bigr)  \\
= - \bigl[ \tr(AZ\partial) (\partial F) \bigr] \overleftarrow{\partial}
+ \partial \bigl[ \tr(AZ\partial)(\partial^+ F^+) + A^+ \partial^+ F^+ \bigr]
+ \tr(A)\partial\partial^+ F^+ - \partial\partial^+ F^+ A^+ \\
= - \tr (AZ \partial + A) (\partial F \partial) - (\partial F \partial)A
+ \tr (AZ \partial) (\partial \partial^+ F^+) + \partial A \partial^+ F^+
+ \partial A^+ \partial^+ F^+ + \partial\partial^+ F^+ A  \\
= - \tr (AZ \partial + A) (\partial F \partial - \partial \partial^+ F^+)
- (\partial F \partial - \partial \partial^+ F^+)A;
\end{multline*}
$$
- \M \tr (B \partial) F = - \tr ( B \partial) \M F;
$$
similarly, using
\begin{equation}  \label{bracket2}
\partial [ \tr(ZD\partial)F] = \tr(ZD\partial) (\partial F) + D\partial F,
\qquad
\partial^+ [ \tr(ZD\partial)F] = \tr(ZD\partial) (\partial^+ F)
+ \partial^+ D^+ F,
\end{equation}
we obtain
\begin{multline*}
\frac14 \M \bigl( \tr (ZD \partial +D) F -FD \bigr)
= \bigl[ \tr(ZD\partial +D) (\partial F) + D\partial F - \partial FD \bigr]
\overleftarrow{\partial} \\
- \partial \bigl[ \tr(ZD\partial) (\partial^+ F^+)
+ \partial^+ D^+ F^+ \bigr] - \tr(D)\partial\partial^+ F^+
+ D^+ \partial\partial^+ F^+ \\
= \tr(ZD\partial +D) (\partial F \partial) + D (\partial F \partial)
- \tr (ZD \partial) (\partial \partial^+ F^+) - D \partial \partial^+ F^+
- \tr(D)\partial\partial^+ F^+ \\
= \tr (ZD \partial + D) (\partial F \partial - \partial \partial^+ F^+)
+ D(\partial F \partial - \partial \partial^+ F^+);
\end{multline*}
finally, using
\begin{multline}  \label{del-C}
\partial [\tr (ZCZ \partial +2ZC) F - ZCF - FCZ ]  \\
= \tr (ZCZ \partial +2ZC) (\partial F) + CZ (\partial F) - (\partial F)CZ
- \tr(FC),
\end{multline}
we obtain
\begin{multline*}
\frac14 \M \bigl( \tr (ZCZ \partial +2ZC) F - ZCF - FCZ \bigr)
= \tr (ZCZ \partial +2ZC) (\partial F \partial) + CZ(\partial F \partial)
+ (\partial F \partial)ZC \\ + C\tr(\partial F) - \partial \tr(FC)
- \bigl( \partial^+ [ \tr (ZCZ \partial +2ZC) (\partial F) + CZ \partial F
- (\partial F)CZ - \tr(FC) ] \bigr)^+ \\
= \tr (ZCZ \partial +2ZC) (\partial F \partial) + CZ(\partial F \partial)
+ (\partial F \partial)ZC + C\tr(\partial F) - \partial \tr(FC)  \\
- \tr (ZCZ \partial +2ZC) (\partial^+\partial F^+)
- (\partial^+\partial F^+)ZC - CZ (\partial^+\partial F^+) - F^+ \partial^+ C
- \partial FC + \partial \tr(FC)  \\
= \tr (ZCZ \partial +2ZC) (\partial F \partial - \partial \partial^+ F^+)
+ CZ (\partial F \partial - \partial \partial^+ F^+)
+ (\partial F \partial - \partial \partial^+ F^+) ZC.
\end{multline*}
\end{proof}

\begin{prop}
The map $\tr \circ \partial^+ : (\rho_2, {\cal W}) \to (\rho, \Sh)$
in (\ref{W-sequence}) is $\g{gl}(2,\HC)$-equivariant.
\end{prop}

\begin{proof}
For $F(Z) \in {\cal W}$, by direct computation we obtain:
\begin{multline*}
-\tr \bigl( \partial^+ \tr(AZ \partial) F + \partial^+ \tr(A) F
+ (\partial^+ F)A \bigr)
= -\tr(AZ \partial) \tr(\partial^+ F) - (\tr A) \tr(\partial^+ F) \\
- \tr(A \partial^+ F)
-\tr \biggl[ \begin{pmatrix} a_{12}\partial_{12}+a_{22}\partial_{22} &
-a_{12}\partial_{11}-a_{22}\partial_{21} \\ -a_{11}\partial_{12}-a_{21}\partial_{22} &
a_{11}\partial_{11}+a_{21}\partial_{21} \end{pmatrix} F \biggr]
= -\tr(AZ \partial +2A) \tr(\partial^+ F),
\end{multline*}
\begin{multline*}
\tr \bigl( \partial^+ \tr (ZCZ \partial +2ZC) F + \partial^+ CZF
+ \partial^+ FZC \bigr)  \\
= \tr (ZCZ \partial +2ZC) \tr(\partial^+ F)
+ \tr \bigl( CZ (F\partial^+)  + ZC \partial^+ F \bigr)  \\
+ \tr \left[ \begin{pmatrix}
\begin{smallmatrix} c_{21}z_{11}\partial_{21}+c_{22}z_{21}\partial_{21}
+c_{21}z_{12}\partial_{22}+2c_{22}z_{22}\partial_{22} \\
+c_{12}z_{11}\partial_{12}+c_{12}z_{21}\partial_{22}+c_{22}z_{12}\partial_{12}
\end{smallmatrix} &
\begin{smallmatrix} -c_{11}z_{11}\partial_{21}-2c_{12}z_{21}\partial_{21}
-c_{11}z_{12}\partial_{22}-c_{12}z_{22}\partial_{22} \\
-c_{12}z_{11}\partial_{11}-c_{22}z_{12}\partial_{11}-c_{22}z_{22}\partial_{21}
\end{smallmatrix} \\ \\
\begin{smallmatrix} -c_{21}z_{11}\partial_{11}-c_{22}z_{21}\partial_{11}
-2c_{21}z_{12}\partial_{12}-c_{22}z_{22}\partial_{12} \\
-c_{11}z_{11}\partial_{12}-c_{11}z_{21}\partial_{22}-c_{21}z_{22}\partial_{22}
\end{smallmatrix} &
\begin{smallmatrix} 2c_{11}z_{11}\partial_{11}+c_{12}z_{21}\partial_{11}
+c_{11}z_{12}\partial_{12}+c_{12}z_{22}\partial_{12} \\
+c_{11}z_{21}\partial_{21}+c_{21}z_{12}\partial_{11}+c_{21}z_{22}\partial_{21}
\end{smallmatrix} \end{pmatrix} F \right]  \\
= \tr(ZCZ \partial + 4ZC) \tr(\partial^+ F).
\end{multline*}
The calculations showing that $\tr \circ \partial^+$ intertwines the actions of
$\bigl(\begin{smallmatrix} 0 & B \\ 0 & 0 \end{smallmatrix}\bigr)$ and
$\bigl(\begin{smallmatrix} 0 & 0 \\ 0 & D \end{smallmatrix}\bigr)$ are similar.
\end{proof}

We have another equivariant map that does not appear in (\ref{W-sequence}):

\begin{prop}  \label{sqsq-equiv}
The map $\square \circ \square : (\rho', \Sh') \to (\rho, \Sh)$
is $\g{gl}(2,\HC)$-equivariant.
\end{prop}

\begin{proof}
Note that $\square = 4 \partial \partial^+ = 4 \partial^+ \partial$.
For $f(Z) \in \Sh'$, using (\ref{bracket1}) and (\ref{bracket2}), we obtain:
$$
- \frac14 \square [\tr (AZ \partial) f]
= - \partial [ \tr(AZ\partial) (\partial^+ f) - A^+ \partial^+ f]
= - \tr(AZ\partial) (\partial \partial^+ f) - \tr(A) \partial \partial^+ f,
$$
$$
- (\square \circ \square) [\tr (AZ \partial) f]
= - \tr(AZ\partial +2A) (\square \circ \square f);
$$
\begin{multline*}
\frac14 \square [\tr (ZCZ \partial) f]
= \partial [ \tr(ZCZ \partial) (\partial^+ f)
+ (\partial^+ f) Z^+C^+ + C^+Z^+ \partial^+ f]  \\
= \tr(ZCZ \partial) (\partial \partial^+ f) + CZ (\partial \partial^+ f)
+ \partial[ZC \partial^+ f] - 2C \partial^+ f
+ \partial [ (\partial^+ f) Z^+C^+ + C^+Z^+ \partial^+ f]  \\
= \tr(ZCZ \partial) (\partial \partial^+ f)
+ (CZ + Z^+C^+) (\partial \partial^+ f)
+ \partial[ZC \partial^+ f + C^+Z^+ \partial^+ f] - 2C \partial^+ f
- (\partial f) C^+ \\
= \tr(ZCZ \partial + 2CZ) (\partial \partial^+ f) - \tr(C^+\partial)f,
\end{multline*}
$$
(\square \circ \square) [\tr (ZCZ \partial) f]
= \tr(ZCZ \partial + 4CZ) (\square \circ \square f).
$$
The calculations showing that $\square \circ \square$ intertwines the actions of
$\bigl(\begin{smallmatrix} 0 & B \\ 0 & 0 \end{smallmatrix}\bigr)$ and
$\bigl(\begin{smallmatrix} 0 & 0 \\ 0 & D \end{smallmatrix}\bigr)$ are similar.
\end{proof}

\subsection{Decomposition of $(\rho,\Sh)$ and $(\rho',\Sh')$ into
Irreducible Components}

Similarly to how we proved Theorem 8 in \cite{L}, we can obtain the
following two decomposition results.

\begin{thm}  \label{rho-decomposition}
The only proper $\mathfrak{gl}(2,\HC)$-invariant subspaces of $(\rho,\Sh)$ are
\begin{align*}
\Sh^+ &= \BB C\text{-span of }
\bigl\{ N(Z)^k \cdot t^l_{n\,\underline{m}}(Z);\: k \ge 0 \bigr\}, \\
\Sh^- &= \BB C\text{-span of }
\bigl\{ N(Z)^k \cdot t^l_{n\,\underline{m}}(Z);\: k \le -(2l+4) \bigr\}, \\
{\cal I}^+ &= \BB C\text{-span of }
\bigl\{ N(Z)^k \cdot t^l_{n\,\underline{m}}(Z);\: k \ge -(2l+1) \bigr\}, \\
{\cal I}^- &= \BB C\text{-span of }
\bigl\{ N(Z)^k \cdot t^l_{n\,\underline{m}}(Z);\: k \le -3 \bigr\}, \\
{\cal J} &= \BB C\text{-span of }
\bigl\{ N(Z)^k \cdot t^l_{n\,\underline{m}}(Z);\: -(2l+1) \le k \le -3 \bigr\}
\end{align*}
and their sums (see Figure \ref{decomposition-fig1}).

The irreducible components of $(\rho,\Sh)$ are the subrepresentations
$$
(\rho, \Sh^+), \qquad (\rho, \Sh^-), \qquad (\rho, {\cal J})
$$
and the quotients
\begin{equation*}  
\bigl( \rho, {\cal I}^+/(\Sh^+ \oplus {\cal J}) \bigr), \quad
\bigl( \rho, {\cal I}^-/(\Sh^- \oplus {\cal J}) \bigr), \quad
\bigl( \rho, \Sh/({\cal I}^++{\cal I}^-) \bigr)
\end{equation*}
(see Figure \ref{decomposition-fig1-comp}).
\end{thm}

\begin{figure}
\begin{center}
\setlength{\unitlength}{1mm}
\begin{picture}(120,70)
\multiput(10,10)(10,0){11}{\circle*{1}}
\multiput(10,20)(10,0){11}{\circle*{1}}
\multiput(10,30)(10,0){11}{\circle*{1}}
\multiput(10,40)(10,0){11}{\circle*{1}}
\multiput(10,50)(10,0){11}{\circle*{1}}
\multiput(10,60)(10,0){11}{\circle*{1}}

\thicklines
\put(80,0){\vector(0,1){70}}
\put(0,10){\vector(1,0){120}}

\thinlines
\put(78,10){\line(0,1){55}}
\put(80,8){\line(1,0){35}}
\qbezier(78,10)(78,8)(80,8)

\put(52,30){\line(0,1){35}}
\put(48.6,28.6){\line(-1,1){36.4}}
\qbezier(52,30)(52,25.2)(48.6,28.6)

\put(5,8){\line(1,0){35}}
\put(41.4,11.4){\line(-1,1){36.4}}
\qbezier(40,8)(44.8,8)(41.4,11.4)

\put(73,7){\line(1,0){42}}
\put(66,10){\line(-1,1){55}}
\qbezier(73,7)(69,7)(66,10)

\put(5,7){\line(1,0){45}}
\put(53,10){\line(0,1){55}}
\qbezier(50,7)(53,7)(53,10)


\put(82,67){$2l$}
\put(117,12){$k$}
\put(3,24){$\Sh^-$}
\put(33,62){${\cal J}$}
\put(112,44){$\Sh^+$}
\put(25,3){${\cal I}^-$}
\put(95,3){${\cal I}^+$}
\end{picture}
\end{center}
\caption{Decomposition of $(\rho,\Sh)$ into irreducible components.}
\label{decomposition-fig1}
\end{figure}
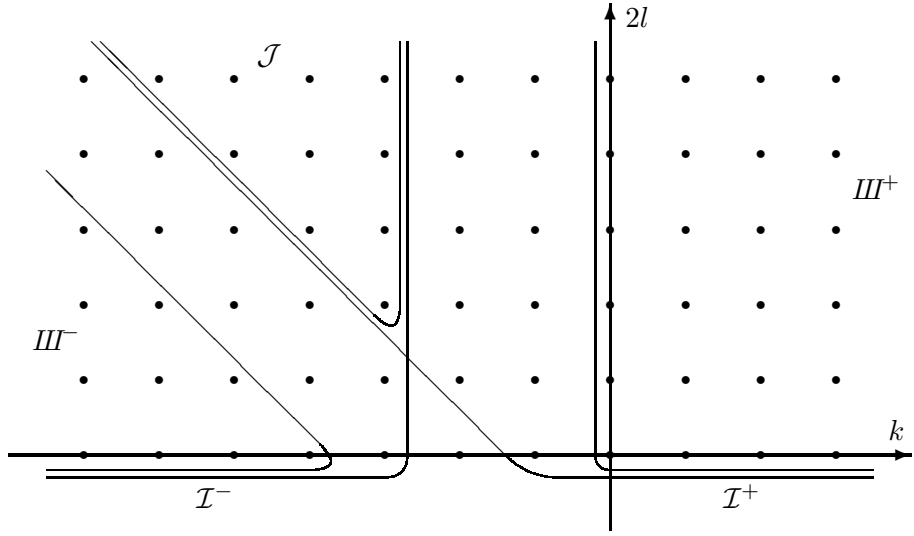

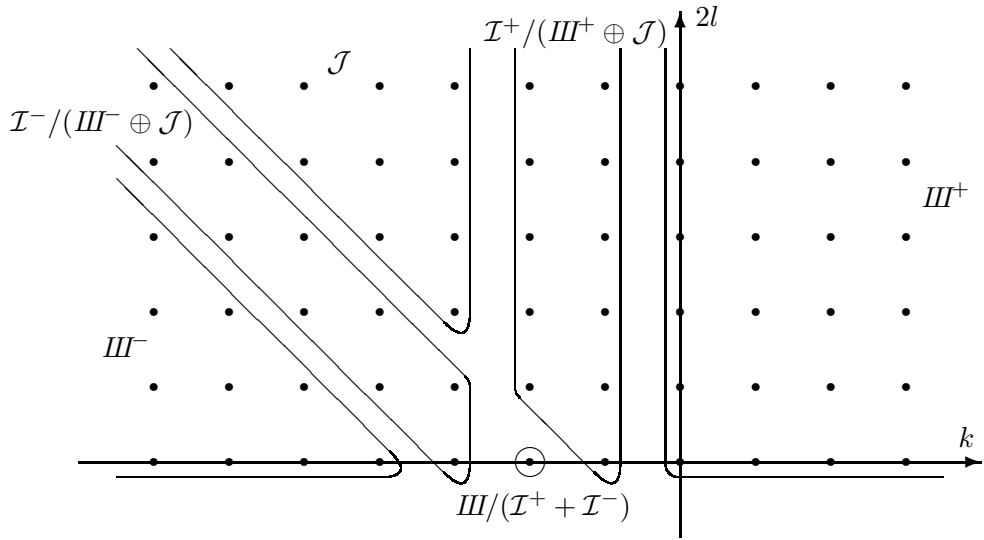
\begin{figure}
\begin{center}
\setlength{\unitlength}{1mm}
\begin{picture}(120,70)
\multiput(10,10)(10,0){11}{\circle*{1}}
\multiput(10,20)(10,0){11}{\circle*{1}}
\multiput(10,30)(10,0){11}{\circle*{1}}
\multiput(10,40)(10,0){11}{\circle*{1}}
\multiput(10,50)(10,0){11}{\circle*{1}}
\multiput(10,60)(10,0){11}{\circle*{1}}

\thicklines
\put(80,0){\vector(0,1){70}}
\put(0,10){\vector(1,0){120}}

\thinlines

\put(60,10){\circle{4}}

\put(72,10){\line(0,1){55}}
\put(58,20){\line(0,1){45}}
\put(68.6,8.6){\line(-1,1){10}}
\qbezier(72,10)(72,5.2)(68.6,8.6)
\qbezier(58,20)(58,19.2)(58.6,18.6)

\put(52,10){\line(0,1){10}}
\put(48.6,8.6){\line(-1,1){43.6}}
\put(51.4,21.4){\line(-1,1){43.6}}
\qbezier(52,10)(52,5.2)(48.6,8.6)
\qbezier(52,20)(52,20.8)(51.4,21.4)

\put(78,10){\line(0,1){55}}
\put(80,8){\line(1,0){35}}
\qbezier(78,10)(78,8)(80,8)

\put(52,30){\line(0,1){35}}
\put(48.6,28.6){\line(-1,1){36.4}}
\qbezier(52,30)(52,25.2)(48.6,28.6)

\put(5,8){\line(1,0){35}}
\put(41.4,11.4){\line(-1,1){36.3}}
\qbezier(40,8)(44.8,8)(41.4,11.4)

\put(82,68){$2l$}
\put(117,12){$k$}
\put(3,24){$\Sh^-$}
\put(33,62){${\cal J}$}
\put(112,44){$\Sh^+$}

\put(54,66){${\cal I}^+/(\Sh^+ \oplus {\cal J})$}
\put(-9,54){${\cal I}^-/(\Sh^- \oplus {\cal J})$}

\put(50,3){$\Sh/({\cal I}^++{\cal I}^-)$}
\end{picture}
\end{center}
\caption{Irreducible components of $(\rho,\Sh)$.}
\label{decomposition-fig1-comp}
\end{figure}

\begin{thm}  \label{rho'-decomposition}
The only proper $\mathfrak{gl}(2,\HC)$-invariant subspaces of $(\rho',\Sh')$ are
\begin{align*}
{\cal I}'_0 &= \BB C = \BB C\text{-span of }
\bigl\{ N(Z)^0 \cdot t^0_{0\,\underline{0}}(Z) \bigr\}, \\
{\cal BH}^+ &= \BB C\text{-span of }
\bigl\{ N(Z)^k \cdot t^l_{n\,\underline{m}}(Z);\: 0 \le k \le 1 \bigr\}, \\
{\cal BH}^- &= \BB C\text{-span of }
\bigl\{ N(Z)^k \cdot t^l_{n\,\underline{m}}(Z);\: -1 \le 2l+k \le 0 \bigr\}, \\
\Sh^+ &= \BB C\text{-span of }
\bigl\{ N(Z)^k \cdot t^l_{n\,\underline{m}}(Z);\: k \ge 0 \bigr\}, \\
\Sh'^- &= \BB C\text{-span of }
\bigl\{ N(Z)^k \cdot t^l_{n\,\underline{m}}(Z);\: k \le -2l \bigr\}, \\
{\cal I}'^+ &= \BB C\text{-span of }
\bigl\{ N(Z)^k \cdot t^l_{n\,\underline{m}}(Z);\: k \ge -(2l+1) \bigr\}, \\
{\cal I}'^- &= \BB C\text{-span of }
\bigl\{ N(Z)^k \cdot t^l_{n\,\underline{m}}(Z);\: k \le 1 \bigr\}, \\
{\cal J}' &= \BB C\text{-span of }
\bigl\{ N(Z)^k \cdot t^l_{n\,\underline{m}}(Z);\: -(2l+1) \le k \le 1 \bigr\}
\end{align*}
and their sums (see Figure \ref{decomposition-fig2}).

The irreducible components of $(\rho',\Sh')$ are the trivial
subrepresentation $(\rho', {\cal I}'_0)$ and the quotients
$$
(\rho', {\cal BH}^+/{\cal I}'_0), \quad (\rho', {\cal BH}^-/{\cal I}'_0),
$$
$$
(\rho', \Sh^+/{\cal BH}^+) = (\rho',\Sh/{\cal I}'^-), \quad
(\rho', \Sh'^-/{\cal BH}^-) = (\rho',\Sh/{\cal I}'^+),
$$
$$
\bigl( \rho', \Sh/({\cal I}'^++{\cal I}'^-) \bigr) =
\bigl( \rho', {\cal I}'^+/(\Sh^++{\cal BH}^-) \bigr) =
\bigl( \rho', {\cal I}'^-/(\Sh'^-+{\cal BH}^+) \bigr) =
\bigl( \rho', {\cal J}'/({\cal BH}^++{\cal BH}^-) \bigr)
$$
(see Figure \ref{decomposition-fig3},
which is essentially a shifted Figure \ref{decomposition-fig1-comp}).
\end{thm}

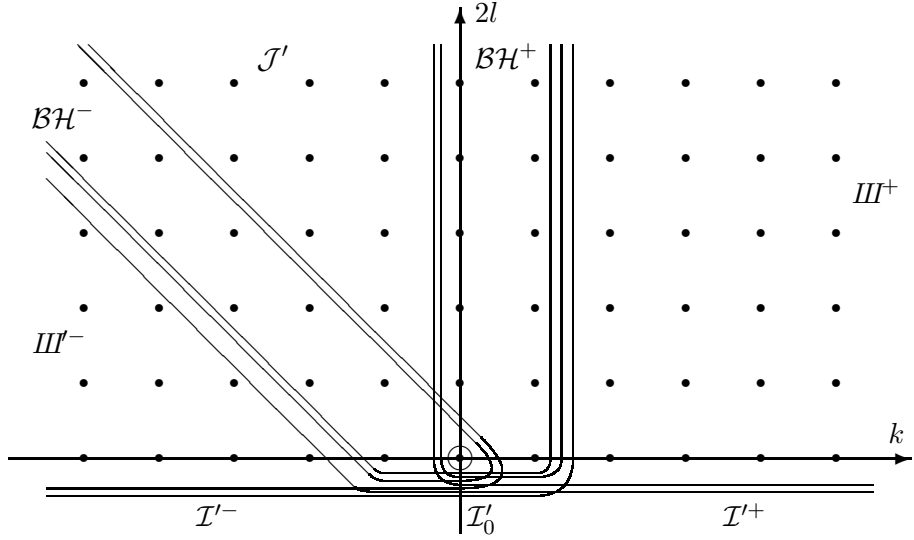
\begin{figure}
\begin{center}
\setlength{\unitlength}{1mm}
\begin{picture}(120,70)
\multiput(10,10)(10,0){11}{\circle*{1}}
\multiput(10,20)(10,0){11}{\circle*{1}}
\multiput(10,30)(10,0){11}{\circle*{1}}
\multiput(10,40)(10,0){11}{\circle*{1}}
\multiput(10,50)(10,0){11}{\circle*{1}}
\multiput(10,60)(10,0){11}{\circle*{1}}

\thicklines
\put(60,0){\vector(0,1){70}}
\put(0,10){\vector(1,0){120}}

\thinlines

\put(60,10){\circle{3}}

\put(50,8){\line(1,0){20}}
\put(72,10){\line(0,1){55}}
\put(48.6,8.6){\line(-1,1){43.6}}
\qbezier(48.6,8.6)(49.2,8)(50,8)
\qbezier(72,10)(72,8)(70,8)

\put(57.5,10){\line(0,1){55}}
\put(73.5,10){\line(0,1){55}}
\put(60,7.5){\line(1,0){10}}
\qbezier(57.5,10)(57.5,7.5)(60,7.5)
\qbezier(73.5,10)(73.5,7.5)(70,7.5)

\put(56.5,10){\line(0,1){55}}
\put(60,6.5){\line(1,0){55}}
\qbezier(56.5,10)(56.5,6.5)(60,6.5)

\put(50,7){\line(1,0){10}}
\put(62.1,12.1){\line(-1,1){52.9}}
\put(47.9,7.9){\line(-1,1){42.8}}
\qbezier(60,7)(67.2,7)(62.1,12.1)
\qbezier(47.9,7.9)(48.8,7)(50,7)

\put(5,6){\line(1,0){55}}
\put(62.8,12.8){\line(-1,1){52.2}}
\qbezier(60,6)(69.6,6)(62.8,12.8)

\put(50,5.5){\line(1,0){65}}
\put(46.1,6.1){\line(-1,1){41.1}}
\qbezier(46.1,6.1)(46.7,5.5)(50,5.5)

\put(5,5){\line(1,0){65}}
\put(75,10){\line(0,1){55}}
\qbezier(70,5)(75,5)(75,10)

\put(62,68){$2l$}
\put(117,12){$k$}
\put(62,62){${\cal BH}^+$}
\put(3,54){${\cal BH}^-$}
\put(3,24){$\Sh'^-$}
\put(33,62){${\cal J}'$}
\put(61,1){${\cal I}'_0$}
\put(112,44){$\Sh^+$}
\put(25,1){${\cal I}'^-$}
\put(95,1){${\cal I}'^+$}
\end{picture}
\end{center}
\caption{Decomposition of $(\rho',\Sh')$ into irreducible components.}
\label{decomposition-fig2}
\end{figure}

\begin{figure}
\begin{center}
\setlength{\unitlength}{1mm}
\begin{picture}(120,70)
\multiput(10,10)(10,0){11}{\circle*{1}}
\multiput(10,20)(10,0){11}{\circle*{1}}
\multiput(10,30)(10,0){11}{\circle*{1}}
\multiput(10,40)(10,0){11}{\circle*{1}}
\multiput(10,50)(10,0){11}{\circle*{1}}
\multiput(10,60)(10,0){11}{\circle*{1}}

\thicklines
\put(60,0){\vector(0,1){70}}
\put(0,10){\vector(1,0){120}}

\thinlines

\put(60,10){\circle{4}}

\put(72,10){\line(0,1){55}}
\put(58,20){\line(0,1){45}}
\put(68.6,8.6){\line(-1,1){10}}
\qbezier(72,10)(72,5.2)(68.6,8.6)
\qbezier(58,20)(58,19.2)(58.6,18.6)

\put(52,10){\line(0,1){10}}
\put(48.6,8.6){\line(-1,1){43.6}}
\put(51.4,21.4){\line(-1,1){43.6}}
\qbezier(52,10)(52,5.2)(48.6,8.6)
\qbezier(52,20)(52,20.8)(51.4,21.4)

\put(78,10){\line(0,1){55}}
\put(80,8){\line(1,0){35}}
\qbezier(78,10)(78,8)(80,8)

\put(52,30){\line(0,1){35}}
\put(48.6,28.6){\line(-1,1){36.4}}
\qbezier(52,30)(52,25.2)(48.6,28.6)

\put(5,8){\line(1,0){35}}
\put(41.4,11.4){\line(-1,1){36.3}}
\qbezier(40,8)(44.8,8)(41.4,11.4)

\put(62,68){$2l$}
\put(117,12){$k$}
\put(63,3){${\cal BH}^+/{\cal I}'_0$}
\put(-4,54){${\cal BH}^-/{\cal I}'_0$}
\put(-4,24){$\Sh'^-/{\cal BH}^-$}
\put(18,64){${\cal J}'/({\cal BH}^++{\cal BH}^-)$}
\put(55,5){${\cal I}'_0$}
\put(110,44){$\Sh^+/{\cal BH}^+$}
\end{picture}
\end{center}
\caption{Irreducible components of $(\rho',\Sh')$.}
\label{decomposition-fig3}
\end{figure}
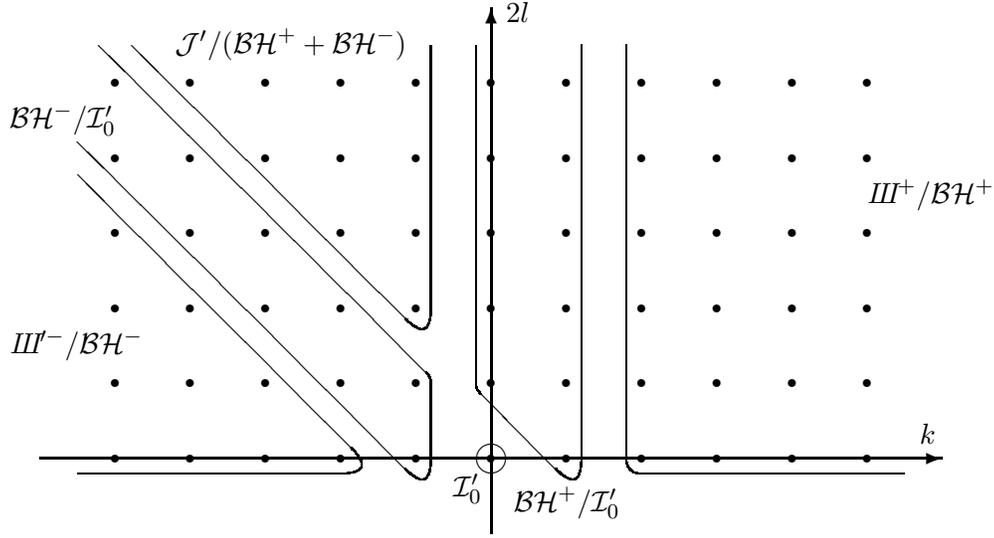

\begin{cor}  \label{sq-sq-im-cor}
The image under the $\g{gl}(2,\HC)$-equivariant map
$\square \circ \square : (\rho', \Sh') \to (\rho, \Sh)$
from Proposition \ref{sqsq-equiv} is $\Sh^+ \oplus {\cal J} \oplus \Sh^-$;
this map provides isomorphisms
$$
(\rho', \Sh^+/{\cal BH}^+) \simeq (\rho, \Sh^+), \quad
\bigl( \rho', {\cal J}'/({\cal BH}^++{\cal BH}^-) \bigr)
\simeq (\rho, {\cal J}), \quad
(\rho', \Sh'^-/{\cal BH}^-) \simeq (\rho, \Sh^-).
$$
\end{cor}

\begin{proof}
The result follows from Theorems \ref{rho-decomposition},
\ref{rho'-decomposition} and an identity:
\begin{equation}  \label{sq-tN}
\square \bigl( N(Z)^k \cdot t^l_{n\,\underline{m}}(Z) \bigr)
= 4k(2l+k+1) N(Z)^{k-1} \cdot t^l_{n\,\underline{m}}(Z),
\end{equation}
which can be verified by direct computation.
\end{proof}

We call a function $f$ {\em biharmonic} if $(\square \circ \square) f=0$.
Using (\ref{sq-tN}), we can characterize the space of biharmonic functions.

\begin{prop}  \label{biharmonic-prop}
We have:
$$
\{ f \in \Sh ;\: (\square \circ \square) f=0 \} = {\cal BH}^+ + {\cal BH}^-.
$$
In other words, a function $f \in \Sh$ is biharmonic if and only if it can
be written as
$$
f(Z) = h_0(Z)+h_1(Z) \cdot N(Z)
$$
with $h_0$ and $h_1$ harmonic.
\end{prop}


\section{Realization of Doubly Regular Functions in
the Quaternionic Chain Complex}  \label{DRinQCC-section}

In this section we decompose $\ker \M$ -- which is an invariant subspace of
$(\rho'_2, {\cal W}')$ -- into irreducible components. We will see that it has
ten irreducible components: five coming from $(\rho',\Sh')$ under the map
$\partial^+$, one trivial one-dimensional representation and four components
that are isomorphic to the spaces of doubly regular maps mentioned in
Proposition \ref{DR-irred-prop}.
Results of this section will be later used in
Subsection \ref{W'-irred-decomp-subsection} to decompose $(\rho'_2, {\cal W}')$
into irreducible components.

\subsection{The Structure of $\ker \M \subset {\cal W}'$}

We introduce a $\g{gl}(2,\HC)$-equivariant map on $\ker \M$;
its kernel is automatically an invariant subspace of $\ker \M$.

\begin{prop}  \label{tr-d-sq-prop}
The map $\tr \circ \partial \circ \square: (\rho'_2, \ker \M) \to (\rho, \Sh)$
is $\g{gl}(2,\HC)$-equivariant.
\end{prop}

\begin{proof}
Recall that the matrices
$\bigl(\begin{smallmatrix} 1 & 0 \\ 0 & 0 \end{smallmatrix}\bigr)$,
$\bigl(\begin{smallmatrix} 0 & B \\ 0 & 0 \end{smallmatrix}\bigr)$ and
$\bigl(\begin{smallmatrix} 0 & 0 \\ C & 0 \end{smallmatrix}\bigr)
\in \mathfrak{gl}(2,\HC)$, $B,C \in \HC$, generate $\mathfrak{gl}(2,\HC)$.
Thus, it is sufficient to show that the map $\tr \circ \partial \circ \square$
commutes with actions of dilation matrices
$\bigl(\begin{smallmatrix} \lambda & 0 \\ 0 & 1 \end{smallmatrix}\bigr)
\in GL(2,\HC)$, $\lambda \in \BB R$, $\lambda>0$ and
$\bigl(\begin{smallmatrix} 0 & B \\ 0 & 0 \end{smallmatrix}\bigr),
\bigl(\begin{smallmatrix} 0 & 0 \\ C & 0 \end{smallmatrix}\bigr)
\in \mathfrak{gl}(2,\HC)$, $B, C \in \HC$.
It is clear from Lemmas \ref{rho-action-lem} and \ref{rho'_2-action-lem} that
$\tr \circ \partial \circ \square$ commutes with the actions of
$\bigl(\begin{smallmatrix} 0 & B \\ 0 & 0 \end{smallmatrix}\bigr)$.
The dilation matrices
$\bigl(\begin{smallmatrix} \lambda & 0 \\ 0 & 1 \end{smallmatrix}\bigr)$
act by
$$
F(Z) \mapsto \lambda^{-1} \cdot F(\lambda^{-1} Z)
\quad \text{on ${\cal W}'$ \qquad and \qquad}
f(Z) \mapsto \lambda^{-4} \cdot f(\lambda^{-1} Z) \quad \text{on $\Sh$},
$$
and it is clear that $\tr \circ \partial \circ \square$
commutes with these actions as well.

If $F(Z) \in \ker \M$, using our previous calculations (\ref{del-C}),
we obtain:
$$
\tr \partial \bigl( \rho'_2
\bigl(\begin{smallmatrix} 0 & 0 \\ C & 0 \end{smallmatrix}\bigr) F \bigr)
= \tr (ZCZ \partial +2ZC) (\tr \partial F) - 2 \tr (CF).
$$
Then we apply
$\square = 4\partial \partial^+ = 4\bigl(
\frac{\partial^2}{\partial z_{11} \partial z_{22}} -
\frac{\partial^2}{\partial z_{12} \partial z_{21}} \bigr)$:
\begin{multline*}
\tr \partial \square \bigl( \rho'_2
\bigl(\begin{smallmatrix} 0 & 0 \\ C & 0 \end{smallmatrix}\bigr) F \bigr)
= \tr (ZCZ \partial +2ZC) (\tr \partial \square F) - 2 \tr (C \square F)  \\
- 4 \tr(C\partial^+) \tr(\partial F)
+ \tr(CZ+ZC) (\tr \partial \square F) + 8 \tr(C\partial^+) \tr(\partial F)  \\
= \tr (ZCZ \partial +4CZ) (\tr \partial \square F)
+ 4 \tr(C\partial^+) \tr(\partial F) - 2 \tr (C \square F)  \\
= \rho \bigl(\begin{smallmatrix} 0 & 0 \\ C & 0 \end{smallmatrix}\bigr)
(\tr \partial \square F)
+ 4 \tr \bigl( (C\partial^+ +\partial C^+) \partial F \bigr)
- 2 \tr (C \square F)
= \rho \bigl(\begin{smallmatrix} 0 & 0 \\ C & 0 \end{smallmatrix}\bigr)
(\tr \partial \square F),
\end{multline*}
since $\M F =0$.
\end{proof}

We introduce a subspace ${\cal M}$ of ${\cal W}'$ -- the kernel of the above
equivariant map:
$$
{\cal M} = \bigl\{ F \in {\cal W}';\:
\M F=0, \: \tr ( \partial \circ \square F) =0 \bigr\}.
$$

\begin{cor}  \label{M-biharm-cor}
The subspace ${\cal M}$
is invariant under the $\rho'_2$ action of $\mathfrak{gl}(2,\HC)$ on
${\cal W}'$.
All elements of ${\cal M}$ are biharmonic
(i.e. annihilated by $\square \circ \square$).
\end{cor}

\begin{proof}
The invariance of ${\cal M}$ under the $\rho'_2$ action is immediate
from the above proposition.
Pick any $F \in {\cal M}$.
Since $\tr ( \partial \circ \square F) =0$,
$$
0 = \square \circ \partial F + \square F^+ \partial^+.
$$
And since $\M F =0$,
$$
0 = \bigl( \square \circ \partial F + \square F^+ \partial^+ \bigr) \partial
= \frac12 \square \circ \square F^+.
$$
This proves that every element of ${\cal M}$ is biharmonic.
\end{proof}

\begin{prop}  \label{kerMx-lem}
We have: $\partial^+(\Sh') + {\cal M} = \ker \M$.
\end{prop}

\begin{proof}
Clearly, $\partial^+(\Sh') + {\cal M} \subset \ker \M$.
Thus, it is sufficient to prove that the images of $\partial^+(\Sh')$
and $\ker \M$ under the map $\tr \circ \partial \circ \square$ from
Proposition \ref{tr-d-sq-prop} are the same.
Then, by Corollary \ref{sq-sq-im-cor}, we need to show that
$\tr ( \partial \circ \square (\ker \M)) = \Sh^+ \oplus {\cal J} \oplus \Sh^-$.
And, by Theorem \ref{rho-decomposition} and Proposition \ref{tr-d-sq-prop},
it is sufficient to show that $\tr ( \partial \circ \square ({\cal W}'))$
does not contain $N(Z)^{-1}$ nor $N(Z)^{-3}$. For this purpose we use an identity
\begin{multline}  \label{Zt-identity}
Z \cdot t^l_{n\,\underline{m}}(Z) = \frac1{2l+1} \begin{pmatrix}
(l-n+1) t^{l+\frac12}_{n-\frac12\,\underline{m-\frac12}}(Z) &
(l-n+1) t^{l+\frac12}_{n-\frac12\,\underline{m+\frac12}}(Z) \\
(l+n+1) t^{l+\frac12}_{n+\frac12\,\underline{m-\frac12}}(Z) &
(l+n+1) t^{l+\frac12}_{n+\frac12\,\underline{m+\frac12}}(Z) \end{pmatrix}  \\
+ \frac{N(Z)}{2l+1} \cdot \begin{pmatrix}
(l+m) t^{l-\frac12}_{n-\frac12\,\underline{m-\frac12}}(Z) &
-(l-m) t^{l-\frac12}_{n-\frac12\,\underline{m+\frac12}}(Z) \\
-(l+m) t^{l-\frac12}_{n+\frac12\,\underline{m-\frac12}}(Z) &
(l-m) t^{l-\frac12}_{n+\frac12\,\underline{m+\frac12}}(Z) \end{pmatrix},
\end{multline}
which can be verified using Lemma 23 in \cite{FL1}, and equation (\ref{sq-tN})
to check that
$$
\partial_{ij} \circ \square (\bigl( N(Z)^k \cdot t^l_{n\,\underline{m}}(Z) \bigr),
\qquad i,j=1,2,
$$
is a linear combination of
$$
4k(2l+k+1) N(Z)^{k-1} \cdot t^{l-\frac12}_{n \pm \frac12 \,\underline{m \pm \frac 12}}(Z),
\qquad
4k(2l+k+1) N(Z)^{k-2} \cdot t^{l+\frac12}_{n \pm \frac12 \,\underline{m \pm \frac 12}}(Z).
$$
But none of these terms can be $N(Z)^{-1}$ or $N(Z)^{-3}$.
\end{proof}

The intersection of $\partial^+(\Sh')$ and ${\cal M}$ will be described in
Corollary \ref{XinterA}.
Next, we show that elements of ${\cal M}$ have a particular form; this
will be used to identify the $K$-types of $(\rho'_2,{\cal M})$.

\begin{lem}  \label{elts-of-A}
Let $F_d: \HC^{\times} \to \HC$ be homogeneous of degree $d$ and such that
$\square F_d=0$, then
$$
F(Z)= F_d(Z) + \frac{N(Z)}{d+1} \cdot \bigl( \partial F_d(Z) \partial \bigr)^+
$$
satisfies $\M F=0$ and $\tr (\partial \square F)=0$.
In particular, if $F_d(Z) \in {\cal W}'$, then $F(Z) \in {\cal M}$.
\end{lem}

\begin{proof}
Note that there are no homogeneous harmonic functions of degree $-1$,
so division by $d+1$ is permissible.
First we check that $\M F=0$. Using (\ref{deg-nabla}) and the fact that
$\partial \partial^+ F_d=0$, we obtain:
\begin{multline*}
\frac{d+1}4 \M F = (d+1) \partial F_d \partial
+ \partial \bigl( N(Z) \cdot (\partial F_d \partial)^+ \bigr) \partial
- \partial \partial^+ \bigl( N(Z) \cdot (\partial F_d \partial) \bigr)  \\
= (d+1) \partial F_d \partial
+ \bigl( Z^+ \cdot (\partial F_d \partial)^+ \bigr) \partial
- \partial \bigl( Z \cdot (\partial F_d \partial) \bigr)
= (d+1) \partial F_d \partial - \partial F_d \partial - d \partial F_d \partial
=0.
\end{multline*}
Then we check that $\tr (\partial \square F)=0$:
\begin{multline*}
\frac{d+1}4 \tr (\partial \square F) =
\tr \Bigl( \partial^+ \partial \partial 
\bigl( N(Z) \cdot (\partial F_d \partial)^+ \bigr) \Bigr)  \\
= \tr \Bigl( \partial^+ \partial
\bigl( (\partial F_d \partial)^+ \cdot Z^+\bigr) \Bigr)
= - \tr \bigl( \partial^+ (\partial F_d \partial) \bigr) =0.
\end{multline*}
\end{proof}

\begin{lem}  \label{Mx-calculations}
We have\footnote{In this formula, whenever the indices of, say,
$\alpha_{11} t^l_{n\,\underline{m}}(Z)$ happen to be outside of the allowed range
$l=0,\frac12,1,\frac32,\dots$, $-l \le m,n \le l$, the coefficient $\alpha_{11}$
must be set to be zero. The same considerations apply to other formulas
in this Lemma.}:
\begin{multline}  \label{Mx(A)}
\frac14 \M \begin{pmatrix}
\alpha_{11} t^l_{n\,\underline{m}}(Z) & \alpha_{12} t^l_{n\,\underline{m+1}}(Z) \\
\alpha_{21} t^l_{n+1\,\underline{m}}(Z) & \alpha_{22} t^l_{n+1\,\underline{m+1}}(Z)
\end{pmatrix}
= \partial \begin{pmatrix}
\alpha_{11} t^l_{n\,\underline{m}}(Z) & \alpha_{12} t^l_{n\,\underline{m+1}}(Z) \\
\alpha_{21} t^l_{n+1\,\underline{m}}(Z) & \alpha_{22} t^l_{n+1\,\underline{m+1}}(Z)
\end{pmatrix} \partial  \\
= \bigl( (l-m)(\alpha_{11}+\alpha_{21}) + (l+m+1)(\alpha_{12}+\alpha_{22}) \bigr) \\
\times \begin{pmatrix}
(l-m-1) t^{l-1}_{n+1\,\underline{m+1}}(Z) & (l-m-1) t^{l-1}_{n\,\underline{m+1}}(Z) \\
(l+m) t^{l-1}_{n+1\,\underline{m}}(Z) & (l+m) t^{l-1}_{n\,\underline{m}}(Z)
\end{pmatrix}
\end{multline}
and, if $l > 1$,
\begin{multline}  \label{Mx(NB)}
\frac14 \M \left( N(Z) \cdot \begin{pmatrix}
\beta_{11} t^{l-1}_{n\,\underline{m}}(Z) & \beta_{12} t^{l-1}_{n\,\underline{m+1}}(Z) \\
\beta_{21} t^{l-1}_{n+1\,\underline{m}}(Z) & \beta_{22} t^{l-1}_{n+1\,\underline{m+1}}(Z)
\end{pmatrix} \right)  \\
= 3l \begin{pmatrix}
-\beta_{22} t^{l-1}_{n+1\,\underline{m+1}}(Z) & \beta_{12} t^{l-1}_{n\,\underline{m+1}}(Z) \\
\beta_{21} t^{l-1}_{n+1\,\underline{m}}(Z) & -\beta_{11} t^{l-1}_{n\,\underline{m}}(Z)
\end{pmatrix}
+ \frac1{l-1} \begin{pmatrix}
\gamma_{11} t^{l-1}_{n+1\,\underline{m+1}}(Z) & \gamma_{12} t^{l-1}_{n\,\underline{m+1}}(Z) \\
\gamma_{21} t^{l-1}_{n+1\,\underline{m}}(Z) & \gamma_{22} t^{l-1}_{n\,\underline{m}}(Z)
\end{pmatrix}  \\
+ \frac{N(Z) \cdot l}{l-1} \bigl( (l-m-1)(\beta_{11}+\beta_{21})
+ (l+m)(\beta_{12}+\beta_{22}) \bigr) \\
\times \begin{pmatrix}
(l-m-2) t^{l-2}_{n+1\,\underline{m+1}}(Z) & (l-m-2) t^{l-2}_{n\,\underline{m+1}}(Z) \\
(l+m-1) t^{l-2}_{n+1\,\underline{m}}(Z) & (l+m-1) t^{l-2}_{n\,\underline{m}}(Z)
\end{pmatrix},
\end{multline}
where the coefficients
\begin{align*}
\gamma_{11} &= (l-m-1)(l+n)\beta_{11} + (m+1)(l+n)\beta_{12}
+ (l-m-1)(n+1)\beta_{21} + (m+1)(n+1)\beta_{22},  \\
\gamma_{12} &= (l-m-1)n\beta_{11} + (m+1)n\beta_{12}
- (l-m-1)(l-n-1)\beta_{21} - (m+1)(l-n-1)\beta_{22},  \\
\gamma_{21} &= m(l+n)\beta_{11} - (l+m)(l+n)\beta_{12} + m(n+1)\beta_{21}
-(l+m)(n+1)\beta_{22},  \\
\gamma_{22} &= mn\beta_{11} - (l+m)n\beta_{12}
- m(l-n-1)\beta_{21} + (l+m)(l-n-1)\beta_{22};
\end{align*}
in the special case of $l=1$,
$$
\frac14 \M \left( N(Z) \cdot \begin{pmatrix}
\beta_{11} t^0_{n\,\underline{m}}(Z) & \beta_{12} t^0_{n\,\underline{m+1}}(Z) \\
\beta_{21} t^0_{n+1\,\underline{m}}(Z) & \beta_{22} t^0_{n+1\,\underline{m+1}}(Z)
\end{pmatrix} \right)
= 3 \begin{pmatrix}
-\beta_{22} t^0_{n+1\,\underline{m+1}}(Z) & \beta_{12} t^0_{n\,\underline{m+1}}(Z) \\
\beta_{21} t^0_{n+1\,\underline{m}}(Z) & -\beta_{11} t^0_{n\,\underline{m}}(Z)
\end{pmatrix}.
$$

Similarly,
\begin{multline*}
\frac14 \M \left( \frac1{N(Z)} \begin{pmatrix}
\alpha'_{11} t^l_{n+1\,\underline{m+1}}(Z^{-1}) &
\alpha'_{12} t^l_{n\,\underline{m+1}}(Z^{-1}) \\
\alpha'_{21} t^l_{n+1\,\underline{m}}(Z^{-1}) &
\alpha'_{22} t^l_{n\,\underline{m}}(Z^{-1})
\end{pmatrix} \right)  \\
= \bigl((l-n)(\alpha'_{11}+\alpha'_{21})+(l+n+1)(\alpha'_{12}+\alpha'_{22})\bigr) \\
\times \frac1{N(Z)} \begin{pmatrix}
(l-n+1) t^{l+1}_{n\,\underline{m}}(Z^{-1}) & (l-n+1) t^{l+1}_{n\,\underline{m+1}}(Z^{-1}) \\
(l+n+2) t^{l+1}_{n+1\,\underline{m}}(Z^{-1}) & (l+n+2) t^{l+1}_{n+1\,\underline{m+1}}(Z^{-1})
\end{pmatrix}
\end{multline*}
and, if $l \ge -1/2$,
\begin{multline*}
\frac14 \M \begin{pmatrix} \beta'_{11} t^{l+1}_{n+1\,\underline{m+1}}(Z^{-1}) &
\beta'_{12} t^{l+1}_{n\,\underline{m+1}}(Z^{-1}) \\
\beta'_{21} t^{l+1}_{n+1\,\underline{m}}(Z^{-1}) &
\beta'_{22} t^{l+1}_{n\,\underline{m}}(Z^{-1}) \end{pmatrix}  \\
= \frac{3(l+1)}{N(Z)} \begin{pmatrix} \beta'_{22} t^{l+1}_{n\,\underline{m}}(Z^{-1}) &
-\beta'_{12} t^{l+1}_{n\,\underline{m+1}}(Z^{-1}) \\
-\beta'_{21} t^{l+1}_{n+1\,\underline{m}}(Z^{-1}) &
\beta'_{11} t^{l+1}_{n+1\,\underline{m+1}}(Z^{-1}) \end{pmatrix}  \\
+ \frac{N(Z)^{-1}}{l+2}
\begin{pmatrix} \gamma'_{11} t^{l+1}_{n\,\underline{m}}(Z^{-1}) &
\gamma'_{12} t^{l+1}_{n\,\underline{m+1}}(Z^{-1})  \\
\gamma'_{21} t^{l+1}_{n+1\,\underline{m}}(Z^{-1}) &
\gamma'_{22} t^{l+1}_{n+1\,\underline{m+1}}(Z^{-1}) \end{pmatrix}  \\
+ \frac{l+1}{l+2} \bigl( (l-n+1)(\beta'_{11}+\beta'_{21})
+ (l+n+2)(\beta'_{12}+\beta'_{22}) \bigr)  \\
\times \begin{pmatrix}
(l-n+2) t^{l+2}_{n\,\underline{m}}(Z^{-1}) & (l-n+2) t^{l+2}_{n\,\underline{m+1}}(Z^{-1}) \\
(l+n+3) t^{l+2}_{n+1\,\underline{m}}(Z^{-1}) & (l+n+3) t^{l+2}_{n+1\,\underline{m+1}}(Z^{-1})
\end{pmatrix},
\end{multline*}
where the coefficients
\begin{align*}
\gamma'_{11} &= -(l+m+2)(l-n+1)\beta'_{11} - (l+m+2)n\beta'_{12}
- m(l-n+1)\beta'_{21} - mn\beta'_{22},  \\
\gamma'_{12} &= -(m+1)(l-n+1)\beta'_{11} - (m+1)n\beta'_{12}
+ (l-m+1)(l-n+1)\beta'_{21} + (l-m+1)n\beta'_{22},  \\
\gamma'_{21} &= -(l+m+2)(n+1)\beta'_{11} + (l+m+2)(l+n+2)\beta'_{12}
- m(n+1)\beta'_{21} + m(l+n+2)\beta'_{22},  \\
\gamma'_{22} &= -(m+1)(n+1)\beta'_{11} + (m+1)(l+n+2)\beta'_{12}
+ (l-m+1)(n+1)\beta'_{21} \\ & \qquad - (l-m+1)(l+n+2)\beta'_{22}.
\end{align*}
\end{lem}

\begin{proof}
The result is obtained by rather tedious, yet completely straightforward
calculations using equation (\ref{Zt-identity}), Lemma 22 from \cite{FL1}
and identity
\begin{multline}  \label{del-t(Z^{-1})}
\partial \bigl( N(Z)^{-1} \cdot t^l_{n\,\underline{m}}(Z^{-1}) \bigr)  \\
= - \frac1{N(Z)} \begin{pmatrix}
(l-n+1) t^{l+\frac12}_{n-\frac12\,\underline{m-\frac12}}(Z^{-1}) &
(l-n+1) t^{l+\frac12}_{n-\frac12\,\underline{m+\frac12}}(Z^{-1})  \\
(l+n+1) t^{l+\frac12}_{n+\frac12\,\underline{m-\frac12}}(Z^{-1}) &
(l+n+1) t^{l+\frac12}_{n+\frac12\,\underline{m+\frac12}}(Z^{-1}) \end{pmatrix},
\end{multline}
which in turn can be verified using Lemmas 22 and 23 in \cite{FL1}.
\end{proof}




\begin{prop}  \label{A-space-prop}
We have:
$$
{\cal M} \oplus \BB C\text{-span of } \bigl\{ N(Z) \cdot Z \bigr\} =
\bigl\{ F \in {\cal W}';\: \M F=0, \: ( \square \circ \square) F =0 \bigr\}.
$$
Moreover, every element $A \in {\cal M}$ is a linear combination of
$N(Z)^{-1} \cdot Z$ and homogeneous elements of the form
\begin{equation}  \label{A(Z)}
A_d(Z) + \frac{N(Z)}{d+1} \cdot \bigl( \partial A_d(Z) \partial \bigr)^+
\quad \in {\cal M}
\end{equation}
for some $A_d: \HC^{\times} \to \HC$ which is homogeneous of degree $d$ and
harmonic.

In particular, the space ${\cal M}$ can be characterized as the unique
maximal $\mathfrak{gl}(2,\HC)$-invariant subspace of $\ker \M$ consisting of
biharmonic functions.
\end{prop}

\begin{proof}
It is easy to see that $N(Z)^{-1} \cdot Z \in {\cal M}$ and
$N(Z) \cdot Z \notin {\cal M}$. Let
$$
{\cal M}' = \bigl\{ F \in {\cal W}';\:
\M F=0, \: ( \square \circ \square) F =0 \bigr\}.
$$
By Corollary \ref{M-biharm-cor},
$$
{\cal M} \oplus \BB C\text{-span of } \bigl\{ N(Z) \cdot Z \bigr\}
\subset {\cal M}',
$$
and we need to prove the opposite inclusion.
By Lemma \ref{elts-of-A} it is sufficient to show that every element
$A \in {\cal M}'$ is a linear combination of $N(Z) \cdot Z$,
$N(Z)^{-1} \cdot Z$ and homogeneous elements of the form (\ref{A(Z)}).
By Proposition \ref{biharmonic-prop}, $A(Z)$ has to be a linear combination
of homogeneous elements that appear in Lemma \ref{Mx-calculations}:
$$
\begin{pmatrix}
\alpha_{11} t^l_{n\,\underline{m}}(Z) & \alpha_{12} t^l_{n\,\underline{m+1}}(Z) \\
\alpha_{21} t^l_{n+1\,\underline{m}}(Z) & \alpha_{22} t^l_{n+1\,\underline{m+1}}(Z)
\end{pmatrix},
\qquad
N(Z) \cdot \begin{pmatrix}
\beta_{11} t^{l-1}_{n\,\underline{m}}(Z) & \beta_{12} t^{l-1}_{n\,\underline{m+1}}(Z) \\
\beta_{21} t^{l-1}_{n+1\,\underline{m}}(Z) & \beta_{22} t^{l-1}_{n+1\,\underline{m+1}}(Z)
\end{pmatrix},
$$
$$
\frac1{N(Z)} \begin{pmatrix}
\alpha'_{11} t^l_{n+1\,\underline{m+1}}(Z^{-1}) &
\alpha'_{12} t^l_{n\,\underline{m+1}}(Z^{-1}) \\
\alpha'_{21} t^l_{n+1\,\underline{m}}(Z^{-1}) &
\alpha'_{22} t^l_{n\,\underline{m}}(Z^{-1})
\end{pmatrix},
\qquad
\begin{pmatrix} \beta'_{11} t^{l+1}_{n+1\,\underline{m+1}}(Z^{-1}) &
\beta'_{12} t^{l+1}_{n\,\underline{m+1}}(Z^{-1}) \\
\beta'_{21} t^{l+1}_{n+1\,\underline{m}}(Z^{-1}) &
\beta'_{22} t^{l+1}_{n\,\underline{m}}(Z^{-1}) \end{pmatrix}.
$$
Our element $A(Z) \in {\cal M}'$ must also be annihilated by $\M$,
and, in view of Lemma \ref{Mx-calculations}, without loss of generality
we may assume that it is either a linear combination of the first two
or the last two types.
We provide a sketch for the first case with $l \ne 1$ only, the subcase $l=1$
and other case are similar.
Thus we assume
$$
A(Z) = \begin{pmatrix}
\alpha_{11} t^l_{n\,\underline{m}}(Z) & \alpha_{12} t^l_{n\,\underline{m+1}}(Z) \\
\alpha_{21} t^l_{n+1\,\underline{m}}(Z) & \alpha_{22} t^l_{n+1\,\underline{m+1}}(Z)
\end{pmatrix}
+ N(Z) \cdot \begin{pmatrix}
\beta_{11} t^{l-1}_{n\,\underline{m}}(Z) & \beta_{12} t^{l-1}_{n\,\underline{m+1}}(Z) \\
\beta_{21} t^{l-1}_{n+1\,\underline{m}}(Z) & \beta_{22} t^{l-1}_{n+1\,\underline{m+1}}(Z)
\end{pmatrix}.
$$
If the second summand is zero, then $A(Z)$ is harmonic with
$\partial A(Z) \partial =0$ (since $\M A = 0$),
and $A(Z)$ is of the form (\ref{A(Z)}).
Thus we further assume
\begin{equation}  \label{not_zero}
 N(Z) \cdot \begin{pmatrix}
\beta_{11} t^{l-1}_{n\,\underline{m}}(Z) & \beta_{12} t^{l-1}_{n\,\underline{m+1}}(Z) \\
\beta_{21} t^{l-1}_{n+1\,\underline{m}}(Z) & \beta_{22} t^{l-1}_{n+1\,\underline{m+1}}(Z)
\end{pmatrix} \ne 0
\end{equation}
and that $-l+1 \le m,n \le l-2$ so that $t^{l-1}_{n\,\underline{m}}(Z)$,
$t^{l-1}_{n\,\underline{m+1}}(Z)$, $t^{l-1}_{n+1\,\underline{m}}(Z)$,
$t^{l-1}_{n+1\,\underline{m+1}}(Z) \ne 0$. (The case when some of these functions
are zero needs to be considered separately.)
Since $\M A = 0$, $\M A$ has no harmonic component.
This means that in the harmonic component of (\ref{Mx(NB)}) the coefficients
in the same rows are equal and proportional to the coefficients in
(\ref{Mx(A)}):
$$
\begin{cases}
-3l(l-1)\beta_{22}+\gamma_{11} = 3l(l-1)\beta_{12}+\gamma_{12} \\
3l(l-1)\beta_{21}+\gamma_{21} = -3l(l-1)\beta_{11}+\gamma_{22} \\
(l+m) \bigl( -3l(l-1)\beta_{22}+\gamma_{11} \bigr)
= (l-m-1) \bigl( 3l(l-1)\beta_{21}+\gamma_{21} \bigr).
\end{cases}
$$
The first two equations simplify to
\begin{equation}  \label{lin-sys}
\begin{cases}
l\bigl((l-m-1)(\beta_{11}+\beta_{21})-(3l-m-4)(\beta_{12}+\beta_{22})\bigr)=0  \\
l\bigl((3l+m-3)(\beta_{11}+\beta_{21})-(l+m)(\beta_{12}+\beta_{22})\bigr)=0,
\end{cases}
\end{equation}
and
$$
\det \begin{pmatrix} l-m-1 & -3l+m+4 \\ 3l+m-3 & -(l+m) \end{pmatrix}
= 4(l-1)(2l-3).
$$
Assume for now that $l \ne 3/2$, then
$\beta_{11}+\beta_{21}=0$ and $\beta_{12}+\beta_{22}=0$.
Substituting $\beta_{21}=-\beta_{11}$ and $\beta_{12}=-\beta_{22}$ into the third
equation yields
$$
4l(l-1) \bigl( (l-m-1)\beta_{11} - (l+m)\beta_{22} \bigr) =0,
$$
hence $(\beta_{11},\beta_{22})$ is proportional to $(l+m,l-m-1)$.
It follows from $\M A(Z) =0$ that $A(Z)$ is of the form (\ref{A(Z)}).


In the exceptional case $l=3/2$, the system (\ref{lin-sys})
has rank one and simplifies to a single equation 
$\beta_{11}+\beta_{21}=\beta_{12}+\beta_{22}$.
If $m=n=-1/2$ we get exactly one additional linearly independent
solution in ${\cal M}'$ that is not in ${\cal M}$:
$$
\beta_{11}=\beta_{12}=\beta_{21}=\beta_{22}=1, \qquad A(Z)=N(Z) \cdot Z.
$$


Finally, to prove the maximality property of ${\cal M}$,
it is sufficient to show that a $\mathfrak{gl}(2,\HC)$-invariant subspace
of ${\cal M}'$ cannot contain an element of the form $A+N(Z)\cdot Z$ with
$A \in {\cal M}$.
Indeed, $N(Z)\cdot Z$ is the image of an element $\frac12 N(Z)^2$ under
a map $\partial^+ : \Sh' \to {\cal W}'$, which is
$\mathfrak{gl}(2,\HC)$-equivariant by Proposition \ref{equiv-prop1}.
Since $\frac12 N(Z)^2$ generates $\Sh^+$ (Theorem \ref{rho'-decomposition}),
$N(Z)\cdot Z$ generates $\partial^+(\Sh^+)$, which contains
$N(Z)^2\cdot Z \notin {\cal M}'$.
Therefore, a $\mathfrak{gl}(2,\HC)$-invariant subspace of ${\cal W}'$
containing $A+N(Z)\cdot Z$ with $A \in {\cal M}$ also contains elements not in
${\cal M}'$.
\end{proof}

\begin{cor}  \label{XinterA}
The intersection of the image of the map $\partial^+ : \Sh' \to {\cal W}'$
and ${\cal M}$ in ${\cal W}'$ is precisely
$\partial^+({\cal BH}^+) \oplus \partial^+({\cal BH}^-)$.
\end{cor}

\begin{proof}
The result follows from Proposition \ref{equiv-prop1},
Theorem \ref{rho'-decomposition} and Proposition \ref{A-space-prop}.
\end{proof}

\subsection{The $K$-types of ${\cal M}$}

In this subsection we describe the $K$-types of ${\cal M}$.
For $d \in \BB Z$, define
$$
{\cal M}(d) =
\{ A(Z) \in {\cal M};\: \text{$A(Z)$ is homogeneous of degree $d$} \}.
$$
We realize $\mathfrak{sl}(2,\BB C) \times \mathfrak{sl}(2,\BB C)$
as diagonal elements of $\mathfrak{gl}(2,\HC)$:
$$
\mathfrak{sl}(2,\BB C) \times \mathfrak{sl}(2,\BB C) = \left\{
\left(\begin{smallmatrix} A & 0 \\ 0 & D \end{smallmatrix}\right)
\in \mathfrak{gl}(2,\HC);\: A,D \in \HC, \re(A)=\re(D)=0 \right\}.
$$

\begin{prop}  \label{A-K-types}
Each ${\cal M}(d)$ is invariant under the $\rho'_2$ action restricted to
$\mathfrak{sl}(2,\BB C) \times \mathfrak{sl}(2,\BB C)$, and we have the
following decomposition into irreducible components:
$$
{\cal M}(-1) = V_0 \boxtimes V_0
= \BB C\text{-span of } \bigl\{ N(Z)^{-1} \cdot Z \bigr\},
$$
\begin{equation}  \label{A(2l)}
{\cal M}(2l) = \bigl( V_{l-\frac12} \boxtimes V_{l-\frac12} \bigr) \oplus
\bigl( V_{l+\frac12} \boxtimes V_{l-\frac12} \bigr) \oplus
\bigl( V_{l-\frac12} \boxtimes V_{l+\frac12} \bigr) \oplus
\bigl( V_{l+\frac12} \boxtimes V_{l+\frac12} \bigr),
\end{equation}
\begin{equation}  \label{A(-2l-1)}
{\cal M}(-2l-2) = \bigl( V_{l-\frac12} \boxtimes V_{l-\frac12} \bigr) \oplus
\bigl( V_{l+\frac12} \boxtimes V_{l-\frac12} \bigr) \oplus
\bigl( V_{l-\frac12} \boxtimes V_{l+\frac12} \bigr) \oplus
\bigl( V_{l+\frac12} \boxtimes V_{l+\frac12} \bigr),
\end{equation}
$l=0,\frac12,1,\frac32,\dots$.
Explicitly, these irreducible components are generated by homogeneous
elements of the form (\ref{A(Z)}) with harmonic parts:
\begin{multline*}
V_{l-\frac12} \boxtimes V_{l-\frac12} = 
\BB C\text{-span of }
\left\{ \left(\begin{smallmatrix}
(l-n) t^l_{n\,\underline{m}}(Z) & (l-n) t^l_{n\,\underline{m+1}}(Z) \\
(l+n+1) t^l_{n+1\,\underline{m}}(Z) & (l+n+1) t^l_{n+1\,\underline{m+1}}(Z)
\end{smallmatrix}\right) ;\: \begin{smallmatrix} l=\frac12,1,\frac32,2,\dots \\
-l \le m,n \le l-1 \end{smallmatrix} \right\}  \\
\text{or} \quad \left\{ \tilde H'_{l,m,n}(Z) = \frac1{N(Z)}
\left(\begin{smallmatrix}
(l+n+1) t^l_{n+1\,\underline{m+1}}(Z^{-1}) & -(l-n) t^l_{n\,\underline{m+1}}(Z^{-1}) \\
-(l+n+1) t^l_{n+1\,\underline{m}}(Z^{-1}) & (l-n) t^l_{n\,\underline{m}}(Z^{-1})
\end{smallmatrix}\right) ;\: \begin{smallmatrix} l=\frac12,1,\frac32,2,\dots \\
-l \le m,n \le l-1 \end{smallmatrix} \right\},
\end{multline*}
\begin{multline*}
V_{l+\frac12} \boxtimes V_{l-\frac12} = 
\BB C\text{-span of } \left\{ \tilde G_{l,m,n}(Z) = \left(\begin{smallmatrix}
t^l_{n\,\underline{m}}(Z) & t^l_{n\,\underline{m+1}}(Z) \\
-t^l_{n+1\,\underline{m}}(Z) & -t^l_{n+1\,\underline{m+1}}(Z) \end{smallmatrix}\right) ;\:
\begin{smallmatrix} l=\frac12,1,\frac32,2,\dots \\
-l \le m \le l-1 \\  -l-1 \le n \le l \end{smallmatrix} \right\}  \\
\text{or} \quad \left\{ \tilde F'_{l,m,n}(Z) = \frac1{N(Z)}
\left(\begin{smallmatrix} (l-m)(l+n+1) t^l_{n+1\,\underline{m+1}}(Z^{-1}) &
-(l-m)(l-n) t^l_{n\,\underline{m+1}}(Z^{-1}) \\
(l+m+1)(l+n+1) t^l_{n+1\,\underline{m}}(Z^{-1}) &
-(l+m+1)(l-n) t^l_{n\,\underline{m}}(Z^{-1}) \end{smallmatrix}\right) ;\:
\begin{smallmatrix} l=\frac12,1,\frac32,2,\dots \\
-l-1 \le m \le l \\  -l \le n \le l-1 \end{smallmatrix} \right\},
\end{multline*}
\begin{multline*}
V_{l-\frac12} \boxtimes V_{l+\frac12} = 
\BB C\text{-span of }  \\
\left\{ \tilde F_{l,m,n}(Z) = \left(\begin{smallmatrix}
(l+m+1)(l-n) t^l_{n\,\underline{m}}(Z) & -(l-m)(l-n) t^l_{n\,\underline{m+1}}(Z) \\
(l+m+1)(l+n+1) t^l_{n+1\,\underline{m}}(Z) & -(l-m)(l+n+1) t^l_{n+1\,\underline{m+1}}(Z)
\end{smallmatrix}\right) ;\:
\begin{smallmatrix} l=\frac12,1,\frac32,2,\dots \\
-l-1 \le m \le l \\  -l \le n \le l-1 \end{smallmatrix} \right\}  \\
\text{or} \quad \left\{ \tilde G'_{l,m,n}(Z) = \frac1{N(Z)}
\left(\begin{smallmatrix}
t^l_{n+1\,\underline{m+1}}(Z^{-1}) & t^l_{n\,\underline{m+1}}(Z^{-1}) \\
-t^l_{n+1\,\underline{m}}(Z^{-1}) & -t^l_{n\,\underline{m}}(Z^{-1})
\end{smallmatrix}\right) ;\:
\begin{smallmatrix} l=\frac12,1,\frac32,2,\dots \\
-l \le m \le l-1 \\  -l-1 \le n \le l \end{smallmatrix} \right\},
\end{multline*}
\begin{multline*}
V_{l+\frac12} \boxtimes V_{l+\frac12} = \BB C\text{-span of }
\left\{ \tilde H_{l,m,n}(Z) = \left(\begin{smallmatrix}
(l+m+1) t^l_{n\,\underline{m}}(Z) & -(l-m) t^l_{n\,\underline{m+1}}(Z) \\
-(l+m+1) t^l_{n+1\,\underline{m}}(Z) & (l-m) t^l_{n+1\,\underline{m+1}}(Z)
\end{smallmatrix}\right) ;\: \begin{smallmatrix} l=0,\frac12,1,\frac32,\dots \\
-l-1 \le m,n \le l \end{smallmatrix} \right\}  \\
\text{or} \quad \left\{ \frac1{N(Z)} \left(\begin{smallmatrix}
(l-m) t^l_{n+1\,\underline{m+1}}(Z^{-1}) & (l-m) t^l_{n\,\underline{m+1}}(Z^{-1}) \\
(l+m+1) t^l_{n+1\,\underline{m}}(Z^{-1}) & (l+m+1) t^l_{n\,\underline{m}}(Z^{-1})
\end{smallmatrix}\right) ;\: \begin{smallmatrix} l=0,\frac12,1,\frac32,\dots \\
-l-1 \le m,n \le l \end{smallmatrix} \right\}.
\end{multline*}

The functions in ${\cal M}(2l)$ that lie in
$$
\bigl( V_{l+\frac12} \boxtimes V_{l-\frac12} \bigr) \oplus
\bigl( V_{l-\frac12} \boxtimes V_{l+\frac12} \bigr) \oplus
\bigl( V_{l+\frac12} \boxtimes V_{l+\frac12} \bigr)
$$
and the functions in ${\cal M}(-2l-2)$ that lie in
$$
\bigl( V_{l-\frac12} \boxtimes V_{l-\frac12} \bigr) \oplus
\bigl( V_{l+\frac12} \boxtimes V_{l-\frac12} \bigr) \oplus
\bigl( V_{l-\frac12} \boxtimes V_{l+\frac12} \bigr)
$$
have harmonic parts only, their non-harmonic parts are zero.
\end{prop}

\begin{proof}
Note that, as usual, if the indices of $t^l_{n\,\underline{m}}(Z)$
happen to be outside of the allowed range $l=0,\frac12,1,\frac32,\dots$,
$m,n \in \BB Z +l$, $-l \le m,n \le l$, then such matrix coefficients
are declared to be zero.
The result follows from Proposition \ref{A-space-prop},
Lemma \ref{rho'_2-action-lem} and explicit realization of the isomorphism
of representations of $\mathfrak{sl}(2,\BB C)$
$$
V_l \otimes V_{\frac12} \simeq V_{l-\frac12} \oplus V_{l+\frac12}.
$$

The assertion about non-harmonic parts follows from Lemma \ref{Mx-calculations}.
\end{proof}

Combining this result with Corollary \ref{XinterA} and comparing
the decompositions into the
$\mathfrak{sl}(2,\BB C) \times \mathfrak{sl}(2,\BB C)$ components,
we obtain:

\begin{cor}
$$
\partial^+({\cal BH}^+) = \bigoplus_{l \ge 0} \bigl(
\text{$V_{l-\frac12} \boxtimes V_{l-\frac12}$ component of ${\cal M}(2l)$} \bigr)
\oplus \bigl(
\text{$V_{l+\frac12} \boxtimes V_{l+\frac12}$ component of ${\cal M}(2l)$} \bigr),
$$
\begin{multline*}
\partial^+({\cal BH}^-) \\ = \bigoplus_{l \ge 0} \bigl(
\text{$V_{l-\frac12} \boxtimes V_{l-\frac12}$ comp. of ${\cal M}(-2l-2)$} \bigr)
\oplus \bigl(
\text{$V_{l+\frac12} \boxtimes V_{l+\frac12}$ comp. of ${\cal M}(-2l-2)$} \bigr).
\end{multline*}
\end{cor}

\begin{lem}  \label{FGH-reg}
The functions $(\tilde F_{l,m,n})^+$, $(\tilde F'_{l,m,n})^+ : \HC \to \HC$
are left regular.
The functions $(\tilde G_{l,m,n})^+$, $(\tilde G'_{l,m,n})^+ : \HC \to \HC$
are right regular.
The functions $(\tilde H_{l,m,n})^+$, $(\tilde H'_{l,m,n})^+ : \HC \to \HC$
are both left and right regular.
\end{lem}

\begin{proof}
The result follows by comparing the columns and rows of the functions
in question with the basis of left and right regular functions given
in Proposition 24 in \cite{FL1}.
\end{proof}

The harmonic functions
$$
\tilde F_{l,m,n}(Z),\: \tilde F'_{l,m,n}(Z),\: \tilde G_{l,m,n}(Z),\:
\tilde G'_{l,m,n}(Z),\: \tilde H_{l,m,n}(Z),\: \tilde H'_{l,m,n}(Z)
$$
belong to ${\cal M}$.
We complete these and $N(Z)^{-1} \cdot Z$ to a basis of ${\cal M}$ by letting
$\tilde I_{l,m,n}(Z)$ be the elements in ${\cal M}$ that have harmonic parts
$$
\begin{pmatrix}
(l-n) t^l_{n\,\underline{m}}(Z) & (l-n) t^l_{n\,\underline{m+1}}(Z) \\
(l+n+1) t^l_{n+1\,\underline{m}}(Z) & (l+n+1) t^l_{n+1\,\underline{m+1}}(Z)
\end{pmatrix}, \qquad \begin{smallmatrix} l=\frac12,1,\frac32,2,\dots \\
-l \le m,n \le l-1 \end{smallmatrix},
$$
these generate the $V_{l-\frac12} \boxtimes V_{l-\frac12}$ component of
${\cal M}(2l)$,
and 
$\tilde I'_{l,m,n}(Z)$ be the elements in ${\cal M}$ that have harmonic parts
$$
\frac1{N(Z)} \begin{pmatrix}
(l-m) t^l_{n+1\,\underline{m+1}}(Z^{-1}) & (l-m) t^l_{n\,\underline{m+1}}(Z^{-1}) \\
(l+m+1) t^l_{n+1\,\underline{m}}(Z^{-1}) & (l+m+1) t^l_{n\,\underline{m}}(Z^{-1})
\end{pmatrix}, \qquad \begin{smallmatrix} l=0,\frac12,1,\frac32,\dots \\
-l-1 \le m,n \le l \end{smallmatrix},
$$
these generate the $V_{l+\frac12} \boxtimes V_{l+\frac12}$ component of
${\cal M}(-2l-2)$.
The following technical lemma will be used to construct equivariant maps
$\tau^{\pm}_a$ and $\tau^{\pm}_s$ from ${\cal W}'$ to doubly regular functions.

\begin{lem}  \label{del-A}
We have the following identities:
$$
(\tilde F_{l,m,n} \overleftarrow{\partial}) (Z)
= (\tilde F'_{l,m,n} \overleftarrow{\partial}) (Z)
= \overrightarrow{\partial} \tilde G_{l,m,n}(Z) =
\overrightarrow{\partial} \tilde G'_{l,m,n}(Z) =0,
$$
$$
\overrightarrow{\partial} \tilde H_{l,m,n}(Z)
= (\tilde H_{l,m,n} \overleftarrow{\partial}) (Z)
= \overrightarrow{\partial} \tilde H'_{l,m,n}(Z)
= (\tilde H'_{l,m,n} \overleftarrow{\partial}) (Z) =0,
$$
$$
\overrightarrow{\partial} \bigl( N(Z)^{-1} \cdot Z \bigr)
= \bigl( N(Z)^{-1} \cdot Z \bigr) \overleftarrow{\partial} = N(Z)^{-1},
$$
\begin{multline*}
\overrightarrow{\partial} \tilde F_{l,m,n}(Z) = (2l+1) \left( \begin{smallmatrix}
(l-m)(l+m+1) t^{l-\frac12}_{n+\frac12\,\underline{m+\frac12}}(Z) &
-(l-m-1)(l-m) t^{l-\frac12}_{n+\frac12\,\underline{m+\frac32}}(Z) \\
(l+m)(l+m+1) t^{l-\frac12}_{n+\frac12\,\underline{m-\frac12}}(Z) &
-(l+m+1)(l-m) t^{l-\frac12}_{n+\frac12\,\underline{m+\frac12}}(Z)
\end{smallmatrix} \right)  \\
= (2l+1) Z^{-1} \cdot \tilde F_{l,m,n}(Z),
\end{multline*}
\begin{multline*}
\overrightarrow{\partial} \tilde F'_{l,m,n}(Z) =
\frac{2l+1}{N(Z)} \left( \begin{smallmatrix}
-(l-n)(l+n+1) t^{l+\frac12}_{n+\frac12\,\underline{m+\frac12}}(Z^{-1}) &
(l-n+1)(l-n) t^{l+\frac12}_{n-\frac12\,\underline{m+\frac12}}(Z^{-1}) \\
-(l+n+2)(l+n+1) t^{l+\frac12}_{n+\frac32\,\underline{m+\frac12}}(Z^{-1}) &
(l+n+1)(l-n) t^{l+\frac12}_{n+\frac12\,\underline{m+\frac12}}(Z^{-1})
\end{smallmatrix} \right) \\
= -(2l+1) Z^{-1} \cdot \tilde F'_{l,m,n}(Z),
\end{multline*}
$$
(\tilde G_{l,m,n} \overleftarrow{\partial}) (Z)
= (2l+1) \left( \begin{smallmatrix} t^{l-\frac12}_{n+\frac12\,\underline{m+\frac12}}(Z) &
t^{l-\frac12}_{n-\frac12\,\underline{m+\frac12}}(Z) \\
-t^{l-\frac12}_{n+\frac32\,\underline{m+\frac12}}(Z) &
-t^{l-\frac12}_{n+\frac12\,\underline{m+\frac12}}(Z) \end{smallmatrix} \right)
= (2l+1) G_{l,m,n}(Z) \cdot Z^{-1},
$$
$$
(\tilde G'_{l,m,n} \overleftarrow{\partial}) (Z)
= \frac{2l+1}{N(Z)} \left( \begin{smallmatrix}
-t^{l+\frac12}_{n+\frac12\,\underline{m+\frac12}}(Z^{-1}) &
-t^{l+\frac12}_{n+\frac12\,\underline{m+\frac32}}(Z^{-1}) \\
t^{l+\frac12}_{n+\frac12\,\underline{m-\frac12}}(Z^{-1}) &
t^{l+\frac12}_{n+\frac12\,\underline{m+\frac12}}(Z^{-1}) \end{smallmatrix} \right)
= -(2l+1) \tilde G'_{l,m,n}(Z) \cdot Z^{-1},
$$
$$
\overrightarrow{\partial} \tilde I_{l,m,n}(Z)
= (2l+1) Z^{-1} \cdot \tilde I_{l,m,n}(Z), \qquad
(\tilde I_{l,m,n} \overleftarrow{\partial}) (Z)
= (2l+1) \tilde I_{l,m,n}(Z) \cdot Z^{-1},
$$
$$
\overrightarrow{\partial} \tilde I'_{l,m,n}(Z)
= -(2l+1) Z^{-1} \cdot \tilde I'_{l,m,n}(Z), \qquad
(\tilde I'_{l,m,n} \overleftarrow{\partial}) (Z)
= -(2l+1) \tilde I'_{l,m,n}(Z) \cdot Z^{-1}.
$$
\end{lem}

\begin{proof}
The result follows by direct computations from Lemmas 22, 23 from \cite{FL1}
and identity (\ref{del-t(Z^{-1})}).
\end{proof}

\subsection{Equivariant Maps from ${\cal W}'$ to Doubly Regular Functions}

The goal of this subsection is to construct equivariant maps $\tau^{\pm}_a$
and $\tau^{\pm}_s$ from ${\cal W}'$ into the spaces of
doubly left and right regular functions.
We conclude that $(\rho'_2, {\cal M})$ contains irreducible components
isomorphic to $(\pi_{dl},{\cal F}^+)$, $(\pi_{dl},{\cal F}^-)$,
$(\pi_{dr},{\cal G}^+)$ and $(\pi_{dr},{\cal G}^-)$ introduced in
Proposition \ref{DR-irred-prop}.

Recall that we refer to Section 2 of \cite{FL3} for a summary of notations.
In particular, we treat $U(2)$ as a subgroup of $\HC^{\times}$,
let $U(2)_R = \{ RZ ;\: Z \in U(2) \}$, where $R>0$. 
We will also need the open domains $\BB D_R^+$, $\BB D_R^-$ defined by
equation (22) in \cite{FL3}.
We introduce four maps $p_a^+$, $p_a^-$, $p_s^+$ and
$p_s^-: {\cal W}' \to {\cal W}$:
\begin{align*}
p_a^+: \:F(Z) \: &\mapsto \:
\frac{i}{2\pi^3} \left[ \int_{Z \in U(2)_R} \frac{(Z-W)^{-1}}{N(Z-W)}
\cdot Z \cdot F(Z)^+ \,\frac{dV}{N(Z)} \right]^+,
\qquad W \in \BB D_R^+,  \\
p_a^-: \:F(Z) \: &\mapsto \:
- \frac{i}{2\pi^3} \left[ \int_{Z \in U(2)_R} \frac{(Z-W)^{-1}}{N(Z-W)}
\cdot Z \cdot F(Z)^+ \,\frac{dV}{N(Z)} \right]^+,
\qquad W \in \BB D_R^-,  \\
p_s^+: \:F(Z) \: &\mapsto \:
\frac{i}{2\pi^3} \left[ \int_{Z \in U(2)_R} F(Z)^+ \cdot Z \cdot
\frac{(Z-W)^{-1}}{N(Z-W)} \,\frac{dV}{N(Z)} \right]^+,
\qquad W \in \BB D_R^+,  \\
p_s^-: \:F(Z) \: &\mapsto \:
- \frac{i}{2\pi^3} \left[ \int_{Z \in U(2)_R} F(Z)^+ \cdot Z \cdot
\frac{(Z-W)^{-1}}{N(Z-W)} \,\frac{dV}{N(Z)} \right]^+,
\qquad W \in \BB D_R^-.
\end{align*}
(These maps do not depend on the choice of $R>0$.)
First, we show that these maps annihilate the non-harmonic parts.

\begin{lem}  \label{p(NF)}
Let $F: \HC^{\times} \to \HC$ be such that $\square F=0$ and $k \ne 0$, then
$$
p_a^+ \bigl( N(Z)^k \cdot F(Z) \bigr) =
p_a^- \bigl( N(Z)^k \cdot F(Z) \bigr) =
p_s^+ \bigl( N(Z)^k \cdot F(Z) \bigr) =
p_s^- \bigl( N(Z)^k \cdot F(Z) \bigr) = 0.
$$
In particular, these maps annihilate the non-harmonic parts of
functions of the form (\ref{A(Z)}).
\end{lem}

\begin{proof}
Observe that
$$
\frac{(Z-W)^{-1}}{N(Z-W)} \cdot Z
\qquad \text{and} \qquad Z \cdot \frac{(Z-W)^{-1}}{N(Z-W)}
$$
are harmonic with respect to the $Z$ variable:
$$
\square_Z \left( \frac{(Z-W)^{-1}}{N(Z-W)} \cdot Z \right)
= \square_Z \left( \frac{(Z-W)^{-1}}{N(Z-W)} \cdot (Z-W) \right)
= \square_Z \left( \frac1{N(Z-W)} \right) =0,
$$
since $\frac{(Z-W)^{-1}}{N(Z-W)}$, $\frac{(Z-W)^{-1}}{N(Z-W)} \cdot W$ and
$\frac1{N(Z-W)}$ are harmonic with respect to $Z$. Then the integrals
\begin{align*}
\int_{Z \in U(2)_R} \biggl( \frac{(Z-W)^{-1}}{N(Z-W)} \cdot Z \biggr)
\cdot N(Z)^{k-1} \cdot F(Z)^+ \,dV &=0,  \\
\int_{Z \in U(2)_R} N(Z)^{k-1} \cdot F(Z)^+ \cdot
\biggl( Z \cdot \frac{(Z-W)^{-1}}{N(Z-W)} \biggr) \,dV &=0
\end{align*}
by the orthogonality relations (19) in \cite{FL3}, since we never get the
power of $N(Z)$ that can potentially result in non-zero integral.
\end{proof}

Then we determine the effect of the maps $p_a^{\pm}$ and $p_s^{\pm}$
on harmonic functions.
Note that any harmonic function in ${\cal W}'$ is a linear combination
of harmonic parts of functions from Proposition \ref{A-K-types}.

\begin{prop}  \label{p(A)}
For each $l$, $l=0,\frac12,1,\frac32,\dots$,
\begin{enumerate}
\item
The map $p_a^+$ annihilates ${\cal M}(-2l-2)$ and the components of
${\cal M}(2l)$ isomorphic to $V_{l-\frac12} \boxtimes V_{l-\frac12}$ and
$V_{l+\frac12} \boxtimes V_{l-\frac12}$ from Proposition \ref{A-K-types};
it leaves the components of ${\cal M}(2l)$ isomorphic to
$V_{l-\frac12} \boxtimes V_{l+\frac12}$ and $V_{l+\frac12} \boxtimes V_{l+\frac12}$
unchanged.
\item
The map $p_a^-$ annihilates ${\cal M}(2l)$ and the components of
${\cal M}(-2l-2)$ isomorphic to $V_{l-\frac12} \boxtimes V_{l+\frac12}$ and
$V_{l+\frac12} \boxtimes V_{l+\frac12}$;
it leaves the components of ${\cal M}(-2l-2)$ isomorphic to
$V_{l-\frac12} \boxtimes V_{l-\frac12}$ and $V_{l+\frac12} \boxtimes V_{l-\frac12}$
unchanged.
\item
The map $p_s^+$ annihilates ${\cal M}(-2l-2)$ and the components of
${\cal M}(2l)$ isomorphic to $V_{l-\frac12} \boxtimes V_{l-\frac12}$ and
$V_{l-\frac12} \boxtimes V_{l+\frac12}$;
it leaves the components of ${\cal M}(2l)$ isomorphic to
$V_{l+\frac12} \boxtimes V_{l-\frac12}$ and $V_{l+\frac12} \boxtimes V_{l+\frac12}$
unchanged.
\item
The map $p_s^-$ annihilates ${\cal M}(2l)$ and the components of
${\cal M}(-2l-2)$ isomorphic to $V_{l+\frac12} \boxtimes V_{l-\frac12}$ and
$V_{l+\frac12} \boxtimes V_{l+\frac12}$;
it leaves the components of ${\cal M}(-2l-2)$ isomorphic to
$V_{l-\frac12} \boxtimes V_{l-\frac12}$ and $V_{l-\frac12} \boxtimes V_{l+\frac12}$
unchanged.
\item
The component ${\cal M}(-1)$ is annihilated by all four maps
$p_a^{\pm}$, $p_s^{\pm}$.
\end{enumerate}
\end{prop}

\begin{proof}
We show calculations for $p_a^+$ acting on the harmonic part of the component
${\cal M}(2l)$ isomorphic to $V_{l-\frac12} \boxtimes V_{l+\frac12}$.
Suppose $W \in \BB D_R^+$,
using Lemma 23 and Proposition 26 from \cite{FL1}
(see also Proposition \ref{Prop26}) we can rewrite the integrand as
\begin{multline*}
\frac{(Z-W)^{-1}}{N(Z-W)} \cdot Z \cdot F(Z)^+  \\
= \frac{(Z-W)^{-1}}{N(Z-W)} \cdot Z \cdot \begin{pmatrix}
-(l-m)(l+n+1) t^l_{n+1\,\underline{m+1}}(Z) & (l-m)(l-n) t^l_{n\,\underline{m+1}}(Z) \\
-(l+m+1)(l+n+1) t^l_{n+1\,\underline{m}}(Z) & (l+m+1)(l-n) t^l_{n\,\underline{m}}(Z)
\end{pmatrix}  \\
= \frac1{N(Z)} \sum_{l',m',n'}
\begin{pmatrix} (l'-m'+\frac12) t^{l'}_{n'\,\underline{m'+\frac12}}(W)  \\
(l'+m'+\frac12) t^{l'}_{n'\,\underline{m'-\frac12}}(W) \end{pmatrix}
\cdot
\Bigl( t^{l'+\frac 12}_{m'\,\underline{n'-\frac 12}}(Z^{-1}),
t^{l'+\frac 12}_{m'\,\underline{n'+\frac 12}}(Z^{-1}) \Bigr)  \\
\times \begin{pmatrix} -(l-n)(l+n+1) t^{l+\frac12}_{n+\frac12\,\underline{m+\frac12}}(Z) &
(l-n+1)(l-n) t^{l+\frac12}_{n-\frac12\,\underline{m+\frac12}}(Z) \\
-(l+n+2)(l+n+1) t^{l+\frac12}_{n+\frac32\,\underline{m+\frac12}}(Z) &
(l+n+1)(l-n) t^{l+\frac12}_{n+\frac12\,\underline{m+\frac12}}(Z) \end{pmatrix}.
\end{multline*}
When integrated over $Z \in U(2)_R$, by the orthogonality relations
(19) in \cite{FL3}, only two terms survive -- with $l'=l$, $m'=m+\frac12$ and
$n'=n$ or $n+1$:
$$
\begin{pmatrix} -(l-m)(l+n+1) t^l_{n+1\,\underline{m+1}}(W) &
(l-m)(l-n) t^l_{n\,\underline{m+1}}(W)  \\
-(l+m+1)(l+n+1) t^l_{n+1\,\underline{m}}(W) &
(l+m+1)(l-n) t^l_{n\,\underline{m}}(W) \end{pmatrix}^+.
$$
The other cases are similar.
\end{proof}

We introduce $\BB C$-linear maps $\BB S \to \BB S'$ and $\BB S' \to \BB S$:
$$
\begin{pmatrix} s_1 \\ s_2 \end{pmatrix}^{\dagger} = (s_2, -s_1),
\qquad
(s'_1, s'_2)^{\dagger} = \begin{pmatrix} -s'_2 \\ s'_1 \end{pmatrix},
\qquad
\begin{pmatrix} s_1 \\ s_2 \end{pmatrix} \in \BB S ,\:
(s'_1, s'_2) \in \BB S'.
$$
These maps are similar to quaternionic conjugation:
$$
(s^{\dagger})^{\dagger}=s, \qquad (s'^{\dagger})^{\dagger}=s', \qquad
(Zs)^{\dagger} = s^{\dagger} Z^+, \qquad
(s'Z)^{\dagger} = Z^+ s'^{\dagger},
$$
for all $Z \in \HC,$ $s \in \BB S$, $s' \in \BB S'$.
Then we have two more isomorphisms induced by these maps:
\begin{align*}
\sigma &: \HC \simeq \BB S \otimes \BB S' \to \BB S \otimes \BB S, \qquad
s \otimes s' \mapsto s \otimes s'^{\dagger},  \\
\sigma' &: \HC \simeq \BB S \otimes \BB S' \to \BB S' \otimes \BB S', \qquad
s \otimes s' \mapsto s^{\dagger} \otimes s'.
\end{align*}
If we identify $\BB S \otimes \BB S$ with $4$-columns
and $\BB S' \otimes \BB S'$ with $4$-rows as in (\ref{StensorS}), the maps
$\sigma$ and $\sigma'$ can be expressed as
$$
\sigma \begin{pmatrix} z_{11} & z_{12} \\ z_{21} & z_{22} \end{pmatrix}
= \begin{pmatrix} z_{12} \\ -z_{11} \\ z_{22} \\ -z_{21} \end{pmatrix}, \qquad
\sigma' \begin{pmatrix} z_{11} & z_{12} \\ z_{21} & z_{22} \end{pmatrix}
= (-z_{21}, -z_{22}, z_{11}, z_{12}).
$$
Note that $\Sh \otimes \BB S \otimes \BB S$ and
$\Sh \otimes \BB S' \otimes \BB S'$ denote respectively
$\BB S \otimes \BB S$ and $\BB S' \otimes \BB S'$-valued
polynomial functions on $\HC^{\times}$.
The maps $\sigma$ and $\sigma'$ naturally extend to maps
$$
\sigma: {\cal W} \to \Sh \otimes \BB S \otimes \BB S
\qquad \text{and} \qquad
\sigma': {\cal W} \to \Sh \otimes \BB S' \otimes \BB S'.
$$
We are particularly interested in the following four maps:
$$
\tau_a^+, \: \tau_a^-:
{\cal W}' \to \Sh \otimes \BB S \otimes \BB S, \qquad
\tau_s^+, \: \tau_s^-:
{\cal W}' \to \Sh \otimes \BB S' \otimes \BB S':
$$
\begin{center}
\begin{tabular}{lcl}
$\tau_a^+(F) = - \sigma \bigl(\partial \bigl( p_a^+(F) \bigr) \bigr)$,
& \qquad &
$\tau_s^+(F) = \sigma' \bigl(\bigl( p_s^+(F) \bigr)
\overleftarrow{\partial} \bigr)$,  \\
$\tau_a^-(F) = - \sigma \bigl(\partial \bigl( p_a^-(F) \bigr) \bigr)$,
& \qquad &
$\tau_s^-(F) = \sigma' \bigl(\bigl( p_s^-(F) \bigr)
\overleftarrow{\partial} \bigr)$.
\end{tabular}
\end{center}

Combining Lemmas \ref{del-A}, \ref{p(NF)} and Proposition \ref{p(A)},
we obtain the following description of these maps:

\begin{lem}  \label{tau-proj}
The maps $\tau_a^{\pm}$ and $\tau_s^{\pm}$ annihilate all functions of the form
$N(Z)^k \cdot F(Z)$ with $F: \HC^{\times} \to \HC$ harmonic and $k \ne 0$.
For each $l$, $l=\frac12,1,\frac32,2,\dots$,
\begin{enumerate}
\item
$\tau_a^+$ annihilates ${\cal M}(-2l-2)$, projects ${\cal M}(2l)$ onto
the $V_{l-\frac12} \boxtimes V_{l+\frac12}$ component in the decomposition
(\ref{A(2l)}), then applies $\partial$ and identifies the result with
a polynomial function $\HC^{\times} \to \BB S \otimes \BB S$;
\item
$\tau_a^-$ annihilates ${\cal M}(2l)$, projects ${\cal M}(-2l-2)$ onto
the $V_{l+\frac12} \boxtimes V_{l-\frac12}$ component in the decomposition
(\ref{A(-2l-1)}), then applies $\partial$ and identifies the result with
a polynomial function $\HC^{\times} \to \BB S \otimes \BB S$;
\item
$\tau_s^+$ annihilates ${\cal M}(-2l-2)$, projects ${\cal M}(2l)$ onto
the $V_{l+\frac12} \boxtimes V_{l-\frac12}$ component in the decomposition
(\ref{A(2l)}), then applies $\partial$ on the right and identifies the result
with a polynomial function $\HC^{\times} \to \BB S' \otimes \BB S'$;
\item
$\tau_s^-$ annihilates ${\cal M}(2l)$, projects ${\cal M}(-2l-2)$ onto
the $V_{l-\frac12} \boxtimes V_{l+\frac12}$ component in the decomposition
(\ref{A(-2l-1)}), then applies $\partial$ on the right and identifies the
result with a polynomial function $\HC^{\times} \to \BB S' \otimes \BB S'$;
\item
The component ${\cal M}(-1)$ is annihilated by all four maps
$\tau_a^{\pm}$, $\tau_s^{\pm}$.
\end{enumerate}
\end{lem}

\begin{thm}  \label{tau-thm}
For each $F \in {\cal W}'$, $\tau_a^+(F)$ and $\tau_a^-(F)$ are polynomial
functions $\HC^{\times} \to \BB S \odot \BB S$ that are doubly left
regular, and $\tau_s^+(F)$ and $\tau_s^-(F)$ are polynomial functions
$\HC^{\times} \to \BB S' \odot \BB S'$ that are doubly right regular.

Moreover, the maps $\tau_a^{\pm}$, $\tau_s^{\pm}$ produce isomorphisms of
representations of $\mathfrak{sl}(2,\HC)$:
$$
\tau_a^+: (\rho'_2,{\cal W}'/ \ker \tau_a^+) \simeq (\pi_{dl},{\cal F}^+), \qquad
\tau_a^-: (\rho'_2,{\cal W}'/ \ker \tau_a^-) \simeq (\pi_{dl},{\cal F}^-),
$$
$$
\tau_s^+: (\rho'_2,{\cal W}'/ \ker \tau_s^+) \simeq (\pi_{dr},{\cal G}^+), \qquad
\tau_s^-: (\rho'_2,{\cal W}'/ \ker \tau_s^-) \simeq (\pi_{dr},{\cal G}^-).
$$
(Recall that the spaces ${\cal F}^{\pm}$ and ${\cal G}^{\pm}$ were introduced
in equations (\ref{FF'})-(\ref{GG'}).)
\end{thm}

\begin{rem}
The maps $\tau_a^{\pm}$, $\tau_s^{\pm}$ do not produce isomorphisms of
representations of $\mathfrak{gl}(2,\HC)$ because the scalar matrices
act trivially via the $\rho'_2$ action and non-trivially via the
$\pi_{dl}$ and $\pi_{dr}$ actions.
\end{rem}

\begin{proof}
By direct computation, using Lemmas \ref{del-A}, \ref{tau-proj},
identity (\ref{del-t(Z^{-1})}) and Lemma 22 from \cite{FL1},
\begin{align}
\tau_a^+ \bigl(\tilde F_{l,m,n}(Z)\bigr) &=
(2l+1) F_{l-\frac12, m+\frac12, n+\frac12}(Z), \label{tau(F)}  \\
\tau_a^- \bigl(\tilde F'_{l,m,n}(Z)\bigr) &=
-(2l+1) F'_{l-\frac12, m+\frac12, n+\frac12}(Z), \label{tau(F')}  \\
\tau_s^+ \bigl(\tilde G_{l,m,n}(Z)\bigr) &=
(2l+1) G_{l-\frac12, n+\frac12, m+\frac12}(Z), \label{tau(G)}  \\
\tau_s^- \bigl(\tilde G'_{l,m,n}(Z)\bigr) &=
-(2l+1) G'_{l-\frac12, n+\frac12, m+\frac12}(Z);  \label{tau(G')}
\end{align}
\begin{multline}  \label{BF}
\left(\begin{smallmatrix} \partial_{11} \tilde F_{l, m, n}(Z) &
\partial_{21} \tilde F_{l, m, n}(Z) \\
\partial_{12} \tilde F_{l, m, n}(Z) &
\partial_{22} \tilde F_{l, m, n}(Z) \end{smallmatrix}\right)
= \frac{2l+1}{2l}
\left(\begin{smallmatrix} (l-m) \tilde F_{l-\frac12, m+\frac12, n+\frac12}(Z) &
(l-m) \tilde F_{l-\frac12, m+\frac12, n-\frac12}(Z) \\
(l+m+1) \tilde F_{l-\frac12, m-\frac12, n+\frac12}(Z) &
(l+m+1) \tilde F_{l-\frac12, m-\frac12, n-\frac12}(Z) \end{smallmatrix}\right)  \\
+ \frac1{2l}
\left(\begin{smallmatrix} (l-m)(l+n+1) \tilde H_{l-\frac12, m+\frac12, n+\frac12}(Z) &
-(l-m)(l-n) \tilde H_{l-\frac12, m+\frac12, n-\frac12}(Z) \\
(l+m+1)(l+n+1) \tilde H_{l-\frac12, m-\frac12, n+\frac12}(Z) &
-(l+m+1)(l-n) \tilde H_{l-\frac12, m-\frac12, n-\frac12}(Z) \end{smallmatrix}\right),
\end{multline}
\begin{multline}  \label{BF'}
\left(\begin{smallmatrix} \partial_{11} \tilde F'_{l, m, n}(Z) &
\partial_{21} \tilde F'_{l, m, n}(Z) \\
\partial_{12} \tilde F'_{l, m, n}(Z) &
\partial_{22} \tilde F'_{l, m, n}(Z) \end{smallmatrix}\right)
= - \frac{2l+1}{2l+2}
\left(\begin{smallmatrix} (l-n) \tilde F'_{l+\frac12, m-\frac12, n-\frac12}(Z) &
(l-n) \tilde F'_{l+\frac12, m+\frac12, n-\frac12}(Z) \\
(l+n+1) \tilde F'_{l+\frac12, m-\frac12, n+\frac12}(Z) &
(l+n+1) \tilde F'_{l+\frac12, m+\frac12, n+\frac12}(Z) \end{smallmatrix}\right)  \\
+ \frac1{2l+2}
\left(\begin{smallmatrix} (l+m+1)(l-n) \tilde H'_{l+\frac12, m-\frac12, n-\frac12}(Z) &
-(l-m)(l-n) \tilde H'_{l+\frac12, m+\frac12, n-\frac12}(Z) \\
(l+m+1)(l+n+1) \tilde H'_{l+\frac12, m-\frac12, n+\frac12}(Z) &
-(l-m)(l+n+1) \tilde H'_{l+\frac12, m+\frac12, n+\frac12}(Z) \end{smallmatrix}\right),
\end{multline}
\begin{multline*}
\left(\begin{smallmatrix} \partial_{11} \tilde G_{l, m, n}(Z) &
\partial_{21} \tilde G_{l, m, n}(Z) \\
\partial_{12} \tilde G_{l, m, n}(Z) &
\partial_{22} \tilde G_{l, m, n}(Z) \end{smallmatrix}\right)
= \frac{2l+1}{2l}
\left(\begin{smallmatrix} (l-m-1) \tilde G_{l-\frac12, m+\frac12, n+\frac12}(Z) &
(l-m-1) \tilde G_{l-\frac12, m+\frac12, n-\frac12}(Z) \\
(l+m) \tilde G_{l-\frac12, m-\frac12, n+\frac12}(Z) &
(l+m) \tilde G_{l-\frac12, m-\frac12, n-\frac12}(Z) \end{smallmatrix}\right)  \\
+ \frac1{2l}
\left(\begin{smallmatrix} \tilde H_{l-\frac12, m+\frac12, n+\frac12}(Z) &
\tilde H_{l-\frac12, m+\frac12, n-\frac12}(Z) \\
- \tilde H_{l-\frac12, m-\frac12, n+\frac12}(Z) &
- \tilde H_{l-\frac12, m-\frac12, n-\frac12}(Z) \end{smallmatrix}\right),
\end{multline*}
\begin{multline*}
\left(\begin{smallmatrix} \partial_{11} \tilde G'_{l, m, n}(Z) &
\partial_{21} \tilde G'_{l, m, n}(Z) \\
\partial_{12} \tilde G'_{l, m, n}(Z) &
\partial_{22} \tilde G'_{l, m, n}(Z) \end{smallmatrix}\right)
= -\frac{2l+1}{2l+2}
\left(\begin{smallmatrix} (l-n+1) \tilde G'_{l+\frac12, m-\frac12, n-\frac12}(Z) &
(l-n+1) \tilde G'_{l+\frac12, m+\frac12, n-\frac12}(Z) \\
(l+n+2) \tilde G'_{l+\frac12, m-\frac12, n+\frac12}(Z) &
(l+n+2) \tilde G'_{l+\frac12, m+\frac12, n+\frac12}(Z) \end{smallmatrix}\right)  \\
+ \frac1{2l+2}
\left(\begin{smallmatrix} \tilde H'_{l+\frac12, m-\frac12, n-\frac12}(Z) &
\tilde H'_{l+\frac12, m+\frac12, n-\frac12}(Z) \\
- \tilde H'_{l+\frac12, m-\frac12, n+\frac12}(Z) &
- \tilde H'_{l+\frac12, m+\frac12, n+\frac12}(Z) \end{smallmatrix}\right);
\end{multline*}
{\small
$$
\begin{pmatrix} \partial_{11} F_{l, m, n}(Z) & \partial_{21} F_{l, m, n}(Z) \\
\partial_{12} F_{l, m, n}(Z) & \partial_{22} F_{l, m, n}(Z) \end{pmatrix}
= \begin{pmatrix} (l-m+1) F_{l-\frac12, m+\frac12, n+\frac12}(Z) &
(l-m+1) F_{l-\frac12, m+\frac12, n-\frac12}(Z) \\
(l+m+1) F_{l-\frac12, m-\frac12, n+\frac12}(Z) &
(l+m+1) F_{l-\frac12, m-\frac12, n-\frac12}(Z) \end{pmatrix},
$$
$$
\begin{pmatrix} \partial_{11} F'_{l, m, n}(Z) & \partial_{21} F'_{l, m, n}(Z) \\
\partial_{12} F'_{l, m, n}(Z) & \partial_{22} F'_{l, m, n}(Z) \end{pmatrix}
= -\begin{pmatrix} (l-n+1) F'_{l+\frac12, m-\frac12, n-\frac12}(Z) &
(l-n+1) F'_{l+\frac12, m+\frac12, n-\frac12}(Z) \\
(l+n+1) F'_{l+\frac12, m-\frac12, n+\frac12}(Z) &
(l+n+1) F'_{l+\frac12, m+\frac12, n+\frac12}(Z) \end{pmatrix},
$$
$$
\begin{pmatrix} \partial_{11} G_{l, m, n}(Z) & \partial_{21} G_{l, m, n}(Z) \\
\partial_{12} G_{l, m, n}(Z) & \partial_{22} G_{l, m, n}(Z) \end{pmatrix}
= \begin{pmatrix} (l-n) G_{l-\frac12, m+\frac12, n+\frac12}(Z) &
(l-n) G_{l-\frac12, m-\frac12, n+\frac12}(Z) \\
(l+n) G_{l-\frac12, m+\frac12, n-\frac12}(Z) &
(l+n) G_{l-\frac12, m-\frac12, n-\frac12}(Z) \end{pmatrix},
$$
$$
\begin{pmatrix} \partial_{11} G'_{l, m, n}(Z) & \partial_{21} G'_{l, m, n}(Z) \\
\partial_{12} G'_{l, m, n}(Z) & \partial_{22} G'_{l, m, n}(Z) \end{pmatrix}
= -\begin{pmatrix} (l-m+2) G'_{l+\frac12, m-\frac12, n-\frac12}(Z) &
(l-m+2) G'_{l+\frac12, m-\frac12, n+\frac12}(Z) \\
(l+m+2) G'_{l+\frac12, m+\frac12, n-\frac12}(Z) &
(l+m+2) G'_{l+\frac12, m+\frac12, n+\frac12}(Z) \end{pmatrix}.
$$
}
It follows from Lemmas \ref{Lie-alg-action} and \ref{rho'_2-action-lem}
that, if $B \in \HC$,
\begin{align*}
\tau_a^+ \circ \rho'_2\bigl(\begin{smallmatrix} 0 & B \\
0 & 0 \end{smallmatrix}\bigr) \tilde F_{l,m,n}(Z) &=
\pi_{dl}\bigl(\begin{smallmatrix} 0 & B \\ 0 & 0 \end{smallmatrix}\bigr)
\circ \tau_a^+ \tilde F_{l,m,n}(Z),  \\
\tau_a^- \circ \rho'_2\bigl(\begin{smallmatrix} 0 & B \\
0 & 0 \end{smallmatrix}\bigr) \tilde F'_{l,m,n}(Z) &=
\pi_{dl}\bigl(\begin{smallmatrix} 0 & B \\ 0 & 0 \end{smallmatrix}\bigr)
\circ \tau_a^- \tilde F'_{l,m,n}(Z),  \\
\tau_s^+ \circ \rho'_2\bigl(\begin{smallmatrix} 0 & B \\
0 & 0 \end{smallmatrix}\bigr) \tilde G_{l,m,n}(Z) &=
\pi_{dr}\bigl(\begin{smallmatrix} 0 & B \\ 0 & 0 \end{smallmatrix}\bigr)
\circ \tau_s^+ \tilde G_{l,m,n}(Z), \\
\tau_s^- \circ \rho'_2\bigl(\begin{smallmatrix} 0 & B \\
0 & 0 \end{smallmatrix}\bigr) \tilde G'_{l,m,n}(Z) &=
\pi_{dr}\bigl(\begin{smallmatrix} 0 & B \\ 0 & 0 \end{smallmatrix}\bigr)
\circ \tau_s^- \tilde G'_{l,m,n}(Z).
\end{align*}

Let $\bigl(\begin{smallmatrix} 0 & 1 \\ 1 & 0 \end{smallmatrix}\bigr) \in
GL(2,\HC)$, then
\begin{equation}  \label{inversion}
\rho'_2 \bigl(\begin{smallmatrix} 0 & 1 \\ 1 & 0 \end{smallmatrix}\bigr)
\tilde F_{l,m,n}(Z) = \tilde F'_{l,m,n}(Z), \qquad
\rho'_2 \bigl(\begin{smallmatrix} 0 & 1 \\ 1 & 0 \end{smallmatrix}\bigr)
\tilde G_{l,m,n}(Z) = \tilde G'_{l,m,n}(Z);
\end{equation}
$$
\pi_{dl} \bigl(\begin{smallmatrix} 0 & 1 \\ 1 & 0 \end{smallmatrix}\bigr)
F_{l,m,n}(Z) = F'_{l,m,n}(Z), \qquad
\pi_{dr} \bigl(\begin{smallmatrix} 0 & 1 \\ 1 & 0 \end{smallmatrix}\bigr)
G_{l,m,n}(Z) = G'_{l,m,n}(Z).
$$
This implies that the maps $\tau_a^{\pm}$, $\tau_s^{\pm}$ commute with actions of
$\bigl(\begin{smallmatrix} 0 & 0 \\ C & 0 \end{smallmatrix}\bigr)$,
$C \in \HC$.
Since matrices of the form
$\bigl(\begin{smallmatrix} 0 & B \\ 0 & 0 \end{smallmatrix}\bigr)$ and
$\bigl(\begin{smallmatrix} 0 & 0 \\ C & 0 \end{smallmatrix}\bigr)$,
$B, C \in \HC$, generate $\mathfrak{sl}(2,\HC)$, the maps
$\tau_a^{\pm}$, $\tau_s^{\pm}$ are $\mathfrak{sl}(2,\HC)$-equivariant.
\end{proof}

\begin{cor}  \label{FG-irred}
The following representations
\begin{align*}
(\rho'_2,{\cal W}'/ \ker \tau_a^+)
= (\rho'_2,{\cal M}/ (\ker \tau_a^+ \cap {\cal M}))
&\simeq (\pi_{dl},{\cal F}^+), \\
(\rho'_2,{\cal W}'/ \ker \tau_a^-)
= (\rho'_2,{\cal M}/ (\ker \tau_a^- \cap {\cal M}))
&\simeq (\pi_{dl},{\cal F}^-),  \\
(\rho'_2,{\cal W}'/ \ker \tau_s^+)
= (\rho'_2,{\cal M}/ (\ker \tau_s^+ \cap {\cal M}))
&\simeq (\pi_{dr},{\cal G}^+), \\
(\rho'_2,{\cal W}'/ \ker \tau_s^-)
= (\rho'_2,{\cal M}/ (\ker \tau_s^- \cap {\cal M}))
&\simeq (\pi_{dr},{\cal G}^-)
\end{align*}
of $\mathfrak{sl}(2,\HC)$ are irreducible.
\end{cor}

\begin{proof}
By Lemma \ref{tau-proj},
$$
{\cal W}'/ \ker \tau_a^+ =
\bigoplus_{l=\frac12,1,\frac32,2,\dots} V_{l-\frac12} \boxtimes V_{l+\frac12}
$$
as a representation of $SU(2) \times SU(2)$, where each direct summand
$V_{l-\frac12} \boxtimes V_{l+\frac12}$ is spanned by $\tilde F_{l,m,n}$'s
with $-l-1 \le m \le l$, $-l \le n \le l-1$.
And by Lemma \ref{rho'_2-action-lem} and equations (\ref{BF}), (\ref{BF'}),
(\ref{inversion}) we have:
$$
\rho'_2\bigl(\begin{smallmatrix} 0 & B \\
0 & 0 \end{smallmatrix}\bigr) \tilde F_{l,m,n}(Z)
= - \frac{2l+1}{2l} \tr \left[ B
\left( \begin{smallmatrix} (l-m) \tilde F_{l-\frac12, m+\frac12, n+\frac12}(Z) &
(l-m) \tilde F_{l-\frac12, m+\frac12, n-\frac12}(Z) \\
(l+m+1) \tilde F_{l-\frac12, m-\frac12, n+\frac12}(Z) &
(l+m+1) \tilde F_{l-\frac12, m-\frac12, n-\frac12}(Z) \end{smallmatrix} \right) \right],
$$
$$
\rho'_2\bigl(\begin{smallmatrix} 0 & 0 \\
C & 0 \end{smallmatrix}\bigr) \tilde F_{l,m,n}(Z)
= \frac{2l+1}{2l+2} \tr \left[ C
\left( \begin{smallmatrix} (l-n) \tilde F_{l+\frac12, m-\frac12, n-\frac12}(Z) &
(l-n) \tilde F_{l+\frac12, m+\frac12, n-\frac12}(Z) \\
(l+n+1) \tilde F_{l+\frac12, m-\frac12, n+\frac12}(Z) &
(l+n+1) \tilde F_{l+\frac12, m+\frac12, n+\frac12}(Z) \end{smallmatrix} \right) \right]
$$
in the quotient space ${\cal W}'/ \ker \tau_a^+$.
It follows that any non-zero vector in ${\cal W}'/ \ker \tau_a^+$
generates the $V_0 \boxtimes V_1$ component and hence the whole space,
thus proving irreducibility of
$(\rho'_2,{\cal W}'/ \ker \tau_a^+) \simeq (\pi_{dl},{\cal F}^+)$.
The other cases are similar.
\end{proof}

\subsection{Quaternionic Analogue of Cauchy's Integral Formula for the
Second Order Pole}

Note that the maps $\tau_a^{\pm}$ and $\tau_s^{\pm}$ on ${\cal W}'$
are given by integrating over a four-dimensional cycle $U(2)_R$. For example,
$$
(\tau_a^+ F)(W) = -\sigma \circ \partial_W \left( \frac{i}{2\pi^3}
\int_{Z \in U(2)_R} F(Z) \cdot Z^+ \cdot \left[ \frac{(Z-W)^{-1}}{N(Z-W)} \right]^+
\frac{dV}{N(Z)} \right), \qquad W \in \BB D_R^+.
$$
When restricted to ${\cal M} \subset {\cal W}'$, these maps can be rewritten
so that the integral is over the sphere $S^3_R$ of radius $R$.

\begin{thm}
The restrictions of $\tau_a^{\pm}$ and $\tau_s^{\pm}$ to ${\cal M}$ can be
rewritten as
\begin{align*}
(\tau_a^+ F)(W) &= -\sigma \circ \partial_W \left[\frac1{2\pi^2}
\int_{Z \in S^3_R} \frac{(Z-W)^{-1}}{N(Z-W)} \cdot Dz \cdot F^+(Z) \right]^+,
\qquad W \in \BB D_R^+,  \\
(\tau_a^- F)(W) &= -\sigma \circ \partial_W
\left[\frac{-1}{2\pi^2} \int_{Z \in S^3_R} \frac{(Z-W)^{-1}}{N(Z-W)} \cdot Dz
\cdot F^+(Z) \right]^+, \qquad W \in \BB D_R^-,  \\
(\tau_s^+ F)(W) &= \sigma' \left( \left[\frac1{2\pi^2}
\int_{Z \in S^3_R} F^+(Z) \cdot Dz \cdot \frac{(Z-W)^{-1}}{N(Z-W)} \right]^+
\overleftarrow{\partial}_W \right), \qquad W \in \BB D_R^+,  \\
(\tau_s^- F)(W) &= \sigma' \left( \left[\frac1{2\pi^2}
\int_{Z \in S^3_R} F^+(Z) \cdot Dz \cdot \frac{(Z-W)^{-1}}{N(Z-W)} \right]^+
\overleftarrow{\partial}_W \right), \qquad W \in \BB D_R^-,
\end{align*}
where $F \in {\cal M}$.
\end{thm}

\begin{rem}
Note that (apart from quaternionic conjugation),
the maps $\tau_a^{\pm}$ and $\tau_s^{\pm}$ restricted to ${\cal M}$
are essentially derivatives of the classical Cauchy-Fueter formulas for
left and right regular functions.
Thus $\tau_a^{\pm}$ and $\tau_s^{\pm}$ can be treated as another analogue of
Cauchy's integral formula for the second order pole (\ref{2pole-intro}).
\end{rem}

\begin{proof}
First, we determine the effect of these maps on the non-harmonic parts
of functions of the form (\ref{A(Z)}).
Suppose that $F_d: \HC^{\times} \to \HC$ be homogeneous of degree $d$ and
harmonic, then
$$
\nabla^+ (\partial F_d(Z) \partial) =
(\partial F_d(Z) \partial) \overleftarrow{\nabla}^+ =0,
$$
i.e. the function $\partial F_d(Z) \partial$ is both left and right regular.
Thus we can apply the Cauchy-Fueter formulas.
In the case of $\tau_a^+$, the integral becomes
$$
\left[ \frac1{2\pi^2} \int_{Z \in S^3_R} \frac{(Z-W)^{-1}}{N(Z-W)} \cdot Dz
\cdot \bigl( \partial F_d(Z) \partial \bigr) \right]^+
= \begin{cases}
(\partial F_d(W) \partial)^+ & \text{if $d \ge 0$};  \\
0 & \text{if $d<0$}.
\end{cases}
$$
Applying $\partial_W$ results in zero.
The cases of $\tau_a^-$ and $\tau_s^{\pm}$ are similar.
Using Lemma 6 from \cite{FL1}, we replace $Dz |_{S^3_R}$ with $R^{-1} Z \,dS$,
where $dS$ is the usual Euclidean volume element on the $3$-sphere $S^3_R$.
Then the proof proceeds exactly as those of Proposition \ref{p(A)}
and Theorem \ref{tau-thm}, except using orthogonality relations (17)
from \cite{FL3} instead of (19).
\end{proof}

\subsection{Invariant Bilinear Pairing on $(\rho'_2, {\cal M})$}

We define a bilinear pairing on ${\cal M}$ by
\begin{equation}  \label{A-pairing}
\langle F_1, F_2 \rangle_{\cal M} = \frac1{2\pi^2} \int_{Z \in S^3_R}
\tr \Bigl( \degt^{-1} \bigl( (F_1 \overleftarrow{\partial})(Z) \cdot Z \bigr)
\cdot
\bigl( Z \cdot (\overrightarrow{\partial} F_2)(Z) \bigr)^+ \Bigr) \frac{dS}R.
\end{equation}
Note that if $F_1$ happens to be $N(Z)^{-1} \cdot Z$, then by Lemma \ref{del-A}
$$
(F_1 \overleftarrow{\partial})(Z) \cdot Z = N(Z)^{-1} \cdot Z,
$$
and $\degt^{-1} \bigl( (F_1 \overleftarrow{\partial})(Z) \cdot Z \bigr)$
is undefined. Thus we declare the pairing (\ref{A-pairing}) to be
zero on ${\cal M}(-1)$.
We will see shortly that the pairing does not depend on the choice of $R>0$.

\begin{thm}
The bilinear pairing $(\ref{A-pairing})$ on ${\cal M}$ is
$\mathfrak{sl}(2,\HC)$-invariant and does not depend on the choice of $R>0$.
We have the following orthogonality relations:
$$
\langle \tilde G_{l,m,n}(Z), \tilde F'_{l',n',m'}(Z) \rangle_{\cal M}
= - \langle \tilde G'_{l,m,n}(Z), \tilde F_{l',n',m'}(Z) \rangle_{\cal M}
= (2l+1)^2 \delta_{ll'}\delta_{mm'} \delta_{nn'};
$$
all other pairing between
$\tilde F_{l,m,n}(Z)$, $\tilde F'_{l,m,n}(Z)$,
$\tilde G_{l,m,n}(Z)$, $\tilde G'_{l,m,n}(Z)$,
$\tilde H_{l,m,n}(Z)$, $\tilde H'_{l,m,n}(Z)$,
$\tilde I_{l,m,n}(Z)$, $\tilde I'_{l,m,n}(Z)$
are zero; in particular,
$$
\langle \tilde F_{l,m,n}(Z), \tilde G'_{l',n',m'}(Z) \rangle_{\cal M}
= \langle \tilde F'_{l,m,n}(Z), \tilde G_{l',n',m'}(Z) \rangle_{\cal M} =0.
$$
\end{thm}

\begin{proof}
First, we check the orthogonality relations; they follow from
Lemma \ref{del-A} and orthogonality relations (17) from \cite{FL3}.
The orthogonality relations also imply that the pairing is independent of
the choice of $R>0$.

Using Proposition \ref{orthogonality-dreg-prop} and equations
(\ref{tau(F)})-(\ref{tau(G')}) we can relate the bilinear pairing
(\ref{A-pairing}) to the pairing for doubly regular functions
(\ref{pairing1}) as
$$
\langle F_1, F_2 \rangle_{\cal M} =
\bigl\langle \tau_a^+(F_2), \tau_s^-(F_1) \bigr\rangle_{\cal DR}
+ \bigl\langle \tau_a^-(F_2), \tau_s^+(F_1) \bigr\rangle_{\cal DR}.
$$
Since the right hand side is $\mathfrak{sl}(2,\HC)$-invariant, 
the pairing (\ref{A-pairing}) is $\mathfrak{sl}(2,\HC)$-invariant as well.
\end{proof}



\section{The Quaternionic Chain Complex and Decomposition of
$(\rho_2,{\cal W})$, $(\rho'_2,{\cal W}')$ into Irreducible Components}  \label{decomp-section}

\subsection{Decomposition of $(\rho_2,{\cal W})$}  \label{W-decomp-subsect}

In this subsection we find explicit decomposition of $(\rho_2,{\cal W})$
into irreducible components.
The idea is to use the quaternionic chain complex (\ref{W-sequence}) and
deal separately with
$$
\ker (\tr \circ \partial^+) \subset {\cal W} \qquad \text{and} \qquad
{\cal W}/\ker (\tr \circ \partial^+) \subset \Sh.
$$
We will see that there is a total of thirteen irreducible components
with $\ker (\tr \circ \partial^+)$ having eight components and
${\cal W}/\ker (\tr \circ \partial^+)$ having five.

\begin{lem}
The image of the map $\tr \circ \partial^+ : {\cal W} \to \Sh$ from
(\ref{W-sequence}) is ${\cal I}^+ + {\cal I}^-$.
\end{lem}

\begin{proof}
Indeed, $\tr \circ \partial^+(N(Z)^{-1} \cdot Z^+) = 2 N(Z)^{-1}$,
which generates
${\cal I}^+$, and $\tr \circ \partial^+(N(Z)^{-3} \cdot Z^+) = -2 N(Z)^{-3}$,
which generates ${\cal I}^-$.
Thus the image of $\tr \circ \partial^+$ contains ${\cal I}^+ + {\cal I}^-$.
It remains to show that the image is ${\cal I}^+ + {\cal I}^-$.
Otherwise, by Theorem \ref{rho-decomposition}, the image of
$\tr \circ \partial^+$ is all of $\Sh$.
In particular, there exists a homogeneous function $F \in {\cal W}$
of degree $-3$ such that $\tr \circ \partial^+ F = N(Z)^{-2}$.
Such a function $F$ must be a linear combination of
$$
N(Z)^{-2} \cdot H_1(Z), \quad N(Z)^{-3} \cdot H_3(Z), \quad
N(Z)^{-4} \cdot H_5(Z), \dots
$$
with $H_1(Z)$, $H_3(Z)$, $H_5(Z), \dots$ harmonic and homogeneous of degrees
$1$, $3$, $5, \dots$ respectively.
On the other hand, the $\rho$ action of
$\mathfrak{sl}(2,\BB C) \times \mathfrak{sl}(2,\BB C)$
on $N(Z)^{-2}$ is trivial, hence the $\rho_2$ action of
$\mathfrak{sl}(2,\BB C) \times \mathfrak{sl}(2,\BB C)$ on
$H_1(Z)$, $H_3(Z)$, $H_5(Z), \dots$ must be trivial as well.
Since each $H_k(Z)$ lies in
$(V_{\frac{k}2} \otimes V_{\frac12}) \boxtimes (V_{\frac{k}2} \otimes V_{\frac12})$,
it follows that $H_3(Z)=H_5(Z)=\dots=0$ and $H_1(Z)$ is proportional to
$N(Z)^{-2} \cdot Z^+$. But $\tr \circ \partial^+(N(Z)^{-2} \cdot Z^+) = 0$,
which gives us a contradiction.
\end{proof}

Combining this with Theorem \ref{rho-decomposition} we obtain:

\begin{cor}  \label{W/ker-decomp-cor}
The quotient $\bigl( \rho_2, {\cal W}/\ker (\tr \circ \partial^+) \bigr)$
has five irreducible components that are isomorphic to
$$
(\rho, \Sh^+), \quad (\rho, \Sh^-), \quad (\rho, {\cal J}), \quad
\bigl( \rho, {\cal I}^+/(\Sh^+ \oplus {\cal J}) \bigr), \quad
\bigl( \rho, {\cal I}^-/(\Sh^- \oplus {\cal J}) \bigr).
$$
\end{cor}

Our next task is to decompose $\ker (\tr \circ \partial^+)$.

\begin{lem}
We have:
\begin{multline*}
\tr \circ \partial^+ \left[ N(Z)^k \cdot \begin{pmatrix}
\alpha_{11} t^l_{n+1\,\underline{m+1}}(Z) & \alpha_{12} t^l_{n\,\underline{m+1}}(Z) \\
\alpha_{21} t^l_{n+1\,\underline{m}}(Z) & \alpha_{22} t^l_{n\,\underline{m}}(Z)
\end{pmatrix} \right]  \\
= \frac{2l+k+1}{2l+1} \bigl( (l+m+1)(\alpha_{11}-\alpha_{12})
+(l-m)(\alpha_{22}-\alpha_{21}) \bigr)
\cdot  N(Z)^k \cdot t^{l-\frac12}_{n+\frac12\,\underline{m+\frac12}}(Z)  \\
+ \frac{k}{2l+1} \bigl( (l-n)(\alpha_{11}+\alpha_{21})
+(l+n+1)(\alpha_{12}+\alpha_{22}) \bigr)
\cdot N(Z)^{k-1} \cdot t^{l+\frac12}_{n+\frac12\,\underline{m+\frac12}}(Z);
\end{multline*}
\begin{multline*}
\tr \circ \partial^+ \left[ N(Z)^{k-1} \cdot \begin{pmatrix}
\beta_{11} t^l_{n\,\underline{m}}(Z^{-1}) & \beta_{12} t^l_{n\,\underline{m+1}}(Z^{-1}) \\
\beta_{21} t^l_{n+1\,\underline{m}}(Z^{-1}) & \beta_{22} t^l_{n+1\,\underline{m+1}}(Z^{-1})
\end{pmatrix} \right]  \\
= -\frac{2l-k+1}{2l+1} \bigl( (l+n+1)(\beta_{11}-\beta_{12})
+(l-n)(\beta_{22}-\beta_{21}) \bigr)
\cdot  N(Z)^{k-1} \cdot t^{l+\frac12}_{n+\frac12\,\underline{m+\frac12}}(Z^{-1})  \\
+ \frac{k}{2l+1} \bigl( (l-m)(\beta_{11}+\beta_{21})
+(l+m+1)(\beta_{12}+\beta_{22}) \bigr)
\cdot N(Z)^{k-2} \cdot t^{l-\frac12}_{n+\frac12\,\underline{m+\frac12}}(Z^{-1}).
\end{multline*}
\end{lem}

\begin{proof}
The result follows by direct computation using Lemma 22 from \cite{FL1}
and identities (\ref{Zt-identity}), (\ref{del-t(Z^{-1})}).
\end{proof}

Let
$$
{\bf F_{l,m,n}}(Z) = \left(\begin{smallmatrix}
-(l-m)(l+n+1) t^l_{n+1\,\underline{m+1}}(Z) & (l-m)(l-n) t^l_{n\,\underline{m+1}}(Z) \\
-(l+m+1)(l+n+1) t^l_{n+1\,\underline{m}}(Z) & (l+m+1)(l-n) t^l_{n\,\underline{m}}(Z)
\end{smallmatrix}\right) ;\quad
\begin{smallmatrix} l=\frac12,1,\frac32,2,\dots \\
-l-1 \le m \le l \\  -l \le n \le l-1 \end{smallmatrix};
$$
$$
{\bf G_{l,m,n}}(Z) = \left(\begin{smallmatrix}
-t^l_{n+1\,\underline{m+1}}(Z) & -t^l_{n\,\underline{m+1}}(Z) \\
t^l_{n+1\,\underline{m}}(Z) & t^l_{n\,\underline{m}}(Z) \end{smallmatrix}\right) ;\quad
\begin{smallmatrix} l=\frac12,1,\frac32,2,\dots \\
-l \le m \le l-1 \\  -l-1 \le n \le l \end{smallmatrix};
$$
\begin{multline*}
{\bf H_{k,l,m,n}}(Z) = \frac{2l+k+2}{2l+1} N(Z)^k \cdot \left(\begin{smallmatrix}
(l-m) t^l_{n+1\,\underline{m+1}}(Z) & (l-m) t^l_{n\,\underline{m+1}}(Z) \\
(l+m+1) t^l_{n+1\,\underline{m}}(Z) & (l+m+1) t^l_{n\,\underline{m}}(Z)
\end{smallmatrix}\right)  \\
- \frac{k}{2l+3} N(Z)^{k-1} \cdot \left(\begin{smallmatrix}
(l+n+2) t^{l+1}_{n+1\,\underline{m+1}}(Z) & -(l-n+1) t^{l+1}_{n\,\underline{m+1}}(Z) \\
-(l+n+2) t^{l+1}_{n+1\,\underline{m}}(Z) & (l-n+1) t^{l+1}_{n\,\underline{m}}(Z)
\end{smallmatrix}\right) ;\quad
\begin{smallmatrix} k=0, \pm1, \pm2, \dots \\
l=0,\frac12,1,\frac32,\dots \\
-l-1 \le m,n \le l \end{smallmatrix}.
\end{multline*}
We also introduce
$$
{\bf F'_{l,m,n}}(Z) = \frac1{N(Z)}
\left(\begin{smallmatrix} -(l+m+1)(l-n) t^l_{n\,\underline{m}}(Z^{-1}) &
(l-m)(l-n) t^l_{n\,\underline{m+1}}(Z^{-1}) \\
-(l+m+1)(l+n+1) t^l_{n+1\,\underline{m}}(Z^{-1}) &
(l-m)(l+n+1) t^l_{n+1\,\underline{m+1}}(Z^{-1}) \end{smallmatrix}\right) ;\quad
\begin{smallmatrix} l=\frac12,1,\frac32,2,\dots \\
-l-1 \le m \le l \\  -l \le n \le l-1 \end{smallmatrix};
$$
$$
{\bf G'_{l,m,n}}(Z) = \frac1{N(Z)} \left(\begin{smallmatrix}
-t^l_{n\,\underline{m}}(Z^{-1}) & -t^l_{n\,\underline{m+1}}(Z^{-1}) \\
t^l_{n+1\,\underline{m}}(Z^{-1}) & t^l_{n+1\,\underline{m+1}}(Z^{-1})
\end{smallmatrix}\right) ;\quad
\begin{smallmatrix} l=\frac12,1,\frac32,2,\dots \\
-l \le m \le l-1 \\  -l-1 \le n \le l \end{smallmatrix};
$$
\begin{multline*}
{\bf H'_{k,l,m,n}}(Z) = \frac{k+1}{2l+1} N(Z)^{k-1} \cdot \left(\begin{smallmatrix}
-(l+m+1) t^l_{n\,\underline{m}}(Z^{-1}) & (l-m) t^l_{n\,\underline{m+1}}(Z^{-1}) \\
(l+m+1) t^l_{n+1\,\underline{m}}(Z^{-1}) & -(l-m) t^l_{n+1\,\underline{m+1}}(Z^{-1})
\end{smallmatrix}\right)  \\
- \frac{2l-k+1}{2l+3} N(Z)^k \cdot \left(\begin{smallmatrix}
(l-n+1) t^{l+1}_{n\,\underline{m}}(Z^{-1}) & (l-n+1) t^{l+1}_{n\,\underline{m+1}}(Z^{-1}) \\
(l+n+2) t^{l+1}_{n+1\,\underline{m}}(Z^{-1}) & (l+n+2) t^{l+1}_{n+1\,\underline{m+1}}(Z^{-1})
\end{smallmatrix}\right) ;\quad
\begin{smallmatrix} k=0, \pm1, \pm2, \dots \\
l=0,\frac12,1,\frac32,\dots \\
-l-1 \le m,n \le l \end{smallmatrix}.
\end{multline*}

Let $K = U(2) \times U(2)$ sitting as the diagonal subgroup of $GL(2,\HC)$.

\begin{prop}
The functions
\begin{equation}  \label{W-K-types}
N(Z)^k \cdot {\bf F_{l,m,n}}(Z), \quad N(Z)^k \cdot {\bf G_{l,m,n}}(Z), \quad
{\bf H_{k,l,m,n}}(Z) \quad \text{and} \quad N(Z)^{-2} \cdot Z^+
\end{equation}
span the kernel of $\tr \circ \partial^+ : {\cal W} \to \Sh$ and generate the
$K$-types of the kernel.
More precisely, as representations of $SU(2) \times SU(2)$,
$$
V_0 \boxtimes V_0 = \BB C\text{-span of } \bigl\{ N(Z)^{-2} \cdot Z^+ \bigr\},
$$
and, for $k$, $l$ fixed,
\begin{align*}
V_{l-\frac12} \boxtimes V_{l+\frac12} &= \BB C\text{-span of }
\left\{ N(Z)^k \cdot {\bf F_{l,m,n}}(Z) ;\: \begin{smallmatrix}
-l-1 \le m \le l \\  -l \le n \le l-1 \end{smallmatrix} \right\},  \\
V_{l+\frac12} \boxtimes V_{l-\frac12} &= \BB C\text{-span of }
\left\{ N(Z)^k \cdot {\bf G_{l,m,n}}(Z) ;\: \begin{smallmatrix}
-l \le m \le l-1 \\  -l-1 \le n \le l \end{smallmatrix} \right\},  \\
V_{l+\frac12} \boxtimes V_{l+\frac12} &= \BB C\text{-span of }
\left\{ {\bf H_{k,l,m,n}}(Z) ;\: -l-1 \le m,n \le l \right\}.
\end{align*}
\end{prop}

\begin{proof}
Clearly, these functions are $K$-finite, linearly independent and annihilated
by $\tr \circ \partial^+$. It remains to show that these functions
span all of the kernel of $\tr \circ \partial^+$.
This is done by checking, for each $d \in \BB Z$, that if one takes the
$K$-types of
$$
{\cal W}(d) =
\{ F(Z) \in {\cal W};\: \text{$F(Z)$ is homogeneous of degree $d$} \}
$$
and ``subtracts'' the $K$-types of
$$
({\cal I}^+ + {\cal I}^-)(d-1) = \{ f(Z) \in {\cal I}^+ + {\cal I}^-;\:
\text{$f(Z)$ is homogeneous of degree $d-1$} \}
$$
then all the remaining $K$-types are accounted for in (\ref{W-K-types}).
\end{proof}

We compute the action of
$\bigl(\begin{smallmatrix} 0 & B \\ 0 & 0 \end{smallmatrix}\bigr)
\in \mathfrak{gl}(2,\HC)$, $B \in \HC$, on these generators:
{\small
\begin{multline}  \label{BFW}
\begin{pmatrix}
\partial_{11} ( N(Z)^k \cdot {\bf F_{l,m,n}}(Z)) &
\partial_{21} ( N(Z)^k \cdot {\bf F_{l,m,n}}(Z)) \\
\partial_{12} ( N(Z)^k \cdot {\bf F_{l,m,n}}(Z)) &
\partial_{22} ( N(Z)^k \cdot {\bf F_{l,m,n}}(Z))
\end{pmatrix}  \\
=
\frac{2l+k+1}{2l} N(Z)^k \begin{pmatrix}
(l-m) {\bf F_{l-\frac12,m+\frac12,n+\frac12}}(Z) &
(l-m) {\bf F_{l-\frac12,m+\frac12,n-\frac12}}(Z) \\
(l+m+1) {\bf F_{l-\frac12,m-\frac12,n+\frac12}}(Z) &
(l+m+1) {\bf F_{l-\frac12,m-\frac12,n-\frac12}}(Z)
\end{pmatrix}  \\
+
\frac{k}{2l+2} N(Z)^{k-1} \begin{pmatrix}
(l+n+1) {\bf F_{l+\frac12,m+\frac12,n+\frac12}}(Z) &
-(l-n) {\bf F_{l+\frac12,m+\frac12,n-\frac12}}(Z) \\
-(l+n+1) {\bf F_{l+\frac12,m-\frac12,n+\frac12}}(Z) &
(l-n) {\bf F_{l+\frac12,m-\frac12,n-\frac12}}(Z)
\end{pmatrix}  \\
+\frac1{2l+1} \begin{pmatrix}
(l-m)(l+n+1) {\bf H_{k,l-\frac12,m+\frac12,n+\frac12}}(Z) &
-(l-m)(l-n) {\bf H_{k,l-\frac12,m+\frac12,n-\frac12}}(Z) \\
(l+m+1)(l+n+1) {\bf H_{k,l-\frac12,m-\frac12,n+\frac12}}(Z) &
-(l+m+1)(l-n) {\bf H_{k,l-\frac12,m-\frac12,n-\frac12}}(Z)
\end{pmatrix};
\end{multline}
\begin{multline}  \label{BGW}
\begin{pmatrix}
\partial_{11} ( N(Z)^k \cdot {\bf G_{l,m,n}}(Z)) &
\partial_{21} ( N(Z)^k \cdot {\bf G_{l,m,n}}(Z)) \\
\partial_{12} ( N(Z)^k \cdot {\bf G_{l,m,n}}(Z)) &
\partial_{22} ( N(Z)^k \cdot {\bf G_{l,m,n}}(Z))
\end{pmatrix}  \\
=
\frac{2l+k+1}{2l} N(Z)^k \begin{pmatrix}
(l-m-1) {\bf G_{l-\frac12,m+\frac12,n+\frac12}}(Z) &
(l-m-1) {\bf G_{l-\frac12,m+\frac12,n-\frac12}}(Z) \\
(l+m) {\bf G_{l-\frac12,m-\frac12,n+\frac12}}(Z) &
(l+m) {\bf G_{l-\frac12,m-\frac12,n-\frac12}}(Z)
\end{pmatrix}  \\
+
\frac{k}{2l+2} N(Z)^{k-1} \begin{pmatrix}
(l+n+2) {\bf G_{l+\frac12,m+\frac12,n+\frac12}}(Z) &
-(l-n+1) {\bf G_{l+\frac12,m+\frac12,n-\frac12}}(Z) \\
-(l+n+2) {\bf G_{l+\frac12,m-\frac12,n+\frac12}}(Z) &
(l-n+1) {\bf G_{l+\frac12,m-\frac12,n-\frac12}}(Z)
\end{pmatrix}  \\
+\frac1{2l+1} \begin{pmatrix}
{\bf H_{k,l-\frac12,m+\frac12,n+\frac12}}(Z) &
{\bf H_{k,l-\frac12,m+\frac12,n-\frac12}}(Z) \\
- {\bf H_{k,l-\frac12,m-\frac12,n+\frac12}}(Z) &
- {\bf H_{k,l-\frac12,m-\frac12,n-\frac12}}(Z)
\end{pmatrix};
\end{multline}
\begin{multline*}
\begin{pmatrix} \partial_{11} {\bf H_{k,l,m,n}}(Z) &
\partial_{21} {\bf H_{k,l,m,n}}(Z) \\
\partial_{12} {\bf H_{k,l,m,n}}(Z) &
\partial_{22} {\bf H_{k,l,m,n}}(Z) \end{pmatrix}  \\
=
\frac{k(2l+k+2) N(Z)^{k-1}}{(l+1)(2l+1)(2l+3)}\begin{pmatrix}
- {\bf F_{l+\frac12,m+\frac12,n+\frac12}}(Z) &
- {\bf F_{l+\frac12,m+\frac12,n-\frac12}}(Z) \\
{\bf F_{l+\frac12,m-\frac12,n+\frac12}}(Z) &
{\bf F_{l+\frac12,m-\frac12,n-\frac12}}(Z)
\end{pmatrix}  \\
+
\frac{k(2l+k+2)N(Z)^{k-1}}{(l+1)(2l+1)(2l+3)} \\
\times \begin{pmatrix} -(l-m)(l+n+2) {\bf G_{l+\frac12,m+\frac12,n+\frac12}}(Z) &
(l-m)(l-n+1) {\bf G_{l+\frac12,m+\frac12,n-\frac12}}(Z) \\
-(l+m+1)(l+n+2) {\bf G_{l+\frac12,m-\frac12,n+\frac12}}(Z) &
(l+m+1)(l-n+1) {\bf G_{l+\frac12,m-\frac12,n-\frac12}}(Z) \end{pmatrix}  \\
+\frac{2l(2l+k+2)}{(2l+1)^2} \begin{pmatrix}
(l-m) {\bf H_{k,l-\frac12,m+\frac12,n+\frac12}}(Z) &
(l-m){\bf H_{k,l-\frac12,m+\frac12,n-\frac12}}(Z) \\
(l+m+1) {\bf H_{k,l-\frac12,m-\frac12,n+\frac12}}(Z) &
(l+m+1) {\bf H_{k,l-\frac12,m-\frac12,n-\frac12}}(Z)
\end{pmatrix}  \\
+ \frac{k(2l+4)}{(2l+3)^2} \begin{pmatrix}
(l+n+2) {\bf H_{k-1,l+\frac12,m+\frac12,n+\frac12}}(Z) &
-(l-n+1) {\bf H_{k-1,l+\frac12,m+\frac12,n-\frac12}}(Z) \\
-(l+n+2) {\bf H_{k-1,l+\frac12,m-\frac12,n+\frac12}}(Z) &
(l-n+1) {\bf H_{k-1,l+\frac12,m-\frac12,n-\frac12}}(Z)
\end{pmatrix};
\end{multline*}
\begin{equation}  \label{BIW}
\begin{pmatrix} \partial_{11} (N(Z)^{-2} \cdot Z^+) &
\partial_{21} (N(Z)^{-2} \cdot Z^+) \\
\partial_{12} (N(Z)^{-2} \cdot Z^+) &
\partial_{22} (N(Z)^{-2} \cdot Z^+) \end{pmatrix}
= \frac32 \begin{pmatrix} -{\bf H_{-2,0,0,0}}(Z) & {\bf H_{-2,0,0,-1}}(Z) \\
{\bf H_{-2,0,-1,0}}(Z) & -{\bf H_{-2,0,-1,-1}}(Z) \end{pmatrix};
\end{equation}
\begin{multline*}
\begin{pmatrix}
\partial_{11} ( N(Z)^k \cdot {\bf F'_{l,m,n}}(Z)) &
\partial_{21} ( N(Z)^k \cdot {\bf F'_{l,m,n}}(Z)) \\
\partial_{12} ( N(Z)^k \cdot {\bf F'_{l,m,n}}(Z)) &
\partial_{22} ( N(Z)^k \cdot {\bf F'_{l,m,n}}(Z))
\end{pmatrix}  \\
=
\frac{k}{2l} N(Z)^{k-1} \begin{pmatrix}
(l+m+1) {\bf F'_{l-\frac12,m-\frac12,n-\frac12}}(Z) &
-(l-m) {\bf F'_{l-\frac12,m+\frac12,n-\frac12}}(Z) \\
-(l+m+1) {\bf F'_{l-\frac12,m-\frac12,n+\frac12}}(Z) &
(l-m) {\bf F'_{l-\frac12,m+\frac12,n+\frac12}}(Z)
\end{pmatrix}  \\
-
\frac{2l-k+1}{2l+2} N(Z)^k \begin{pmatrix}
(l-n) {\bf F'_{l+\frac12,m-\frac12,n-\frac12}}(Z) &
(l-n) {\bf F'_{l+\frac12,m+\frac12,n-\frac12}}(Z) \\
(l+n+1) {\bf F'_{l+\frac12,m-\frac12,n+\frac12}}(Z) &
(l+n+1) {\bf F'_{l+\frac12,m+\frac12,n+\frac12}}(Z)
\end{pmatrix}  \\
+\frac1{2l+1} \begin{pmatrix}
-(l+m+1)(l-n) {\bf H'_{k-1,l-\frac12,m-\frac12,n-\frac12}}(Z) &
(l-m)(l-n) {\bf H'_{k-1,l-\frac12,m+\frac12,n-\frac12}}(Z) \\
-(l+m+1)(l+n+1) {\bf H'_{k-1,l-\frac12,m-\frac12,n+\frac12}}(Z) &
(l-m)(l+n+1) {\bf H'_{k-1,l-\frac12,m+\frac12,n+\frac12}}(Z)
\end{pmatrix};
\end{multline*}
\begin{multline*}
\begin{pmatrix}
\partial_{11} ( N(Z)^k \cdot {\bf G'_{l,m,n}}(Z)) &
\partial_{21} ( N(Z)^k \cdot {\bf G'_{l,m,n}}(Z)) \\
\partial_{12} ( N(Z)^k \cdot {\bf G'_{l,m,n}}(Z)) &
\partial_{22} ( N(Z)^k \cdot {\bf G'_{l,m,n}}(Z))
\end{pmatrix}  \\
=
\frac{k}{2l} N(Z)^{k-1} \begin{pmatrix}
(l+m) {\bf G'_{l-\frac12,m-\frac12,n-\frac12}}(Z) &
-(l-m-1) {\bf G'_{l-\frac12,m+\frac12,n-\frac12}}(Z) \\
-(l+m) {\bf G'_{l-\frac12,m-\frac12,n+\frac12}}(Z) &
(l-m-1) {\bf G'_{l-\frac12,m+\frac12,n+\frac12}}(Z)
\end{pmatrix}  \\
-
\frac{2l-k+1}{2l+2} N(Z)^k \begin{pmatrix}
(l-n+1) {\bf G'_{l+\frac12,m-\frac12,n-\frac12}}(Z) &
(l-n+1) {\bf G'_{l+\frac12,m+\frac12,n-\frac12}}(Z) \\
(l+n+2) {\bf G'_{l+\frac12,m-\frac12,n+\frac12}}(Z) &
(l+n+2) {\bf G'_{l+\frac12,m+\frac12,n+\frac12}}(Z)
\end{pmatrix}  \\
+\frac1{2l+1} \begin{pmatrix}
- {\bf H'_{k-1,l-\frac12,m-\frac12,n-\frac12}}(Z) &
- {\bf H'_{k-1,l-\frac12,m+\frac12,n-\frac12}}(Z) \\
{\bf H'_{k-1,l-\frac12,m-\frac12,n+\frac12}}(Z) &
{\bf H'_{k-1,l-\frac12,m+\frac12,n+\frac12}}(Z)
\end{pmatrix};
\end{multline*}
\begin{multline*}
\begin{pmatrix} \partial_{11} {\bf H'_{k,l,m,n}}(Z) &
\partial_{21} {\bf H'_{k,l,m,n}}(Z) \\
\partial_{12} {\bf H'_{k,l,m,n}}(Z) &
\partial_{22} {\bf H'_{k,l,m,n}}(Z) \end{pmatrix}  \\
=
\frac{(k+1)(2l-k+1) N(Z)^k}{(l+1)(2l+1)(2l+3)}\begin{pmatrix}
- {\bf F'_{l+\frac12,m-\frac12,n-\frac12}}(Z) &
- {\bf F'_{l+\frac12,m+\frac12,n-\frac12}}(Z) \\
{\bf F'_{l+\frac12,m-\frac12,n+\frac12}}(Z) &
{\bf F'_{l+\frac12,m+\frac12,n+\frac12}}(Z)
\end{pmatrix}  \\
+
\frac{(k+1)(2l-k+1)N(Z)^k}{(l+1)(2l+1)(2l+3)} \\
\times \begin{pmatrix} -(l+m+1)(l-n+1) {\bf G'_{l+\frac12,m-\frac12,n-\frac12}}(Z) &
(l-m)(l-n+1) {\bf G'_{l+\frac12,m+\frac12,n-\frac12}}(Z) \\
-(l+m+1)(l+n+2) {\bf G'_{l+\frac12,m-\frac12,n+\frac12}}(Z) &
(l-m)(l+n+2) {\bf G'_{l+\frac12,m+\frac12,n+\frac12}}(Z) \end{pmatrix}  \\
+\frac{2l(k+1)}{(2l+1)^2} \begin{pmatrix}
(l+m+1) {\bf H'_{k-1,l-\frac12,m-\frac12,n-\frac12}}(Z) &
-(l-m){\bf H'_{k-1,l-\frac12,m+\frac12,n-\frac12}}(Z) \\
-(l+m+1) {\bf H'_{k-1,l-\frac12,m-\frac12,n+\frac12}}(Z) &
(l-m) {\bf H'_{k-1,l-\frac12,m+\frac12,n+\frac12}}(Z)
\end{pmatrix}  \\
- \frac{(2l+4)(2l-k+1)}{(2l+3)^2} \begin{pmatrix}
(l-n+1) {\bf H'_{k,l+\frac12,m+\frac12,n+\frac12}}(Z) &
(l-n+1) {\bf H'_{k,l+\frac12,m+\frac12,n-\frac12}}(Z) \\
(l+n+2) {\bf H'_{k,l+\frac12,m-\frac12,n+\frac12}}(Z) &
(l+n+2) {\bf H'_{k,l+\frac12,m+\frac12,n+\frac12}}(Z)
\end{pmatrix}.
\end{multline*}
}

Let $\bigl(\begin{smallmatrix} 0 & 1 \\ 1 & 0 \end{smallmatrix}\bigr) \in
GL(2,\HC)$, then
\begin{align*}
\rho_2 \bigl(\begin{smallmatrix} 0 & 1 \\ 1 & 0 \end{smallmatrix}\bigr)
\bigl( N(Z)^k \cdot {\bf F_{l,m,n}}(Z) \bigr)
&= N(Z)^{-k-2} \cdot {\bf F'_{l,m,n}}(Z), \\
\rho_2 \bigl(\begin{smallmatrix} 0 & 1 \\ 1 & 0 \end{smallmatrix}\bigr)
\bigl( N(Z)^k \cdot {\bf G_{l,m,n}}(Z) \bigr)
&= N(Z)^{-k-2} \cdot {\bf G'_{l,m,n}}(Z), \\
\rho_2 \bigl(\begin{smallmatrix} 0 & 1 \\ 1 & 0 \end{smallmatrix}\bigr)
\bigl( {\bf H_{k,l,m,n}}(Z) \bigr) &= {\bf H'_{-k-2,l,m,n}}(Z), \\
\rho_2 \bigl(\begin{smallmatrix} 0 & 1 \\ 1 & 0 \end{smallmatrix}\bigr)
\bigl( N(Z)^{-2} \cdot Z^+ \bigr) &= - N(Z)^{-2} \cdot Z^+.
\end{align*}
Now we can compute the action of
$\bigl(\begin{smallmatrix} 0 & 0 \\ C & 0 \end{smallmatrix}\bigr)
\in \mathfrak{gl}(2,\HC)$, $C \in \HC$, on these generators:
{\small
\begin{multline*}
\begin{pmatrix}
\rho_2 \bigl(\begin{smallmatrix} 0 & 1 \\ 1 & 0 \end{smallmatrix}\bigr)
\partial_{11}
\rho_2 \bigl(\begin{smallmatrix} 0 & 1 \\ 1 & 0 \end{smallmatrix}\bigr)
(N(Z)^k \cdot {\bf F_{l,m,n}}(Z))  &
\rho_2 \bigl(\begin{smallmatrix} 0 & 1 \\ 1 & 0 \end{smallmatrix}\bigr)
\partial_{21}
\rho_2 \bigl(\begin{smallmatrix} 0 & 1 \\ 1 & 0 \end{smallmatrix}\bigr)
(N(Z)^k \cdot {\bf F_{l,m,n}}(Z))  \\
\rho_2 \bigl(\begin{smallmatrix} 0 & 1 \\ 1 & 0 \end{smallmatrix}\bigr)
\partial_{12}
\rho_2 \bigl(\begin{smallmatrix} 0 & 1 \\ 1 & 0 \end{smallmatrix}\bigr)
(N(Z)^k \cdot {\bf F_{l,m,n}}(Z))  &
\rho_2 \bigl(\begin{smallmatrix} 0 & 1 \\ 1 & 0 \end{smallmatrix}\bigr)
\partial_{22}
\rho_2 \bigl(\begin{smallmatrix} 0 & 1 \\ 1 & 0 \end{smallmatrix}\bigr)
(N(Z)^k \cdot {\bf F_{l,m,n}}(Z))
\end{pmatrix}  \\
=
- \frac{k+2}{2l} N(Z)^{k+1} \begin{pmatrix}
(l+m+1) {\bf F_{l-\frac12,m-\frac12,n-\frac12}}(Z) &
-(l-m) {\bf F_{l-\frac12,m+\frac12,n-\frac12}}(Z) \\
-(l+m+1) {\bf F_{l-\frac12,m-\frac12,n+\frac12}}(Z) &
(l-m) {\bf F_{l-\frac12,m+\frac12,n+\frac12}}(Z)
\end{pmatrix}  \\
-
\frac{2l+k+3}{2l+2} N(Z)^k \begin{pmatrix}
(l-n) {\bf F_{l+\frac12,m-\frac12,n-\frac12}}(Z) &
(l-n) {\bf F_{l+\frac12,m+\frac12,n-\frac12}}(Z) \\
(l+n+1) {\bf F_{l+\frac12,m-\frac12,n+\frac12}}(Z) &
(l+n+1) {\bf F_{l+\frac12,m+\frac12,n+\frac12}}(Z)
\end{pmatrix}  \\
+\frac1{2l+1} \begin{pmatrix}
-(l+m+1)(l-n) {\bf H_{k+1,l-\frac12,m-\frac12,n-\frac12}}(Z) &
(l-m)(l-n) {\bf H_{k+1,l-\frac12,m+\frac12,n-\frac12}}(Z) \\
-(l+m+1)(l+n+1) {\bf H_{k+1,l-\frac12,m-\frac12,n+\frac12}}(Z) &
(l-m)(l+n+1) {\bf H_{k+1,l-\frac12,m+\frac12,n+\frac12}}(Z)
\end{pmatrix};
\end{multline*}
\begin{multline*}
\begin{pmatrix}
\rho_2 \bigl(\begin{smallmatrix} 0 & 1 \\ 1 & 0 \end{smallmatrix}\bigr)
\partial_{11}
\rho_2 \bigl(\begin{smallmatrix} 0 & 1 \\ 1 & 0 \end{smallmatrix}\bigr)
(N(Z)^k \cdot {\bf G_{l,m,n}}(Z))  &
\rho_2 \bigl(\begin{smallmatrix} 0 & 1 \\ 1 & 0 \end{smallmatrix}\bigr)
\partial_{21}
\rho_2 \bigl(\begin{smallmatrix} 0 & 1 \\ 1 & 0 \end{smallmatrix}\bigr)
(N(Z)^k \cdot {\bf G_{l,m,n}}(Z))  \\
\rho_2 \bigl(\begin{smallmatrix} 0 & 1 \\ 1 & 0 \end{smallmatrix}\bigr)
\partial_{12}
\rho_2 \bigl(\begin{smallmatrix} 0 & 1 \\ 1 & 0 \end{smallmatrix}\bigr)
(N(Z)^k \cdot {\bf G_{l,m,n}}(Z))  &
\rho_2 \bigl(\begin{smallmatrix} 0 & 1 \\ 1 & 0 \end{smallmatrix}\bigr)
\partial_{22}
\rho_2 \bigl(\begin{smallmatrix} 0 & 1 \\ 1 & 0 \end{smallmatrix}\bigr)
(N(Z)^k \cdot {\bf G_{l,m,n}}(Z))
\end{pmatrix}  \\
=
- \frac{k+2}{2l} N(Z)^{k+1} \begin{pmatrix}
(l+m) {\bf G_{l-\frac12,m-\frac12,n-\frac12}}(Z) &
-(l-m-1) {\bf G_{l-\frac12,m+\frac12,n-\frac12}}(Z) \\
-(l+m) {\bf G_{l-\frac12,m-\frac12,n+\frac12}}(Z) &
(l-m-1) {\bf G_{l-\frac12,m+\frac12,n+\frac12}}(Z)
\end{pmatrix}  \\
-
\frac{2l+k+3}{2l+2} N(Z)^k \begin{pmatrix}
(l-n+1) {\bf G_{l+\frac12,m-\frac12,n-\frac12}}(Z) &
(l-n+1) {\bf G_{l+\frac12,m+\frac12,n-\frac12}}(Z) \\
(l+n+2) {\bf G_{l+\frac12,m-\frac12,n+\frac12}}(Z) &
(l+n+2) {\bf G_{l+\frac12,m+\frac12,n+\frac12}}(Z)
\end{pmatrix}  \\
+\frac1{2l+1} \begin{pmatrix}
- {\bf H_{k+1,l-\frac12,m-\frac12,n-\frac12}}(Z) &
- {\bf H_{k+1,l-\frac12,m+\frac12,n-\frac12}}(Z) \\
{\bf H_{k+1,l-\frac12,m-\frac12,n+\frac12}}(Z) &
{\bf H_{k+1,l-\frac12,m+\frac12,n+\frac12}}(Z)
\end{pmatrix};
\end{multline*}
\begin{multline*}
\begin{pmatrix}
\rho_2 \bigl(\begin{smallmatrix} 0 & 1 \\ 1 & 0 \end{smallmatrix}\bigr)
\partial_{11}
\rho_2 \bigl(\begin{smallmatrix} 0 & 1 \\ 1 & 0 \end{smallmatrix}\bigr)
{\bf H_{k,l,m,n}}(Z)  &
\rho_2 \bigl(\begin{smallmatrix} 0 & 1 \\ 1 & 0 \end{smallmatrix}\bigr)
\partial_{21}
\rho_2 \bigl(\begin{smallmatrix} 0 & 1 \\ 1 & 0 \end{smallmatrix}\bigr)
{\bf H_{k,l,m,n}}(Z)  \\
\rho_2 \bigl(\begin{smallmatrix} 0 & 1 \\ 1 & 0 \end{smallmatrix}\bigr)
\partial_{12}
\rho_2 \bigl(\begin{smallmatrix} 0 & 1 \\ 1 & 0 \end{smallmatrix}\bigr)
{\bf H_{k,l,m,n}}(Z)  &
\rho_2 \bigl(\begin{smallmatrix} 0 & 1 \\ 1 & 0 \end{smallmatrix}\bigr)
\partial_{22}
\rho_2 \bigl(\begin{smallmatrix} 0 & 1 \\ 1 & 0 \end{smallmatrix}\bigr)
{\bf H_{k,l,m,n}}(Z)
\end{pmatrix}  \\
=
- \frac{(k+1)(2l+k+3) N(Z)^k}{(l+1)(2l+1)(2l+3)}\begin{pmatrix}
- {\bf F_{l+\frac12,m-\frac12,n-\frac12}}(Z) &
- {\bf F_{l+\frac12,m+\frac12,n-\frac12}}(Z) \\
{\bf F_{l+\frac12,m-\frac12,n+\frac12}}(Z) &
{\bf F_{l+\frac12,m+\frac12,n+\frac12}}(Z)
\end{pmatrix}  \\
+ \frac{(k+1)(2l+k+3)N(Z)^k}{(l+1)(2l+1)(2l+3)} \\
\times \begin{pmatrix} (l+m+1)(l-n+1) {\bf G_{l+\frac12,m-\frac12,n-\frac12}}(Z) &
-(l-m)(l-n+1) {\bf G_{l+\frac12,m+\frac12,n-\frac12}}(Z) \\
(l+m+1)(l+n+2) {\bf G_{l+\frac12,m-\frac12,n+\frac12}}(Z) &
-(l-m)(l+n+2) {\bf G_{l+\frac12,m+\frac12,n+\frac12}}(Z) \end{pmatrix}  \\
- \frac{2l(k+1)}{(2l+1)^2} \begin{pmatrix}
(l+m+1) {\bf H_{k+1,l-\frac12,m-\frac12,n-\frac12}}(Z) &
-(l-m){\bf H_{k+1,l-\frac12,m+\frac12,n-\frac12}}(Z) \\
-(l+m+1) {\bf H_{k+1,l-\frac12,m-\frac12,n+\frac12}}(Z) &
(l-m) {\bf H_{k+1,l-\frac12,m+\frac12,n+\frac12}}(Z)
\end{pmatrix}  \\
- \frac{(2l+4)(2l+k+3)}{(2l+3)^2} \begin{pmatrix}
(l-n+1) {\bf H_{k,l+\frac12,m+\frac12,n+\frac12}}(Z) &
(l-n+1) {\bf H_{k,l+\frac12,m+\frac12,n-\frac12}}(Z) \\
(l+n+2) {\bf H_{k,l+\frac12,m-\frac12,n+\frac12}}(Z) &
(l+n+2) {\bf H_{k,l+\frac12,m+\frac12,n+\frac12}}(Z)
\end{pmatrix};
\end{multline*}
\begin{multline}  \label{1-dim-gen-actions2}
\begin{pmatrix}
\rho_2 \bigl(\begin{smallmatrix} 0 & 1 \\ 1 & 0 \end{smallmatrix}\bigr)
\partial_{11}
\rho_2 \bigl(\begin{smallmatrix} 0 & 1 \\ 1 & 0 \end{smallmatrix}\bigr)
(N(Z)^{-2} \cdot Z^+)  &
\rho_2 \bigl(\begin{smallmatrix} 0 & 1 \\ 1 & 0 \end{smallmatrix}\bigr)
\partial_{21}
\rho_2 \bigl(\begin{smallmatrix} 0 & 1 \\ 1 & 0 \end{smallmatrix}\bigr)
(N(Z)^{-2} \cdot Z^+)  \\
\rho_2 \bigl(\begin{smallmatrix} 0 & 1 \\ 1 & 0 \end{smallmatrix}\bigr)
\partial_{12}
\rho_2 \bigl(\begin{smallmatrix} 0 & 1 \\ 1 & 0 \end{smallmatrix}\bigr)
(N(Z)^{-2} \cdot Z^+)  &
\rho_2 \bigl(\begin{smallmatrix} 0 & 1 \\ 1 & 0 \end{smallmatrix}\bigr)
\partial_{22}
\rho_2 \bigl(\begin{smallmatrix} 0 & 1 \\ 1 & 0 \end{smallmatrix}\bigr)
(N(Z)^{-2} \cdot Z^+)
\end{pmatrix}  \\
= \frac32 \begin{pmatrix} {\bf H'_{0,0,0,0}}(Z) & -{\bf H'_{0,0,0,-1}}(Z) \\
-{\bf H'_{0,0,-1,0}}(Z) & {\bf H'_{0,0,-1,-1}}(Z) \end{pmatrix}
= - \frac32 \begin{pmatrix} {\bf H_{-1,0,-1,-1}}(Z) & {\bf H_{-1,0,0,-1}}(Z) \\
{\bf H_{-1,0,-1,0}}(Z) & {\bf H_{-1,0,0,0}}(Z) \end{pmatrix},
\end{multline}
}

From (\ref{BFW})-(\ref{BIW}) we see that any $\rho_2$-invariant subspace of
the kernel of $\tr \circ \partial^+ : {\cal W} \to \Sh$ contains at least one
function ${\bf H_{k,l,m,n}}(Z)$.
We introduce a $U(2) \times U(2)$-invariant subspace of ${\cal W}$
$$
{\bf H} = \BB C\text{-span of } \Bigl\{ {\bf H_{k,l,m,n}}(Z); \:
k=0, \pm1, \pm2, \dots, \: l=0,\frac12,1,\frac32,\dots, \:
-l-1 \le m,n \le l \Bigr\}
$$
-- note that ${\bf H}$ is not $\mathfrak{gl}(2,\HC)$-invariant -- and a
projection
$$
\varpi_{\bf H}: \ker(\tr \circ \partial^+ : {\cal W} \to \Sh) \to {\bf H}
$$
that maps each ${\bf H_{k,l,m,n}}(Z)$ into itself and annihilates
$Z^+ \cdot N(Z)^2$, all $N(Z)^k \cdot {\bf F_{l,m,n}}(Z)$'s and all
$N(Z)^k \cdot {\bf G_{l,m,n}}(Z)$'s. Then
\begin{multline*}
\begin{pmatrix} \varpi_{\bf H} \circ \partial_{11} {\bf H_{k,l,m,n}}(Z) &
\varpi_{\bf H} \circ \partial_{21} {\bf H_{k,l,m,n}}(Z) \\
\varpi_{\bf H} \circ \partial_{12} {\bf H_{k,l,m,n}}(Z) &
\varpi_{\bf H} \circ \partial_{22} {\bf H_{k,l,m,n}}(Z) \end{pmatrix}  \\
= \frac{2l(2l+k+2)}{(2l+1)^2} \begin{pmatrix}
(l-m) {\bf H_{k,l-\frac12,m+\frac12,n+\frac12}}(Z) &
(l-m){\bf H_{k,l-\frac12,m+\frac12,n-\frac12}}(Z) \\
(l+m+1) {\bf H_{k,l-\frac12,m-\frac12,n+\frac12}}(Z) &
(l+m+1) {\bf H_{k,l-\frac12,m-\frac12,n-\frac12}}(Z)
\end{pmatrix}  \\
+ \frac{k(2l+4)}{(2l+3)^2} \begin{pmatrix}
(l+n+2) {\bf H_{k-1,l+\frac12,m+\frac12,n+\frac12}}(Z) &
-(l-n+1) {\bf H_{k-1,l+\frac12,m+\frac12,n-\frac12}}(Z) \\
-(l+n+2) {\bf H_{k-1,l+\frac12,m-\frac12,n+\frac12}}(Z) &
(l-n+1) {\bf H_{k-1,l+\frac12,m-\frac12,n-\frac12}}(Z)
\end{pmatrix}
\end{multline*}
and
\begin{multline*}
\begin{pmatrix}
\varpi_{\bf H} \circ
\rho_2 \bigl(\begin{smallmatrix} 0 & 1 \\ 1 & 0 \end{smallmatrix}\bigr)
\partial_{11}
\rho_2 \bigl(\begin{smallmatrix} 0 & 1 \\ 1 & 0 \end{smallmatrix}\bigr)
{\bf H_{k,l,m,n}}(Z)  &
\varpi_{\bf H} \circ
\rho_2 \bigl(\begin{smallmatrix} 0 & 1 \\ 1 & 0 \end{smallmatrix}\bigr)
\partial_{21}
\rho_2 \bigl(\begin{smallmatrix} 0 & 1 \\ 1 & 0 \end{smallmatrix}\bigr)
{\bf H_{k,l,m,n}}(Z)  \\
\varpi_{\bf H} \circ
\rho_2 \bigl(\begin{smallmatrix} 0 & 1 \\ 1 & 0 \end{smallmatrix}\bigr)
\partial_{12}
\rho_2 \bigl(\begin{smallmatrix} 0 & 1 \\ 1 & 0 \end{smallmatrix}\bigr)
{\bf H_{k,l,m,n}}(Z)  &
\varpi_{\bf H} \circ
\rho_2 \bigl(\begin{smallmatrix} 0 & 1 \\ 1 & 0 \end{smallmatrix}\bigr)
\partial_{22}
\rho_2 \bigl(\begin{smallmatrix} 0 & 1 \\ 1 & 0 \end{smallmatrix}\bigr)
{\bf H_{k,l,m,n}}(Z)
\end{pmatrix}  \\
= - \frac{2l(k+1)}{(2l+1)^2} \begin{pmatrix}
(l+m+1) {\bf H_{k+1,l-\frac12,m-\frac12,n-\frac12}}(Z) &
-(l-m){\bf H_{k+1,l-\frac12,m+\frac12,n-\frac12}}(Z) \\
-(l+m+1) {\bf H_{k+1,l-\frac12,m-\frac12,n+\frac12}}(Z) &
(l-m) {\bf H_{k+1,l-\frac12,m+\frac12,n+\frac12}}(Z)
\end{pmatrix}  \\
- \frac{(2l+4)(2l+k+3)}{(2l+3)^2} \begin{pmatrix}
(l-n+1) {\bf H_{k,l+\frac12,m+\frac12,n+\frac12}}(Z) &
(l-n+1) {\bf H_{k,l+\frac12,m+\frac12,n-\frac12}}(Z) \\
(l+n+2) {\bf H_{k,l+\frac12,m-\frac12,n+\frac12}}(Z) &
(l+n+2) {\bf H_{k,l+\frac12,m+\frac12,n+\frac12}}(Z)
\end{pmatrix}
\end{multline*}

\begin{figure}
\begin{center}
\begin{subfigure}[b]{0.3\textwidth}
\centering
\includegraphics[scale=0.2]{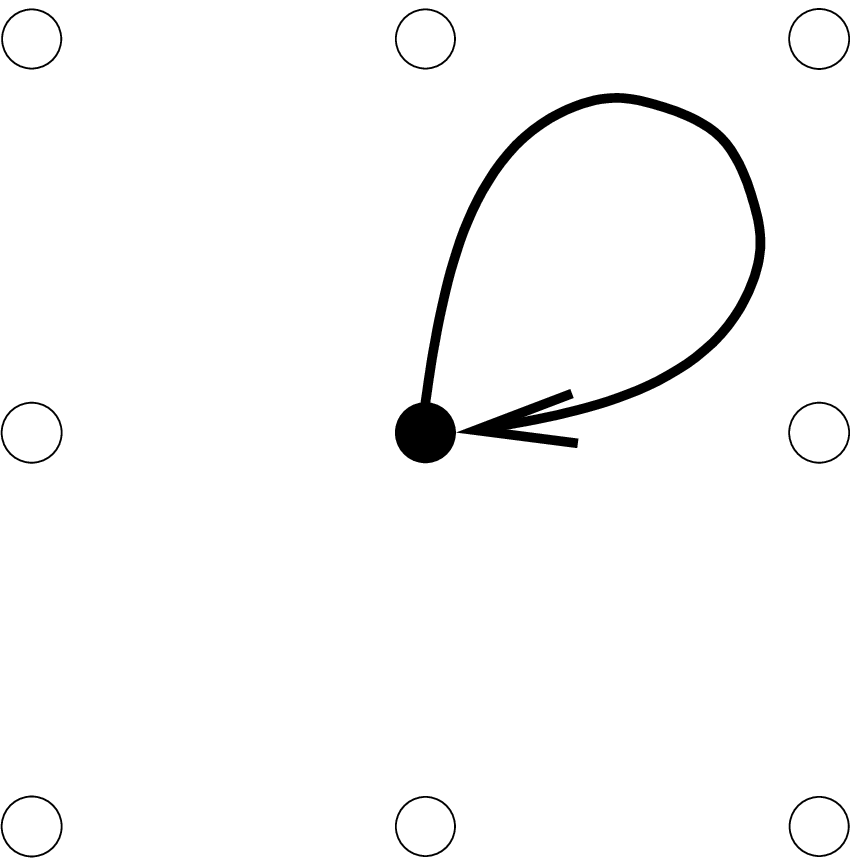}
\caption{Actions of
$\rho_2\bigl(\begin{smallmatrix} A & 0 \\ 0 & 0 \end{smallmatrix}\bigr)$ \\
and $\rho_2\bigl(\begin{smallmatrix} 0 & 0 \\ 0 & D \end{smallmatrix}\bigr)$}
\end{subfigure}
\quad
\begin{subfigure}[b]{0.3\textwidth}
\centering
\includegraphics[scale=0.2]{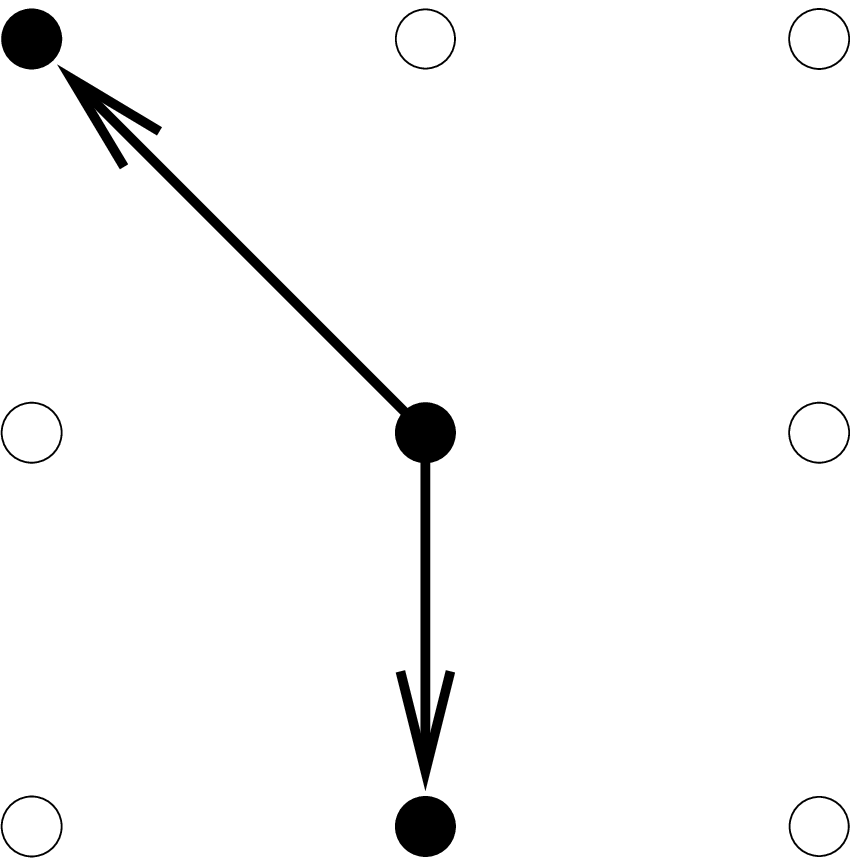}
\caption{Action of $\varpi_{\bf H} \circ
\rho_2\bigl(\begin{smallmatrix} 0 & B \\ 0 & 0 \end{smallmatrix}\bigr)$}
\end{subfigure}
\quad
\begin{subfigure}[b]{0.3\textwidth}
\centering
\includegraphics[scale=0.2]{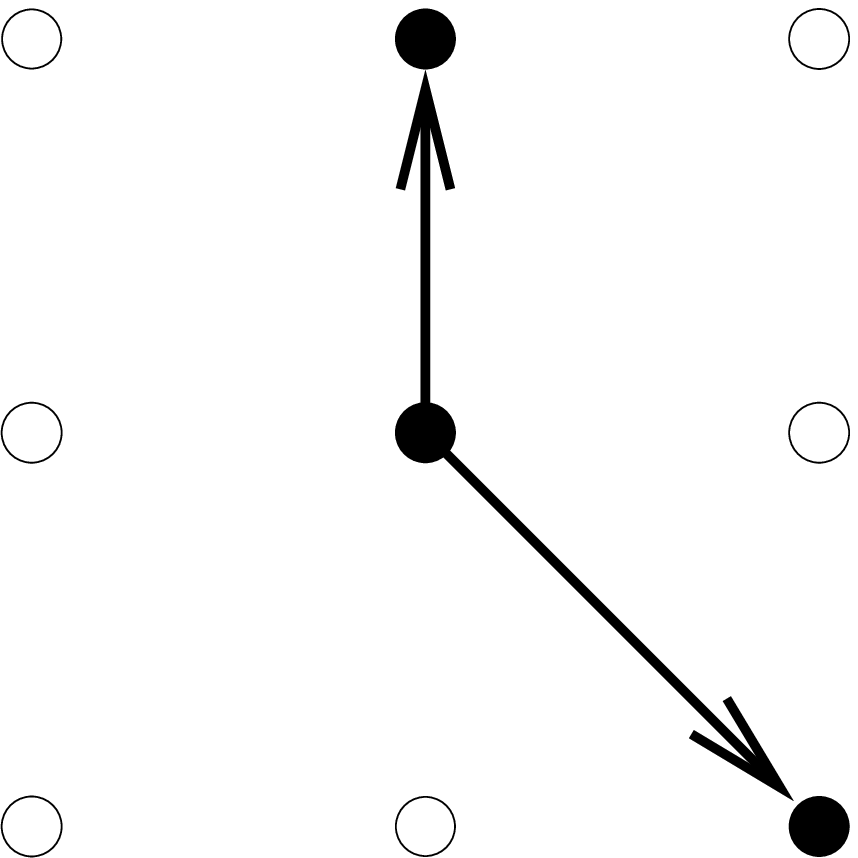}
\caption{Action of $\varpi_{\bf H} \circ
\rho_2\bigl(\begin{smallmatrix} 0 & 0 \\ C & 0 \end{smallmatrix}\bigr)$}
\end{subfigure}
\end{center}
\caption{}
\label{actions}
\end{figure}

The actions of
$\bigl( \begin{smallmatrix} A & 0 \\ 0 & 0 \end{smallmatrix} \bigr)$,
$\bigl( \begin{smallmatrix} 0 & B \\ 0 & 0 \end{smallmatrix} \bigr)$,
$\bigl( \begin{smallmatrix} 0 & 0 \\ C & 0 \end{smallmatrix} \bigr)$ and
$\bigl( \begin{smallmatrix} 0 & 0 \\ 0 & D \end{smallmatrix} \bigr)$
are illustrated in Figure \ref{actions}. In the diagram describing
$\varpi_{\bf H} \circ
\rho_2\bigl( \begin{smallmatrix} 0 & B \\ 0 & 0 \end{smallmatrix} \bigr)$
the vertical arrow disappears if $l=0$ or $2l+k+2=0$
and the diagonal arrow disappears if $k=0$.
Similarly, in the diagram describing $\varpi_{\bf H} \circ
\rho_2\bigl( \begin{smallmatrix} 0 & 0 \\ C & 0 \end{smallmatrix} \bigr)$
the vertical arrow disappears if $2l+k+3=0$ and the
diagonal arrow disappears if $k=-1$ or $l=0$.
Let
\begin{align*}
{\bf H}^+ &= \BB C\text{-span of }
\bigl\{ {\bf H_{k,l,m,n}}(Z);\: k \ge 0 \bigr\}, \\
{\bf H}^- &= \BB C\text{-span of }
\bigl\{ {\bf H_{k,l,m,n}}(Z);\: k \le -(2l+3) \bigr\}, \\
{\bf H}^0 &= \BB C\text{-span of }
\bigl\{ {\bf H_{k,l,m,n}}(Z);\: -(2l+2) \le k \le -1 \bigr\}
\end{align*}
(see Figure \ref{decomposition-fig4} and compare with Figure 2 in \cite{FL3}).
It follows that if ${\cal U}$ is a $\rho_2$-invariant subspace of
$\ker(\tr \circ \partial^+ : {\cal W} \to \Sh)$, then
$\varpi_{\bf H} ({\cal U})$ must be ${\bf H}^-$, ${\bf H}^0$, ${\bf H}^+$,
a direct sum of two of these spaces or all of ${\bf H}$.

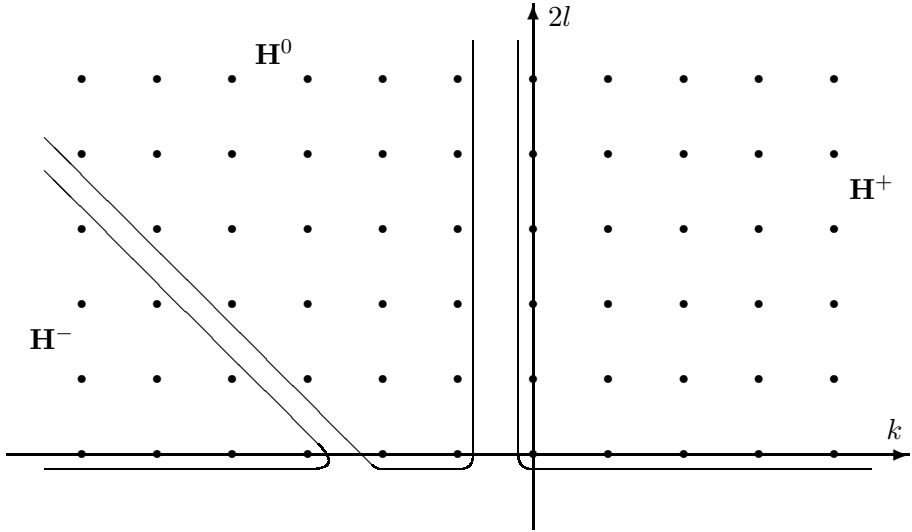
\begin{figure}
\begin{center}
\setlength{\unitlength}{1mm}
\begin{picture}(120,70)
\multiput(10,10)(10,0){11}{\circle*{1}}
\multiput(10,20)(10,0){11}{\circle*{1}}
\multiput(10,30)(10,0){11}{\circle*{1}}
\multiput(10,40)(10,0){11}{\circle*{1}}
\multiput(10,50)(10,0){11}{\circle*{1}}
\multiput(10,60)(10,0){11}{\circle*{1}}

\thicklines
\put(70,0){\vector(0,1){70}}
\put(0,10){\vector(1,0){120}}

\thinlines
\put(68,10){\line(0,1){55}}
\put(70,8){\line(1,0){45}}
\qbezier(68,10)(68,8)(70,8)

\put(62,10){\line(0,1){55}}
\put(48.6,8.6){\line(-1,1){43.6}}
\put(50,8){\line(1,0){10}}
\qbezier(60,8)(62,8)(62,10)
\qbezier(48.6,8.6)(49.2,8)(50,8)

\put(5,8){\line(1,0){35}}
\put(50,8){\line(1,0){10}}
\put(41.4,11.4){\line(-1,1){36.4}}
\qbezier(40,8)(44.8,8)(41.4,11.4)

\put(72,67){$2l$}
\put(117,12){$k$}
\put(3,24){${\bf H}^-$}
\put(33,62){${\bf H}^0$}
\put(112,44){${\bf H}^+$}
\end{picture}
\end{center}
\caption{Decomposition of $(\varpi_{\bf H} \circ \rho_2, {\bf H})$
into invariant components.}
\label{decomposition-fig4}
\end{figure}

Now let
\begin{align*}
{\cal Q}^+ &= \BB C\text{-span of }
\bigl\{ N(Z)^k \cdot {\bf F_{l,m,n}}(Z),\: N(Z)^k \cdot {\bf G_{l,m,n}}(Z),\:
{\bf H_{k,l,m,n}}(Z);\: k \ge 0 \bigr\}, \\
{\cal Q}^- &= \BB C\text{-span of }
\bigl\{ N(Z)^k \cdot {\bf F_{l,m,n}}(Z),\: N(Z)^k \cdot {\bf G_{l,m,n}}(Z),\:
{\bf H_{k,l,m,n}}(Z);\: k \le -(2l+3) \bigr\}, \\
{\cal Q}^0 &= \BB C\text{-span of }
\biggl\{ \begin{matrix}
N(Z)^k \cdot {\bf F_{l,m,n}}(Z),\: N(Z)^k \cdot {\bf G_{l,m,n}}(Z)
\text{ with } -(2l+1) \le k \le -2, \\
{\bf H_{k,l,m,n}}(Z) \text{ with } -(2l+2) \le k \le -1 \end{matrix} \biggr\}.
\end{align*}
It is easy to see that
$$
\varpi_{\bf H}({\cal Q}^+) = {\bf H}^+, \qquad
\varpi_{\bf H}({\cal Q}^-) = {\bf H}^-, \qquad
\varpi_{\bf H}({\cal Q}^0) = {\bf H}^0, \qquad
$$
${\cal Q}^+$, ${\cal Q}^-$, ${\cal Q}^0$ are invariant under the
$\rho_2$-action of $\mathfrak{gl}(2,\HC)$ and irreducible.
Moreover, these are the only irreducible subrepresentations of
$(\rho_2, \ker (\tr \circ \partial^+))$.
Furthermore, the quotient
$$
\frac{\ker (\tr \circ \partial^+ : {\cal W} \to \Sh)}
{{\cal Q}^- \oplus {\cal Q}^0 \oplus {\cal Q}^+}
$$
decomposes into five irreducible subrepresentations; they are the images of
\begin{align*}
\BB C\text{-span of }
&\bigl\{ N(Z)^{-1} \cdot {\bf F_{l,m,n}}(Z);\: l \ge 1/2 \bigr\}, \\
\BB C\text{-span of }
&\bigl\{ N(Z)^{-(2l+2)} \cdot {\bf G_{l,m,n}}(Z) ;\: l \ge 1/2 \bigr\}  \\
&= \BB C\text{-span of }
\bigl\{ N(Z)^{-1} \cdot {\bf F'_{l,m,n}}(Z);\: l \ge 1/2 \bigr\}, \\
\BB C\text{-span of }
&\bigl\{ N(Z)^{-1} \cdot {\bf G_{l,m,n}}(Z);\: l \ge 1/2 \bigr\}, \\
\BB C\text{-span of }
&\bigl\{ N(Z)^{-(2l+2)} \cdot {\bf F_{l,m,n}}(Z) ;\: l \ge 1/2 \bigr\}  \\
&= \BB C\text{-span of }
\bigl\{ N(Z)^{-1} \cdot {\bf G'_{l,m,n}}(Z);\: l \ge 1/2 \bigr\}, \\
\BB C\text{-span of } &\bigl\{ N(Z)^{-2} \cdot Z^+ \bigr\}.
\end{align*}
These subrepresentations are isomorphic to
$$
(\rho'_2,{\cal W}'/ \ker \tau_a^+), \qquad
(\rho'_2,{\cal W}'/ \ker \tau_a^-), \qquad
(\rho'_2,{\cal W}'/ \ker \tau_s^+), \qquad
(\rho'_2,{\cal W}'/ \ker \tau_s^-)
$$
(which appeared in Corollary \ref{FG-irred})
and the trivial one-dimensional representation respectively.
Combining this with Corollary \ref{W/ker-decomp-cor}, we obtain a description
of all thirteen irreducible components of $(\rho_2,{\cal W})$.

For future reference we make the following observation:

\begin{lem}  \label{nesting-lem}
  The element $N(Z)^{-2} \cdot Z^+ \in {\cal W}$ generates a subrepresentation
  ${\cal U}$ of ${\cal W}$ that has exactly two irreducible components:
  ${\cal Q}^0$ and the trivial one-dimensional representation.
  Moreover, the trivial one-dimensional representation does not appear as a
  subrepresentation of ${\cal U}$ -- it is isomorphic to the quotient
  ${\cal U}/{\cal Q}^0$.
\end{lem}

\begin{proof}
The result follows from equations (\ref{BIW}), (\ref{1-dim-gen-actions2}).
\end{proof}

\subsection{Decomposition of $(\rho'_2,{\cal W}')$}  \label{W'-irred-decomp-subsection}

By Proposition 80 in \cite{FL1}, the representations $(\rho_2,{\cal W})$ and
$(\rho'_2,{\cal W}')$ of $\mathfrak{gl}(2,\HC)$ are dual to each other.
Thus the irreducible components of $(\rho'_2,{\cal W}')$ are dual to those of
$(\rho_2,{\cal W})$, and, in particular, these two representations have
the same number of irreducible components -- thirteen.
We would like to describe the irreducible components of $(\rho'_2,{\cal W}')$
more explicitly.

The idea is to use the quaternionic chain complex (\ref{W-sequence}) and deal
separately with $\ker \M \subset {\cal W}'$ and ${\cal W}'/\ker \M$.
Since $\ker \M$ contains the image
$\partial^+(\Sh') \simeq \Sh'/\ker \partial^+ = \Sh'/{\cal I}'_0$, the following
five irreducible components listed in Theorem \ref{rho'-decomposition}
reappear in $(\rho'_2, {\cal W}')$:
$$
(\rho', {\cal BH}^+/{\cal I}'_0), \quad (\rho', {\cal BH}^-/{\cal I}'_0), \quad
(\rho', \Sh^+/{\cal BH}^+), \quad (\rho', \Sh'^-/{\cal BH}^-), \quad
\bigl( \rho', {\cal J}'/({\cal BH}^++{\cal BH}^-) \bigr).
$$
The $\rho'_2$-invariant subspace ${\cal M} \subset \ker \M$ contains
all of the above components plus five more:
four that appeared in Corollary \ref{FG-irred}:
$$
(\rho'_2,{\cal W}'/ \ker \tau_a^+), \quad
(\rho'_2,{\cal W}'/ \ker \tau_a^-), \quad
(\rho'_2,{\cal W}'/ \ker \tau_s^+), \quad
(\rho'_2,{\cal W}'/ \ker \tau_s^-)
$$
as well as the trivial one-dimensional representation spanned by
$N(Z)^{-1} \cdot Z$.
By Proposition \ref{kerMx-lem}, these are the ten irreducible components of
$\ker \M$.
It remains to describe the three irreducible components of ${\cal W}'/\ker \M$
that are dual to $(\rho_2, {\cal Q}^+)$, $(\rho_2, {\cal Q}^-)$ and
$(\rho_2, {\cal Q}^0)$.

Recall the functions $\tilde F_{l,m,n}(Z)$, $\tilde G_{l,m,n}(Z)$,
$\tilde H_{l,m,n}(Z)$ from Proposition \ref{A-K-types}.
We also introduce functions
\begin{multline*}
\tilde I_{k,l,m,n}(Z) =
(2l+k+2) N(Z)^k \cdot \begin{pmatrix}
(l+m+1) t^l_{n\,\underline{m}}(Z) & -(l-m) t^l_{n\,\underline{m+1}}(Z) \\
-(l+m+1) t^l_{n+1\,\underline{m}}(Z) & (l-m) t^l_{n+1\,\underline{m+1}}(Z)
\end{pmatrix}  \\
+ k N(Z)^{k-1} \cdot \begin{pmatrix}
(l-n+1) t^{l+1}_{n\,\underline{m}}(Z) & (l-n+1) t^{l+1}_{n\,\underline{m+1}}(Z) \\
(l+n+2) t^{l+1}_{n+1\,\underline{m}}(Z) & (l+n+2) t^{l+1}_{n+1\,\underline{m+1}}(Z)
\end{pmatrix}, \qquad \begin{smallmatrix} k=\pm1,\pm2,\pm3,\dots \\
l=-\frac12,0,\frac12,1,\frac32,\dots \\
-l-1 \le m,n \le l \end{smallmatrix},
\end{multline*}
(when $l=-\frac12$, the function $\tilde I_{k,-\frac12,-\frac12,-\frac12}(Z)$
reduces to $k N(Z)^{k-1} \cdot Z$) and
$$
\tilde J_{l,m,n}(Z) = N(Z)^{-1} \cdot \begin{pmatrix}
(l-n) t^l_{n\,\underline{m}}(Z) & (l-n) t^l_{n\,\underline{m+1}}(Z) \\
(l+n+1) t^l_{n+1\,\underline{m}}(Z) & (l+n+1) t^l_{n+1\,\underline{m+1}}(Z)
\end{pmatrix}, \qquad \begin{smallmatrix}
l=\frac12,1,\frac32,2,\dots \\
-l \le m,n \le l-1 \end{smallmatrix}.
$$
Note that:
\begin{align*}
\text{when $k=1$, the function } &\tilde I_{1,l,m,n}(Z)
\text{ is $\tilde I_{l+1,m,n}(Z)$ introduced after Lemma \ref{FGH-reg}},  \\
\text{when $2l+k+2=0$, the function } &\tilde I_{k,l,m,n}(Z)
\text{ is proportional to $\tilde H'_{l+1,-n-1,-m-1}(Z)$},  \\
\text{when $2l+k+1=0$, the function } &\tilde I_{k,l,m,n}(Z)
\text{ is proportional to $\tilde I'_{l,-n-1,-m-1}(Z)$},  \\
&\tilde J_{\frac12,-\frac12,-\frac12}(Z) = N(Z)^{-1} \cdot Z.
\end{align*}

Recall that $K = U(2) \times U(2)$ sitting as the diagonal subgroup of
$GL(2,\HC)$.

\begin{prop}  \label{W'-K-types}
The functions
$$
N(Z)^k \cdot \tilde F_{l,m,n}(Z), \quad N(Z)^k \cdot \tilde G_{l,m,n}(Z), \quad
N(Z)^k \cdot \tilde H_{l,m,n}(Z), \quad \tilde I_{k,l,m,n}(Z), \quad
\tilde J_{l,m,n}(Z)
$$
span ${\cal W}'$ and generate the $K$-types of $(\rho'_2,{\cal W}')$.
More precisely, as representations of $SU(2) \times SU(2)$,
\begin{align*}
V_{l-\frac12} \boxtimes V_{l+\frac12} &= \BB C\text{-span of }
\left\{ N(Z)^k \cdot \tilde F_{l,m,n}(Z) ;\: \begin{smallmatrix}
-l-1 \le m \le l \\  -l \le n \le l-1 \end{smallmatrix} \right\},  \\
V_{l+\frac12} \boxtimes V_{l-\frac12} &= \BB C\text{-span of }
\left\{ N(Z)^k \cdot \tilde G_{l,m,n}(Z) ;\: \begin{smallmatrix}
-l \le m \le l-1 \\  -l-1 \le n \le l \end{smallmatrix} \right\},  \\
V_{l+\frac12} \boxtimes V_{l+\frac12} &= \BB C\text{-span of }
\left\{ N(Z)^k \cdot \tilde H_{l,m,n}(Z) ;\: -l-1 \le m,n \le l \right\},  \\
V_{l+\frac12} \boxtimes V_{l+\frac12} &= \BB C\text{-span of }
\left\{ \tilde I_{k,l,m,n}(Z) ;\: -l-1 \le m,n \le l \right\},  \\
V_{l-\frac12} \boxtimes V_{l-\frac12} &= \BB C\text{-span of }
\left\{ \tilde J_{l,m,n}(Z) ;\: -l \le m,n \le l-1 \right\},
\end{align*}
for $k$ and $l$ fixed.
\end{prop}

\begin{proof}
Clearly, these functions are $K$-finite and linearly independent.
One checks that these functions span all of ${\cal W}'$ by checking,
for each $d \in \BB Z$, that these functions generate all the $K$-types of
$$
{\cal W}'(d) =
\{ F(Z) \in {\cal W}';\: \text{$F(Z)$ is homogeneous of degree $d$} \}.
$$
\end{proof}

\begin{lem}
We have:
$$
\partial^+(\Sh) = \BB C\text{-span of } \biggl\{ \tilde H_{l,m,n}(Z);\:
\begin{smallmatrix} l=0,\frac12,1,\frac32,\dots \\
-l-1 \le m,n \le l \end{smallmatrix} \biggr\}
\bigoplus
\BB C\text{-span of } \biggl\{ \tilde I_{k,l,m,n}(Z);\:
\begin{smallmatrix} k=\pm1,\pm2,\pm3,\dots \\
l=-\frac12, 0,\frac12,1,\frac32,\dots \\
-l-1 \le m,n \le l \end{smallmatrix} \biggr\}.
$$
\end{lem}

\begin{proof}
By Corollary 6 from \cite{FL3},
$$
\Sh = \BB C\text{-span of } \biggl\{ N(Z)^k \cdot t^l_{n\,\underline{m}}(Z);\:
\begin{smallmatrix} k=0,\pm1,\pm2,\dots \\
l= 0,\frac12,1,\frac32,\dots \\
-l \le m,n \le l \end{smallmatrix} \biggr\}.
$$
Using Lemma 21 in \cite{FL1} and identity (\ref{Zt-identity}), we obtain:
\begin{align*}
\partial^+ \bigl( t^0_{0\,\underline{0}}(Z) \bigr) &=0,  \\
\partial^+ \bigl( t^l_{n\,\underline{m}}(Z) \bigr)
&= \tilde H_{l-\frac12,m-\frac12,n-\frac12}(Z) \quad \text{if $l>0$},  \\
\partial^+ \bigl( N(Z)^k \cdot t^l_{n\,\underline{m}}(Z) \bigr)
&= \frac1{2l+1} \tilde I_{k,l-\frac12,m-\frac12,n-\frac12}(Z)
\quad \text{if $k \ne 0$}.
\end{align*}
\end{proof}

Now let
\begin{align*}
{\cal Q}'^+ &= \BB C\text{-span of }
\bigl\{ N(Z)^k \cdot \tilde F_{l,m,n}(Z),\: N(Z)^k \cdot \tilde G_{l,m,n}(Z),\:
N(Z)^k \cdot \tilde H_{l,m,n}(Z);\: k \ge 1 \bigr\}, \\
{\cal Q}'^- &= \BB C\text{-span of }
\bigl\{ N(Z)^k \cdot \tilde F_{l,m,n}(Z),\: N(Z)^k \cdot \tilde G_{l,m,n}(Z),\:
N(Z)^k \cdot \tilde H_{l,m,n}(Z);\: k \le -(2l+2) \bigr\}, \\
{\cal Q}'^0 &= \BB C\text{-span of }
\left\{ \begin{matrix}
N(Z)^k \cdot \tilde F_{l,m,n}(Z),\: N(Z)^k \cdot \tilde G_{l,m,n}(Z)
\text{ with } -2l \le k \le -1, \\
N(Z)^k \cdot \tilde H_{l,m,n}(Z) \text{ with } -(2l+1) \le k \le -1,  \\
\tilde J_{l,m,n}(Z) \text{ with } l=1,\frac32,2,\frac52,\dots
\end{matrix} \right\};
\end{align*}
we treat ${\cal Q}'^+$, ${\cal Q}'^-$ and ${\cal Q}'^0$
as subspaces of the quotient space ${\cal W}'/\ker \M$.

\begin{thm}
The quotient representation $(\rho'_2,{\cal W}'/\ker \M)$
is the direct sum of three irreducible components
$(\rho'_2,{\cal Q}'^+)$, $(\rho'_2,{\cal Q}'^-)$ and
$(\rho'_2,{\cal Q}'^0)$. These irreducible components are dual to
$(\rho_2,{\cal Q}^-)$, $(\rho_2,{\cal Q}^+)$ and
$(\rho_2,{\cal Q}^0)$ respectively via the invariant bilinear pairing
given in Proposition 80 in \cite{FL1}.
\end{thm}

\begin{proof}
The result follows by identifying the functions listed in
Proposition \ref{W'-K-types} that are dual to the $K$-types of
$(\rho_2,{\cal Q}^-)$, $(\rho_2,{\cal Q}^+)$ and
$(\rho_2,{\cal Q}^0)$ respectively via the bilinear pairing
given in Proposition 80 in \cite{FL1}.
\end{proof}

Combining this with our description of irreducible components of
$(\rho'_2, \ker \M)$, we obtain a complete list of all thirteen
irreducible components of $(\rho'_2,{\cal W}')$.
Now that we know the irreducible components of $(\rho_2,{\cal W})$ and
$(\rho'_2,{\cal W}')$, it is easy to identify the indecomposable
subrepresentations of $(\rho_2,{\cal W})$ and $(\rho'_2,{\cal W}')$.

Collecting the $K$-types of each irreducible
components $(\rho'_2,{\cal W}')$, we obtain a decomposition of
${\cal W}'$ into a direct sum of thirteen $K$-invariant vector subspaces.
This decomposition is compatible with the decomposition of $(\rho_1,\Zh)$
into irreducible components
$$
(\rho_1,\Zh) = (\rho_1,\Zh^+) \oplus (\rho_1,\Zh^0) \oplus (\rho_1,\Zh^-)
$$
given in Section 4 of \cite{FL3} in the following sense.

\begin{lem}  \label{W-Zh-compatibility-lemma}
As vector spaces,
$\HC \otimes \Zh = {\cal W}' = (\HC \otimes \Zh^+) \oplus
(\HC \otimes \Zh^0) \oplus (\HC \otimes \Zh^-)$,
and
\begin{align*}
\HC \otimes \Zh^+ &= \partial^+\bigl({\cal BH}^+/{\cal I}_0\bigr) \oplus
\partial^+\bigl(\Sh^+/{\cal BH}^+\bigr) \oplus
\bigl({\cal W}'/ \ker \tau_a^+\bigr) \oplus
\bigl({\cal W}'/ \ker \tau_s^+\bigr) \oplus {\cal Q}'^+,  \\
\HC \otimes \Zh^0 &= \partial^+\bigl({\cal BH}^-/{\cal I}_0\bigr) \oplus
\partial^+\bigl({\cal J}'/({\cal BH}^++{\cal BH}^-)\bigr) \oplus
\BB C\text{-span of } \bigl\{ N(Z)^{-1} \cdot Z \bigr\} \\
&\hskip1in \oplus \bigl({\cal W}'/ \ker \tau_a^-\bigr) \oplus
\bigl({\cal W}'/ \ker \tau_s^-\bigr) \oplus {\cal Q}'^0,  \\
\HC \otimes \Zh^- &= \partial^+\bigl(\Sh'^-/{\cal BH}^-\bigr) \oplus {\cal Q}'^-.
\end{align*}

On the other hand, the image of $\M: {\cal W}' \to {\cal W}$
has three irreducible components $(\rho_2,{\cal Q}^+)$,
$(\rho_2,{\cal Q}^0)$, $(\rho_2,{\cal Q}^-)$, and
$$
{\cal Q}^+ \subset \HC \otimes \Zh^+, \qquad
{\cal Q}^0 \subset \HC \otimes \Zh^0, \qquad
{\cal Q}^- \subset \HC \otimes \Zh^-.
$$
\end{lem}

\section{Polarization of Vacuum and Equivariant Maps $(\rho'_2,{\cal W}') \to
(\pi_l \otimes \pi_r, \widetilde{{\cal V} \otimes {\cal V}'})$}  \label{VP-section}

\subsection{Equivariant Maps $(\rho'_2,{\cal W}') \to
(\pi_l \otimes \pi_r, \widetilde{{\cal V} \otimes {\cal V}'})$}

Recall the $\HC$-modules $\BB S$ and $\BB S'$ introduced just before equation
(\ref{SotimesS}).
We denote by $\tilde{\cal V}$ the space of $\BB S$-valued holomorphic
left regular functions on $\HC$ (possibly with singularities) and by
$\tilde{\cal V}'$ the space of $\BB S'$-valued holomorphic right regular
functions on $\HC$ (possibly with singularities).
The group $GL(2,\HC)$ acts on these spaces via
\begin{align*}
\pi_l(h): \: f(Z) \quad &\mapsto \quad (\pi_l(h)f)(Z) =
\frac {(cZ+d)^{-1}}{N(cZ+d)} \cdot f \bigl( (aZ+b)(cZ+d)^{-1} \bigr),  \\
\pi_r(h): \: g(Z) \quad &\mapsto \quad (\pi_r(h)g)(Z) =
g \bigl( (a'-Zc')^{-1}(-b'+Zd') \bigr) \cdot \frac {(a'-Zc')^{-1}}{N(a'-Zc')}
\end{align*}
respectively, where $f \in \tilde{\cal V}$, $g \in \tilde{\cal V}'$,
$h = \bigl(\begin{smallmatrix} a' & b' \\ c' & d' \end{smallmatrix}\bigr)
\in GL(2,\HC)$ and 
$h^{-1} = \bigl(\begin{smallmatrix} a & b \\ c & d \end{smallmatrix}\bigr)$.
Inside $\tilde{\cal V}$ and $\tilde{\cal V}'$, we have subspaces
\begin{align*}
{\cal V}^+ &= \{ f \in \tilde{\cal V};\:
\text{$f: \HC \to \BB S$ is a polynomial map}\}, \\
{\cal V}^- &= \bigl\{ f \in \tilde{\cal V};\:
\pi_l \bigl( \begin{smallmatrix} 0 & 1 \\ 1 & 0 \end{smallmatrix} \bigr) f
= N(Z)^{-1} \cdot Z^{-1} \cdot f(Z^{-1}) \in {\cal V}^+ \bigr\}, \\
{\cal V}'^+ &= \{ g \in \tilde{\cal V}';\:
\text{$g: \HC \to \BB S'$ is a polynomial map}\}, \\
{\cal V}'^- &= \bigl\{ g \in \tilde{\cal V}';\:
\pi_r \bigl( \begin{smallmatrix} 0 & 1 \\ 1 & 0 \end{smallmatrix} \bigr) g
= -N(Z)^{-1} \cdot g(Z^{-1}) \cdot Z^{-1} \in {\cal V}'^+ \bigr\}.
\end{align*}

We can form a tensor product representation
$(\pi_l \otimes \pi_r, \tilde{\cal V} \otimes \tilde{\cal V}')$
and consider a larger space
$$
\widetilde{{\cal V} \otimes {\cal V}'} = \left\{ \begin{matrix}
\text{holomorphic $\HC$-valued functions in two variables} \\
\text{$Z_1,Z_2 \in \HC$ (possibly with singularities) that are left regular} \\
\text{with respect to $Z_1$ and right regular with respect to $Z_2$}
\end{matrix} \right\}.
$$
The action of $GL(2,\HC)$ on these functions is given by
\begin{multline*}
(\pi_l \otimes \pi_r)(h): \: F(Z_1,Z_2) \quad \mapsto \quad
\bigl( (\pi_l \otimes \pi_r)(h)F \bigr)(Z_1,Z_2) \\
= \frac {(cZ_1+d)^{-1}}{N(cZ_1+d)} \cdot
F \bigl( (aZ_1+b)(cZ_1+d)^{-1}, (a'-Z_2c')(-b'+Z_2d')^{-1} \bigr) \cdot
\frac {(a'-Z_2c')^{-1}}{N(a'-Z_2c')}.
\end{multline*}
Differentiating, we obtain the corresponding action of the Lie algebra
$\mathfrak{gl}(2,\HC)$ also denoted by $\pi_l \otimes \pi_r$.

We denote by $\DR$ the restriction to the diagonal map
\begin{equation}  \label{DiagRes}
\DR: F(Z_1,Z_2) \mapsto F(Z,Z).
\end{equation}
Clearly, $\DR$ intertwines the actions of $\pi_l \otimes \pi_r$ and $\rho_2$.

We consider maps ${\cal W}' \to \widetilde{{\cal V} \otimes {\cal V}'}$
\begin{equation}  \label{fork}
(J_R F)(Z_1,Z_2) = \frac {12i}{\pi^3} \int_{W \in U(2)_R}
\frac{(W-Z_1)^{-1}}{N(W-Z_1)} \cdot F(W) \cdot \frac{(W-Z_2)^{-1}}{N(W-Z_2)} \,dV.
\end{equation}
If $Z_1, Z_2 \in \BB D^-_R \sqcup \BB D^+_R$, the integrand has no singularities
and the result is a holomorphic function in two variables $Z_1, Z_2$
which is harmonic in each variable separately.
We will see soon that the result depends on whether $Z_1$ and $Z_2$
are both in $\BB D^+_R$, both in $\BB D^-_R$ or one is in $\BB D^+_R$ and
the other is in $\BB D^-_R$.
Thus the expression (\ref{fork}) determines four different maps.
We use notations $J_R^{++}$ and $J_R^{--}$ to signify
$Z_1, Z_2 \in \BB D^+_R$ and $Z_1, Z_2 \in \BB D^-_R$ respectively.
(Notations $J_R^{+-}$ and $J_R^{-+}$ will be introduced in the next subsection.)
These maps $J_R$ are closely related to the maps $I_R$ given by
equation (34) in Chapter 6 of \cite{FL3}
\begin{equation}  \label{I_R}
\Zh \ni f \quad \mapsto \quad (I_R f)(Z_1,Z_2) =
\frac i{2\pi^3} \int_{W \in U(2)_R} \frac{f(W) \,dV}{N(W-Z_1) \cdot N(W-Z_2)}
\quad \in \widetilde{{\cal H} \otimes {\cal H}},
\end{equation}
where $\widetilde{{\cal H} \otimes {\cal H}}$ denotes the space of
holomorphic $\BB C$-valued functions in two variables $Z_1,Z_2 \in \HC$
(possibly with singularities) that are harmonic in each variable separately.
Indeed, $I_R$ extends to a map on $\HC \otimes \Zh = {\cal W}'$, and
\begin{equation}  \label{I-J-relation}
J_R F(Z_1,Z_2) = 24 \overrightarrow{\partial}_{Z_1} \bigl( I_RF(Z_1,Z_2) \bigr)
\overleftarrow{\partial}_{Z_2}.
\end{equation}

Recall from Section 2 of \cite{FL3} that the group $U(2,2)_R$ is a conjugate
of $U(2,2)$, which is a real form of $GL(2,\HC)$ preserving $U(2)_R$,
$\BB D_R^+$ and $\BB D_R^-$.

\begin{prop}  \label{equivariance}
The maps $F \mapsto (J_R F)(Z_1,Z_2)$ are $U(2,2)_R$ and
$\mathfrak{gl}(2,\HC)$-equivariant maps from $(\rho'_2,{\cal W}')$
to $(\pi_l \otimes \pi_r, \widetilde{{\cal V} \otimes {\cal V}'})$.
\end{prop}

\begin{proof}
We need to show that, for all $h \in U(2,2)_R$, the maps (\ref{fork})
commute with the action of $h$. Writing
$h= \bigl(\begin{smallmatrix} a' & b' \\ c' & d' \end{smallmatrix}\bigr)$,
$h^{-1}= \bigl(\begin{smallmatrix} a & b \\ c & d \end{smallmatrix}\bigr)$,
$$
\tilde Z_1 = (aZ_1+b)(cZ_1+d)^{-1}, \qquad
\tilde Z_2 = (aZ_2+b)(cZ_2+d)^{-1}, \qquad
\tilde W = (aW+b)(cW+d)^{-1}
$$
and using Lemmas 10 and 61 from \cite{FL1} we obtain:
\begin{multline*}
\int_{W \in U(2)_R} \frac{(W-Z_1)^{-1}}{N(W-Z_1)} \cdot (\rho'_2(h)F)(W) \cdot
\frac{(W-Z_2)^{-1}}{N(W-Z_2)} \,dV  \\
= \int_{W \in U(2)_R} \frac{(W-Z_1)^{-1}}{N(W-Z_1)} \cdot
\frac{(a'-Wc')}{N(a'-Wc')} \cdot F(\tilde W) \cdot \frac{(cW+d)}{N(cW+d)}
\cdot \frac{(W-Z_2)^{-1}}{N(W-Z_2)} \,dV  \\
= \int_{W \in U(2)_R} \frac{(cZ_1+d)^{-1} \cdot (\tilde W - \tilde Z_1)^{-1} \cdot
F(\tilde W) \cdot (\tilde W - \tilde Z_2)^{-1} \cdot (a'-Z_2c')^{-1} \,dV}
{N(cZ_1+d) \cdot N(a'-Wc')^2 \cdot N(\tilde W - \tilde Z_1) \cdot
N(\tilde W - \tilde Z_2) \cdot N(cW+d)^2 \cdot N(a'-Z_2c')}  \\
= \frac{(cZ_1+d)^{-1}}{N(cZ_1+d)} \cdot \int_{\tilde W \in U(2)_R}
\frac{(\tilde W - \tilde Z_1)^{-1}}{N(\tilde W - \tilde Z_1)} \cdot F(\tilde W)
\cdot \frac{(\tilde W - \tilde Z_2)^{-1}}{N(\tilde W - \tilde Z_2)} \,dV
\cdot \frac{(a'-Z_2c')^{-1}}{N(a'-Z_2c')}.
\end{multline*}
This proves the $U(2,2)_R$-equivariance.
The $\mathfrak{gl}(2,\HC)$-equivariance then follows since
$\mathfrak{gl}(2,\HC) \simeq \BB C \otimes \mathfrak{u}(2,2)_R$.
\end{proof}

We compose the maps $J_R$ with the restriction onto the diagonal map $\DR$
defined by (\ref{DiagRes}).
Note that the subspace
${\cal V}^+ \otimes {\cal V}'^+ \subset \widetilde{{\cal V} \otimes {\cal V}'}$
can be described as the $\HC$-valued polynomials in
$\widetilde{{\cal V} \otimes {\cal V}'}$.

\begin{thm}  \label{J_R(W^+-)}
The maps $F \mapsto (J_R F)(Z_1,Z_2)$ have the following properties:
\begin{enumerate}
\item  \label{one}
If $Z_1, Z_2 \in \BB D^+_R$, then $J_R^{++}$ maps ${\cal W}'$ into
${\cal V}^+ \otimes {\cal V}'^+ \subset \widetilde{{\cal V} \otimes {\cal V}'}$,
annihilates all irreducible components of $(\rho'_2,{\cal W}')$,
except for ${\cal Q}'^+$, and
$$
\DR \circ (J_R^{++} F)(Z_1,Z_2) = \M F, \qquad
\text{if $F \in \HC \otimes \Zh^+$}.
$$

\item
If $Z_1, Z_2 \in \BB D^-_R$, then $J_R^{--}$ maps ${\cal W}'$ into
${\cal V}^- \otimes {\cal V}'^- \subset
\widetilde{{\cal V} \otimes {\cal V}'}$, annihilates all irreducible
components of $(\rho'_2,{\cal W}')$, except for ${\cal Q}'^-$, and
$$
\DR \circ (J_R^{--} F)(Z_1,Z_2) = \M F, \qquad
\text{if $F \in \HC \otimes \Zh^-$}.
$$
\end{enumerate}
\end{thm}

\begin{proof}
We prove part \ref{one} only, the other part can be proven in the same way.
So, suppose that $Z_1, Z_2 \in \BB D^+_R$.
It follows immediately from Theorem 12 in \cite{FL3},
Lemma \ref{W-Zh-compatibility-lemma} and equation (\ref{I-J-relation})
that the image of $J_R^{++}$ lies in ${\cal V}^+ \otimes {\cal V}'^+$ and that
$J_R^{++}$ annihilates the irreducible components of $(\rho'_2,{\cal W}')$
that lie in $\HC \otimes (\Zh^0 \oplus \Zh^-)$ (as described in
Lemma \ref{W-Zh-compatibility-lemma}).
Then we check the effect of $J_R^{++}$ on a suitable generator of each of
the remaining irreducible component of $(\rho'_2,{\cal W}')$.
For the irreducible components from Corollary \ref{FG-irred}
$$
(\rho'_2,{\cal W}'/ \ker \tau_a^+) \qquad \text{and} \qquad
(\rho'_2,{\cal W}'/ \ker \tau_s^+),
$$
choose generators $\tilde F_{\frac12,\frac12,-\frac12}(W)$ and 
$\tilde G_{\frac12,-\frac12,\frac12}(W)$ respectively.
The image of $\Sh^+$ under $\partial^+$ is generated by
$\partial^+(N(Z)^2) = 2N(Z) \cdot Z$ and contains two irreducible components.
We show the calculations for $N(Z) \cdot Z$ only, the other cases are similar.
Using Lemma 23 and the matrix coefficient expansion of
$\frac{(W-Z_1)^{-1}}{N(W-Z_1)}$ given by Proposition 26 in \cite{FL1}
(see also Proposition \ref{Prop26}), we compute:
\begin{multline*}
J_R^{++} \bigl( N(W) \cdot W \bigr) (Z_1,Z_2)
= \frac{24i}{\pi^3} \int_{W \in U(2)_R} \frac{(W-Z_1)^{-1}}{N(W-Z_1)} \cdot
N(W) \cdot W \cdot \frac{(W-Z_2)^{-1}}{N(W-Z_2)} \,dV  \\
= \frac{24i}{\pi^3} \sum_{l,m,n,l',m',n'}
\left(\begin{smallmatrix} (l-m+ \frac 12) t^l_{n \, \underline{m+ \frac 12}}(Z_1)  \\
(l+m+ \frac 12) t^l_{n \, \underline{m- \frac 12}}(Z_1) \end{smallmatrix}\right)  \\
\times \int_{W \in U(2)_R} 
\left( \begin{smallmatrix} t^{l+\frac 12}_{m \, \underline{n- \frac 12}}(W^{-1}), &
t^{l+\frac 12}_{m \, \underline{n+ \frac 12}}(W^{-1}) \end{smallmatrix} \right)
\cdot \frac{W}{N(W)} \cdot
\left(\begin{smallmatrix} (l'-n'+\frac 12) t^{l'}_{n'-\frac12\,\underline{m'}}(W^{-1}) \\
(l'+n'+\frac12) t^{l'}_{n'+\frac12\,\underline{m'}}(W^{-1})\end{smallmatrix}\right)\,dV\\
\times \left(\begin{smallmatrix} t^{l'-\frac 12}_{m'+\frac12\,\underline{n'}}(Z_2), &
t^{l'-\frac12}_{m'-\frac12\,\underline{n'}}(Z_2) \end{smallmatrix}\right)  \\
= \frac{24i}{\pi^3} \sum_{l,m,n,l',m',n'}
\left(\begin{smallmatrix} (l-m+ \frac 12) t^l_{n \, \underline{m+ \frac 12}}(Z_1)  \\
(l+m+ \frac 12) t^l_{n \, \underline{m- \frac 12}}(Z_1) \end{smallmatrix}\right)
\cdot \left(\begin{smallmatrix} t^{l'-\frac 12}_{m'+\frac12\,\underline{n'}}(Z_2), &
t^{l'-\frac12}_{m'-\frac12\,\underline{n'}}(Z_2) \end{smallmatrix}\right)  \\
\times \int_{W \in U(2)_R} \frac1{N(W)}
\left( \begin{smallmatrix} t^l_{m+\frac12 \, \underline{n}}(W^{-1}), &
t^l_{m-\frac12 \, \underline{n}}(W^{-1}) \end{smallmatrix} \right) \cdot
\left(\begin{smallmatrix} (l'-n'+\frac 12) t^{l'}_{n'-\frac12\,\underline{m'}}(W^{-1}) \\
(l'+n'+\frac12) t^{l'}_{n'+\frac12\,\underline{m'}}(W^{-1})\end{smallmatrix}\right)
\,dV =0
\end{multline*}
by the orthogonality relations (19) in \cite{FL3}, since $l' \ge 1/2$.
We conclude from Proposition \ref{equivariance} that $J_R^{++}$ annihilates
all of $(\rho'_2,\partial^+(\Sh^+))$.

Finally, the statement about the composition $\DR \circ J_R^{++}$ follows from
Theorem 77 in \cite{FL1}. (Note that the differential form $dV$ that appears
in (\ref{fork}) differs from $dZ^4$ that appears in Theorem 77 in \cite{FL1}
by a factor of $1/4$.)
\end{proof}

\subsection{Polarization of Vacuum}

Now we suppose $Z_1 \in \BB D^+_R$ and $Z_2 \in \BB D^-_R$
(or the other way around), this case is much more subtle.
We reduce the spinor case of $J_R$ to the already known scalar case of $I_R$
as much as possible via the relation (\ref{I-J-relation}).

\begin{prop}  \label{J_R(W^0)-prop}
Let $Z_1 \in \BB D^+_R$ and $Z_2 \in \BB D^-_R$, then the map
$J_R$ annihilates all irreducible components of $(\rho'_2,{\cal W}')$,
except for ${\cal Q}'^0$ and the trivial one-dimensional representation
spanned by $N(W)^{-1} \cdot W$. Moreover,
$$
J_R \bigl( N(W)^{-1} \cdot W \bigr) (Z_1,Z_2)
= 24 \frac{(Z_1-Z_2)^{-1}}{N(Z_1-Z_2)}.
$$

Similarly, if $Z_1 \in \BB D^-_R$ and $Z_2 \in \BB D^+_R$, then the map
$J_R$ annihilates all irreducible components of $(\rho'_2,{\cal W}')$,
except for ${\cal Q}'^0$ and the trivial one-dimensional representation
spanned by $N(W)^{-1} \cdot W$. Moreover,
$$
J_R \bigl( N(W)^{-1} \cdot W \bigr) (Z_1,Z_2)
= - 24 \frac{(Z_1-Z_2)^{-1}}{N(Z_1-Z_2)}.
$$
\end{prop}

\begin{proof}
Since $I_R$ annihilates $\Zh^+ \oplus \Zh^-$, it follows from equation
(\ref{I-J-relation}) that $J_R$ annihilates the irreducible components of
$(\rho'_2,{\cal W}')$ that lie in $\HC \otimes (\Zh^+ \oplus \Zh^-)$
(as described in Lemma \ref{W-Zh-compatibility-lemma}).
Then we check the effect of $J_R$ on a suitable generator of each of
the remaining irreducible component of $(\rho'_2,{\cal W}')$.

The image of $\Sh'$ under $\partial^+$ is generated by two generators
$$
\partial^+(N(Z)^2) = 2N(Z) \cdot Z \qquad \text{and} \qquad
\partial^+(N(Z)^{-2}) = -2N(Z)^{-3} \cdot Z;
$$
these take care of irreducible components contained in $\partial^+(\Sh')$.
For the irreducible components
$(\rho'_2,{\cal W}'/ \ker \tau_a^-)$ and $(\rho'_2,{\cal W}'/ \ker \tau_s^-)$
make a choice of generators such as
$$
\tilde F'_{\frac12,\frac12,-\frac12}(W)
= \frac2{N(W)^2} \begin{pmatrix} 0 & 0 \\ w_{11} & w_{12} \end{pmatrix}
\qquad \text{and} \qquad
\tilde G'_{\frac12,-\frac12,\frac12}(W)
= \frac1{N(W)^2} \begin{pmatrix} 0 & w_{11} \\ 0 & w_{21} \end{pmatrix}
$$
respectively.
We show the calculations for $\tilde F'_{\frac12,\frac12,-\frac12}(W)$ with
$Z_1 \in \BB D^+_R$, $Z_2 \in \BB D^-_R$ only;
the other cases are similar and easier.
Using Lemma 23 and the matrix coefficient expansion of
$\frac{(W-Z_1)^{-1}}{N(W-Z_1)}$ given by Proposition 26 in \cite{FL1}
(see also Proposition \ref{Prop26}), we compute:
\begin{multline*}
J_R \bigl( \tilde F'_{\frac12,\frac12,-\frac12}(W) \bigr) (Z_1,Z_2)  \\
= \frac{24i}{\pi^3} \int_{W \in U(2)_R} \frac{(W-Z_1)^{-1}}{N(W-Z_1)} \cdot
\frac1{N(W)^2} \cdot \begin{pmatrix} 0 & 0 \\ w_{11} & w_{12} \end{pmatrix}
\cdot \frac{(W-Z_2)^{-1}}{N(W-Z_2)} \,dV  \\
= - \frac{24i}{\pi^3} \sum_{l,m,n,l',m',n'}
\left(\begin{smallmatrix} (l-m+ \frac 12) t^l_{n \, \underline{m+ \frac 12}}(Z_1)  \\
(l+m+ \frac 12) t^l_{n \, \underline{m- \frac 12}}(Z_1) \end{smallmatrix}\right)  \\
\times \int_{W \in U(2)_R} 
\left( \begin{smallmatrix} t^{l+\frac 12}_{m \, \underline{n- \frac 12}}(W^{-1}), &
t^{l+\frac 12}_{m \, \underline{n+ \frac 12}}(W^{-1}) \end{smallmatrix} \right) \cdot
\frac1{N(W)^3} \cdot \begin{pmatrix} 0 & 0 \\ w_{11} & w_{12} \end{pmatrix} \cdot
\left(\begin{smallmatrix} (l'-m'+\frac 12) t^{l'}_{n'\, \underline{m'+\frac 12}}(W)\\
(l'+m'+\frac 12) t^{l'}_{n'\,\underline{m'-\frac12}}(W) \end{smallmatrix}\right)\,dV \\
\times N(Z_2)^{-1} \cdot
\left(\begin{smallmatrix} t^{l'+\frac12}_{m'\,\underline{n'-\frac12}}(Z_2^{-1}), &
t^{l'+\frac12}_{m' \,\underline{n'+\frac 12}}(Z_2^{-1}) \end{smallmatrix}\right)  \\
= \frac{-24i}{\pi^3 N(Z_2)} \sum_{l,m,n,l',m',n'}
\left(\begin{smallmatrix} (l-m+ \frac 12) t^l_{n \, \underline{m+ \frac 12}}(Z_1)  \\
(l+m+ \frac 12) t^l_{n \, \underline{m- \frac 12}}(Z_1) \end{smallmatrix}\right)
\cdot \left(\begin{smallmatrix} t^{l'+\frac12}_{m'\,\underline{n'-\frac12}}(Z_2^{-1}), &
t^{l'+\frac12}_{m' \,\underline{n'+\frac 12}}(Z_2^{-1}) \end{smallmatrix}\right)  \\
\times \int_{W \in U(2)_R} 
\left( \begin{smallmatrix} t^{l+\frac 12}_{m \, \underline{n- \frac 12}}(W^{-1}), &
t^{l+\frac 12}_{m \, \underline{n+ \frac 12}}(W^{-1}) \end{smallmatrix} \right)
\cdot \frac1{N(W)^3} \cdot \left(\begin{smallmatrix} 0  \\
(l'-n'+1) t^{l'+\frac12}_{n'-\frac12\,\underline{m'}}(W) \end{smallmatrix}\right)
\,dV =0
\end{multline*}
by the orthogonality relations (19) in \cite{FL3}, since the power of $N(W)$
is not $-2$. We conclude from Proposition \ref{equivariance} that $J_R$
annihilates all of $(\rho'_2,{\cal W}'/ \ker \tau_a^-)$.
\end{proof}

For $Z_1, Z_2 \in \HC^{\times}$, let $\lambda_1$ and $\lambda_2$ denote the
eigenvalues of $Z_1Z_2^{-1}$, and introduce an open subset of
$\HC^{\times} \times \HC^{\times}$
$$
\Omega_0 = \left\{ (Z_1,Z_2) \in \HC^{\times} \times \HC^{\times} ;\:
\begin{smallmatrix} \lambda_1 \ne 1,\: \lambda_2 \ne 1,\:
\text{neither $\frac{1-\lambda_1}{1-\lambda_2}$ nor} \\
\text{$\frac{1-\lambda_1^{-1}}{1-\lambda_2^{-1}}$ is a negative real number}
\end{smallmatrix}\right\}.
$$
Let $Z_1 \in \BB D^+_R$ and $Z_2 \in \BB D^-_R$ and recall
the relation (\ref{I-J-relation}) between $J_R$ and $I_R$.
From Section 6 of \cite{FL3}, we see that, for any $F \in {\cal W}'$,
$(J_RF)(Z_1,Z_2)$ extends analytically across $\Omega_0$, and we have
a well defined operator $J_R^{+-}$ on ${\cal W}'$:
$$
(J_R^{+-} F)(Z_1,Z_2) = \text{analytic extension of $(J_RF)(Z_1,Z_2)$
from $\BB D^+_R \times \BB D^-_R$ to $\Omega_0$}.
$$
The operator $J_R^{+-}$ is $U(2) \times U(2)$ and
$\mathfrak{gl}(2,\HC)$-equivariant (which follows from
Proposition \ref{equivariance}) and annihilates all irreducible components of
$(\rho'_2,{\cal W}')$, except for ${\cal Q}'^0$ and the trivial
one-dimensional representation spanned by $N(W)^{-1} \cdot W$
(which follows from Proposition \ref{J_R(W^0)-prop}).
While $J_R^{+-} F$ is independent of the choice of $R>0$,
we keep the subscript $R$ to distinguish this analytic function from
a formal series $J^{+-} F$ that will be defined in
Subsection \ref{convolution-alg-subsection}.

Similarly, we can switch the domains of $Z_1$ and $Z_2$ and define another
operator $J_R^{-+}$ on ${\cal W}'$:
$$
(J_R^{-+} F)(Z_1,Z_2) = \text{analytic extension of $(J_RF)(Z_1,Z_2)$
from $\BB D^-_R \times \BB D^+_R$ to $\Omega_0$}.
$$
The operator $J_R^{-+}$ is also $U(2) \times U(2)$ and
$\mathfrak{gl}(2,\HC)$-equivariant and annihilates all
irreducible components of $(\rho'_2,{\cal W}')$, except for
${\cal Q}'^0$ and the trivial one-dimensional representation spanned by
$N(W)^{-1} \cdot W$.

We introduce the following notation: if $\lambda \in \BB C$, let
$$
\sgn (\im \lambda) = \begin{cases}
1 & \text{ if $\lambda$ lies in the upper half plane of $\BB C$}, \\
-1 & \text{ if $\lambda$ lies in the lower half plane of $\BB C$}, \\
\text{undefined} & \text{ if $\lambda \in \BB R$}.
\end{cases}
$$

\begin{thm}  \label{Mx^0-operator-thm}
We have a well defined operator on ${\cal W}'$
\begin{equation}  \label{Mx^0-operator}
(\M^0 F)(Z) = \lim_{\genfrac{}{}{0pt}{}{Z_1, Z_2 \to Z, \: N(Z_1-Z_2) \ne 0}
{\sgn (\im \lambda_1) = \sgn (\im \lambda_2)}}
- \bigl( (J_R^{+-} + J_R^{-+}) F \bigr)(Z_1,Z_2),
\qquad Z \in U(2)_R,
\end{equation}
where $\lambda_1$ and $\lambda_2$ are the eigenvalues of $Z_1Z_2^{-1}$.
The operator $\M^0$ has values in ${\cal W}$,
is $\mathfrak{gl}(2,\HC)$-equivariant, annihilates all irreducible components
of $(\rho'_2,{\cal W}')$, except for ${\cal Q}'^0$, and equals $\M$ on
${\cal Q}'^0$.

Furthermore, the operator $\M^0$ on ${\cal W}'$ can be computed as follows:
\begin{multline*}
(\M^0 F)(Z) = - \frac{12i}{\pi^3} \lim_{\theta \to 0} \lim_{s \to 1} \biggl(
\int_{W \in U(2)_R} \frac{(W-se^{i\theta}Z)^{-1}}{N(W-se^{i\theta}Z)}
\cdot F(W) \cdot \frac{(W-s^{-1}e^{-i\theta}Z)^{-1}}{N(W-s^{-1}e^{-i\theta}Z)} \,dV\\
+ \int_{W \in U(2)_R} \frac{(W-s^{-1}e^{i\theta}Z)^{-1}}{N(W-s^{-1}e^{i\theta}Z)}
\cdot F(W) \cdot \frac{(W-se^{-i\theta}Z)^{-1}}{N(W-se^{-i\theta}Z)} \,dV
\biggr), \qquad Z \in U(2)_R.
\end{multline*}
\end{thm}

Note that the space ${\cal W}'$ consists of rational functions, and rational
functions on $\HC$ as well as analytic ones are completely determined by
their values on $U(2)_R$.

\begin{proof}
First, we show that the limit (\ref{Mx^0-operator}) exists.
The map $(J_R^{+-} + J_R^{-+})$ is related to the map $(I_R^{+-} + I_R^{-+})$
from Theorem 15 in \cite{FL3} via
\begin{equation}  \label{I-J-relation-repeat}
\bigl( (J_R^{+-} + J_R^{-+}) F \bigr)(Z_1,Z_2)
= 24 \overrightarrow{\partial}_{Z_1} \bigl( (I_R^{+-} + I_R^{-+}) F (Z_1,Z_2)
\bigr) \overleftarrow{\partial}_{Z_2},
\end{equation}
which is essentially equation (\ref{I-J-relation}).
We saw in the course of proof of Theorem 15 in \cite{FL3}
(see also Theorem \ref{Thm15}) that the image of 
$(I_R^{+-} + I_R^{-+})$ is generated by
\begin{equation}  \label{I(N(W)^{-1})}
\bigl( (I_R^{+-}+ I_R^{-+}) N(W)^{-1} \bigr)(Z_1,Z_2)
= -\frac1{N(Z_2)} \cdot
\begin{cases} \frac{\log\lambda_2-\log\lambda_1}{\lambda_2-\lambda_1} &
\text{if $\lambda_1 \ne \lambda_2$;} \\
\lambda^{-1} & \text{if $\lambda_1 = \lambda_2 = \lambda$, $\lambda \ne 1$,}
\end{cases}
\end{equation}
where $\log$ denotes the branch of logarithm with a cut along the positive
real axis. If we restrict $\bigl( (I_R^{+-}+ I_R^{-+}) N(W)^{-1} \bigr)(Z_1,Z_2)$
to the open set where $\sgn (\im \lambda_1) = \sgn (\im \lambda_2)$,
we see that this restriction is 
$$
\frac1{N(Z_2)} \cdot \bigl( \text{function of
  $\lambda_1, \lambda_2 \in \BB C$ that is holomorphic near
  $(\lambda_1,\lambda_2)=(1,1)$} \bigr).
$$
Since the map $(I_R^{+-}+ I_R^{-+})$ is $\mathfrak{gl}(2,\HC)$-equivariant,
it follows that the same is true for
$\bigl( (I_R^{+-}+ I_R^{-+}) F \bigr)(Z_1,Z_2)$ with any $F \in {\cal W}'$.
Therefore, the limit (\ref{Mx^0-operator}) exists.

Clearly, the operator $\M^0$ on ${\cal W}'$ is
$\mathfrak{gl}(2,\HC)$-equivariant and, by Proposition \ref{J_R(W^0)-prop},
annihilates all irreducible components of $(\rho'_2,{\cal W}')$, except for
${\cal Q}'^0$.
It remains to show that $\M^0$ equals $\M$ on ${\cal Q}'^0$ on one
particular generator, then the other statements of the theorem follow
immediately, including the part that the values of $\M^0$ lie in ${\cal W}$.
For this purpose we choose a generator
$$
N(Z)^{-1} = N(Z)^{-1} \cdot \tilde H_{0,0,0}(Z)
+ N(Z)^{-1} \cdot \tilde H_{0,-1,-1}(Z) \quad \in {\cal Q}'^0.
$$
Moreover, since $\M^0$ is $(U(2) \times U(2))$-equivariant, it is sufficient
to show that $\M^0 N(Z)^{-1} = \M N(Z)^{-1}$ when $Z$ is diagonal.
And since the limit (\ref{Mx^0-operator}) is known to exist, we can assume
that $Z_1$, $Z_2$ are diagonal as well. We have:
$$
\M (N(Z)^{-1}) = 4 N(Z)^{-2} + 8 N(Z)^{-3} \cdot (Z^+)^2
$$
and
\begin{equation}  \label{Mx(gen)}
\M \left( \frac1{N(Z)} \right) = \frac4{a^2d^2}
\begin{pmatrix} 1 + 2d/a & 0 \\ 0 & 1 +2a/d \end{pmatrix}, \qquad \text{if }
Z= \bigl(\begin{smallmatrix} a & 0 \\ 0 & d \end{smallmatrix} \bigr).
\end{equation}
Suppose first that $Z_1 \in \BB D^+_R$ and $Z_2 \in \BB D^-_R$.
Recall from Section 6 of \cite{FL3} that
$$
(I_R N(W)^{-1})(Z_1,Z_2) = \sum_{l,m,n} \frac1{2l+1}
t^l_{n\,\underline{m}}(Z_1) \cdot N(Z_2)^{-1} \cdot t^l_{m\,\underline{n}}(Z_2^{-1}).
$$
By (\ref{I-J-relation}),
\begin{multline*}
(J_R N(W)^{-1})(Z_1,Z_2) =  24 \overrightarrow{\partial}_{Z_1}
\bigl( (I_R N(W)^{-1})(Z_1,Z_2) \bigr) \overleftarrow{\partial}_{Z_2}  \\
= - \sum_{l,m,n} \frac{24}{2l+1}
\begin{pmatrix}
(l-m) t^{l-\frac12}_{n+\frac12 \, \underline{m+\frac12}}(Z_1) &
(l-m) t^{l-\frac12}_{n-\frac12 \, \underline{m+\frac12}}(Z_1) \\
(l+m) t^{l-\frac12}_{n+\frac12 \, \underline{m-\frac12}}(Z_1) &
(l+m) t^{l-\frac12}_{n-\frac12 \, \underline{m-\frac12}}(Z_1)
\end{pmatrix}  \\
\times \frac1{N(Z_2)} \begin{pmatrix}
(l-m+1) t^{l+\frac12}_{m-\frac12\,\underline{n-\frac12}}(Z_2^{-1}) &
(l-m+1) t^{l+\frac12}_{m-\frac12\,\underline{n+\frac12}}(Z_2^{-1})  \\
(l+m+1) t^{l+\frac12}_{m+\frac12\,\underline{n-\frac12}}(Z_2^{-1}) &
(l+m+1) t^{l+\frac12}_{m+\frac12\,\underline{n+\frac12}}(Z_2^{-1}) \end{pmatrix}.
\end{multline*}
Additionally, assume that $Z_1$ and $Z_2$ are diagonal:
$Z_1 = \bigl(\begin{smallmatrix} a_1 & 0 \\ 0 & d_1 \end{smallmatrix}\bigr)$
with $|a_1|, |d_1| < R$ and
$Z_2 = \bigl(\begin{smallmatrix} a_2 & 0 \\ 0 & d_2 \end{smallmatrix}\bigr)$
with $|a_2|, |d_2| > R$, then
\begin{equation}  \label{t-diag}
t^l_{n\,\underline{m}}
\bigl(\begin{smallmatrix} a & 0 \\ 0 & d \end{smallmatrix}\bigr)
= \begin{cases}
a^{l-m}d^{l+m} & \text{ if $m=n$};  \\
0 & \text{ otherwise},
\end{cases}
\end{equation}
and only the terms with $n=m-1,m,m+1$ are non-zero, and it is easy to see that
\begin{multline}  \label{J_R^+-(1/N)}
(J_R N(W)^{-1})(Z_1,Z_2) \\
= \sum_{l,m} \frac{24}{2l+1} \begin{pmatrix}
\frac{\partial^2}{\partial a_1 \partial a_2}
+ \frac{\partial^2}{\partial a_1 \partial d_2} & 0 \\
0 & \frac{\partial^2}{\partial d_1 \partial d_2}
+ \frac{\partial^2}{\partial a_2 \partial d_1}
\end{pmatrix}
a_1^{l-m} a_2^{-l+m-1} d_1^{l+m} d_2^{-l-m-1}  \\
= 24 \begin{pmatrix}
\frac{\partial^2}{\partial a_1 \partial a_2}
+ \frac{\partial^2}{\partial a_1 \partial d_2} & 0 \\
0 & \frac{\partial^2}{\partial d_1 \partial d_2}
+ \frac{\partial^2}{\partial a_2 \partial d_1}
\end{pmatrix}
\frac{\log(1-d_1/d_2) - \log(1-a_1/a_2)}{a_2d_2(a_1/a_2 - d_1/d_2)}.
\end{multline}
Now suppose that $Z_1 \in \BB D^-_R$ and $Z_2 \in \BB D^+_R$, then
\begin{multline*}
(J_R N(W)^{-1})(Z_1,Z_2) =  24 \overrightarrow{\partial}_{Z_1}
\bigl( (I_R N(W)^{-1})(Z_1,Z_2) \bigr) \overleftarrow{\partial}_{Z_2}  \\
= - \sum_{l,m,n} \frac{24}{2l+1} \frac1{N(Z_1)} \begin{pmatrix}
(l-m+1) t^{l+\frac12}_{m-\frac12\,\underline{n-\frac12}}(Z_1^{-1}) &
(l-m+1) t^{l+\frac12}_{m-\frac12\,\underline{n+\frac12}}(Z_1^{-1})  \\
(l+m+1) t^{l+\frac12}_{m+\frac12\,\underline{n-\frac12}}(Z_1^{-1}) &
(l+m+1) t^{l+\frac12}_{m+\frac12\,\underline{n+\frac12}}(Z_1^{-1}) \end{pmatrix}  \\
\times \begin{pmatrix}
(l-m) t^{l-\frac12}_{n+\frac12 \, \underline{m+\frac12}}(Z_2) &
(l-m) t^{l-\frac12}_{n-\frac12 \, \underline{m+\frac12}}(Z_2) \\
(l+m) t^{l-\frac12}_{n+\frac12 \, \underline{m-\frac12}}(Z_2) &
(l+m) t^{l-\frac12}_{n-\frac12 \, \underline{m-\frac12}}(Z_2)
\end{pmatrix}.
\end{multline*}
If $Z_1$ and $Z_2$ are diagonal:
$Z_1 = \bigl(\begin{smallmatrix} a_1 & 0 \\ 0 & d_1 \end{smallmatrix}\bigr)$
with $|a_1|, |d_1| > R$ and
$Z_2 = \bigl(\begin{smallmatrix} a_2 & 0 \\ 0 & d_2 \end{smallmatrix}\bigr)$
with $|a_2|, |d_2| < R$, then, by (\ref{t-diag}), only the terms with
$n=m-1,m,m+1$ are non-zero, and it is easy to see that
\begin{multline}  \label{J_R^-+(1/N)}
(J_R N(W)^{-1})(Z_1,Z_2) \\
= \sum_{l,m} \frac{24}{2l+1} \begin{pmatrix}
\frac{\partial^2}{\partial a_1 \partial a_2}
+ \frac{\partial^2}{\partial a_1 \partial d_2} & 0 \\
0 & \frac{\partial^2}{\partial d_1 \partial d_2}
+ \frac{\partial^2}{\partial a_2 \partial d_1}
\end{pmatrix}
a_1^{-l+m-1} a_2^{l-m} d_1^{-l-m-1} d_2^{l+m}  \\
= 24 \begin{pmatrix}
\frac{\partial^2}{\partial a_1 \partial a_2}
+ \frac{\partial^2}{\partial a_1 \partial d_2} & 0 \\
0 & \frac{\partial^2}{\partial d_1 \partial d_2}
+ \frac{\partial^2}{\partial a_2 \partial d_1}
\end{pmatrix}
\frac{\log(1-d_2/d_1)-\log(1-a_2/a_1)}{a_1d_1(a_2/a_1 - d_2/d_1)}.
\end{multline}
Adding (\ref{J_R^+-(1/N)}) and (\ref{J_R^-+(1/N)}), we obtain
\begin{multline}  \label{(J+J)(1/N)}
\bigl( (J_R^{+-} + J_R^{-+}) N(W)^{-1} \bigr)(Z_1,Z_2)  \\
= 24 \begin{pmatrix}
\frac{\partial^2}{\partial a_1 \partial a_2}
+ \frac{\partial^2}{\partial a_1 \partial d_2} & 0 \\
0 & \frac{\partial^2}{\partial d_1 \partial d_2}
+ \frac{\partial^2}{\partial a_2 \partial d_1}
\end{pmatrix}
\frac{\log(a_1/a_2) - \log(d_1/d_2)}{a_2d_2(a_1/a_2 - d_1/d_2)},
\end{multline}
where
$Z_1 = \bigl(\begin{smallmatrix} a_1 & 0 \\ 0 & d_1 \end{smallmatrix}\bigr)$,
$Z_2 = \bigl(\begin{smallmatrix} a_2 & 0 \\ 0 & d_2 \end{smallmatrix}\bigr)
\in U(2)_R$.
Note that in this equation $\log$ denotes the branch of logarithm with a
cut along the positive real axis.
Finally, we take a limit of (\ref{(J+J)(1/N)}) as
$Z_1, Z_2 \to
Z = \bigl(\begin{smallmatrix} a & 0 \\ 0 & d \end{smallmatrix}\bigr)
\in U(2)_R$.
Since we know that the limit exists, to find its value, we can, for example,
set $a_1=a$, $d_1=d_2=d$ and let $a_1/a_2 \to 1$. We obtain
$$
\bigl( (J_R^{+-} + J_R^{-+}) N(W)^{-1} \bigr)(Z_1,Z_2) \to
- \frac4{a^2d^2}
\begin{pmatrix} 1 + 2d/a & 0 \\ 0 & 1 +2a/d \end{pmatrix}
= - \M \left( \frac1{N(Z)} \right)
$$
as $Z_1, Z_2 \to
Z = \bigl(\begin{smallmatrix} a & 0 \\ 0 & d \end{smallmatrix}\bigr)
\in U(2)_R$ (recall our earlier computation (\ref{Mx(gen)})).
This completes our proof that $\M^0 N(Z)^{-1} = \M N(Z)^{-1}$.
\end{proof}

We note that Theorem \ref{Mx^0-operator-thm} in this paper and
Theorem 15 in \cite{FL3} can be viewed as mathematical versions of the
regularization of vacuum polarization in QED and scalar QED, as it was
discussed in Subsection 4.5 of \cite{FL1}.
One just has to add, respectively, operators $J_R^{++} + J_R^{--}$ and
$I_R^{++}+I_R^{--}$, which do not contain any singularities and are also
related by the same identity (\ref{I-J-relation-repeat}).
Note, however, that the non-scalar case (Theorem \ref{Mx^0-operator-thm})
contains an additional subtraction of the one-dimensional representation
component from $F \in {\cal W}'$ (after factorization by the intersection
of the kernels of the maps $J_R^{\pm\pm}$).
Thus, one can say that the trivial one-dimensional component of ${\cal W}'$
that appears in the decomposition of $(\rho'_2,{\cal W}')$ into
irreducible components in Subsection \ref{W'-irred-decomp-subsection}
lies at the heart of the regularization of the vacuum polarization.
The subtraction of this component in the Minkowski picture is a
subtle procedure that is a part of the art of renormalization in
four-dimensional QED.

\begin{figure}
\begin{center}
\begin{subfigure}[b]{0.25\textwidth}
\centering
\includegraphics[scale=0.4]{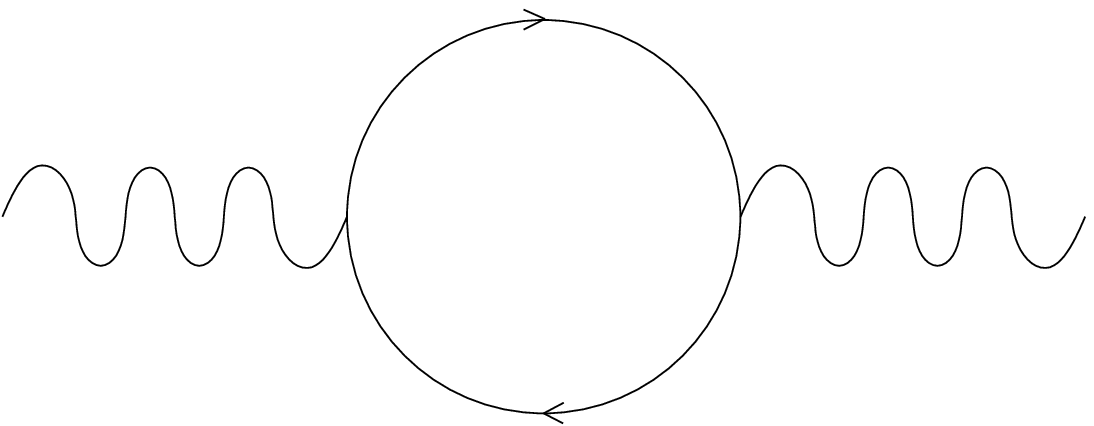}
\caption{Spinor case}
\end{subfigure}
\qquad
\begin{subfigure}[b]{0.4\textwidth}
\centering
\includegraphics[scale=0.4]{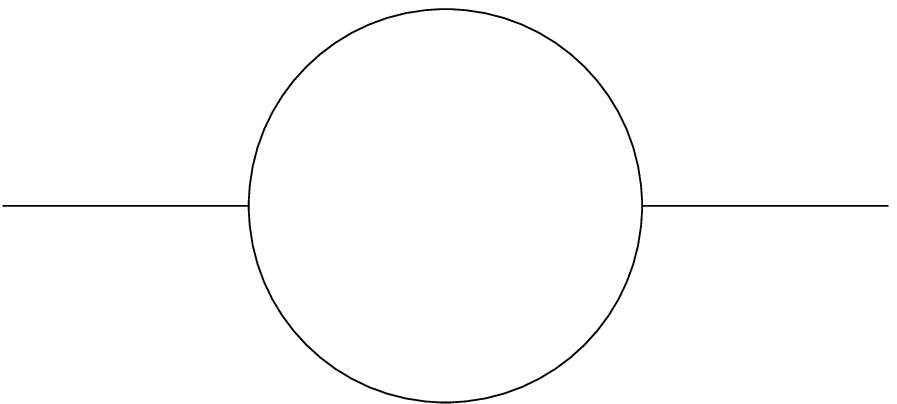}
\caption{Scalar case}
\end{subfigure}
\end{center}
\caption{Vacuum polarization Feynman diagrams.}
\label{vpolar-diag}
\end{figure}

\section{Algebras of Quaternionic Functions}  \label{Alg-section}

In this section we construct $\mathfrak{gl}(2,\HC)$-invariant algebra
structures on $(\rho_1,\Zh)$ and $(\rho'_2, {\cal W}'/\ker\M)$.
The first two subsections deal with the scalar case $(\rho_1,\Zh)$
and the following two subsections deal with the spinor case
$(\rho'_2, {\cal W}'/\ker\M)$.
In the last subsection we conjecture that these algebras have the structures
of weak cyclic $A_{\infty}$-algebras.

\subsection{Scalar Version of the Convolution Algebra}

Recall harmonic polynomial functions on $\HC^{\times}$:
\begin{align*}
  {\cal H}^+ &=   \bigl\{ \phi \in \BB C[z_{11},z_{12},z_{21},z_{22}] ;\:
  \square \phi =0 \bigr\}  \\
  &= \BB C\text{-span of } \bigl\{  t^l_{n \, \underline{m}}(Z) \bigr\},  \\
  {\cal H}^- &= \bigl\{ \phi \in \BB C[z_{11},z_{12},z_{21},z_{22}, N(Z)^{-1}] ;\:
  N(Z)^{-1} \cdot \phi(Z^{-1}) \in {\cal H^+} \bigr\}  \\
  &= \BB C\text{-span of } \bigl\{
  N(Z)^{-1} \cdot t^l_{m \, \underline{n}}(Z^{-1}) \bigr\},  \\
  {\cal H} &= \bigl\{ \phi \in \BB C[z_{11},z_{12},z_{21},z_{22}, N(Z)^{-1}] ;\:
  \square \phi =0 \bigr\}  \\
  &= {\cal H}^+ \oplus {\cal H}^-,
\end{align*}
where $l = 0, \frac12, 1, \frac32, \dots$, $m,n \in \BB Z +l$,
$-l \le m,n \le l$.
The Lie algebra $\mathfrak{gl}(2,\HC)$ acts on ${\cal H}^+$, ${\cal H}^-$ and
${\cal H}$ via two slightly different actions $\pi^0_l$ and $\pi^0_r$
(Subsections 2.4-2.5 in \cite{FL1}).
There is a non-degenerate $\mathfrak{gl}(2,\HC)$-invariant
bilinear pairing between $(\pi^0_l, {\cal H}^+)$ and $(\pi^0_r, {\cal H}^-)$
given by the integral formula
$$
\langle \phi_1, \phi_2 \rangle_{\cal H} = \frac1{2\pi^2} \int_{Z \in SU(2)}
(\degt_Z \phi_1)(Z) \cdot \phi_2(Z) \,dS,
\qquad \phi_1 \in {\cal H}^+, \: \phi_2 \in {\cal H}^-,
$$
(equation (32) in \cite{FL1}).
We extend this pairing to ${\cal H}^- \times {\cal H}^+$ by the same formula
(or antisymmetry), and then declare pairings on ${\cal H}^+ \times {\cal H}^+$
and ${\cal H}^- \times {\cal H}^-$ to be zero (even though the integral
need not be zero in these cases).
Thus we obtain a non-degenerate antisymmetric $\mathfrak{gl}(2,\HC)$-invariant
bilinear pairing between $(\pi^0_l, {\cal H})$ and $(\pi^0_r, {\cal H})$.

We consider ${\cal H} \otimes {\cal H}$.
This space consists of polynomial $\BB C$-valued functions
in two variables $Z_1, Z_2 \in \HC^{\times}$ that are harmonic
with respect to $Z_1$ and $Z_2$.
We define a convolution operation on ${\cal H} \otimes {\cal H}$ as
\begin{equation}  \label{conv-HtensorH}
(f \ast g) (Z_1,Z_2) =
  \langle \phi_2, \phi'_1 \rangle_{\cal H} \cdot \phi_1(Z_1) \otimes \phi'_2(Z_2),
\end{equation}
where $f=\phi_1(Z_1) \otimes \phi_2(Z_2)$,
$g=\phi'_1(Z_1) \otimes \phi'_2(Z_2) \in {\cal H} \otimes {\cal H}$.
This operation gives ${\cal H} \otimes {\cal H}$ the structure of an
associative algebra.
Since the bilinear pairing is $\mathfrak{gl}(2,\HC)$-invariant,
the above convolution product is $\mathfrak{gl}(2,\HC)$-equivariant
with respect to the $\pi^0_l \otimes \pi^0_r$ action on
${\cal H} \otimes {\cal H}$.

Next, we consider a space ${\cal HH}$ consisting of infinite sums
$\sum_i \phi_i(Z_1) \otimes \phi'_i(Z_2)$, where $\phi_i, \phi'_i \in {\cal H}$,
such that
$$
|\operatorname{degree\:of\,} \phi_i + \operatorname{degree\:of\,} \phi'_i|
$$
is bounded.
More precisely, ${\cal HH}$ consists of formal series of the form
\begin{multline} \label{f-series}
  \sum_{\genfrac{}{}{0pt}{}{l,m,n}{l',m',n'}}
  a_{\genfrac{}{}{0pt}{}{l,m,n}{l',m',n'}}
  t^l_{n \, \underline{m}}(Z_1) \cdot t^{l'}_{n' \, \underline{m'}}(Z_2)
  + \sum_{\genfrac{}{}{0pt}{}{l,m,n}{l',m',n'}}
  b_{\genfrac{}{}{0pt}{}{l,m,n}{l',m',n'}}
  t^l_{n \, \underline{m}}(Z_1) \cdot
  N(Z_2)^{-1} \cdot t^{l'}_{m' \, \underline{n'}}(Z_2^{-1})  \\
  + \sum_{\genfrac{}{}{0pt}{}{l,m,n}{l',m',n'}}
  c_{\genfrac{}{}{0pt}{}{l,m,n}{l',m',n'}}
  N(Z_1)^{-1} \cdot t^l_{m \, \underline{n}}(Z_1^{-1}) \cdot
  t^{l'}_{n' \, \underline{m'}}(Z_2)  \\
  + \sum_{\genfrac{}{}{0pt}{}{l,m,n}{l',m',n'}}
  d_{\genfrac{}{}{0pt}{}{l,m,n}{l',m',n'}}
  N(Z_1)^{-1} \cdot t^l_{m \, \underline{n}}(Z_1^{-1}) \cdot
  N(Z_2)^{-1} \cdot t^{l'}_{m' \, \underline{n'}}(Z_2^{-1})
\end{multline}
such that only finitely many coefficients
$a_{\genfrac{}{}{0pt}{}{l,m,n}{l',m',n'}}$'s,
$d_{\genfrac{}{}{0pt}{}{l,m,n}{l',m',n'}}$'s are non-zero and
non-zero coefficients $b_{\genfrac{}{}{0pt}{}{l,m,n}{l',m',n'}}$'s,
$c_{\genfrac{}{}{0pt}{}{l,m,n}{l',m',n'}}$'s have bounded difference of indices
$|l-l'|$.

\begin{lem}  \label{conv-prod-lem}
The convolution operation (\ref {conv-HtensorH}) extends to ${\cal HH}$
and gives it the structure of a $\mathfrak{gl}(2,\HC)$-invariant associative
algebra. Moreover, if $f \in {\cal H} \otimes {\cal H}$ and $g \in {\cal HH}$,
then both $f \ast g$ and $g \ast f$ lie in ${\cal H} \otimes {\cal H}$.
\end{lem}

\begin{proof}
The result follows from an observation that, for a fixed index $l_0$,
there are only finitely many non-zero coefficients
$$
a_{\genfrac{}{}{0pt}{}{l_0,m,n}{l',m',n'}}, \quad
b_{\genfrac{}{}{0pt}{}{l_0,m,n}{l',m',n'}}, \quad
c_{\genfrac{}{}{0pt}{}{l_0,m,n}{l',m',n'}}, \quad
d_{\genfrac{}{}{0pt}{}{l_0,m,n}{l',m',n'}}
$$
with that particular index, and similarly for index $l'_0$.
\end{proof}

Unlike ${\cal H} \otimes {\cal H}$, ${\cal HH}$ has a unit.
The expression for the unit is obtained by formally combining two copies of
matrix coefficient expansions of $N(Z_1-Z_2)^{-1}$
given in Proposition 25 from \cite{FL1} (see also Proposition \ref{Prop25}):
\begin{equation}  \label{1_HH}
1_{\cal HH} = - \sum_{l,m,n}
t^l_{n\,\underline{m}}(Z_1) \cdot N(Z_2)^{-1} \cdot t^l_{m\,\underline{n}}(Z_2^{-1})
+ \sum_{l,m,n} N(Z_1)^{-1} \cdot t^l_{m\,\underline{n}}(Z_1^{-1})
\cdot t^l_{n\,\underline{m}}(Z_2) \quad \in {\cal HH}.
\end{equation}
The fact that $1_{\cal HH}$ is indeed a unit follows from the definition of the
convolution operation and orthogonality relations (17) in \cite{FL3}.

The Lie algebra $\mathfrak{gl}(2,\HC)$ can act on ${\cal H} \otimes {\cal H}$
by at least three different actions:
$\pi^0_l \otimes 1$, $1 \otimes \pi^0_r$ and $\pi^0_l \otimes \pi^0_r$.
Clearly, all three actions extend to ${\cal HH}$.
Then the convolution operation on ${\cal HH}$ is
$(\pi^0_l \otimes \pi^0_r)$-equivariant.


For each $d \in \BB Z$, we define operators
$(\deg+d)^{-1}_{Z_1}$ and $(\deg+d)^{-1}_{Z_2}$ on ${\cal HH}$
as follows. If $f \in {\cal HH}$ is expressed as a series
(\ref{f-series}), then
\begin{multline*}
(\deg+d)^{-1}_{Z_1} f(Z_1,Z_2) = \sum_{\genfrac{}{}{0pt}{}{l>-d/2}{m,n,l',m',n'}}
\frac{a_{\genfrac{}{}{0pt}{}{l,m,n}{l',m',n'}}}{2l+d}
t^l_{n \, \underline{m}}(Z_1) \cdot t^{l'}_{n' \, \underline{m'}}(Z_2)  \\
+ \sum_{\genfrac{}{}{0pt}{}{l>-d/2}{m,n,l',m',n'}}
\frac{b_{\genfrac{}{}{0pt}{}{l,m,n}{l',m',n'}}}{2l+d}
t^l_{n \, \underline{m}}(Z_1) \cdot
N(Z_2)^{-1} \cdot t^{l'}_{m' \, \underline{n'}}(Z_2^{-1})  \\
- \sum_{\genfrac{}{}{0pt}{}{l>d/2-1}{m,n,l',m',n'}}
\frac{c_{\genfrac{}{}{0pt}{}{l,m,n}{l',m',n'}}}{2l+2-d}
N(Z_1)^{-1} \cdot t^l_{m \, \underline{n}}(Z_1^{-1}) \cdot
t^{l'}_{n' \, \underline{m'}}(Z_2)  \\
- \sum_{\genfrac{}{}{0pt}{}{l>d/2-1}{m,n,l',m',n'}}
\frac{d_{\genfrac{}{}{0pt}{}{l,m,n}{l',m',n'}}}{2l+2-d}
N(Z_1)^{-1} \cdot t^l_{m \, \underline{n}}(Z_1^{-1}) \cdot
N(Z_2)^{-1} \cdot t^{l'}_{m' \, \underline{n'}}(Z_2^{-1}).
\end{multline*}
Note that if $d \ne 1$, certain terms get discarded:
if $d>1$, terms
$$
N(Z_1)^{-1} \cdot t^l_{m \, \underline{n}}(Z_1^{-1}) \cdot
t^{l'}_{n' \, \underline{m'}}(Z_2)
\quad \text{and} \quad
N(Z_1)^{-1} \cdot t^l_{m \, \underline{n}}(Z_1^{-1}) \cdot
N(Z_2)^{-1} \cdot t^{l'}_{m' \, \underline{n'}}(Z_2^{-1})
$$
with $0 \le l \le d/2-1$ are discarded; and if $d<1$, terms
$$
t^l_{n \, \underline{m}}(Z_1) \cdot t^{l'}_{n' \, \underline{m'}}(Z_2)
\quad \text{and} \quad
t^l_{n \, \underline{m}}(Z_1) \cdot
N(Z_2)^{-1} \cdot t^{l'}_{m' \, \underline{n'}}(Z_2^{-1})
$$
with $0 \le l \le -d/2$ are discarded.
Then $(\deg+d)^{-1}_{Z_2} f(Z_1,Z_2)$ is defined similarly.

\begin{lem}  \label{deg-inverse-comm-rels}
Let $f(Z_1,Z_2) \in {\cal HH}$, $d \in \BB Z$ and $A, B, C, D \in \HC$.
We have the following commutation relations:
\begin{align*}
(\deg+d)_{Z_1}^{-1} (\pi^0_l \otimes 1)
\bigl(\begin{smallmatrix} A & 0 \\ 0 & D \end{smallmatrix}\bigr) f(Z_1,Z_2)
&= (\pi^0_l \otimes 1) \bigl(\begin{smallmatrix}
  A & 0 \\ 0 & D \end{smallmatrix}\bigr)
(\deg+d)_{Z_1}^{-1} f(Z_1,Z_2), \\
(\deg+d)_{Z_1}^{-1} (\pi^0_l \otimes 1)
\bigl(\begin{smallmatrix} 0 & B \\ 0 & 0 \end{smallmatrix}\bigr) f(Z_1,Z_2)
&= (\pi^0_l \otimes 1) \bigl(\begin{smallmatrix}
  0 & B \\ 0 & 0 \end{smallmatrix}\bigr)
(\deg+d-1)_{Z_1}^{-1} f(Z_1,Z_2), \\
(\deg+d)_{Z_1}^{-1} (\pi^0_l \otimes 1)
\bigl(\begin{smallmatrix} 0 & 0 \\ C & 0 \end{smallmatrix}\bigr) f(Z_1,Z_2)
&= (\pi^0_l \otimes 1) \bigl(\begin{smallmatrix}
  0 & 0 \\ C & 0 \end{smallmatrix}\bigr)
(\deg+d+1)_{Z_1}^{-1} f(Z_1,Z_2),
\end{align*}
and similarly for $(\deg+d)_{Z_2}^{-1}$ and $1 \otimes \pi^0_r$.
\end{lem}

\begin{lem}  \label{deg-inverse-star-rels}
Let $f, g \in {\cal HH}$ and $d \in \BB Z$. We have the following relations:
\begin{align*}
(\deg+d)_{Z_1}^{-1} (f \ast g) &= \bigl( (\deg+d)_{Z_1}^{-1} f \bigr) \ast g,  \\
(\deg+d)_{Z_2}^{-1} (f \ast g) &= f \ast \bigl( (\deg+d)_{Z_2}^{-1} g \bigr),  \\
\bigl( (\deg+d)_{Z_2}^{-1} f \bigr) \ast g &=
  - f \ast \bigl( (\deg+2-d)_{Z_1}^{-1} g \bigr).
\end{align*}
\end{lem}

We define an equivariant map
$$
I=I^{++} - (I^{+-}+I^{-+}) + I^{--} :
(\rho_1, \Zh) \to (\pi^0_l \otimes \pi^0_r, {\cal HH})
$$
as follows.
Recall the maps $I_R$ given by equation (34) in Chapter 6 of \cite{FL3},
their definition is copied here in (\ref{I_R}).
By Lemma 11 and Theorem 12 in \cite{FL3}, if $Z_1, Z_2 \in \BB D^+_R$,
$I_R$ is a $\mathfrak{gl}(2,\HC)$-equivariant map
$(\rho_1,\Zh) \to (\pi^0_l \otimes \pi^0_r, {\cal H}^+ \otimes {\cal H}^+)$
independent of the choice of $R>0$; we call this map $I^{++}$.
Similarly, if $Z_1, Z_2 \in \BB D^-_R$, $I_R$ is a
$\mathfrak{gl}(2,\HC)$-equivariant map
$(\rho_1,\Zh) \to (\pi^0_l \otimes \pi^0_r, {\cal H}^- \otimes {\cal H}^-)$
also independent of the choice of $R>0$; we call this map $I^{--}$.
If $f \in \Zh$,
$$
I^{+-}(f) = \sum_{\genfrac{}{}{0pt}{}{l,m,n}{l',m',n'}}
b(f)_{\genfrac{}{}{0pt}{}{l,m,n}{l',m',n'}}
t^l_{n \, \underline{m}}(Z_1) \cdot
N(Z_2)^{-1} \cdot t^{l'}_{m' \, \underline{n'}}(Z_2^{-1}),
$$
where
$$
b(f)_{\genfrac{}{}{0pt}{}{l,m,n}{l',m',n'}} =
\frac{i}{2\pi^3} \int_{U(2)_R} f(W) \cdot N(W)^{-1} \cdot
t^l_{m\,\underline{n}}(W^{-1}) \cdot t^{l'}_{n'\,\underline{m'}}(W) \,dV.
$$
On the one hand, this integral does not depend on $R>0$.
On the other hand, by the matrix coefficient expansion of $N(Z-W)^{-1}$
given in Proposition 25 from \cite{FL1} (see also Proposition \ref{Prop25}),
for each $R>0$, the series $I^{+-}(f)$ converges to $(I_Rf)(Z_1,Z_2)$
whenever $Z_1 \in \BB D^+_R$ and $Z_2 \in \BB D^-_R$.
Similarly,
$$
I^{-+}(f) = \sum_{\genfrac{}{}{0pt}{}{l,m,n}{l',m',n'}}
c(f)_{\genfrac{}{}{0pt}{}{l,m,n}{l',m',n'}}
N(Z_1)^{-1} \cdot t^l_{m \, \underline{n}}(Z_1^{-1}) \cdot
t^{l'}_{n' \, \underline{m'}}(Z_2),
$$
where
$$
c(f)_{\genfrac{}{}{0pt}{}{l,m,n}{l',m',n'}} =
b(f)_{\genfrac{}{}{0pt}{}{l',m',n'}{l,m,n}} =
\frac{i}{2\pi^3} \int_{U(2)_R} f(W) \cdot t^l_{n \, \underline{m}}(W) \cdot
N(W)^{-1} \cdot t^{l'}_{m' \, \underline{n'}}(W^{-1}) \,dV.
$$
This integral is independent of $R>0$ and, for each $R>0$,
the series $I^{-+}(f)$ converges to $(I_Rf)(Z_1,Z_2)$ whenever
$Z_1 \in \BB D^-_R$ and $Z_2 \in \BB D^+_R$.

\begin{prop}  \label{gen-mult-prop}
We have:
  $$
  I(N(Z)^{-1}) = (\deg+1)^{-1}_{Z_1} 1_{\cal HH}
  = - (\deg+1)^{-1}_{Z_2} 1_{\cal HH}.
  $$
  In particular, for any $g \in {\cal HH}$,
  $$
  I(N(Z)^{-1}) \ast g = (\deg+1)^{-1}_{Z_1} g, \qquad
  g \ast I(N(Z)^{-1}) = -(\deg+1)^{-1}_{Z_2} g.
  $$
\end{prop}

\begin{proof}
From Section 6 in \cite{FL3}, we have:
\begin{multline*}
I(N(Z)^{-1})(Z_1,Z_2)  \\
= - \sum_{l,m,n} \frac{N(Z_2)^{-1}}{2l+1} t^l_{n \, \underline{m}}(Z_1) \cdot
t^l_{m \, \underline{n}}(Z_2^{-1})
- \sum_{l,m,n} \frac{N(Z_1)^{-1}}{2l+1} t^l_{m \, \underline{n}}(Z_1^{-1}) \cdot
t^l_{n \, \underline{m}}(Z_2).
\end{multline*}
Then the result follows from (\ref{1_HH}) and
Lemma \ref{deg-inverse-star-rels}.
\end{proof}

\subsection{Scalar Version of the Algebra of Quaternionic Functions}  \label{scalar-algebra-subsection}

In this subsection we give $(\rho_1,\Zh)$ the structure of a
$\mathfrak{gl}(2,\HC)$-invariant algebra.

\begin{df}
  Let ${\cal HH}^{\omega}$ denote the subspace of ${\cal HH}$ generated by
  ${\cal H} \otimes {\cal H}$, $I(\Zh)$, application of operators
  $(\deg+d)^{-1}_{Z_1}$ and $(\deg+d)^{-1}_{Z_2}$, $d \in \BB Z$,
  as well as actions $\pi^0_l \otimes 1$ and $1 \otimes \pi^0_r$ of
  $\mathfrak{gl}(2,\HC)$.
\end{df}


Thus, by definition, ${\cal HH}^{\omega}$ is invariant under the
$\pi^0_l \otimes \pi^0_r$ action of $\mathfrak{gl}(2,\HC)$.
We want to reduce the number of generators of ${\cal HH}^{\omega}$.

\begin{lem}  \label{HH-omega-generators}
  The space ${\cal HH}^{\omega}$ is generated by ${\cal H} \otimes {\cal H}$,
  elements of the type
  $$
  I(N(Z)^{-1}), \quad
  (\deg+d_1)^{-1}_{Z_1} I(N(Z)^{-1}), \quad
  (\deg+d_1)^{-1}_{Z_2} I(N(Z)^{-1}),
  $$
  $$
  (\deg+d_2)^{-1}_{Z_1} (\deg+d_1)^{-1}_{Z_1} I(N(Z)^{-1}), \quad
  (\deg+d_2)^{-1}_{Z_1} (\deg+d_1)^{-1}_{Z_2} I(N(Z)^{-1}),
  $$
  $$
  (\deg+d_2)^{-1}_{Z_2} (\deg+d_1)^{-1}_{Z_2} I(N(Z)^{-1}), \quad
  (\deg+d_3)^{-1}_{Z_1} (\deg+d_2)^{-1}_{Z_1}(\deg+d_1)^{-1}_{Z_1} I(N(Z)^{-1}),\dots
  $$
  as well as actions $\pi^0_l \otimes 1$ and $1 \otimes \pi^0_r$ of
  $\mathfrak{gl}(2,\HC)$.
\end{lem}

\begin{proof}
Since $\Zh = \Zh^- \oplus \Zh^0 \oplus \Zh^+$,
$I(\Zh^- \oplus \Zh^+) \subset {\cal H} \otimes {\cal H}$
and $\Zh^0$ is generated by $N(Z)^{-1}$,
${\cal HH}^{\omega}$ can be generated by ${\cal H} \otimes {\cal H}$ and
$I(N(Z)^{-1})$ instead of ${\cal H} \otimes {\cal H}$ and $I(\Zh)$.
Notice that applying $(\deg+d)^{-1}_{Z_1}$ or $(\deg+d)^{-1}_{Z_2}$ to an element
of ${\cal H} \otimes {\cal H}$ results in another element of
${\cal H} \otimes {\cal H}$.
Then the result follows from Lemma \ref{deg-inverse-comm-rels}.
\end{proof}

\begin{prop}  \label{HH-mult-closure-prop}
  The space ${\cal HH}^{\omega}$ is closed under the convolution operation:
  if $f, g \in {\cal HH}^{\omega}$, then $f \ast g$ also lies in
  ${\cal HH}^{\omega}$.
\end{prop}

\begin{proof}
  First, observe that if $f$ or $g \in {\cal H} \otimes {\cal H}$,
  then, by Lemma \ref{conv-prod-lem},
  $$
  f \ast g \in {\cal H} \otimes {\cal H} \subset {\cal HH}^{\omega}.
  $$
  Since the convolution operation is $(\pi^0_l \otimes \pi^0_r)$-equivariant,
  $$
  \bigl( (\pi^0_l \otimes \pi^0_r)(X)f \bigr) \ast g =
  (\pi^0_l \otimes \pi^0_r)(X) (f \ast g)
  - f \ast \bigl( (\pi^0_l \otimes \pi^0_r)(X)g \bigr),
  $$
  for all $X \in \mathfrak{gl}(2,\HC)$.
  And by Lemmas \ref{deg-inverse-star-rels} and \ref{HH-omega-generators},
  it is sufficient to prove the proposition for the case $f$ is one of the
  generators of ${\cal HH}^{\omega}$ of the form $I(N(Z)^{-1})$ with operators
  $(\deg+d_i)^{-1}_{Z_1}$ and $(\deg+d_j)^{-1}_{Z_2}$ applied several times.
  Then Lemma \ref{deg-inverse-star-rels} reduces this further to the case
  $f=I(N(Z)^{-1})$, and the result follows from
  Proposition \ref{gen-mult-prop}.
\end{proof}

Next, we realize elements of ${\cal HH}^{\omega}$ as analytic functions as
opposed to formal series.
For $Z_1, Z_2 \in \HC^{\times}$, let $\lambda_1$ and $\lambda_2$ denote the
eigenvalues of $Z_1Z_2^{-1}$, and introduce open subsets of
$\HC^{\times} \times \HC^{\times}$
\begin{align*}
\Omega &= \{ (Z_1,Z_2) \in \HC^{\times} \times \HC^{\times} ;\:
\text{neither $\lambda_1$ nor $\lambda_2$ is a positive real number} \}, \\
\Lambda^+ &= \{ (Z_1,Z_2) \in \HC^{\times} \times \HC^{\times} ;\:
\sgn (\im \lambda_1)>0,\: \sgn (\im \lambda_2)>0 \},  \\
\Lambda^- &= \{ (Z_1,Z_2) \in \HC^{\times} \times \HC^{\times} ;\:
\sgn (\im \lambda_1)<0,\: \sgn (\im \lambda_2)<0 \}.
\end{align*}
Then $\Lambda^+, \Lambda^- \subset \Omega$.

\begin{lem}
  The elements of ${\cal HH}^{\omega}$ can be realized as analytic functions
  on $\Omega$.
\end{lem}

\begin{proof}
Elements of ${\cal H} \otimes {\cal H}$ are polynomials, hence can be treated
as analytic functions on $\Omega$.
On the other hand,
$$
I(N(Z)^{-1}) = -I^{+-}(N(Z)^{-1}) -I^{-+}(N(Z)^{-1}).
$$
As was shown in Section 6 of \cite{FL3},
the series $I^{+-}(N(Z)^{-1})$ converges on $\BB D^+_R \times \BB D^-_R$
and extends analytically to $\Omega$.
Similarly, the series $I^{-+}(N(Z)^{-1})$ converges on
$\BB D^-_R \times \BB D^+_R$ and also extends analytically to $\Omega$.
Just as we expressed $(\deg+2)^{-1}$ as an integral operator in
Subsection \ref{deg+2-inverse-subsection}, integral expressions can be found
for $(\deg+d)^{-1}_{Z_1}$ and $(\deg+d)^{-1}_{Z_2}$, $d \in \BB Z$.
This shows that the result of application of such operators to 
$I^{+-}(N(Z)^{-1})$ and $I^{-+}(N(Z)^{-1})$ is analytic on $\Omega$ as well.
Finally, the actions $\pi^0_l \otimes 1$ and $1 \otimes \pi^0_r$ of
$\mathfrak{gl}(2,\HC)$ preserve analyticity of functions.
\end{proof}

Each $f \in {\cal HH}^{\omega}$ is harmonic with respect to $Z_1$ and $Z_2$:
$$
\square_{Z_1} f(Z_1,Z_2) = 0 = \square_{Z_2} f(Z_1,Z_2).
$$
We have the following inclusions:
$$
{\cal H} \otimes {\cal H} \subset {\cal HH}^{\omega} \subset {\cal HH}.
$$
Thus, we can think of ${\cal HH}^{\omega}$ as a completion of
${\cal H} \otimes {\cal H}$.

\begin{df}  \label{extendable-scalar-def}
  We call a function $f \in {\cal HH}^{\omega}$ {\em extendable} if,
  for each $Z_0 \in \HC^{\times}$, there exists an open neighborhood
  $V \subset \HC^{\times} \times \HC^{\times}$ of $(Z_0,Z_0)$ and two
  functions $f^+$ and $f^-$ analytic on
  $V$ such that $f = f^+$ for all points in $\Lambda^+ \cap V$ and
  $f = f^-$ for all points in $\Lambda^- \cap V$.
\end{df}

Informally, extendable functions are functions $f$ whose restrictions
$f|_{\Lambda^+}$ and $f|_{\Lambda^-}$ extend analytically across the diagonal in
$\HC^{\times} \times \HC^{\times}$.
Clearly, elements of ${\cal H} \otimes {\cal H}$ are extendable.
Also, the following observation is obvious, but will be used in the future.

\begin{lem}  \label{extendable-invariant-lem}
  The extendable functions in ${\cal HH}^{\omega}$ form a subspace that is
  invariant under the actions $\pi^0_l \otimes 1$, $1 \otimes \pi^0_r$ and
  $\pi^0_l \otimes \pi^0_r$ of $\mathfrak{gl}(2,\HC)$.
\end{lem}

\begin{rem}
We expect all functions in ${\cal HH}^{\omega}$ to be extendable.
\end{rem}

\begin{lem}  \label{I(Zh)-extendable}
  For each $f \in \Zh$, $I(f)$ is extendable and can be written as a
  finite linear combination of analytic functions on $\Omega$ that are
  homogeneous in $Z_1$ and $Z_2$.
\end{lem}

\begin{proof}
Recall that $\Zh = \Zh^- \oplus \Zh^0 \oplus \Zh^+$.
If $f \in I(\Zh^- \oplus \Zh^+)$, then
$I(f) \in {\cal H} \otimes {\cal H}$ is extendable
and can be written as a finite linear combination of homogeneous functions.
On the other hand, expression (\ref{I(N(W)^{-1})}) shows that
$I(N(Z)^{-1})$ is extendable and homogeneous in $Z_1$ and $Z_2$
(of degree $-1$ in each variable).
Since $N(Z)^{-1}$ generates $\Zh^0$, it follows from
Lemma \ref{extendable-invariant-lem} that the result is true for
all $f \in \Zh^0$.
\end{proof}

Let $f$ be an extendable function. Even though it may be singular along
the diagonal, we still can construct two analytic functions on $\HC^{\times}$:
\begin{align*}
(\DR^+ f)(Z) &=
\lim_{\genfrac{}{}{0pt}{}{Z_1, Z_2 \to Z}{(Z_1,Z_2) \in \Lambda^+}} f(Z_1,Z_2)
\qquad \text{and}  \\
(\DR^- f)(Z) &=
\lim_{\genfrac{}{}{0pt}{}{Z_1, Z_2 \to Z}{(Z_1,Z_2) \in \Lambda^-}} f(Z_1,Z_2).
\end{align*}
In some cases, applying $\DR^+$ and $\DR^-$ may yield different results,
we will see a concrete example in Proposition \ref{n-mult-scalar}.
Note also that functions $\DR^+ f$ and $\DR^- f$ need not be elements of $\Zh$
because they may not be polynomials on $\HC^{\times}$.
Nevertheless, the operators $\DR^+$ and $\DR^-$ intertwine the
$\mathfrak{gl}(2,\HC)$-actions $\pi^0_l \otimes \pi^0_r$ on extendable
functions in ${\cal HH}^{\omega}$ and $\rho_1$ on $\BB C$-valued
analytic functions on $\HC^{\times}$.

\begin{lem}  \label{Diag_on_HxH-lem}
  If $f(Z_1,Z_2) \in {\cal H} \otimes {\cal H}$, then
  $$
  \DR^+ f = \DR^- f \quad \in \Zh.
  $$
\end{lem}

\begin{proof}
  Note that, when restricted to ${\cal H} \otimes {\cal H}$, both $\DR^+$ and
  $\DR^-$ reduce to the multiplication map ${\cal H} \otimes {\cal H} \to \Zh$. 
\end{proof}

\begin{thm}  \label{mult-extendable-thm}
  Let $f, g \in \Zh$, then $I(f) \ast I(g)$ is extendable and
  \begin{equation}  \label{rest_in_Zh-eqn}
  \DR^+ \bigl( I(f) \ast I(g) \bigr), \quad
  \DR^- \bigl( I(f) \ast I(g) \bigr) \quad \in \Zh.
  \end{equation}
\end{thm}

\begin{proof}
Recall that $\Zh = \Zh^- \oplus \Zh^0 \oplus \Zh^+$.
If $f$ or $g \in \Zh^- \oplus \Zh^+$, then
$I(f)$ or $I(g) \in {\cal H} \otimes {\cal H}$.
Then, by Lemma \ref{conv-prod-lem},
$I(f) \ast I(g) \in {\cal H} \otimes {\cal H}$, hence $I(f) \ast I(g)$
is extendable and (\ref{rest_in_Zh-eqn}) is true.

It remains to consider the case $f, g \in \Zh^0$.
Since $\Zh^0$ is irreducible and generated by $N(Z)^{-1}$,
the result follows from Proposition \ref{gen-mult-prop} and
Lemmas \ref{deg-inverse-comm-rels}, \ref{extendable-invariant-lem} and
\ref{dilogarithm-calculations-lem}.
\end{proof}

\begin{lem}  \label{dilogarithm-calculations-lem}
For each $d \in \BB Z$, $(\deg+d)^{-1}_{Z_1} I(N(Z)^{-1})$ and
$(\deg+d)^{-1}_{Z_2} I(N(Z)^{-1})$ are extendable. Furhermore,
$$
\DR^+ \bigl( (\deg+d)^{-1}_{Z_1} I(N(Z)^{-1}) \bigr), \quad
\DR^- \bigl( (\deg+d)^{-1}_{Z_1} I(N(Z)^{-1}) \bigr),
$$
$$
\DR^+ \bigl( (\deg+d)^{-1}_{Z_2} I(N(Z)^{-1}) \bigr), \quad
\DR^- \bigl( (\deg+d)^{-1}_{Z_2} I(N(Z)^{-1}) \bigr)
$$
are proportional to $N(Z)^{-1}$, and hence elements of $\Zh$.
\end{lem}

\begin{proof}
First, we consider the case $d=1$. Recall the dilogarithm function:
$$
\operatorname{Li}_2(z) = \sum_{k=1}^{\infty} \frac{z^k}{k^2},
$$
it extends to an analytic function on $\BB C$ with a cut along $[1,\infty)$.
It also satisfies a well known identity
$$
\operatorname{Li}_2(z) + \operatorname{Li}_2(1/z)
= -\frac{\pi^2}6 - \frac12 \log^2(-z)
$$
(see, for example, \cite{Z}).
As we saw in the proof of Proposition \ref{gen-mult-prop},
\begin{align*}
I^{+-}(N(Z)^{-1})(Z_1,Z_2)
&= \sum_{l,m,n} \frac{N(Z_2)^{-1}}{2l+1} t^l_{n \, \underline{m}}(Z_1) \cdot
t^l_{m \, \underline{n}}(Z_2^{-1}),  \\
I^{-+}(N(Z)^{-1})(Z_1,Z_2)
&= \sum_{l,m,n} \frac{N(Z_1)^{-1}}{(2l+1)^2} t^l_{m \, \underline{n}}(Z_1^{-1})
\cdot t^l_{n \, \underline{m}}(Z_2).
\end{align*}
We have:
\begin{multline*}
  (\deg+1)^{-1}_{Z_1} I^{+-}(N(Z)^{-1})(Z_1,Z_2)
  = -(\deg+1)^{-1}_{Z_2} I^{+-}(N(Z)^{-1})(Z_1,Z_2)  \\
  = \sum_{l,m,n} \frac{N(Z_2)^{-1}}{(2l+1)^2} t^l_{n \, \underline{m}}(Z_1) \cdot
  t^l_{m \, \underline{n}}(Z_2^{-1}).
\end{multline*}
Recall that $\lambda_1$ and $\lambda_2$ are the eigenvalues of $Z_1 Z_2^{-1}$.
Then, following calculations (36) from \cite{FL3},
\begin{multline*}
 (\deg+1)^{-1}_{Z_1} I^{+-}(N(Z)^{-1})(Z_1,Z_2)
  = -(\deg+1)^{-1}_{Z_2} I^{+-}(N(Z)^{-1})(Z_1,Z_2)  \\ 
  = \sum_l \frac{N(Z_2)^{-1}}{(2l+1)^2}
  \frac{\lambda_1^{2l+1}-\lambda_2^{2l+1}}{\lambda_1-\lambda_2}
  = \frac1{N(Z_2)} \frac{\operatorname{Li}_2(\lambda_1)
    - \operatorname{Li}_2(\lambda_2)}{\lambda_1-\lambda_2}.
\end{multline*}
Similarly,
\begin{multline*}
  (\deg+1)^{-1}_{Z_1} I^{-+}(N(Z)^{-1})(Z_1,Z_2)
  = - (\deg+1)^{-1}_{Z_2} I^{-+}(N(Z)^{-1})(Z_1,Z_2)  \\
  = - \sum_{l,m,n} \frac{N(Z_1)^{-1}}{(2l+1)^2} t^l_{m \, \underline{n}}(Z_1^{-1})
  \cdot t^l_{n \, \underline{m}}(Z_2)
  = - \sum_l \frac{N(Z_1)^{-1}}{(2l+1)^2}
  \frac{\lambda_1^{-(2l+1)}-\lambda_2^{-(2l+1)}}{\lambda_1^{-1}-\lambda_2^{-1}}  \\
  = - \frac1{N(Z_1)} \frac{\operatorname{Li}_2(\lambda_1^{-1})
    - \operatorname{Li}_2(\lambda_2^{-1})}{\lambda_1^{-1}-\lambda_2^{-1}}
  = \frac1{N(Z_2)} \frac{\operatorname{Li}_2(\lambda_1^{-1})
    - \operatorname{Li}_2(\lambda_2^{-1})}{\lambda_1-\lambda_2}
\end{multline*}
(recall that $N(Z_1) = \lambda_1 \lambda_2 \cdot N(Z_2)$).
Thus,
\begin{multline}  \label{(deg+1)-inverse}
(\deg+1)^{-1}_{Z_1} I(N(Z)^{-1})(Z_1,Z_2)
= -(\deg+1)^{-1}_{Z_2} I(N(Z)^{-1})(Z_1,Z_2)  \\
= - \frac1{N(Z_2)} \frac{\operatorname{Li}_2(\lambda_1)
+ \operatorname{Li}_2(\lambda_1^{-1}) - \operatorname{Li}_2(\lambda_2)
- \operatorname{Li}_2(\lambda_2^{-1})}{\lambda_1-\lambda_2}  \\
= \frac1{N(Z_2)} \frac{\log^2(-\lambda_1) - \log^2(-\lambda_2)}
{2(\lambda_1-\lambda_2)}.
\end{multline}
If $\lambda_1=\lambda_2=\lambda \ne 1$, we get
$$
\frac1{N(Z_2)} \frac{\log(-\lambda)}{\lambda}.
$$
This shows that $(\deg+1)^{-1}_{Z_1} I(N(Z)^{-1})$ and
$(\deg+1)^{-1}_{Z_2} I(N(Z)^{-1})$ are extendable and
applying $\DR^+$ or $\DR^-$ results in scalar multiples of $N(Z)^{-1}$.

Now, suppose that $d>1$. Consider a power series
$$
p(z) = \sum_{k=1}^{\infty} \frac{z^k}{k(k+d-1)}.
$$
Then
$$
\bigl( z^{d-1} p(z) \bigr)' = \sum_{k=1}^{\infty} \frac{z^{k+d-2}}k
= - z^{d-2} \log(1-z),
$$
and
\begin{multline*}
p(z) = - z^{1-d} \int_0^z t^{d-2} \log(1-t) \,dt
= - \frac{z^{1-d}}{d-1} \biggl( z^{d-1} \log(1-z)
- \int_0^z \frac{t^{d-1}}{t-1} \,dt \biggr)  \\
=  \frac1{d-1} (z^{1-d}-1) \log(1-z)
+ \frac1{d-1} \sum_{k=1}^{d-1} \frac{z^{k-d+1}}{k}.
\end{multline*}
Similarly, consider another power series
$$
q(z) = \sum_{k=d}^{\infty} \frac{z^k}{k(k+1-d)}.
$$
Then
$$
q'(z) = \sum_{k=d}^{\infty} \frac{z^{k-1}}{k+1-d}
= - z^{d-2} \log(1-z),
$$
and
$$
q(z) = \frac1{d-1} (1-z^{d-1}) \log(1-z)
+ \frac1{d-1} \sum_{k=1}^{d-1} \frac{z^k}{k}.
$$

We have:
\begin{multline*}
  (\deg+d)^{-1}_{Z_1} I^{+-}(N(Z)^{-1})(Z_1,Z_2)
  = \sum_{l,m,n} \frac{N(Z_2)^{-1}}{(2l+1)(2l+d)} t^l_{n \, \underline{m}}(Z_1)
  \cdot t^l_{m \, \underline{n}}(Z_2^{-1})  \\
  = \sum_l \frac{N(Z_2)^{-1}}{(2l+1)(2l+d)}
  \frac{\lambda_1^{2l+1}-\lambda_2^{2l+1}}{\lambda_1-\lambda_2}
  = \frac1{N(Z_2)} \frac{ p(\lambda_1) - p(\lambda_2)}{\lambda_1-\lambda_2}.
\end{multline*}
Similarly,
\begin{multline*}
  (\deg+d)^{-1}_{Z_1} I^{-+}(N(Z)^{-1})(Z_1,Z_2)
  = - \sum_{l,m,n} \frac{N(Z_1)^{-1}}{(2l+1)(2l+2-d)}
  t^l_{m \, \underline{n}}(Z_1^{-1}) \cdot t^l_{n \, \underline{m}}(Z_2)  \\
  = - \sum_{l>d/2-1} \frac{N(Z_1)^{-1}}{(2l+1)(2l+2-d)}
  \frac{\lambda_1^{-(2l+1)}-\lambda_2^{-(2l+1)}}{\lambda_1^{-1}-\lambda_2^{-1}} \\
  = - \frac1{N(Z_1)} \frac{ q(\lambda_1^{-1}) - q(\lambda_2^{-1})}
  {\lambda_1^{-1}-\lambda_2^{-1}}
   = \frac1{N(Z_2)} \frac{ q(\lambda_1^{-1}) - q(\lambda_2^{-1})}
  {\lambda_1-\lambda_2}
\end{multline*}
(recall that $N(Z_1) = \lambda_1 \lambda_2 \cdot N(Z_2)$).
Using
$$
\log(1-\lambda) - \log(1-\lambda^{-1}) = \log(-\lambda),
$$
we obtain
\begin{multline} \label{deg_Z1}
  (\deg+d)^{-1}_{Z_1} I(N(Z)^{-1})(Z_1,Z_2)
  = - \frac1{N(Z_2)} \frac{p(\lambda_1)+ q(\lambda_1^{-1})
    - p(\lambda_2)- q(\lambda_2^{-1})}{\lambda_1-\lambda_2}  \\
  = \frac1{N(Z_2)} \frac{(1-\lambda_1^{1-d})\log(-\lambda_1)
    - (1-\lambda_2^{1-d})\log(-\lambda_2)}{(d-1)(\lambda_1-\lambda_2)}  \\
  - \frac{N(Z_2)^{-1}}{d-1} \sum_{k=1}^{d-1}
  \frac{\lambda_1^{k-d+1} + \lambda_1^{-k} - \lambda_2^{k-d+1} - \lambda_2^{-k}}
       {k(\lambda_1-\lambda_2)}.
\end{multline}

In the same fashion we also obtain:
\begin{multline} \label{deg_Z2}
(\deg+d)^{-1}_{Z_2} I(N(Z)^{-1})(Z_1,Z_2)
  = - \frac1{N(Z_2)} \frac{q(\lambda_1)+ p(\lambda_1^{-1})
    - q(\lambda_2)- p(\lambda_2^{-1})}{\lambda_1-\lambda_2}  \\
  = \frac1{N(Z_2)} \frac{(\lambda_1^{d-1}-1)\log(-\lambda_1)
    - (\lambda_2^{d-1}-1)\log(-\lambda_2)}{(d-1)(\lambda_1-\lambda_2)}  \\
  - \frac{N(Z_2)^{-1}}{d-1} \sum_{k=1}^{d-1}
  \frac{\lambda_1^{k} + \lambda_1^{d-k-1} - \lambda_2^{k} - \lambda_2^{d-k-1}}
       {k(\lambda_1-\lambda_2)}.
\end{multline}
Then equations (\ref{deg_Z1}) and (\ref{deg_Z2}) show that
$(\deg+d)^{-1}_{Z_1} I(N(Z)^{-1})$ and
$(\deg+d)^{-1}_{Z_2} I(N(Z)^{-1})(Z_1,Z_2)$ are extendable and
applying $\DR^+$ or $\DR^-$ results in scalar multiples of $N(Z)^{-1}$.
The case $d<1$ is similar.
\end{proof}

Theorem \ref{mult-extendable-thm} allows us to define two
$\mathfrak{gl}(2,\HC)$-invariant multiplication operations on $\Zh$ as follows.

\begin{df}
  Let $f, g \in \Zh$, define
  \begin{align*}
    f \ast^+ g &= \DR^+ \bigl( I(f) \ast I(g) \bigr), \\
    f \ast^- g &= \DR^- \bigl( I(f) \ast I(g) \bigr).
  \end{align*}
\end{df}

\begin{lem}  \label{2-mult-lem}
The two multiplication operations are related to each other as follows:
\begin{align*}
&f \ast^+ g = f \ast^- g \qquad
\text{if $f$ or $g$ is in $\Zh^- \oplus \Zh^+$},  \\
&N(Z)^{-1} \ast^+ N(Z)^{-1} = - N(Z)^{-1} \ast^- N(Z)^{-1}
= - \pi i \cdot N(Z)^{-1}.
\end{align*}
\end{lem}

\begin{proof}
If $f$ or $g \in \Zh^- \oplus \Zh^+$, then
$I(f)$ or $I(g) \in {\cal H} \otimes {\cal H}$.
Then, by Lemma \ref{conv-prod-lem},
$I(f) \ast I(g) \in {\cal H} \otimes {\cal H}$ and,
by Lemma \ref{Diag_on_HxH-lem},
$$
\DR^+ \bigl( I(f) \ast I(g) \bigr) = \DR^- \bigl( I(f) \ast I(g) \bigr).
$$
On the other hand, Proposition \ref{gen-mult-prop} and our previous
calculations (\ref{(deg+1)-inverse}) show that
\begin{align*}
\DR^+ \bigl( IN(Z)^{-1} \ast IN(Z)^{-1} \bigr) &=
\lim_{\genfrac{}{}{0pt}{}{Z_1, Z_2 \to Z}{(Z_1,Z_2) \in \Lambda^+}}
(\deg+1)^{-1}_{Z_1} I(N(Z)^{-1})(Z_1,Z_2)
= - \frac{\pi i}{N(Z)}, \\
\DR^- \bigl( IN(Z)^{-1} \ast IN(Z)^{-1} \bigr) &=
\lim_{\genfrac{}{}{0pt}{}{Z_1, Z_2 \to Z}{(Z_1,Z_2) \in \Lambda^-}}
(\deg+1)^{-1}_{Z_1} I(N(Z)^{-1})(Z_1,Z_2)
= \frac{\pi i}{N(Z)}.
\end{align*}
\end{proof}

\begin{prop}
  One can also consider multiplications obtained by taking linear combinations
  of $f \ast^+ g$ and $f \ast^- g$.
  Thus we obtain a one-parameter family of $\mathfrak{gl}(2,\HC)$-invariant
  multiplications on $\Zh$.
\end{prop}

\begin{ex}  \label{non-associativity-ex}
In this example we show that the multiplication operations $\ast^+$ and
$\ast^-$ on $\Zh$ are {\em not} associative.
\end{ex}

\begin{proof}
Consider elements $f_1(Z)=N(Z)^{-1}$, $f_2(Z)=1$ and $f_3(Z)=N(Z)^{-2}$ of $\Zh$.
Then
$$
I(f_2) = 1, \qquad I(f_3)=N(Z_1)^{-1} \cdot N(Z_2)^{-1}.
$$
By Proposition \ref{gen-mult-prop},
$$
I(f_1) \ast I(f_2) =1, \qquad I(f_2) \ast I(f_3) = N(Z_2)^{-1},
\qquad I(f_1) \ast I(f_2) \ast I(f_3) = N(Z_2)^{-1}.
$$
Hence
$$
f_1 \ast^{\pm} f_2 = 1, \qquad f_2 \ast^{\pm} f_3 = N(Z)^{-1}, \qquad
(f_1 \ast^{\pm} f_2) \ast^{\pm} f_3 = N(Z)^{-1}.
$$
On the other hand, by Lemma \ref{2-mult-lem},
$$
f_1 \ast^+ (f_2 \ast^{\pm} f_3) = -\pi i \cdot N(Z)^{-1}, \qquad
f_1 \ast^- (f_2 \ast^{\pm} f_3) = \pi i \cdot N(Z)^{-1}.
$$
We also have
$$
\DR^+ \bigl( I(f_1) \ast I(f_2) \ast I(f_3) \bigr)
= \DR^- \bigl( I(f_1) \ast I(f_2) \ast I(f_3) \bigr)
= N(Z)^{-1}.
$$
\end{proof}

\begin{prop}  \label{n-mult-scalar}
  For each positive integer $n$,
  $$
  \underbrace{I(N(Z)^{-1}) \ast \dots \ast I(N(Z)^{-1})}_{\text{$n$ times}}
  \quad \in {\cal HH}^{\omega}
  $$
  is extendable and
  \begin{align*}
  \DR^+ \bigl(
  \underbrace{I(N(Z)^{-1}) \ast \dots \ast I(N(Z)^{-1})}_{\text{$n$ times}} \bigr)
  &= \kappa_n^+ \cdot N(Z)^{-1},  \\
  \DR^- \bigl(
  \underbrace{I(N(Z)^{-1}) \ast \dots \ast I(N(Z)^{-1})}_{\text{$n$ times}} \bigr)
  &= \kappa_n^- \cdot N(Z)^{-1},
  \end{align*}
  where
  $$
  \kappa_n^+ = -(2\pi i)^{n-1} \frac{B_{n-1}(1)}{(n-1)!}, \qquad
  \kappa_n^- = -(2\pi i)^{n-1} \frac{B_{n-1}(0)}{(n-1)!}.
  $$
\end{prop}

\begin{proof}
Recall the polylogarithm function:
$$
\operatorname{Li}_n(z) = \sum_{k=1}^{\infty} \frac{z^k}{k^n},
$$
it extends to an analytic function on $\BB C$ with a cut along $[1,\infty)$.
It also satisfies an identity
$$
\operatorname{Li}_n(z) + (-1)^n \operatorname{Li}_n(1/z)
= - \frac{(2\pi i)^n}{n!} B_n \Bigl( \frac12 - \frac{\log(-z)}{2\pi i} \Bigr),
\qquad z \notin (0,\infty),
$$
where $B_n$ are the Bernoulli polynomials
(equation 1.11(18) in \cite{Ba} as well as its correction on Wikipedia's
Polylogarithm page).

By Proposition \ref{gen-mult-prop}, we have:
\begin{multline*}
  \underbrace{I(N(Z)^{-1}) \ast \dots \ast I(N(Z)^{-1})}_{\text{$n-1$ times}}
  \ast I^{+-}(N(Z)^{-1}) \\
  = \underbrace{ (\deg+1)^{-1}_{Z_1} \dots (\deg+1)^{-1}_{Z_1}}_{\text{$n-1$ times}}
  \sum_{l,m,n} \frac{N(Z_2)^{-1}}{2l+1} t^l_{n \, \underline{m}}(Z_1) \cdot
  t^l_{m \, \underline{n}}(Z_2^{-1})  \\
  = \sum_l \frac{N(Z_2)^{-1}}{(2l+1)^n}
  \frac{\lambda_1^{2l+1}-\lambda_2^{2l+1}}{\lambda_1-\lambda_2}
  = \frac1{N(Z_2)} \frac{\operatorname{Li}_n(\lambda_1)
    - \operatorname{Li}_n(\lambda_2)}{\lambda_1-\lambda_2}.
\end{multline*}
Similarly,
\begin{multline*}
  \underbrace{I(N(Z)^{-1}) \ast \dots \ast I(N(Z)^{-1})}_{\text{$n-1$ times}}
  \ast I^{+-}(N(Z)^{-1}) \\
  = \underbrace{ (\deg+1)^{-1}_{Z_1} \dots (\deg+1)^{-1}_{Z_1}}_{\text{$n-1$ times}}
  \sum_{l,m,n} \frac{N(Z_1)^{-1}}{2l+1} t^l_{m \, \underline{n}}(Z_1^{-1}) \cdot
  t^l_{n \, \underline{m}}(Z_2)  \\
  = (-1)^{n-1} \sum_l \frac{N(Z_1)^{-1}}{(2l+1)^n}
  \frac{\lambda_1^{-(2l+1)}-\lambda_2^{-(2l+1)}}{\lambda_1^{-1}-\lambda_2^{-1}}
  = \frac{(-1)^{n-1}}{N(Z_1)} \frac{\operatorname{Li}_n(\lambda_1^{-1})
    - \operatorname{Li}_n(\lambda_2^{-1})}{\lambda_1^{-1}-\lambda_2^{-1}}  \\
  = \frac{(-1)^n}{N(Z_2)} \frac{\operatorname{Li}_n(\lambda_1^{-1})
    - \operatorname{Li}_n(\lambda_2^{-1})}{\lambda_1-\lambda_2}
\end{multline*}
(recall that $N(Z_1) = \lambda_1 \lambda_2 \cdot N(Z_2)$).
Thus,
\begin{multline*}
  \underbrace{I(N(Z)^{-1}) \ast \dots \ast I(N(Z)^{-1})}_{\text{$n$ times}}  \\
= - \frac1{N(Z_2)} \frac{\operatorname{Li}_n(\lambda_1)
+ (-1)^n \operatorname{Li}_n(\lambda_1^{-1}) - \operatorname{Li}_n(\lambda_2)
- (-1)^n \operatorname{Li}_n(\lambda_2^{-1})}{\lambda_1-\lambda_2}  \\
= \frac{(2\pi i)^n}{n! N(Z_2)}
\frac{B_n\bigl( \frac12 - \frac{\log(-\lambda_1)}{2\pi i} \bigr)
  - B_n\bigl( \frac12 - \frac{\log(-\lambda_2)}{2\pi i} \bigr)}
     {\lambda_1-\lambda_2}.
\end{multline*}

If $\lambda_1,\lambda_2 \to 1$ with $\sgn(\im\lambda_1), \sgn(\im\lambda_2)>0$,
$$
\frac{(2\pi i)^n}{n! N(Z_2)}
\frac{B_n\bigl( \frac12 - \frac{\log(-\lambda_1)}{2\pi i} \bigr)
  - B_n\bigl( \frac12 - \frac{\log(-\lambda_2)}{2\pi i} \bigr)}
     {\lambda_1-\lambda_2}
     \to -(2\pi i)^{n-1} \frac{B'_n(1)}{n!} \frac1{N(Z_2)};
$$
and if $\lambda_1,\lambda_2 \to 1$ with
$\sgn(\im\lambda_1), \sgn(\im\lambda_2)<0$,
$$
\frac{(2\pi i)^n}{n! N(Z_2)}
\frac{B_n\bigl( \frac12 - \frac{\log(-\lambda_1)}{2\pi i} \bigr)
  - B_n\bigl( \frac12 - \frac{\log(-\lambda_2)}{2\pi i} \bigr)}
     {\lambda_1-\lambda_2}
\to -(2\pi i)^{n-1} \frac{B'_n(0)}{n!} \frac1{N(Z_2)}.
$$
Since the Bernoulli polynomials satisfy
$$
B'_n(z) = n B_{n-1}(z)
$$
(see, for example, \cite{Ba}), the result follows.
\end{proof}

\begin{rem}
  It is well known (see, for example, \cite{Ba}) that
  $B_n(0)$ is the n-th Bernoulli number $B_n$ and that
  $$
  B_n(1) = B_n(0), \qquad n \ge 2,
  $$
  $$
  B_0(0)=B_0(1)=1, \qquad B_1(0)=-1/2, \qquad B_1(1)=1/2,
  $$
  $$
  B_{2n+1}=0, \qquad n=1,2,3,\dots.
  $$
\end{rem}

\subsection{Convolution Algebra}  \label{convolution-alg-subsection}

Recall the bases of left and right regular polynomial functions on $\HC^{\times}$
with values in $\BB S$ and $\BB S'$ respectively that were introduced in
Proposition 24 of \cite{FL1}:
\begin{center}
\begin{tabular}{lcl}
$v^+_{l,m,n}(Z) = \left( \begin{smallmatrix}
(l-m+ \frac 12) t^l_{n \, \underline{m+ \frac 12}}(Z)  \\
(l+m+ \frac 12) t^l_{n \, \underline{m- \frac 12}}(Z)
\end{smallmatrix} \right)$,
& \qquad &
$\begin{smallmatrix}
l=0,\frac12,1,\frac32,\dots, \\
m =-l-\frac 12 ,-l+\frac 32,\dots,l+\frac 12,  \\
n =-l,-l+1,\dots,l;
\end{smallmatrix}$  \\
$v^-_{l,m,n}(Z) = \left(\begin{smallmatrix}
(l-m+ \frac 12) N(Z)^{-1} \cdot t^l_{m- \frac 12 \, \underline{n}}(Z^{-1})  \\
(l+m+ \frac 12) N(Z)^{-1} \cdot t^l_{m+ \frac 12 \, \underline{n}}(Z^{-1})
\end{smallmatrix}\right)$,
& \qquad &
$\begin{smallmatrix}
l=\frac12,1,\frac32,2,\dots, \\
m =-l+\frac 12 ,-l+\frac 32,\dots,l-\frac 12, \\
n=-l,-l+1,\dots,l;
\end{smallmatrix}$  \\
$v'^+_{l,m,n}(Z) = \Bigl( t^{l-\frac 12}_{n + \frac 12 \, \underline{m}}(Z),
t^{l-\frac 12}_{n - \frac 12 \, \underline{m}}(Z) \Bigr)$,
& \qquad &
$\begin{smallmatrix}
l=\frac12,1,\frac32,2,\dots, \\
m =-l+\frac 12 ,-l+\frac 32,\dots,l-\frac 12, \\
n=-l,-l+1,\dots,l;  \\
\end{smallmatrix}$  \\
$v'^-_{l,m,n}(Z) = \Bigl(
N(Z)^{-1} \cdot t^{l+\frac 12}_{m \, \underline{n- \frac 12}}(Z^{-1}),
N(Z)^{-1} \cdot t^{l+\frac 12}_{m \, \underline{n+ \frac 12}}(Z^{-1}) \Bigr)$,
& \qquad &
$\begin{smallmatrix}
l=0,\frac12,1,\frac32,\dots, \\
m =-l-\frac 12 ,-l+\frac 32,\dots,l+\frac 12,  \\
n =-l,-l+1,\dots,l.
\end{smallmatrix}$
\end{tabular}
\end{center}
Then
\begin{align*}
  {\cal V}^+ &= \bigl\{ f \in \BB S \otimes \BB C[z_{11},z_{12},z_{21},z_{22}] ;\:
  \nabla^+ f =0 \bigr\}  \\
  &= \BB C\text{-span of } \biggl\{  v^+_{l,m,n}(Z) ;\:
  \begin{smallmatrix} l=0,\frac12,1,\frac32,\dots, \\
    m =-l-\frac 12 ,-l+\frac 32,\dots,l+\frac 12,  \\
    n =-l,-l+1,\dots,l \end{smallmatrix} \biggr\},  \\
  {\cal V}^- &= \bigl\{
  f \in \BB S \otimes \BB C[z_{11},z_{12},z_{21},z_{22}, N(Z)^{-1}] ;\:
  N(Z)^{-1} \cdot Z^{-1} \cdot f(Z^{-1}) \in {\cal V^+} \bigr\}  \\
  &= \BB C\text{-span of } \biggl\{ v^-_{l,m,n}(Z) ;\:
  \begin{smallmatrix} l=\frac12,1,\frac32,2,\dots, \\
    m =-l+\frac 12 ,-l+\frac 32,\dots,l-\frac 12, \\
    n=-l,-l+1,\dots,l
  \end{smallmatrix} \biggr\},  \\
{\cal V} &= \bigl\{
  f \in \BB S \otimes \BB C[z_{11},z_{12},z_{21},z_{22}, N(Z)^{-1}] ;\:
  \nabla^+ f =0 \bigr\}  \\
  &= {\cal V}^+ \oplus {\cal V}^-;
\end{align*}
\begin{align*}
{\cal V}'^+ &= \bigl\{ g \in \BB S' \otimes \BB C[z_{11},z_{12},z_{21},z_{22}] ;\:
g \nabla^+ =0 \bigr\}  \\
&= \BB C\text{-span of } \biggl\{  v'^+_{l,m,n}(Z) ;\:
\begin{smallmatrix} l=\frac12,1,\frac32,2,\dots, \\
    m =-l+\frac 12 ,-l+\frac 32,\dots,l-\frac 12, \\
    n=-l,-l+1,\dots,l
\end{smallmatrix} \biggr\},  \\
  {\cal V}'^- &= \bigl\{
  g \in \BB S' \otimes \BB C[z_{11},z_{12},z_{21},z_{22}, N(Z)^{-1}] ;\:
  N(Z)^{-1} \cdot g(Z^{-1}) \cdot Z^{-1} \in {\cal V'^+} \bigr\}  \\
  &= \BB C\text{-span of } \biggl\{ v'^-_{l,m,n}(Z) ;\:
\begin{smallmatrix} l=0,\frac12,1,\frac32,\dots, \\
    m =-l-\frac 12 ,-l+\frac 32,\dots,l+\frac 12,  \\
    n =-l,-l+1,\dots,l \end{smallmatrix} \biggr\},  \\
{\cal V}' &= \bigl\{
g \in \BB S' \otimes \BB C[z_{11},z_{12},z_{21},z_{22}, N(Z)^{-1}] ;\:
g \nabla^+ =0 \bigr\}  \\
&= {\cal V}'^+ \oplus {\cal V}'^-.
\end{align*}
The Lie algebra $\mathfrak{gl}(2,\HC)$ acts on ${\cal V}$ and ${\cal V}'$
via $\pi_l$ and $\pi_r$ respectively.
There is a non-degenerate $\mathfrak{gl}(2,\HC)$-invariant
bilinear pairing between $(\pi_r, {\cal V}')$ and $(\pi_l, {\cal V})$
given by the integral formula
$$
\langle g, f \rangle_{\cal V} = \frac1{2\pi^2} \int_{Z \in SU(2)}
g(Z) \cdot Dz \cdot f(Z),
\qquad g \in {\cal V}',\: f \in {\cal V},
$$
(equation (29) in \cite{FL1}).
The above basis functions satisfy the following orthogonality relations:
\begin{align}
\langle v'^-_{l,m,n}, v^+_{l',m',n'} \rangle_{\cal V}
&= \delta_{ll'} \delta_{mm'} \delta_{nn'},  \label{V-orthog1}  \\
\langle v'^+_{l,m,n}, v^-_{l',m',n'} \rangle_{\cal V}
&= \delta_{ll'} \delta_{mm'} \delta_{nn'},  \label{V-orthog2}  \\
\langle v'^+_{l,m,n}, v^+_{l',m',n'} \rangle_{\cal V}
&= \langle v'^-_{l,m,n}, v^-_{l',m',n'} \rangle_{\cal V} = 0  \label{V-orthog3}
\end{align}
(Proposition 24 in \cite{FL1}).

We consider ${\cal V} \otimes {\cal V}'$.
This space consists of polynomial $\HC$-valued functions
in two variables $Z_1, Z_2 \in\HC^{\times}$ that are left regular with respect
to $Z_1$ and right regular with respect to $Z_2$.
We have a multiplication (or restriction to the diagonal) operation
$$
\operatorname{Mult}: {\cal V} \otimes {\cal V}' \to {\cal W}, \quad
v(Z_1) \otimes v'(Z_2) \mapsto v(Z) \cdot v'(Z).
$$
It is clearly $\mathfrak{gl}(2,\HC)$-equivariant with respect to the
$\pi_l \otimes \pi_r$ action on ${\cal V} \otimes {\cal V}'$ and
$\rho_2$ action on ${\cal W}$.

\begin{prop}  \label{mult-image-prop}
The image under the multiplication map
$\operatorname{Mult}: {\cal V} \otimes {\cal V}' \to {\cal W}$ is precisely
${\cal Q}^+ \oplus {\cal Q}^0 \oplus {\cal Q}^-$.
\end{prop}

\begin{proof}
  Multiplying generators of ${\cal V}$ and ${\cal V}'$, one checks that the
  image of ${\cal V} \otimes {\cal V}'$ contains
  ${\cal Q}^+$, ${\cal Q}^0$ and ${\cal Q}^-$.
  Thus it remains to show the other inclusion.
  It is easy to see that $\tr \circ \partial^+$ annihilates the image,
  hence $\operatorname{Mult} ({\cal V} \otimes {\cal V}') \subset
  \ker(\tr \circ \partial^+)$.
  Recall that the decomposition of $\ker(\tr \circ \partial^+)$ into
  irreducible components was obtained in Subsection \ref{W-decomp-subsect}.
  By Theorem 28 from \cite{FL1}, ${\cal V}$, ${\cal V}'$ and hence
  ${\cal V} \otimes {\cal V}'$ have inner products such that the real form
  $\mathfrak u(2,2)$ of $\mathfrak{gl}(2,\HC)$ acts unitarily.
  In particular, $(\pi_l \otimes \pi_r, {\cal V} \otimes {\cal V}')$ is
  semisimple. It follows that the other irreducible components of
  $\ker(\tr \circ \partial^+)$ cannot appear in the image of the
  multiplication map because they are not semisimple.
\end{proof}

We define convolution operation on ${\cal V} \otimes {\cal V}'$ as
$$
(F \ast G)(Z_1,Z_2) = \frac1{2\pi^2} \int_{W \in SU(2)}
F(Z_1, W) \cdot Dw \cdot G(W, Z_2),
\qquad F, G \in {\cal V} \otimes {\cal V}'.
$$
Alternatively, this operation can be defined on pure tensors as follows.
If $F(Z_1,Z_2) = v(Z_1) \otimes v'(Z_2)$,
$G(Z_1,Z_2) = w(Z_1) \otimes w'(Z_2) \in {\cal V} \otimes {\cal V}'$, then
\begin{equation}  \label{conv-VtensorV}
  (F \ast G)(Z_1,Z_2)
  = \langle v',w \rangle_{\cal V} \cdot v(Z_1) \otimes w'(Z_2).
\end{equation}
This operation gives ${\cal V} \otimes {\cal V}'$ the structure of an
associative algebra.
Since the bilinear pairing is $\mathfrak{gl}(2,\HC)$-invariant,
the above convolution product is $\mathfrak{gl}(2,\HC)$-equivariant with
respect to the $\pi_l \otimes \pi_r$ action on ${\cal V} \otimes {\cal V}'$.

Next, we consider a space ${\cal A}$ consisting of infinite sums
$\sum_i v_i(Z_1) \otimes v'_i(Z_2)$, where $v_i \in {\cal V}$,
$v'_i \in {\cal V}'$, such that
$$
|\operatorname{degree\:of\,} v_i + \operatorname{degree\:of\,} v'_i|
$$
is bounded.
More precisely, ${\cal A}$ consists of formal series of the form
\begin{multline} \label{F-series}
  \sum_{\genfrac{}{}{0pt}{}{l,m,n}{l',m',n'}}
  a_{\genfrac{}{}{0pt}{}{l,m,n}{l',m',n'}} v^+_{l,m,n}(Z_1) \cdot v'^+_{l',m',n'}(Z_2)
  + \sum_{\genfrac{}{}{0pt}{}{l,m,n}{l',m',n'}}
  b_{\genfrac{}{}{0pt}{}{l,m,n}{l',m',n'}}
  v^+_{l,m,n}(Z_1) \cdot v'^-_{l',m',n'}(Z_2)  \\
  + \sum_{\genfrac{}{}{0pt}{}{l,m,n}{l',m',n'}}
  c_{\genfrac{}{}{0pt}{}{l,m,n}{l',m',n'}} v^-_{l,m,n}(Z_1) \cdot v'^+_{l',m',n'}(Z_2)
  + \sum_{\genfrac{}{}{0pt}{}{l,m,n}{l',m',n'}}
  d_{\genfrac{}{}{0pt}{}{l,m,n}{l',m',n'}} v^-_{l,m,n}(Z_1) \cdot v'^-_{l',m',n'}(Z_2)
\end{multline}
such that only finitely many coefficients
$a_{\genfrac{}{}{0pt}{}{l,m,n}{l',m',n'}}$'s,
$d_{\genfrac{}{}{0pt}{}{l,m,n}{l',m',n'}}$'s are non-zero and
non-zero coefficients $b_{\genfrac{}{}{0pt}{}{l,m,n}{l',m',n'}}$'s,
$c_{\genfrac{}{}{0pt}{}{l,m,n}{l',m',n'}}$'s have bounded difference of indices
$|l-l'|$.
We have an analogue of Lemma \ref{conv-prod-lem},
its proof is exactly the same.

\begin{lem}  \label{conv-prod-lem-spinor}
The convolution operation (\ref{conv-VtensorV}) extends to ${\cal A}$ and gives
it the structure of a $\mathfrak{gl}(2,\HC)$-invariant associative algebra.
Moreover, if $F \in {\cal V} \otimes {\cal V}'$ and $G \in {\cal A}$, then
both $F \ast G$ and $G \ast F$ lie in ${\cal V} \otimes {\cal V}'$.
\end{lem}

Unlike ${\cal V} \otimes {\cal V}'$, ${\cal A}$ has a unit.
The expression for the unit is obtained by formally combining the two
matrix coefficient expansions of $\frac{(Z_1-Z_2)^{-1}}{N(Z_1-Z_2)}$
given in Proposition 26 from \cite{FL1} (see also Proposition \ref{Prop26}):
\begin{equation}  \label{1_A}
1_{\cal A} = \sum_{l,m,n} v^+_{l,m,n}(Z_1) \cdot v'^-_{l,m,n}(Z_2)
+ \sum_{l,m,n} v^-_{l,m,n}(Z_1) \cdot v'^+_{l,m,n}(Z_2) \quad \in {\cal A}.
\end{equation}
The fact that $1_{\cal A}$ is indeed a unit follows from the definition of the
convolution operation and orthogonality relations
(\ref{V-orthog1})-(\ref{V-orthog3}).

The Lie algebra $\mathfrak{gl}(2,\HC)$ can act on ${\cal V} \otimes {\cal V}'$
by at least three different actions:
$\pi_l \otimes 1$, $1 \otimes \pi_r$ and $\pi_l \otimes \pi_r$.
Clearly, all three actions extend to ${\cal A}$.
Then the convolution operation on ${\cal A}$ is
$(\pi_l \otimes \pi_r)$-equivariant.

Similarly to how we did in the case of ${\cal HH}$, we define operators
$(\deg+d)^{-1}_{Z_1}$ and $(\deg+d)^{-1}_{Z_2}$ on ${\cal A}$, where $d \in \BB Z$.
If $F \in {\cal A}$ is expressed as a series
(\ref{F-series}), then
\begin{multline*}
(\deg+d)^{-1}_{Z_1} F(Z_1,Z_2) = \sum_{\genfrac{}{}{0pt}{}{l>-d/2}{m,n,l',m',n'}}
\frac{a_{\genfrac{}{}{0pt}{}{l,m,n}{l',m',n'}}}{2l+d}
v^+_{l,m,n}(Z_1) \cdot v'^+_{l',m',n'}(Z_2)  \\
+ \sum_{\genfrac{}{}{0pt}{}{l>-d/2}{m,n,l',m',n'}}
\frac{b_{\genfrac{}{}{0pt}{}{l,m,n}{l',m',n'}}}{2l+d}
v^+_{l,m,n}(Z_1) \cdot v'^-_{l',m',n'}(Z_2)  \\
- \sum_{\genfrac{}{}{0pt}{}{l>d/2-1}{m,n,l',m',n'}}
\frac{c_{\genfrac{}{}{0pt}{}{l,m,n}{l',m',n'}}}{2l+2-d}
v^-_{l,m,n}(Z_1) \cdot v'^+_{l',m',n'}(Z_2)  \\
- \sum_{\genfrac{}{}{0pt}{}{l>d/2-1}{m,n,l',m',n'}}
\frac{d_{\genfrac{}{}{0pt}{}{l,m,n}{l',m',n'}}}{2l+2-d}
v^-_{l,m,n}(Z_1) \cdot v'^-_{l',m',n'}(Z_2).
\end{multline*}
Note that if $d \ne 1$ or $2$, certain terms get discarded.
Then $(\deg+d)^{-1}_{Z_2} F(Z_1,Z_2)$ is defined similarly.
We have analogues of Lemmas \ref{deg-inverse-comm-rels}
and \ref{deg-inverse-star-rels}:

\begin{lem}  \label{deg-inverse-comm-rels-spinor}
Let $F(Z_1,Z_2) \in {\cal A}$, $d \in \BB Z$ and $A, B, C, D \in \HC$.
We have the following commutation relations:
\begin{align*}
(\deg+d)_{Z_1}^{-1} (\pi_l \otimes 1)
\bigl(\begin{smallmatrix} A & 0 \\ 0 & D \end{smallmatrix}\bigr) F(Z_1,Z_2)
&= (\pi_l \otimes 1) \bigl(\begin{smallmatrix}
  A & 0 \\ 0 & D \end{smallmatrix}\bigr)
(\deg+d)_{Z_1}^{-1} F(Z_1,Z_2), \\
(\deg+d)_{Z_1}^{-1} (\pi_l \otimes 1)
\bigl(\begin{smallmatrix} 0 & B \\ 0 & 0 \end{smallmatrix}\bigr) F(Z_1,Z_2)
&= (\pi_l \otimes 1) \bigl(\begin{smallmatrix}
  0 & B \\ 0 & 0 \end{smallmatrix}\bigr)
(\deg+d-1)_{Z_1}^{-1} F(Z_1,Z_2), \\
(\deg+d)_{Z_1}^{-1} (\pi_l \otimes 1)
\bigl(\begin{smallmatrix} 0 & 0 \\ C & 0 \end{smallmatrix}\bigr) F(Z_1,Z_2)
&= (\pi_l \otimes 1) \bigl(\begin{smallmatrix}
  0 & 0 \\ C & 0 \end{smallmatrix}\bigr)
(\deg+d+1)_{Z_1}^{-1} F(Z_1,Z_2),
\end{align*}
and similarly for $(\deg+d)_{Z_2}^{-1}$ and $1 \otimes \pi_r$.
\end{lem}

\begin{lem}  \label{deg-inverse-star-rels-spinor}
Let $F, G \in {\cal A}$ and $d \in \BB Z$. We have the following relations:
\begin{align*}
(\deg+d)_{Z_1}^{-1} (F \ast G) &= \bigl( (\deg+d)_{Z_1}^{-1} F \bigr) \ast G,  \\
(\deg+d)_{Z_2}^{-1} (F \ast G) &= F \ast \bigl( (\deg+d)_{Z_2}^{-1} G \bigr),  \\
\bigl( (\deg+d)_{Z_2}^{-1} F \bigr) \ast G &=
  - F \ast \bigl( (\deg+3-d)_{Z_1}^{-1} G \bigr).
\end{align*}
\end{lem}

We define an equivariant map
$$
J=J^{++} - (J^{+-}+J^{-+}) + J^{--} :
(\rho'_2, {\cal W}') \to (\pi_l \otimes \pi_r, {\cal A})
$$
as follows.
Recall the maps $J_R$ given by equation (\ref{fork}).
By Theorem \ref{J_R(W^+-)}, if $Z_1, Z_2 \in \BB D^+_R$,
$J_R$ is a $\mathfrak{gl}(2,\HC)$-equivariant map
$(\rho'_2,{\cal W}') \to (\pi_l \otimes \pi_r,{\cal V}^+ \otimes {\cal V}'^+)$
independent of the choice of $R>0$; we call this map $J^{++}$.
Similarly, if $Z_1, Z_2 \in \BB D^-_R$, $J_R$ is a
$\mathfrak{gl}(2,\HC)$-equivariant map
$(\rho'_2,{\cal W}') \to (\pi_l \otimes \pi_r, {\cal V}^- \otimes {\cal V}'^-)$
also independent of $R>0$; we call this map $J^{--}$.
If $F \in {\cal W}'$,
$$
J^{+-}(F) = \sum_{\genfrac{}{}{0pt}{}{l,m,n}{l',m',n'}}
b(F)_{\genfrac{}{}{0pt}{}{l,m,n}{l',m',n'}} v^+_{l,m,n}(Z_1) \cdot v'^-_{l',m',n'}(Z_2),
$$
where
$$
b(F)_{\genfrac{}{}{0pt}{}{l,m,n}{l',m',n'}} =
\frac{12}{\pi^3i} \int_{U(2)_R}
v'^-_{l,m,n}(W) \cdot F(W) \cdot v^+_{l',m',n'}(W) \,dV.
$$
On the one hand, this integral does not depend on $R>0$.
On the other hand, by the matrix coefficient expansions of
$\frac{(Z-W)^{-1}}{N(Z-W)}$
given in Proposition 26 from \cite{FL1} (see also Proposition \ref{Prop26}),
for each $R>0$, the series $J^{+-}(F)$ converges to $(J_RF)(Z_1,Z_2)$
whenever $Z_1 \in \BB D^+_R$ and $Z_2 \in \BB D^-_R$.
Similarly,
$$
J^{-+}(F) = \sum_{\genfrac{}{}{0pt}{}{l,m,n}{l',m',n'}}
c(F)_{\genfrac{}{}{0pt}{}{l,m,n}{l',m',n'}} v^-_{l,m,n}(Z_1) \cdot v'^+_{l',m',n'}(Z_2),
$$
where
$$
c(F)_{\genfrac{}{}{0pt}{}{l,m,n}{l',m',n'}} =
\frac{12}{\pi^3i} \int_{U(2)_R}
v'^+_{l,m,n}(W) \cdot F(W) \cdot v^-_{l',m',n'}(W) \,dV.
$$
This integral is independent of $R>0$ and, for each $R>0$,
the series $J^{-+}(F)$ converges to $(J_RF)(Z_1,Z_2)$ whenever
$Z_1 \in \BB D^-_R$ and $Z_2 \in \BB D^+_R$.

\begin{lem}  \label{1_A-lem}
  We have:
  $$
  J\bigl( N(Z)^{-1} \cdot Z \bigr) = 24 \cdot 1_{\cal A}.
  $$
\end{lem}

\begin{proof}
The result follows immediately from Proposition \ref{J_R(W^0)-prop}
and expression (\ref{1_A}).
\end{proof}

Choose a generator
$$
N(Z)^{-1} \cdot \tilde H_{0,0,0}(Z) =
N(Z)^{-1} \cdot \bigl(\begin{smallmatrix} 1 & 0 \\ 0 & 0 \end{smallmatrix}\bigr)
\quad \in {\cal Q}'^0.
$$
We conclude this subsection with an analogue of Proposition \ref{gen-mult-prop}:

\begin{prop}  \label{gen-mult-spinor-prop}
  We have:
  $$
  J \bigl( N(W)^{-1} \cdot \tilde H_{0,0,0}(W) \bigr)
  = 24 (\partial_{11})_{Z_1} (\deg+1)_{Z_1}^{-1} 1_{\cal A}
  = 24 (\partial_{11})_{Z_2} (\deg+1)_{Z_2}^{-1} 1_{\cal A}.
  $$
In particular, for any $G \in {\cal A}$,
\begin{align*}
  J \bigl( N(W)^{-1} \cdot \tilde H_{0,0,0}(W) \bigr) \ast G &=
  24 (\partial_{11})_{Z_1} (\deg+1)_{Z_1}^{-1} G(Z_1,Z_2),  \\
  G \ast J \bigl( N(W)^{-1} \cdot \tilde H_{0,0,0}(W) \bigr) &=
  24 (\partial_{11})_{Z_2} (\deg+1)_{Z_2}^{-1} G(Z_1,Z_2).
\end{align*}
\end{prop}

\begin{proof}
We start by computing $J^{+-}$ of $N(W)^{-1} \cdot \tilde H_{0,0,0}(W)$:
$$
J^{+-} \bigl( N(W)^{-1} \cdot \tilde H_{0,0,0}(W) \bigr) (Z_1,Z_2)
= \sum_{\genfrac{}{}{0pt}{}{l,m,n}{l',m',n'}}
b(F)_{\genfrac{}{}{0pt}{}{l,m,n}{l',m',n'}} v^+_{l,m,n}(Z_1) \cdot v'^-_{l',m',n'}(Z_2),
$$
where
\begin{multline*}
\frac{\pi^3i}{12} \cdot b(F)_{\genfrac{}{}{0pt}{}{l,m,n}{l',m',n'}}
= \int_{U(2)_R} v'^-_{l,m,n}(W) \cdot N(W)^{-1} \cdot
\bigl(\begin{smallmatrix} 1 & 0 \\ 0 & 0 \end{smallmatrix}\bigr) \cdot
v^+_{l',m',n'}(W) \,dV  \\
= \int_{U(2)_R}
\left( \begin{smallmatrix} t^{l+\frac 12}_{m \, \underline{n- \frac 12}}(W^{-1}), &
t^{l+\frac 12}_{m \, \underline{n+ \frac 12}}(W^{-1}) \end{smallmatrix} \right)
\cdot N(W)^{-2} \cdot
\bigl(\begin{smallmatrix} 1 & 0 \\ 0 & 0 \end{smallmatrix}\bigr) \cdot
\left(\begin{smallmatrix} (l'-m'+\frac 12) t^{l'}_{n'\, \underline{m'+\frac 12}}(W)\\
(l'+m'+\frac 12) t^{l'}_{n'\,\underline{m'-\frac12}}(W)\end{smallmatrix}\right)\,dV\\
= \Bigl( l'-m'+\frac12 \Bigr)
\int_{U(2)_R} N(W)^{-2} \cdot t^{l+\frac 12}_{m \, \underline{n-\frac 12}}(W^{-1})
\cdot t^{l'}_{n'\,\underline{m'+\frac 12}}(W) \,dV.
\end{multline*}
By the orthogonality relations (19) in \cite{FL3} this coefficient is zero
unless $l+1/2=l'$, $m=m'+1/2$ and $n-1/2=n'$.
So, let us assume that this is the case.
Using Lemma 22 from \cite{FL1}, we obtain:
\begin{multline*}
J^{+-} \bigl( N(W)^{-1} \cdot \tilde H_{0,0,0}(W) \bigr) (Z_1,Z_2)  \\
= \frac{-24}{N(Z_2)} \sum_{l,m,n}
\frac{\begin{smallmatrix} l-m+\frac32 \end{smallmatrix}}{2l+2} \cdot
\left(\begin{smallmatrix} (l-m+\frac12) t^l_{n\,\underline{m+ \frac 12}}(Z_1)  \\
(l+m+\frac12) t^l_{n\,\underline{m-\frac 12}}(Z_1) \end{smallmatrix}\right)
\cdot \left(\begin{smallmatrix} t^{l+1}_{m-\frac12\,\underline{n-1}}(Z_2^{-1}), &
t^{l+1}_{m-\frac12 \,\underline{n}}(Z_2^{-1}) \end{smallmatrix}\right)  \\
= - 24 \sum_{l,m,n} \frac{(\partial_{11})_{Z_1}}{2l+2} \cdot
\left(\begin{smallmatrix}(l-m+\frac32) t^{l+\frac12}_{n-\frac12\,\underline{m}}(Z_1)\\
(l+m+\frac12) t^{l+\frac12}_{n-\frac12\,\underline{m-1}}(Z_1) \end{smallmatrix}\right)
\cdot N(Z_2)^{-1} \cdot
\left(\begin{smallmatrix} t^{l+1}_{m-\frac12\,\underline{n-1}}(Z_2^{-1}), &
t^{l+1}_{m-\frac12 \,\underline{n}}(Z_2^{-1}) \end{smallmatrix}\right).
\end{multline*}
Alternatively, using (\ref{del-t(Z^{-1})}), we can rewrite
$J^{+-} \bigl( N(W)^{-1} \cdot \tilde H_{0,0,0}(W) \bigr)$ as
\begin{multline*}
J^{+-} \bigl( N(W)^{-1} \cdot \tilde H_{0,0,0}(W) \bigr) (Z_1,Z_2)  \\
= 24 \sum_{l,m,n}
\frac{(\partial_{11})_{Z_2}}{2l+2} \cdot
\left(\begin{smallmatrix} (l-m+\frac12) t^l_{n\,\underline{m+\frac 12}}(Z_1) \\
(l+m+\frac12) t^l_{n\,\underline{m-\frac 12}}(Z_1) \end{smallmatrix}\right)
\cdot N(Z_2)^{-1} \cdot
\left(\begin{smallmatrix} t^{l+\frac12}_{m\,\underline{n-\frac12}}(Z_2^{-1}), &
t^{l+\frac12}_{m \,\underline{n+\frac12}}(Z_2^{-1}) \end{smallmatrix}\right).
\end{multline*}

Next, we find $J^{-+}$ of $N(W)^{-1} \cdot \tilde H_{0,0,0}(W)$:
$$
J_R \bigl( N(W)^{-1} \cdot \tilde H_{0,0,0}(W) \bigr) (Z_1,Z_2)
= \sum_{\genfrac{}{}{0pt}{}{l,m,n}{l',m',n'}}
c(F)_{\genfrac{}{}{0pt}{}{l,m,n}{l',m',n'}} v^-_{l,m,n}(Z_1) \cdot v'^+_{l',m',n'}(Z_2),
$$
where
\begin{multline*}
\frac{\pi^3i}{12} \cdot c(F)_{\genfrac{}{}{0pt}{}{l,m,n}{l',m',n'}}
= \int_{U(2)_R} v'^+_{l,m,n}(W) \cdot N(W)^{-1} \cdot
\bigl(\begin{smallmatrix} 1 & 0 \\ 0 & 0 \end{smallmatrix}\bigr) \cdot
v^-_{l',m',n'}(W) \,dV  \\
= \int_{U(2)_R}
\left(\begin{smallmatrix} t^{l-\frac 12}_{m + \frac 12 \, \underline{n}}(W), &
t^{l-\frac 12}_{m - \frac 12 \, \underline{n}}(W) \end{smallmatrix}\right)
\cdot N(W)^{-2} \cdot
\bigl(\begin{smallmatrix} 1 & 0 \\ 0 & 0 \end{smallmatrix}\bigr)
\cdot \left(\begin{smallmatrix}
(l'-n'+\frac12) t^{l'}_{n'-\frac 12 \, \underline{m'}}(W^{-1})  \\
(l'+n'+\frac12) t^{l'}_{n'+\frac 12 \, \underline{m'}}(W^{-1})
\end{smallmatrix}\right) \,dV  \\
= \Bigl( l'-n'+\frac12 \Bigr) \int_{U(2)_R}
N(W)^{-2} \cdot t^{l-\frac 12}_{m + \frac 12 \, \underline{n}}(W)
\cdot t^{l'}_{n'-\frac 12 \, \underline{m'}}(W^{-1}) \,dV.
\end{multline*}
By the orthogonality relations (19) in \cite{FL3} this coefficient is zero
unless $l-1/2=l'$, $m+1/2=m'$ and $n=n'-1/2$.
So, let us assume that this is the case.
Using (\ref{del-t(Z^{-1})}), we obtain:
\begin{multline*}
J^{-+} \bigl( N(W)^{-1} \cdot \tilde H_{0,0,0}(W) \bigr) (Z_1,Z_2)  \\
= - 24 \sum_{\genfrac{}{}{0pt}{}{l,m,n}{l \ge 1}}
\frac{\begin{smallmatrix} l-n-\frac12 \end{smallmatrix}}{2l \cdot N(Z_1)}
\left(\begin{smallmatrix} (l-n+\frac12) t^l_{n-\frac12 \,\underline{m}}(Z_1^{-1})  \\
(l+n+\frac 12) t^l_{n+\frac12 \,\underline{m}}(Z_1^{-1})\end{smallmatrix}\right)
\cdot \left(\begin{smallmatrix} t^{l-1}_{m+1 \, \underline{n+\frac12}}(Z_2), &
t^{l-1}_{m \, \underline{n+\frac12}}(Z_2) \end{smallmatrix}\right)  \\
= \frac{24}{N(Z_1)} \sum_{\genfrac{}{}{0pt}{}{l,m,n}{l \ge 1}}
\frac{(\partial_{11})_{Z_1}}{2l} \left(\begin{smallmatrix}
(l-n-\frac12) t^{l-\frac12}_{n \, \underline{m+\frac12}}(Z_1^{-1})  \\
(l+n+\frac 12) t^{l-\frac12}_{n+1 \,\underline{m+\frac12}}(Z_1^{-1})
\end{smallmatrix}\right) \cdot
\left(\begin{smallmatrix} t^{l-1}_{m+1 \, \underline{n+\frac12}}(Z_2), &
t^{l-1}_{m \, \underline{n+\frac12}}(Z_2) \end{smallmatrix}\right).
\end{multline*}
Alternatively, using Lemma 22 from \cite{FL1}, we can rewrite
$J^{-+} \bigl( N(W)^{-1} \cdot \tilde H_{0,0,0}(W) \bigr)$ as
\begin{multline*}
J^{-+} \bigl( N(W)^{-1} \cdot \tilde H_{0,0,0}(W) \bigr) (Z_1,Z_2)  \\
= \frac{-24}{N(Z_1)} \sum_{\genfrac{}{}{0pt}{}{l,m,n}{l \ge 1}}
\frac{(\partial_{11})_{Z_2}}{2l}
\left(\begin{smallmatrix} (l-n+\frac12) t^l_{n-\frac12\,\underline{m}}(Z_1^{-1})  \\
(l+n+\frac 12) t^l_{n+\frac12\,\underline{m}}(Z_1^{-1})\end{smallmatrix}\right)
\cdot \left(\begin{smallmatrix} t^{l-\frac12}_{m+\frac12\,\underline{n}}(Z_2),&
t^{l-\frac12}_{m-\frac12 \, \underline{n}}(Z_2) \end{smallmatrix}\right).
\end{multline*}

By Theorem \ref{J_R(W^+-)},
$$
J^{++} \bigl( N(W)^{-1} \cdot \tilde H_{0,0,0}(W) \bigr) =
J^{--} \bigl( N(W)^{-1} \cdot \tilde H_{0,0,0}(W) \bigr) =0.
$$
Thus,
\begin{multline*}
J\bigl( N(W)^{-1} \cdot \tilde H_{0,0,0}(W) \bigr) (Z_1,Z_2)
= -(J^{+-}+J^{-+}) \bigl( N(W)^{-1} \cdot \tilde H_{0,0,0}(W) \bigr) (Z_1,Z_2)  \\
= 24 (\partial_{11})_{Z_1} (\deg+1)_{Z_1}^{-1} \Bigl(
\sum_{l,m,n} v^+_{l,m,n}(Z_1) \cdot v'^-_{l,m,n}(Z_2)
+ \sum_{l,m,n} v^-_{l,m,n}(Z_1) \cdot v'^+_{l,m,n}(Z_2) \Bigr)  \\
= 24 (\partial_{11})_{Z_2} (\deg+1)_{Z_2}^{-1} \Bigl(
\sum_{l,m,n} v^+_{l,m,n}(Z_1) \cdot v'^-_{l,m,n}(Z_2)
+ \sum_{l,m,n} v^-_{l,m,n}(Z_1) \cdot v'^+_{l,m,n}(Z_2) \Bigr).
\end{multline*}
Then the result follows from (\ref{1_A}) and Lemmas
\ref{deg-inverse-comm-rels-spinor}, \ref{deg-inverse-star-rels-spinor}.
\end{proof}

\begin{rem}
The same argument shows that, for $A \in \HC$,
  $$
  J \bigl( N(W)^{-1} \cdot A \bigr)
  = 24 \tr(A \partial)_{Z_1} (\deg+1)_{Z_1}^{-1} 1_{\cal A}
  = 24 \tr(A \partial)_{Z_2} (\deg+1)_{Z_2}^{-1} 1_{\cal A}.
  $$
  In particular,
  $$
  J \bigl( N(W)^{-1} \bigr)
  = 24 (\partial_{11} + \partial_{22})_{Z_1} (\deg+1)_{Z_1}^{-1} 1_{\cal A}
  = 24 (\partial_{11} + \partial_{22})_{Z_2} (\deg+1)_{Z_2}^{-1} 1_{\cal A}.
  $$
\end{rem}

\subsection{Algebra of Quaternionic Functions}

In this subsection we give $(\rho'_2,{\cal W}'/\ker J)$ the structure of a
$\mathfrak{gl}(2,\HC)$-invariant algebra. Many steps are proved by reduction
to the already developed scalar case of $(\rho_1,\Zh)$.

\begin{df}
  Let ${\cal A}^{\omega}$ denote the subspace of ${\cal A}$ generated by
  ${\cal V} \otimes {\cal V}'$, $J({\cal W}')$, application of operators
  $(\deg+d)^{-1}_{Z_1}$ and $(\deg+d)^{-1}_{Z_2}$, $d \in \BB Z$,
  as well as actions $\pi_l \otimes 1$ and $1 \otimes \pi_r$ of
  $\mathfrak{gl}(2,\HC)$.
\end{df}

Thus, by definition, ${\cal A}^{\omega}$ is invariant under the
$\pi_l \otimes \pi_r$ action of $\mathfrak{gl}(2,\HC)$.
As was done in the case of ${\cal HH}^{\omega}$, we want to reduce the number
of generators of ${\cal A}^{\omega}$.

\begin{lem}  \label{A-omega-generators}
  The space ${\cal A}^{\omega}$ is generated by ${\cal V} \otimes {\cal V}'$,
  elements of the type
  $$
  J(N(Z)^{-1}), \quad
  (\deg+d_1)^{-1}_{Z_1} J(N(Z)^{-1}), \quad
  (\deg+d_1)^{-1}_{Z_2} J(N(Z)^{-1}),
  $$
  $$
  (\deg+d_2)^{-1}_{Z_1} (\deg+d_1)^{-1}_{Z_1} J(N(Z)^{-1}), \quad
  (\deg+d_2)^{-1}_{Z_1} (\deg+d_1)^{-1}_{Z_2} J(N(Z)^{-1}),
  $$
  $$
  (\deg+d_2)^{-1}_{Z_2} (\deg+d_1)^{-1}_{Z_2} J(N(Z)^{-1}), \quad
  (\deg+d_3)^{-1}_{Z_1} (\deg+d_2)^{-1}_{Z_1}(\deg+d_1)^{-1}_{Z_1} J(N(Z)^{-1}),\dots
  $$
  as well as actions $\pi_l \otimes 1$ and $1 \otimes \pi_r$ of
  $\mathfrak{gl}(2,\HC)$.
\end{lem}

\begin{proof}
  By Theorem \ref{J_R(W^+-)} and Proposition \ref{J_R(W^0)-prop},
  $J({\cal W}') \simeq {\cal W}'/ \ker J$ has four irreducible components:
  $$
  {\cal Q}'^+, \qquad {\cal Q}'^0, \qquad{\cal Q}'^-
  $$
  and the trivial one-dimensional representation generated by
  $N(Z)^{-1} \cdot Z$.
  Since $(\rho'_2, {\cal W}')$ and $(\rho_2, {\cal W})$ are linear dual
  to each other, by Lemma \ref{nesting-lem}, the one-dimensional representation
  appears as a subrepresentation in $J({\cal W}')$ and the irreducible
  component ${\cal Q}'^0$ only as a subquotient.
  By Theorem \ref{J_R(W^+-)} again,
  $$
  J({\cal Q}'^+ \oplus {\cal Q}'^-) \subset {\cal V} \otimes {\cal V}'.
  $$
  On the other hand, $N(Z)^{-1}$ generates both ${\cal Q}'^0$ and the
  one-dimensional component.
  Then the proof proceeds the same way as that of
  Lemma \ref{HH-omega-generators}.
\end{proof}

\begin{prop}  \label{A-mult-closure-prop}
  The space ${\cal A}^{\omega}$ is closed under the convolution operation:
  if $F, G \in {\cal A}^{\omega}$, then $F \ast G$ also lies in
  ${\cal A}^{\omega}$.
\end{prop}

\begin{proof}
  The proof is essentially the same as that of
  Proposition \ref{HH-mult-closure-prop}.
\end{proof}

As was done in the case of ${\cal HH}^{\omega}$, we want to realize
elements of ${\cal A}^{\omega}$ as analytic functions.
Recall open subset $\Omega$ of $\HC^{\times} \times \HC^{\times}$
introduced in Subsection \ref{scalar-algebra-subsection}.
There is a natural map $\omega$ from ${\cal A}^{\omega}$ into
$\HC$-valued analytic functions on $\Omega$.
Indeed, elements of ${\cal V} \otimes {\cal V}'$ are polynomials,
hence can be treated as analytic functions on $\Omega$.
On the other hand, by (\ref{I-J-relation}),
$$
J(N(Z)^{-1}) = 24 \overrightarrow{\partial}_{Z_1} I(N(Z)^{-1})
\overleftarrow{\partial}_{Z_2}
$$
is an analytic function on $\Omega$ also.
Then operators $(\deg+d)^{-1}_{Z_1}$ and $(\deg+d)^{-1}_{Z_2}$, $d \in \BB Z$,
as well as actions  $\pi_l \otimes 1$ and $1 \otimes \pi_r$ of
$\mathfrak{gl}(2,\HC)$ preserve analyticity of functions.
Unlike the case of ${\cal HH}^{\omega}$, this map $\omega$ has a non-trivial
kernel: $\omega(1_{\cal A}) =0$, by Proposition \ref{J_R(W^0)-prop}.
We denote the composition $\omega \circ J$ by $\tilde J$ and
the space of analytic functions on $\Omega$ that are in the image
of $\omega$ by $\tilde{\cal A}^{\omega}$.
Note that $\ker\tilde J = \ker\M$ and
\begin{equation}  \label{I-J-relation2}
\tilde J(F) = 24 \overrightarrow{\partial}_{Z_1} I(F)
\overleftarrow{\partial}_{Z_2}, \qquad F \in {\cal W}'.
\end{equation}
Each $F \in \tilde{\cal A}^{\omega}$ is left regular with respect to $Z_1$ and
right regular with respect to $Z_2$:
$$
\overrightarrow{\nabla}_{Z_1} F(Z_1,Z_2) = 0
= F(Z_1,Z_2) \overleftarrow{\nabla}_{Z_2}.
$$
Since
$$
{\cal V} \otimes {\cal V}' \subset \tilde{\cal A}^{\omega},
$$
we can think of $\tilde{\cal A}^{\omega}$ as a completion of
${\cal V} \otimes {\cal V}'$.

Recall open subsets $\Lambda^+$, $\Lambda^-$ of $\HC^{\times} \times \HC^{\times}$
introduced in Subsection \ref{scalar-algebra-subsection}.
Definition \ref{extendable-scalar-def} extends to $\tilde{\cal A}^{\omega}$.

\begin{df}
  We call a function $F \in \tilde{\cal A}^{\omega}$ {\em extendable} if,
  for each $Z_0 \in \HC^{\times}$, there exists an open neighborhood
  $V \subset \HC^{\times} \times \HC^{\times}$ of $(Z_0,Z_0)$ and two
  $\HC$-valued functions $F^+$ and $F^-$ analytic on
  $V$ such that $F = F^+$ for all points in $\Lambda^+ \cap V$ and
  $F = F^-$ for all points in $\Lambda^- \cap V$.
\end{df}

Clearly, elements of ${\cal V} \otimes {\cal V}'$ are extendable.

\begin{lem}  \label{extendable-invariant-lem-spinor}
The extendable functions in $\tilde{\cal A}^{\omega}$ form a subspace
that is invariant under the actions $\pi_l \otimes 1$, $1 \otimes \pi_r$ and
$\pi_l \otimes \pi_r$ of $\mathfrak{gl}(2,\HC)$.
\end{lem}

\begin{rem}
We expect all functions in $\tilde{\cal A}^{\omega}$ to be extendable.
\end{rem}

Let $F \in \tilde{\cal A}^{\omega}$ be an extendable function,
define two analytic functions on $\HC^{\times}$:
\begin{align*}
(\DR^+ F)(Z) &=
\lim_{\genfrac{}{}{0pt}{}{Z_1, Z_2 \to Z}{(Z_1,Z_2) \in \Lambda^+}} F(Z_1,Z_2)
\qquad \text{and}  \\
(\DR^- F)(Z) &=
\lim_{\genfrac{}{}{0pt}{}{Z_1, Z_2 \to Z}{(Z_1,Z_2) \in \Lambda^-}} F(Z_1,Z_2).
\end{align*}
As in the scalar case, applying $\DR^+ F$ and $\DR^- F$ may yield different
results; functions $\DR^+ F$ and $\DR^- F$ need not be elements of ${\cal W}$
because they may not be polynomials on $\HC^{\times}$.
Nevertheless, the operators $\DR^+$ and $\DR^-$ intertwine the
$\mathfrak{gl}(2,\HC)$-actions $\pi_l \otimes \pi_r$ on extendable
functions in $\tilde{\cal A}^{\omega}$ and $\rho_2$ on $\HC$-valued
analytic functions on $\HC^{\times}$.

\begin{lem}  \label{Diag_on_VxV'-lem}
  If $F(Z_1,Z_2) \in {\cal V} \otimes {\cal V}'$, then
  $$
  \DR^+ F = \DR^- F
  \quad \in {\cal Q}^+ \oplus {\cal Q}^0 \oplus {\cal Q}^-.
  $$
\end{lem}

\begin{proof}
  Note that, when restricted to ${\cal V} \otimes {\cal V}$, both $\DR^+$ and
  $\DR^-$ reduce to the multiplication map
  $\operatorname{Mult}: {\cal V} \otimes {\cal V}' \to {\cal W}$, then the
  result follows from Proposition \ref{mult-image-prop}.
\end{proof}

\begin{lem}  \label{J(W')-extendable-lem}
  For each $F \in {\cal W'}$, $\tilde J(F)$ is extendable and can be written as
  a finite linear combination of $\HC$-valued analytic functions on $\Omega$
  that are homogeneous in $Z_1$ and $Z_2$. Moreover,
  $$
  \DR^+ \tilde J(F) = \DR^- \tilde J(F) = \M F
  \quad \in {\cal Q}^+ \oplus {\cal Q}^0 \oplus {\cal Q}^-.
  $$
\end{lem}

\begin{proof}
The first part follows from Lemma \ref{I(Zh)-extendable} and equation
(\ref{I-J-relation2}). The second part follows from
Theorems \ref{J_R(W^+-)} and \ref{Mx^0-operator-thm}.
\end{proof}

\begin{lem}  \label{deg(1/N)-extendable-lem}
For each $d \in \BB Z$, $(\deg+d)^{-1}_{Z_1} \tilde J(N(Z)^{-1})$ and
$(\deg+d)^{-1}_{Z_2} \tilde J(N(Z)^{-1})$ are extendable.
\end{lem}

\begin{proof}
  The result follows from Lemma \ref{dilogarithm-calculations-lem} and
  equation (\ref{I-J-relation2}).
\end{proof}

\begin{thm}  \label{mult-extendable-spinor-thm}
  Let $F, G \in {\cal W}'$, then
  $\omega(J(F) \ast J(G)) \in \tilde{\cal A}^{\omega}$ is extendable and
  \begin{equation}  \label{rest_in_W-eqn}
  \DR^+ \bigl( \omega \bigl( J(F) \ast J(G) \bigr) \bigr), \quad
  \DR^- \bigl( \omega \bigl( J(F) \ast J(G) \bigr) \bigr) \quad
  \in {\cal Q}^+ \oplus {\cal Q}^0 \oplus {\cal Q}^-.
  \end{equation}
\end{thm}

\begin{proof}
  Recall that ${\cal W}'/\ker J$ has four irreducible components:
  $$
  {\cal Q}'^+, \qquad {\cal Q}'^0, \qquad {\cal Q}'^-
  $$
  and the trivial one-dimensional subrepresentation generated by
  $N(Z)^{-1} \cdot Z$.
  If $F$ or $G \in {\cal Q}'^- \oplus {\cal Q}'^+$, then,
  by Theorem \ref{J_R(W^+-)}, $J(F)$ or $J(G) \in {\cal V} \otimes {\cal V}'$,
  and the result follows from
  Lemmas \ref{conv-prod-lem-spinor}, \ref{Diag_on_VxV'-lem}.

  If $F$ or $G$ is proportional to $N(Z)^{-1} \cdot Z$, then, by
  Lemma \ref{1_A-lem}, $J(F) \ast J(G)$ is proportional to $J(F)$ or $J(G)$,
  and the result follows from Lemma \ref{J(W')-extendable-lem}.

It remains to consider the case $F, G \in {\cal Q}'^0$.
Since ${\cal Q}'^0$ is irreducible and can be generated by either
$N(Z)^{-1}$ or $N(Z)^{-1} \cdot \tilde H_{0,0,0}(Z)$,
Proposition \ref{gen-mult-spinor-prop} and
Lemmas \ref{deg-inverse-comm-rels-spinor},
\ref{extendable-invariant-lem-spinor} and \ref{deg(1/N)-extendable-lem}
imply that $\omega(J(F) \ast J(G))$ is extendable.
We still need to prove property (\ref{rest_in_W-eqn}).
As usual, let $K= U(2) \times U(2)$ and observe that $F, G \in {\cal W}'$ are
$K$-finite. Since all operations involved intertwine the actions of
$\mathfrak{gl}(2,\HC)$, analytic functions
$\DR^+ \bigl( \omega \bigl( J(F) \ast J(G) \bigr) \bigr)$ and
$\DR^- \bigl( \omega \bigl( J(F) \ast J(G) \bigr) \bigr)$
on $\HC^{\times}$ are $K$-finite as well.
By Proposition \ref{W'-K-types}, $(\rho'_2,{\cal W}')$ and its dual
$(\rho_2,{\cal W})$ are admissible $(\mathfrak g, K)$-modules.
Since the elements of ${\cal W}$ are dense in the space of all analytic
functions on $\HC^{\times}$, it follows that
$$
\DR^+ \bigl( \omega \bigl( J(F) \ast J(G) \bigr) \bigr), \quad
\DR^- \bigl( \omega \bigl( J(F) \ast J(G) \bigr) \bigr) \quad \in {\cal W}.
$$
On the other hand, by Proposition \ref{mult-image-prop},
$\DR^+ \bigl( \omega \bigl( J(F) \ast J(G) \bigr) \bigr)$ and
$\DR^- \bigl( \omega \bigl( J(F) \ast J(G) \bigr) \bigr)$
must be in the closure of
${\cal Q}^+ \oplus {\cal Q}^0 \oplus {\cal Q}^-$,
then (\ref{rest_in_W-eqn}) follows.
\end{proof}

Theorem \ref{mult-extendable-spinor-thm} allows us to define two
$\mathfrak{gl}(2,\HC)$-invariant multiplication-like operations
$$
({\cal W}'/\ker J) \otimes ({\cal W}'/\ker J) \to {\cal W}'/\ker\M
$$
as follows.
Recall $\M: {\cal W}' \to {\cal W}$, its image is
${\cal Q}^+ \oplus {\cal Q}^0 \oplus {\cal Q}^-$.
Let
$$
\M^{-1}: {\cal Q}^+ \oplus {\cal Q}^0 \oplus {\cal Q}^-
\to {\cal W}'/\ker\M
$$
be the inverse isomorphism to
$\M: {\cal W}'/\ker\M \to \M({\cal W}') \subset {\cal W}$.

\begin{df}  \label{*-def}
  Let $F, G \in {\cal W}'/\ker J$, define
  \begin{align*}
    F \ast^+ G &=
    \M^{-1} \circ \DR^+ \bigl( \omega \bigl( J(F) \ast J(G) \bigr) \bigr), \\
    F \ast^- G &=
    \M^{-1} \circ \DR^- \bigl( \omega \bigl( J(F) \ast J(G) \bigr) \bigr).
  \end{align*}
\end{df}

\begin{rem}
Note that ${\cal W}'/\ker\M = {\cal W}'/\ker \tilde J$ is the quotient of
${\cal W}'/\ker J$ by the one-dimensional trivial subrepresentation spanned
by $N(Z)^{-1} \cdot Z$.
\end{rem}

\begin{prop}
One can also consider multiplications obtained by taking linear combinations
of $F \ast^+ G$ and $F \ast^- G$.
Thus we obtain a one-parameter family of $\mathfrak{gl}(2,\HC)$-equivariant maps
$$
({\cal W}'/\ker J) \otimes ({\cal W}'/\ker J) \to {\cal W}'/\ker\M.
$$
\end{prop}

If $F_1,\dots,F_n \in {\cal W}'/\ker J$, we expect
$$
\omega( J(F_1) \ast \dots \ast J(F_n)) \quad \in \tilde {\cal A}^{\omega}
$$
to be extendable. Once this is established, we can define
$\mathfrak{gl}(2,\HC)$-invariant $n$-multiplications:

\begin{df}
Let $F_1,\dots,F_n \in {\cal W}'/\ker J$, define
$$
\underbrace
{({\cal W}'/\ker J) \times \dots \times ({\cal W}'/\ker J)}_{\text{$n$ times}}
\to {\cal W}'/\ker\M:
$$
\begin{align*}
F_1 \ast^+ \dots \ast^+ F_n &=
\M^{-1} \circ \DR^+ \bigl( \omega \bigl( J(F_1) \ast \dots \ast J(F_n)
\bigr) \bigr), \\
F_1 \ast^- \dots \ast^- F_n &=
\M^{-1} \circ \DR^- \bigl( \omega \bigl( J(F_1) \ast \dots \ast J(F_n)
\bigr) \bigr).
\end{align*}
\end{df}

\begin{rem}  \label{n-mult-spinor}
  Once again, ${\cal W}'/\ker\M$ is the quotient of ${\cal W}'/\ker J$
  by the one-dimensional trivial subrepresentation. In a future work we intend
  to ``lift'' these operations to genuine multiplications
$$
\underbrace
{({\cal W}'/\ker J) \times \dots \times ({\cal W}'/\ker J)}_{\text{$n$ times}}
\to {\cal W}'/\ker J.
$$
  
  Furthermore, we can also consider $n$-multiplications obtained by taking
  linear combinations of $F_1 \ast^+ \dots \ast^+ F_n$ and
  $F_1 \ast^- \dots \ast^- F_n$.
  Thus we obtain a one-parameter family of $\mathfrak{gl}(2,\HC)$-invariant
  multiplications
  $$
  \underbrace
  {({\cal W}'/\ker J) \otimes \dots \otimes ({\cal W}'/\ker J)}_{\text{$n$ times}}
  \to {\cal W}'/\ker J.
  $$
\end{rem}



We conclude this subsection with some thoughts about lifting operations
$$
\ast^{\pm}: ({\cal W}'/\ker J) \otimes ({\cal W}'/\ker J) \to {\cal W}'/\ker\M
$$
to $\mathfrak{gl}(2,\HC)$-invariant multiplication operations on
${\cal W}'/\ker J$ and properties of these multiplication operations.
Recall that the definition of $\ast^+$ and $\ast^-$ involves a composition of
maps factoring through a subspace of ${\cal W}$ (Definition \ref{*-def}).
Our first observation is that a multiplication operation on ${\cal W}'/\ker J$
cannot be factored through a subspace of ${\cal W}$ because
there does not exist a $\rho_2$-invariant subspace of ${\cal W}$
isomorphic to ${\cal W}'/\ker J$.
Indeed, the only $\rho_2$-invariant subspace of ${\cal W}$ that has the same
 irreducible components as ${\cal W}'/\ker J$ is the subspace of
$(\rho_2, {\cal W})$ generated by
${\cal Q}^+ \oplus {\cal Q}^0 \oplus {\cal Q}^-$
and $N(Z)^{-2} \cdot Z^+$. But this subspace is still not isomorphic
to ${\cal W}'/\ker J$ because of Lemma \ref{nesting-lem}.



Next, we assume that a $\mathfrak{gl}(2,\HC)$-invariant multiplication
operation lifting $\ast^+$ or $\ast^-$ is defined and try to derive some of
its properties. As usual, let $K= U(2) \times U(2)$.
Recall that ${\cal W}'/ \ker J$ has four irreducible components:
$$
{\cal Q}'^+, \qquad {\cal Q}'^0, \qquad {\cal Q}'^-
$$
and the trivial one-dimensional representation generated by
$N(Z)^{-1} \cdot Z$.
The one-dimensional representation appears as a subrepresentation in
${\cal W}'/ \ker J$ and the irreducible component ${\cal Q}'^0$
only as a subquotient.
We decompose ${\cal W}'/ \ker J$ as a direct sum of $K$-invariant subspaces:
$$
{\cal W}'/ \ker J \simeq \BB C \oplus {\cal Q}'^+ \oplus
{\cal Q}'^0 \oplus {\cal Q}'^-,
$$
where $\BB C$ denotes the trivial one-dimensional representation generated by
$N(Z)^{-1} \cdot Z$. We emphasize that this direct sum decomposition is
{\em not} $\mathfrak{gl}(2,\HC)$-invariant, since the subspace
${\cal Q}'^0$ is not $\mathfrak{gl}(2,\HC)$-invariant,
it is only $K$-invariant. Since
$$
{\cal W}'/\ker\M \simeq
{\cal Q}'^+ \oplus {\cal Q}'^0 \oplus {\cal Q}'^-,
$$
we have a direct sum decomposition of $K$-invariant subspaces
\begin{equation}  \label{C+W/kerMx}
{\cal W}'/\ker J \simeq \BB C \oplus ({\cal W}'/\ker\M),
\end{equation}
and we can write elements of ${\cal W}'/\ker J$ as pairs
$$
(c,w), \qquad c \in \BB C,\: w \in {\cal W}'/\ker\M.
$$
For example, by Lemma \ref{1_A-lem}, element $N(Z)^{-1} \cdot Z$ corresponds
to a pair $(24,0) \in \BB C \oplus ({\cal W}'/\ker\M)$.
Then lifting an operation from
$$
\ast^{\pm}: ({\cal W}'/\ker J) \otimes ({\cal W}'/\ker J) \to {\cal W}'/\ker\M
$$
to multiplication
$$
\tilde\ast^{\pm}:
({\cal W}'/\ker J) \otimes ({\cal W}'/\ker J) \to {\cal W}'/\ker J
$$
amounts to specifying a $\BB C$-valued bilinear pairing
$$
c^{\pm}(F_1,F_2), \qquad F_1, F_2 \in {\cal W}'/\ker J
$$
so that the lift is
$$
F_1 \tilde\ast^{\pm} F_2 = \bigl( c^{\pm}(F_1,F_2), F_1 \ast^{\pm} F_2 \bigr)
\quad \in \BB C \oplus ({\cal W}'/\ker\M).
$$
If $\tilde\ast^{\pm}$ is $K$-invariant, then so is $c^{\pm}(F_1,F_2)$.

\begin{lem}
  Suppose that the multiplication operation $\tilde\ast^{\pm}$ on
  ${\cal W}'/\ker J$ is $\mathfrak{gl}(2,\HC)$-invariant.
  Then the bilinear pairing $c^{\pm}(F_1,F_2)$ {\em cannot} be
  $\mathfrak{gl}(2,\HC)$-invariant.
\end{lem}

\begin{proof}
First, we spell out the $\mathfrak{gl}(2,\HC)$-invariance of $\tilde\ast^{\pm}$:
\begin{equation}  \label{*-invariance}
\bigl( \rho'_2(X)F_1 \bigr) \tilde\ast^{\pm} F_2
+ F_1 \tilde\ast^{\pm} \bigl( \rho'_2(X)F_2 \bigr)
= \rho'_2(X)(F_1 \tilde\ast^{\pm} F_2),
\end{equation}
for all $F_1, F_2 \in {\cal W}'/\ker J$ and all $X \in \mathfrak{gl}(2,\HC)$.
Since $\mathfrak{gl}(2,\HC)$ acts on the $\BB C$-component of (\ref{C+W/kerMx})
trivially,
$$
\rho'_2(X)(F_1 \tilde\ast^{\pm} F_2) = \rho'_2(X)(F_1 \ast^{\pm} F_2).
$$
If the bilinear pairing $c^{\pm}(F_1,F_2)$ is $\mathfrak{gl}(2,\HC)$-invariant,
then the $\BB C$-component of the left hand side of (\ref{*-invariance})
is zero.

On the other hand, the image of
$$
\ast^{\pm}: ({\cal W}'/\ker J) \otimes ({\cal W}'/\ker J) \to {\cal W}'/\ker\M
$$
contains ${\cal Q}'^0$.
But ${\cal Q}'^0$ is not a $\mathfrak{gl}(2,\HC)$-invariant subspace of
${\cal W}'/\ker J$ -- there exist an $F \in {\cal Q}'^0$ and
an $X \in \mathfrak{gl}(2,\HC)$ such that
$\rho'_2(X)F$ has non-zero $\BB C$-component.
Hence there exist elements
$F_1, F_2 \in {\cal W}'/\ker J$ and an $X \in \mathfrak{gl}(2,\HC)$ such that
the $\BB C$-component of
$$
\rho'_2(X)(F_1 \tilde\ast^{\pm} F_2) = \rho'_2(X)(F_1 \ast^{\pm} F_2)
\quad \in {\cal W}'/\ker J
$$
is not zero.
Substituting such $F_1$, $F_2$ and $X$ into (\ref{*-invariance})
produces a contradiction.
\end{proof}

Thus, lifting $\ast^+$ (or $\ast ^-$) to a $\mathfrak{gl}(2,\HC)$-invariant
multiplication operation $\tilde\ast^+$ (or $\tilde\ast^-$) on
${\cal W}'/\ker J$ amounts to specifying a certain $K$-invariant bilinear
pairing $c^+$ (or $c^-$) on ${\cal W}'/\ker J$.
But, since this pairing cannot be $\mathfrak{gl}(2,\HC)$-invariant, finding
such a pairing that would make the resulting multiplication operation
$\mathfrak{gl}(2,\HC)$-invariant is not trivial.

\subsection{Properties of Multiplications on Quaternionic Algebras}  \label{last_subsection}

We know that the quaternionic algebra ${\cal W}'/\ker J$ and its scalar
counterpart $\Zh$ are not associative (Example \ref{non-associativity-ex}).
On the other hand, in both cases we have indicated how to define
$\mathfrak{gl}(2,\HC)$-invariant $n$-multiplications
(Proposition \ref{n-mult-scalar} and Remark \ref{n-mult-spinor}).
It is natural to conjecture that these $n$-multiplications satisfy some sort
of relaxed associativity-type relations.
There is a well-known structure of this kind known as an $A_{\infty}$-algebra.
There are several types of $A_{\infty}$-algebras, and, in order to formulate
our conjecture more explicitly, we recall some basic definitions
(see, for example, \cite{K} for details).

An {\em $A_{\infty}$ algebra} over $\BB C$ is a complex vector space $A$ endowed
with maps
$$
\nu_n: A^{\otimes n} = \underbrace{A \otimes \dots \otimes A}_{\text{$n$ times}}
\to A, \qquad n=1,2,3,\dots,
$$
such that, for all $n \ge 1$, we have associativity-type identities of the form
\begin{equation}  \label{a-infinity-relations}
\sum_{a+b+c=n} \nu_{a+1+c} \circ
( \mathbbm 1^{\otimes a} \otimes \nu_b \otimes \mathbbm 1^{\otimes c}) =0
\end{equation}
as maps from $A^{\otimes n}$ to $A$.
For $n=1, 2, 3$, these identities become
\begin{align}
& \nu_1 \circ \nu_1 =0,  \label{1st_id} \\
& \nu_1 \circ \nu_2
+ \nu_2( \nu_1 \otimes \mathbbm 1 + \mathbbm 1 \otimes \nu_1) =0,  \label{2nd_id} \\
&\nu_1 \circ \nu_3 + \nu_2( \nu_2 \otimes \mathbbm 1 + \mathbbm 1 \otimes \nu_2)
+ \nu_3( \nu_1 \otimes \mathbbm 1 \otimes \mathbbm 1
+ \mathbbm 1 \otimes \nu_1 \otimes \mathbbm 1
+ \mathbbm 1 \otimes \mathbbm 1 \otimes \nu_1) =0.  \label{3rd_id}
\end{align}
The first identity (\ref{1st_id}) states that $A$ is a complex,
the second (\ref{2nd_id}) means that $\nu_2: A \otimes A \to A$
is a morphism of complexes.
Neither of these two assertions seem natural for the quaternionic algebras.
However, there is a more general notion of a {\em weak $A_{\infty}$-algebra}
with an additional map
\begin{equation*}
\nu_0 : \BB C \to A
\end{equation*}
such that the identities (\ref{a-infinity-relations}) hold after one includes
the additional terms with $\nu_0$. For example, for $n=1,2$, identities
(\ref{1st_id})-(\ref{2nd_id}) become
\begin{align}
& \nu_1 \circ \nu_1
+ \nu_2(\nu_0 \otimes \mathbbm 1 + \mathbbm 1 \otimes \nu_0) =0,  \label{1st-weak_id} \\
& \nu_1 \circ \nu_2
+ \nu_2( \nu_1 \otimes \mathbbm 1 + \mathbbm 1 \otimes \nu_1)
+ \nu_3( \nu_0 \otimes \mathbbm 1 \otimes \mathbbm 1
+ \mathbbm 1 \otimes \nu_0 \otimes \mathbbm 1
+ \mathbbm 1 \otimes \mathbbm 1 \otimes \nu_0) =0.  \label{2st-weak_id}
\end{align}
Identity (\ref{1st-weak_id}) is satisfied when $\nu_1$ is the identity map
on $A$ and
\begin{equation}  \label{1st-weak_id2}
\nu_2(\nu_0(1),a) + \nu_2(a,\nu_0(1)) =-a, \qquad \forall a \in A.
\end{equation}
Thus we can define multiplication on $A$ by setting
$$
\mu_2 = - \frac12 \nu_2,
$$
then (\ref{1st-weak_id2}) will be satisfied if $\nu_0(1)$ is the
(left and right) unit for the multiplication operation $\mu_2$.
Similarly, choosing appropriate coefficients, we can relate $\nu_n$ with
$n$-multiplications $\mu_n$ for our quaternionic algebras.

We also expect that our $n$-multiplications $\mu_n$ satisfy a certain
cyclic symmetry. Namely, there is a $\mathfrak{gl}(2,\HC)$-invariant
bilinear product $\langle \,\cdot\,,\,\cdot\, \rangle$ such that
\begin{equation}  \label{cyclic-symmetry}
\bigl\langle \mu_n(a_1,\dots,a_n), a_{n+1} \bigr\rangle
= \bigl\langle a_1, \mu_n(a_2,\dots,a_{n+1}) \bigr\rangle,
\qquad \forall a_1,\dots,a_{n+1} \in A.
\end{equation}
Those $A_{\infty}$-algebras that  satisfy the additional property
(\ref{cyclic-symmetry}) are called {\em cyclic $A_{\infty}$-algebras},
in our case {\em weak cyclic $A_{\infty}$-algebras}.

In the case of quaternionic algebras, the cyclic symmetry of
$n$-multiplications should follow from the conjectural identification of
the paired $n$-products as in (\ref{cyclic-symmetry}) with variants of
the $(n+1)$-photon diagrams (Figure \ref{n-photon}).
In fact, our map
$$
J: {\cal W}' \to {\cal A}
= \operatorname{completion \: of} {\cal V} \otimes {\cal V}'
$$
corresponds to the vertices in this diagram,
with spaces ${\cal V}$ and ${\cal V}'$ identified with the solutions of
the massless Dirac equations.
Note that we are considering the chiral case of $n$-photon diagrams;
in the non-chiral case all the diagrams with odd number of vertices cancel
out and yield a zero result.
Also, the Maxwell equations -- which define the classical photon space and
play an important role in the quantum theory -- are crucial to our definition
of algebra of quaternionic functions.
One can ask then, what might be the physical meaning of the associativity-type
identities (\ref{a-infinity-relations})-(\ref{2st-weak_id}).
It turns out that, despite the long history of four-dimensional quantum field
theory, only recently there were discovered certain quadratic relations,
first, in quantum Yang-Mills theory \cite{BCFW}.
These relations were later extended to one-loop multi-photon diagrams in
QED (see \cite{BBBV} and references therein), and they might provide a source
of associativity in quaternionic algebra.
Note that there is also a scalar counterpart of QED, which has an analogous
structure, including the quadratic identities.
This scalar version of QED is expected to match our scalar quaternionic
algebra, including the associativity-type identities.

To make the identification of our quaternionic constructions with physics
more transparent, we can ``translate'' all our definitions related to
quaternionic algebra from the quaternionic space into the Minkowski space
using the Cayley transform, as we did in \cite{FL1}.
In \cite{FL1}, the switch to the Minkowski space via the Cayley transform
was crucial for demonstrating unitarity of the spaces of (left and right)
regular functions.
The unitarity of $(\rho_1, \Zh)$ was shown in Proposition 9 of \cite{FL3}.
The same methods and motivations also apply to the spaces of (left and right)
doubly regular functions and to the underlying space of quaternionic algebra
${\cal W}'/\ker J$ factored by the one dimensional subrepresentation.
On the physics side, the unitarity is fundamental in four-dimensional
quantum field theory, in particular in QED.
Minkowski space realization of our current results also suggests an
interesting problem of finding of physical meaning of our present results
in quaternionic analysis, including the decomposition of spaces ${\cal W}$
and ${\cal W}'$ into irreducible components, the role of the one-dimensional
representation in vacuum polarization and the meaning of the quaternionic
algebras.

Comparing the structures of quaternionic analysis with those of
four-dimensional quantum field theory will be beneficial to both disciplines.
On the one hand, various techniques of calculations and regularizations of
Feynman integrals should apply to different constructions of quaternionic
analysis, including the quaternionic algebras.
On the other hand, our clear conceptual program of quaternionic analysis
developed along the lines of well-established complex analysis and
carried out to a new step in this paper might eventually provide a purely
mathematical foundation of the vast number of scattered calculations,
curious identities and remarkable cancellations in the still mysterious
subject of four-dimensional quantum physics.

\section{Appendix: Comments about \cite{FL1} and \cite{FL3}}  \label{errata-section}

We would like to add some comments about \cite{FL1} and \cite{FL3}
that are relevant to the present article.

\subsection{Comments about \cite{FL1}}

Lemma 17 describing the Lie algebra actions $\pi^0_l$ and $\pi^0_r$ of
$\mathfrak{gl}(2,\BB H)$ on the space of harmonic functions should state
\begin{equation}
\pi_l^0 \bigl(\begin{smallmatrix} 0 & 0 \\ C & 0 \end{smallmatrix}\bigr) =
\pi_r^0 \bigl(\begin{smallmatrix} 0 & 0 \\ C & 0 \end{smallmatrix}\bigr) :
\phi \mapsto \tr \bigl( C \cdot \bigl(
X \cdot (\partial \phi) \cdot X + X\phi \bigr) \bigr)
= \tr \bigl( C \cdot \bigl( X \cdot \partial (X\phi) \bigr) - X\phi \bigr).
\end{equation}

The matrix coefficient expansions from Propositions 25, 26 and 27 have much
larger regions of convergence than stated, the proofs remain the same.
Since we use these expansions so often, we provide more precise statements.

\begin{prop}  \label{Prop25}
We have the following matrix coefficient expansion
\begin{equation}  \label{1/N-expansion}
k_0(Z-W) = \frac 1{N(Z-W)}= N(W)^{-1} \cdot \sum_{l,m,n}
t^l_{n\,\underline{m}}(Z) \cdot t^l_{m\,\underline{n}}(W^{-1}),
\end{equation}
which converges uniformly on compact subsets in the region
$\{ (Z,W) \in \HC \times \HC^{\times}; \: ZW^{-1} \in \BB D^+ \}$.
The sum is taken first over all $m,n = -l, -l+1, \dots, l$, then over
$l=0,\frac 12, 1, \frac 32,\dots$.
\end{prop}

\begin{prop}  \label{Prop26}
We have the following matrix coefficient expansions
$$
k(Z-W) = \frac{(Z-W)^{-1}}{N(Z-W)}
= \frac1{N(Z)} \sum_{l,m,n} \left(\begin{smallmatrix}
(l-m+ \frac 12) t^l_{n \, \underline{m+ \frac 12}}(W)  \\
(l+m+ \frac 12) t^l_{n \, \underline{m- \frac 12}}(W) \end{smallmatrix}\right)
\cdot
\left(\begin{smallmatrix}
t^{l+\frac 12}_{m \, \underline{n- \frac 12}}(Z^{-1}), &
t^{l+\frac 12}_{m \, \underline{n+ \frac 12}}(Z^{-1})
\end{smallmatrix}\right),
$$
which converges uniformly on compact subsets in the region
$\{ (Z,W) \in \HC^{\times} \times \HC; \: WZ^{-1} \in \BB D^+ \}$.
The sum is taken first over all
$m =-l-\frac 12 ,-l+\frac 32,\dots,l+\frac 12$ and $n =-l,-l+1,\dots,l$,
then over $l=0,\frac 12, 1, \frac 32,\dots$.
Similarly,
$$
k(Z-W) = \frac{(Z-W)^{-1}}{N(Z-W)} = - \sum_{l,m,n} \frac1{N(W)}
\left(\begin{smallmatrix} (l-m+ \frac 12) t^l_{m- \frac 12 \, \underline{n}}(W^{-1}) \\
(l+m+ \frac 12) t^l_{m+ \frac 12 \, \underline{n}}(W^{-1}) \end{smallmatrix}\right)
\cdot
\left(\begin{smallmatrix} t^{l-\frac 12}_{n + \frac 12 \, \underline{m}}(Z), &
t^{l-\frac 12}_{n - \frac 12 \, \underline{m}}(Z) \end{smallmatrix}\right),
$$
which converges uniformly on compact subsets in the region
$\{ (Z,W) \in \HC \times \HC^{\times}; \: ZW^{-1} \in \BB D^+ \}$.
The sum is taken first over all
$m =-l+\frac 12 ,-l+\frac 32,\dots,l-\frac 12$ and $n =-l,-l+1,\dots,l$,
then over $l =\frac 12, 1, \frac 32, 2, \dots$.
\end{prop}

\begin{prop}  \label{Prop27}
We have the following matrix coefficient expansions
$$
\frac 1{N(Z-W)^2} = \sum_{k,l,m,n} 
(2l+1) N(Z)^k \cdot t^l_{n \, \underline{m}}(Z) \cdot
N(W)^{-k-2} \cdot t^l_{m \, \underline{n}}(W^{-1}),
$$
which converges uniformly on compact subsets in the region
$\{ (Z,W) \in \HC \times \HC^{\times}; \: ZW^{-1} \in \BB D^+ \}$.
The sum is taken first over all $m,n = -l, -l+1, \dots, l$, then over
$k=0,1,2,3,\dots$ and $l=0,\frac 12, 1, \frac 32,\dots$.
\end{prop}

The representation $(\rho_2,{\cal W}^+)$ introduced at the beginning of
Subsection 4.2 is not irreducible.
In fact, it is easy to see from Subsection \ref{W-decomp-subsect} that
$(\rho_2,{\cal W}^+)$ has two irreducible components:
$(\rho_2,{\cal Q}^+)$ and $(\rho,\Sh^+)$.

There are several sign errors in Subsection 4.3.
In particular, $\tilde A$ should be defined as
$$
\tilde A = -A_0 \tilde e_0 + A_1e_1 + A_2e_2 + A_3e_3
= - \tilde e_0 A_0 + \overrightarrow{A}
$$
(negative of the original $\tilde A$).
Either way, the main conclusion still holds.
Namely, that $\M \tilde A=0$ if and only if the Maxwell equations
(expressed as equation (56) in \cite{FL1}) are satisfied.


The main purpose of Subsection 5.1 was to describe the decomposition of the
tensor product representation
$(\pi^0_l, {\cal H}^+) \otimes (\pi^0_r, {\cal H}^+)$ of $\mathfrak{gl}(2,\HC)$
into irreducible components due to \cite{JV2}.
Unfortunately, the representations $(\rho_n, \Zh^+_n)$ of $\mathfrak{gl}(2,\HC)$
are not irreducible for $n \ge 2$.
Indeed, we saw in Section \ref{decomp-section} that
$(\rho_2, \Zh^+_2) = (\rho_2,{\cal W}^+)$ is not irreducible.
Thus, Theorem 82 in \cite{FL1} can be corrected as

\begin{thm}
The image of the intertwining map $M_n$ from Theorem 85 in \cite{FL1} is an
irreducible subrepresentation of $(\rho_n,\Zh_n^+)$, $n=1,2,3,\dots$.

Let us denote this image by $(\Zh_n^+)_{irr}$.
The irreducible representations $\bigl( \rho_n, (\Zh_n^+)_{irr} \bigr)$,
$n=1,2,3,\dots$, of $\mathfrak{sl}(2,\HC)$ are pairwise non-isomorphic
and possess inner products which make them unitary representations of
the real form $\mathfrak{su}(2,2)$ of $\mathfrak{sl}(2,\HC)$.
\end{thm}

When $n=1$, we have $(\Zh_1^+)_{irr} = \Zh_1^+ = \Zh^+$.
Then equation (61) in \cite{FL1} should read as follows.
\begin{equation}  \label{tensor-decomp}
(\pi^0_l, {\cal H}^+) \otimes (\pi^0_r, {\cal H}^+) \simeq
\bigoplus_{n=1}^{\infty} \bigl( \rho_n,(\Zh_n^+)_{irr} \bigr),
\end{equation}
This decomposition is obtained by treating ${\cal H}^+ \otimes {\cal H}^+$
as functions of two variables $Z, Z' \in \HC$ and filtering them by the
degree of vanishing on the diagonal $\HC \subset \HC \times \HC$. Then
$$
(\rho_1,\Zh^+) \quad \text{generated by} \quad 1 \otimes 1,
$$
$$
\bigl( \rho_n, (\Zh_n^+)_{irr} \bigr) \quad \text{generated by} \quad
(z_{ij}-z'_{ij})^{n-1} , \qquad n \ge 2.
$$


Subsection 5.3 was written so it could be later used to give a proof of
the ``magic identities'' for the conformal four-point integrals described by
the box diagrams.
Magic identities are proved in \cite{L2} using different methods.

\subsection{Comments about \cite{FL3}}

In the expression above Theorem 15
$$
\bigl( (I_R^{+-}+ I_R^{-+}) N(W)^{-1} \bigr)(Z_1,Z_2)
= -\frac1{N(Z_2)} \cdot
\begin{cases} \frac{\log\lambda_2-\log\lambda_1}{\lambda_2-\lambda_1} &
\text{if $\lambda_1 \ne \lambda_2$;} \\
\lambda^{-1} & \text{if $\lambda_1 = \lambda_2 = \lambda$,} \end{cases}
$$
$\log$ denotes the branch of logarithm with a cut along
the {\em positive} real axis and $\lambda \ne 1$.
Thus when we let $Z_1, Z_2 \to Z$ we need the eigenvalues
$\lambda_1$, $\lambda_2$ of $Z_1 Z_2^{-1}$ to stay on the same side of the cut:
$$
\lim_{\genfrac{}{}{0pt}{}{Z_1, Z_2 \to Z,\: N(Z_1-Z_2) \ne 0}
{\sgn (\im \lambda_1) = \sgn (\im \lambda_2)}}
\bigl( (I_R^{+-}+ I_R^{-+}) N(W)^{-1} \bigr)(Z_1,Z_2)
= - N(Z)^{-1}, \qquad Z \in U(2)_R.
$$
Recall that $\widetilde{{\cal H} \otimes {\cal H}}$ denotes the space of
holomorphic $\BB C$-valued functions in two variables $Z_1,Z_2 \in \HC$
(possibly with singularities) that are harmonic in each variable separately.
Then Theorem 15 should be restated as

\begin{thm}  \label{Thm15}
The $\mathfrak{gl}(2,\HC)$-equivariant map
$$
f \mapsto \bigl((I_R^{+-}+ I_R^{-+})f\bigr)(Z_1,Z_2)
\quad \in \widetilde{{\cal H} \otimes {\cal H}}, \qquad f \in \Zh,
$$
where $Z_1,Z_2 \in U(2)_R$, $N(Z_1-Z_2) \ne 0$,
is well-defined and annihilates $\Zh^- \oplus \Zh^+$.

Moreover, we have a well defined operator $\P^0$ on $\Zh$
$$
f \mapsto (\P^0 f)(Z) =
\lim_{\genfrac{}{}{0pt}{}{Z_1, Z_2 \to Z,\: N(Z_1-Z_2) \ne 0}
{\sgn (\im \lambda_1) = \sgn (\im \lambda_2)}}
- \bigl( (I_R^{+-}+ I_R^{-+})f \bigr)(Z_1,Z_2), \qquad Z \in U(2)_R,
$$
which annihilates $\Zh^- \oplus \Zh^+$ and is the identity mapping on $\Zh^0$.

Furthermore, the projector $\P^0$ on $\Zh$ can be computed as follows:
\begin{multline*}
(\P^0 f)(Z) = \frac1{2\pi^3i} \lim_{\theta \to 0} \lim_{s \to 1} \biggl(
\int_{W \in U(2)_R} \frac{f(W)\,dV}{N(W-se^{i\theta}Z) \cdot
N(W-s^{-1}e^{-i\theta}Z)}\\
+ \int_{W \in U(2)_R} \frac{f(W)\,dV}{N(W-s^{-1}e^{i\theta}Z) \cdot
N(W-se^{-i\theta}Z)}
\biggr), \qquad Z \in U(2)_R.
\end{multline*}
\end{thm}

The fact that $\P^0$ is a projector onto $\Zh^0$ may be interpreted as
$$
M \circ \bigl((I_R^{+-}+ I_R^{-+})f\bigr) = f \quad \text{if $f \in \Zh^0$}.
$$

Theorem 22 should be restated in a similar manner. In particular,
the operator $\P^0_{\BB M}$ on $\rho_1(\Zh)$ should be defined as
$$
(\P^0_{\BB M} f)(Z) =
\lim_{\genfrac{}{}{0pt}{}{Z_1, Z_2 \to Z,\: N(Z_1-Z_2) \ne 0}
{\sgn (\im \lambda_1) = \sgn (\im \lambda_2)}}
- \bigl( (I_{\BB M}^{+-}+ I_{\BB M}^{-+})f \bigr)(Z_1,Z_2), \qquad Z \in \BB M,
$$
where $\lambda_1$ and $\lambda_2$ denote the eigenvalues of
$(Z_1+1)(Z_1-1)^{-1}(Z_2-1)(Z_2+1)^{-1}$.

%
%


\noindent
{\em Department of Mathematics, Yale University,
P.O. Box 208283, New Haven, CT 06520-8283}\\
{\em Department of Mathematics, Indiana University,
Rawles Hall, 831 East 3rd St, Bloomington, IN 47405}   


\begin{thebibliography}{[BCFW]}
\bibitem[BBBV]{BBBV} S.~Badger, N.~Bjerrum-Bohr, P.~Vanhove,
{\em Simplicity in the structure of QED and gravity amplitudes},
JHEP {\bf 02} (2009) 038.
\bibitem[BCFW]{BCFW} R.~Britto, F.~Cachazo, B.~Feng, E.~Witten,
{\em Direct Proof of the Tree-Level Scattering Amplitude Recursion Relation
in Yang-Mills Theory}, Phys. Rev. Lett. {\bf 94} (18), 181602 (2005).
\bibitem[ES]{ES} M.~Eastwood, M.~Singer, {\em A conformally invariant
Maxwell gauge}, Phys. Lett. A {\bf 107} (1985), 73-74.
\bibitem[Er]{Ba} A.~Erd\'elyi et al,
{\em Higher Transcendental Functions.
Based, in part, on notes left by Harry Bateman}, Vol. I,
McGraw-Hill, New York-Toronto-London, 1953.
\bibitem[FL1]{FL1} I.~Frenkel, M.~Libine,
{\em Quaternionic analysis, representation theory and physics},
Advances in Math {\bf 218} (2008), 1806-1877; also arXiv:0711.2699.
\bibitem[FL2]{FL2} I.~Frenkel, M.~Libine,
{\em Split quaternionic analysis and the separation of the series for
$SL(2,\mathbb R)$ and $SL(2,\mathbb C)/SL(2,\mathbb R)$},
Advances in Math {\bf 228} (2011), 678-763; also arXiv:1009.2532.
\bibitem[FL3]{FL3} I.~Frenkel, M.~Libine,
{\em Anti de Sitter deformation
of quaternionic analysis and the second-order pole},
IMRN, 2015 (2015), 4840-4900; also arXiv:1404.7098.
\bibitem[FL4]{FL4} I.~Frenkel, M.~Libine,
{\em $n$-regular functions in quaternionic analysis}, submitted.
\bibitem[JV1]{JV1} H.~P.~Jakobsen, M.~Vergne,
{\em Wave and Dirac operators, and representations of the conformal group},
J. Functional Analysis {\bf 24} (1977), no. 1, 52-106.
\bibitem[JV2]{JV2} H.~P.~Jakobsen, M.~Vergne,
{\em Restrictions and expansions of holomorphic representations},
J. Funct. Anal. {\bf 34} (1979), no. 1, 29-53.
\bibitem[K]{K} B.~Keller, {\em $A$-infinity algebras, modules and functor
categories} in Trends in Representation Theory of Algebras and Related topics,
67-93, Contemp. Math. {\bf 406}, Amer. Math. Soc., Providence, RI, 2006. 
\bibitem[Le]{Le} S.~T.~Lee, {\em On some degenerate principal series
representations of $U(n,n)$}, J. Funct. Anal. {\bf 126} (1994), 305-366.
\bibitem[L1]{L} M.~Libine, {\em The two-loop ladder diagram and representations
of $U(2,2)$}, Jour. of Lie Theory {\bf 27} (2017), 771-800;
also arXiv:1309.5665.
\bibitem[L2]{L2} M.~Libine, {\em The conformal four-point integrals,
magic identities and representations of $U(2,2)$},
Advances in Math {\bf 301} (2016), 289-321; also arXiv:1407.2507.
\bibitem[P]{P} S.~Paneitz, {\em Analysis in space-time bundles, III.
Higher spin bundles}, J. Functional Analysis {\bf 54} (1983), 18-112.
\bibitem[V]{V} N.~Ja.~Vilenkin, {\em Special functions and the theory of
group representations}, translated from the Russian by V.~N.~Singh,
Translations of Mathematical Monographs, Vol. 22
American Mathematical Society, Providence, RI 1968.
\bibitem[Z]{Z} D.~Zagier, {\em The Dilogarithm Function}.
  In: P.~Cartier, P.~Moussa, B.~Julia, P.~Vanhove (eds)
  Frontiers in Number Theory, Physics, and Geometry II.
  Springer, Berlin, Heidelberg 2007.
\end{thebibliography}
\end{document}